\documentclass[11pt,a4paper,reqno]{article}

\usepackage{cite}
\usepackage{amsmath}
\usepackage{amssymb}
\usepackage[colorlinks]{hyperref}
\usepackage{amsfonts}
\usepackage{mathrsfs}
\usepackage{mathtools}
\usepackage{theoremref}
\usepackage{amsthm}
\usepackage[dvipsnames]{xcolor}
\usepackage[utf8]{inputenc}
\usepackage{array}
\usepackage{enumitem}
\usepackage[english]{babel}
\usepackage{comment}
\usepackage{geometry}
\usepackage{multicol}

\usepackage{stmaryrd}

\usepackage{hhline}
\usepackage{multirow}

\usepackage[margin=10pt,font=small,labelfont=bf, labelsep=period]{caption}

\usepackage{tikz}
\usetikzlibrary{shapes.geometric}
\usetikzlibrary {arrows.meta}

\makeatletter
\def\blfootnote{\gdef\@thefnmark{}\@footnotetext}
\makeatother

\numberwithin{equation}{section}

\theoremstyle{plain}

\newtheorem{thm}[equation]{Theorem}
\newtheorem{lemma}[equation]{Lemma}
\newtheorem{cor}[equation]{Corollary}
\newtheorem{prop}[equation]{Proposition}

\newtheorem{claim}{Claim}

\theoremstyle{definition}

\newtheorem{defn}[equation]{Definition}
\newtheorem{rem}[equation]{Remark}

\newtheorem{prob}[equation]{Problem}

\newcommand\nc\newcommand
\nc\rnc\renewcommand

\nc\mr\mathrel
\rnc\implies{\ \Rightarrow\ }
\rnc\iff{\ \Leftrightarrow\ }
\nc\sub\subseteq
\nc\mt\mapsto
\nc\bit{\begin{nitemize}}
\nc\eit{\end{nitemize}}

\nc\GP{G^\supPG}
\nc\GI{G^{\sfE}}

\nc\set[2]{\{#1\colon#2\}}
\nc\bigset[2]{\big\{#1\colon#2\big\}}
\nc\ol\overline
\nc\la\langle
\nc\ra\rangle
\nc\Ga\Gamma
\nc\es\varnothing

\nc\sgap{\;\!\;\!}

\nc\AND{\qquad\text{and}\qquad}
\nc\ANd{\quad\text{and}\quad}
\nc\COMMA{,\qquad}
\nc\COMMa{,\quad}

\newcommand{\R}{\mathscr{R}}
\renewcommand{\L}{\mathscr{L}}
\renewcommand{\H}{\mathscr{H}}
\newcommand{\J}{\mathscr{J}}
\newcommand{\D}{\mathscr{D}}

\newcommand{\F}{\mathscr{F}}

\renewcommand{\P}{\mathcal{P}}
\newcommand{\T}{\mathcal{T}}
\newcommand{\B}{\mathcal{B}}
\newcommand\TL{\mathcal T\!\mathcal L}
\renewcommand{\S}{\mathcal{S}}

\newcommand{\Ptw}{\P^\Phi}

\newcommand{\rels}{\mathcal{R}}

\newcommand{\Z}{\mathbb{Z}}

\DeclareMathOperator{\PG}{\textup{\textsf{PG}}}
\DeclareMathOperator{\IG}{\textup{\textsf{IG}}}
\DeclareMathOperator{\RIG}{\textup{\textsf{RIG}}}

\DeclareMathOperator{\sfE}{\textup{\textsf{E}}}
\DeclareMathOperator{\sfP}{\textup{\textsf{P}}}
\DeclareMathOperator{\sfI}{\textup{\textsf{I}}}
\DeclareMathOperator{\sfR}{\textup{\textsf{R}}}

\DeclareMathOperator{\sfG}{\textup{\textsf{G}}}
\DeclareMathOperator{\sfA}{\textup{\textsf{A}}}

\rnc\SS{\operatorname{\textup{\textsf{S}}}}

\newcommand{\Eq}{\textup{\textsf{Eq}}}

\DeclareMathOperator{\GH}{\textup{\textsf{GH}}}

\DeclareMathOperator{\rank}{\textup{\textsf{rank}}}
\let\ker\relax
\DeclareMathOperator{\ker}{\textup{\textsf{ker}}}

\DeclareMathOperator{\coker}{\textup{\textsf{coker}}}
\DeclareMathOperator{\dom}{\textup{\textsf{dom}}}
\DeclareMathOperator{\codom}{\textup{\textsf{codom}}}
\DeclareMathOperator{\KER}{\textup{\textsf{KER}}}
\DeclareMathOperator{\id}{\textup{\textsf{id}}}
\DeclareMathOperator{\NT}{\textup{\textsf{NT}}}
\DeclareMathOperator{\NTu}{\textup{\textsf{U}}}
\DeclareMathOperator{\NTd}{\textup{\textsf{L}}}

\newcommand{\mapj}{\textup{\textsf{j}}}
\newcommand{\mapq}{\textup{\textsf{q}}}
\newcommand{\sfp}{\textup{\textsf{p}}}

\newcommand{\mapC}{\textup{\textsf{C}}}

\newcommand{\mapV}{\textup{\textsf{V}}}
\newcommand{\subSq}{\textsf{sq}}
\newcommand{\subinv}{\textsf{inv}}
\newcommand{\supPG}{\textsf{P}}

\newcommand{\supRIG}{\textsf{R}}

\DeclareMathOperator{\presn}{\Pi}
\newcommand{\lam}{\lambda}
\newcommand{\lamp}{\lambda'}

\newcommand{\lex}{\textup{lex}}
\newcommand{\fd}{\textup{fd}}
\newcommand{\fc}{\textup{fc}}

\newcommand{\Lk}{\mathfrak{L}}
\newcommand{\Trs}{\mathfrak{T}}
\newcommand{\Bp}{\mathfrak{B}}
\newcommand{\Sq}{\mathfrak{S}}

\newcommand{\Esub}{F}

\newcommand{\Tlex}{T_\lex}
\newcommand{\Tfd}{T_\fd}
\newcommand{\Tfc}{T_\fc}

\newcommand{\labU}{O}

\newcommand{\ED}{D}
\newcommand{\ER}{R}

\newcommand{\Ghat}{\widehat{G}}

\newcommand{\fg}[1]{\llbracket #1\rrbracket}

\newcounter{ncols}
\newcounter{incols}
\newenvironment{partn}[1]{
  \setcounter{ncols}{#1} \setcounter{incols}{\thencols - 1}\setlength{\arraycolsep}{1pt}
  \Bigl( \hspace{-1.5truemm}\scriptsize 
    \begin{array}{@{\hskip 3pt} c *{\theincols}{|c} @{\hskip 3pt}  }
}{
     \end{array}
     \normalsize \hspace{-1.5truemm}\Bigr)\setlength{\arraycolsep}{6pt}
}

\newcommand{\smat}[4]{\bigl(\begin{smallmatrix} #1 & #2 \\ #3 & #4 \end{smallmatrix}\bigr)}
\newcommand{\lmat}[4]{\begin{pmatrix} #1 & #2 \\ #3 & #4 \end{pmatrix}}

\newcommand{\trans}[2]{\left(\begin{smallmatrix} #1 \\ #2 \end{smallmatrix}\right)}

\nc\uv[1]{\fill (#1,1.5)circle(.15);}
\nc\uvs[1]{{\foreach \x in {#1}{\uv{\x}}}}
\nc\lv[1]{\fill (#1,0)circle(.15);}
\nc\lvs[1]{{\foreach \x in {#1}{\lv{\x}}}}
\nc\darcx[3]{\draw(#1,0)arc(180:90:#3) (#1+#3,#3)--(#2-#3,#3) (#2-#3,#3) arc(90:0:#3);}
\nc\uarcx[3]{\draw(#1,1.5)arc(180:270:#3) (#1+#3,1.5-#3)--(#2-#3,1.5-#3) (#2-#3,1.5-#3) arc(270:360:#3);}
\nc\darc[2]{\darcx{#1}{#2}{.4}}
\nc\uarc[2]{\uarcx{#1}{#2}{.4}}

\nc\udline[2]{\draw(#1,1.5)--(#2,0);}
\nc\uuline[2]{\draw(#1,1.5)--(#2,1.5);}
\nc\ddline[2]{\draw(#1,0)--(#2,0);}

\newcommand{\olduv}[1]{\fill (#1,2)circle(.17);}
\newcommand{\oldlv}[1]{\fill (#1,0)circle(.17);}
\newcommand{\olduvs}[1]{{\foreach \x in {#1} { \olduv{\x}}}}
\newcommand{\oldlvs}[1]{{\foreach \x in {#1} { \oldlv{\x}}}}
\newcommand{\olddarcx}[3]{\draw(#1,0)arc(180:90:#3) (#1+#3,#3)--(#2-#3,#3) (#2-#3,#3) arc(90:0:#3);}
\newcommand{\olddarc}[2]{\olddarcx{#1}{#2}{.4}}
\newcommand{\olduarcx}[3]{\draw(#1,2)arc(180:270:#3) (#1+#3,2-#3)--(#2-#3,2-#3) (#2-#3,2-#3) arc(270:360:#3);}
\newcommand{\olduarc}[2]{\olduarcx{#1}{#2}{.4}}
\newcommand{\oldstline}[2]{\draw(#1,2)--(#2,0);}

\nc\ediagram[3]{
\begin{scope}[shift = {(0,4)}]
\node[above left] () at (1,1.5) {$#1$};
\uvs{1,...,#2}
\lvs{1,...,#2}
#3
\node (eL) at (1,1.5/2) {\phantom{$E$}};
\node (eR) at (#2,1.5/2) {\phantom{$E$}};
\node (eU) at (.5+#2/2,1.5) {\phantom{\Large$E$}};
\node (eD) at (.5+#2/2,0) {\phantom{\Large$E$}};
\end{scope}
}

\nc\fdiagram[3]{
\begin{scope}[shift = {(#2+2,4)}]
\node[above left] () at (1,1.5) {$#1$};
\uvs{1,...,#2}
\lvs{1,...,#2}
#3
\node (fL) at (1,1.5/2) {\phantom{$E$}};
\node (fR) at (#2,1.5/2) {\phantom{$E$}};
\node (fU) at (.5+#2/2,1.5) {\phantom{\Large$E$}};
\node (fD) at (.5+#2/2,0) {\phantom{\Large$E$}};
\end{scope}
}

\nc\gdiagram[3]{
\begin{scope}[shift = {(0,0)}]
\node[above left] () at (1,1.5) {$#1$};
\uvs{1,...,#2}
\lvs{1,...,#2}
#3
\node (gL) at (1,1.5/2) {\phantom{$E$}};
\node (gR) at (#2,1.5/2) {\phantom{$E$}};
\node (gU) at (.5+#2/2,1.5) {\phantom{\Large$E$}};
\node (gD) at (.5+#2/2,0) {\phantom{\Large$E$}};
\end{scope}
}

\nc\hdiagram[3]{
\begin{scope}[shift = {(#2+2,0)}]
\node[above left] () at (1,1.5) {$#1$};
\uvs{1,...,#2}
\lvs{1,...,#2}
#3
\node (hL) at (1,1.5/2) {\phantom{$E$}};
\node (hR) at (#2,1.5/2) {\phantom{$E$}};
\node (hU) at (.5+#2/2,1.5) {\phantom{\Large$E$}};
\node (hD) at (.5+#2/2,0) {\phantom{\Large$E$}};
\end{scope}
}

\nc\udiagram[3]{
\begin{scope}[shift = {(#2/2+1,8)}]
\node[above left] () at (1,1.5) {$#1$};
\uvs{1,...,#2}
\lvs{1,...,#2}
#3
\end{scope}
}

\nc\diagramshading[1]{
\fill[lightgray!50](-.5,-1.5)--(2*#1+3.5,-1.5)--(2*#1+3.5,7)--(-.5,7)--(-.5,-1.5);
}

\nc\LRarrows{
\draw[->] (eL) edge [loop left] ();
\draw[->] (gL) edge [loop left] ();
\draw[->] (eR)--(fL);
\draw[->] (gR)--(hL);
}

\nc\RLarrows{
\draw[->] (fR) edge [loop right] ();
\draw[->] (hR) edge [loop right] ();
\draw[<-] (eR)--(fL);
\draw[<-] (gR)--(hL);
}

\nc\UDarrows{
\draw[->] (eU) edge [loop above] ();
\draw[->] (fU) edge [loop above] ();
\draw[->] (eD)--(gU);
\draw[->] (fD)--(hU);
}

\nc\DUarrows{
\draw[->] (gD) edge [loop below] ();
\draw[->] (hD) edge [loop below] ();
\draw[<-] (eD)--(gU);
\draw[<-] (fD)--(hU);
}

\nc\partitionedges[5]{
\foreach \x/\y in {#1} {\udline\x\y}
\foreach \x/\y in {#2} {\uuline\x\y}
\foreach \x/\y in {#3} {\uarc\x\y}
\foreach \x/\y in {#4} {\ddline\x\y}
\foreach \x/\y in {#5} {\darc\x\y}
}

\nc\bluetrans[6]{
\fill[blue!20](#1,1.5)--(#2,1.5)--(#4,0)--(#3,0);
\draw[|-|](#1,1.75)--(#2,1.75); \node () at (#1/2+#2/2,2) {\footnotesize$#5$};
\draw[|-|](#3,-.25)--(#4,-.25); \node () at (#3/2+#4/2,-.5) {\footnotesize$#6$};
}

\nc\brokenbluetrans[6]{
\blueupper{#1}{#2}{#5}
\bluelower{#3}{#4}{#6}
}

\nc\blueupper[3]{
\fill[blue!20](#1,1.25)--(#2,1.25)--(#2,1.5)--(#1,1.5);
\draw[|-|](#1,1.75)--(#2,1.75); \node () at (#1/2+#2/2,2) {\footnotesize$#3$};
}

\nc\bluelower[3]{
\fill[blue!20](#1,.25)--(#2,.25)--(#2,0)--(#1,0);
\draw[|-|](#1,-.25)--(#2,-.25); \node () at (#1/2+#2/2,-.5) {\footnotesize$#3$};
}

\nc\bluepartlabel[1]{\node () at (0.4,2.1) {\large$#1$};}

\nc\Ebluetrans[4]{\fill[blue!20](#1,1.5)--(#2,1.5)--(#4,0)--(#3,0);}
\nc\Eblueupper[2]{\fill[blue!20](#1,1.25)--(#2,1.25)--(#2,1.5)--(#1,1.5);}
\nc\Ebluelower[2]{\fill[blue!20](#1,.25)--(#2,.25)--(#2,0)--(#1,0);}
\nc\Ebluepartlabel[1]{\node[left] () at (1,1.9) {\large$#1$};}
\nc\Eredtrans[4]{\fill[red!20](#1,1.5)--(#2,1.5)--(#4,0)--(#3,0);}
\nc\Eredupper[2]{\fill[red!20](#1,1.25)--(#2,1.25)--(#2,1.5)--(#1,1.5);}
\nc\Eredlower[2]{\fill[red!20](#1,.25)--(#2,.25)--(#2,0)--(#1,0);}

\newenvironment{thmenumerate}{\begin{enumerate}[label=\textup{(\roman*)},leftmargin=10mm]}{\end{enumerate}}
\newenvironment{nitemize}{\begin{itemize}[label=\textbullet, leftmargin=5mm]}{\end{itemize}}

\begin{document}

\title{Maximal subgroups of free projection- and idempotent-generated semigroups with applications to partition monoids}

\date{}
\author{}

\maketitle

\vspace{-15mm}

\begin{center}
{\large 
James East,%
\hspace{-.25em}\footnote{\label{fn:JE}Centre for Research in Mathematics and Data Science, Western Sydney University, Locked Bag 1797, Penrith NSW 2751, Australia. {\it Emails:} {\tt J.East@westernsydney.edu.au}, {\tt A.ParayilAjmal@westernsydney.edu.au}.}
Robert D.~Gray,%
\hspace{-.25em}\footnote{School of Engineering, Mathematics and Physics, University of East Anglia, Norwich NR4 7TJ, England, UK. {\it Email:} {\tt Robert.D.Gray@uea.ac.uk}.}
P.A.~Azeef Muhammed,%
\hspace{-.25em}\textsuperscript{\ref{fn:JE}}
Nik Ru\v{s}kuc%
\footnote{Mathematical Institute, School of Mathematics and Statistics, University of St Andrews, St Andrews, Fife KY16 9SS, UK. {\it Email:} {\tt Nik.Ruskuc@st-andrews.ac.uk}}
}
\end{center}

\maketitle

\begin{abstract}
This paper investigates the maximal subgroups of a free projection-generated regular $*$-semigroup $\PG(P)$ over a projection algebra $P$, and their relationship to the maximal subgroups of the free idempotent-generated semigroup $\IG(E)$ over the corresponding biordered set $E = \sfE(P)$.  
In the first part of the paper we obtain a number of general presentations by generators and defining relations, in each case reflecting salient combinatorial/topological properties of the groups.  In the second part we apply these to explicitly compute the groups when $P = \sfP(\P_n)$ and $E = \sfE(\P_n)$ arise from the partition monoid $\P_n$.
Specifically, we show that the maximal subgroup of $\PG(\sfP(\P_n))$ corresponding to a projection of rank $r\leq n-2$ is (isomorphic to) the symmetric group $\S_r$.  In $\IG(\sfE(\P_n))$, the corresponding subgroup is the direct product~$\Z\times\S_r$.
The appearance of the infinite cyclic group $\Z$ is explained by a connection to a certain twisted partition monoid $\Ptw_n$, which has the same biordered set as~$\P_n$.

\medskip

\medskip
%


\end{abstract}

\tableofcontents

\section{Introduction}
\label{sec:intro}

Among the properties of elements vis-\`{a}-vis algebraic operations, idempotency is certainly one of the most ubiquitous ones.
As it refers to a single operation, semigroup theory is a natural context in which to study this property in full generality.
The most general theoretical concept for capturing the idempotent structure of a semigroup is that of a biordered set.
To every biordered set $E$ one can associate a free semigroup object $\IG(E)$, which is the free-est semigroup with the biordered set $E$ -- intuitively, $\IG(E)$ possesses only the properties that are necessarily implied by $E$.
Since it is well known that idempotents in any semigroup are in one-one correspondence with maximal subgroups, the problem of studying maximal subgroups of $\IG(E)$ naturally comes to the fore. 

Regular $*$-semigroups are a class of semigroups with involution, which prominently include various diagram monoids.
In such a semigroup one can identify certain special idempotents, called projections, which typically have simpler structure than generic idempotents, but are still sufficient to determine the entire biordered set.
The framework for describing the structure formed by these elements is that of a projection algebra. To every such algebra $P$ one can again associate a free object $\PG(P)$, the free-est regular $*$-semigroup with projections $P$.
And again, analogous to the case of $\IG(E)$, the question of their maximal subgroups arises naturally.

The purpose of this paper is to study maximal subgroups of $\PG(P)$ and compare them with those of $\IG(E)$.
This will first be done on a general level, by establishing a number of presentations (by generators and defining relations) for the maximal subgroups of $\PG(P)$, and comparing them to the known presentations for the maximal subgroups of $\IG(E)$. In particular, presentations will be obtained that reflect the fact that the former are homomorphic images of the latter. Then we will apply this theory to compute the maximal subgroups in the case where the biordered set $E$ and projection algebra $P$ arise from the partition monoid.
The determination of the maximal subgroups of $\IG(E)$ in this case has been an important open problem for some time.

There are many examples of interactions between semigroup theory and other areas of mathematics revolving around idempotents.
Some examples worth mentioning are: the work on the idempotents of  Stone--\v{C}ech compactifications \cite{Ze14,Ze17}, idempotent-generated Banach/C*-algebras \cite{Ha69,Ha50,Sh23,BS10,KR96,KS00,Sp94},
idempotent tropical matrices \cite{JK14}, the fundamental role played by idempotents in Putcha--Renner theory of reductive algebraic monoids
\cite{Pu88a,Pu88,Pu06,Re05}, 
 as well as the intrinsic importance of idempotents in the study of rings and representation theory.

Of course, the idempotents of a particular semigroup $S$ need not form a subsemigroup, and, in many cases, may generate the entire $S$.
These idempotent-generated semigroups are of particular importance in our context, and possess
many features that confirm their significance.
For example, they have the universal property that every semigroup embeds into an idempotent generated semigroup \cite{Ho66}, and if the semigroup is countable, it can be embedded in a semigroup generated by three idempotents \cite{By84}. Also, many naturally occurring semigroups have the property that they are idempotent-generated. Important examples include semigroups of transformations \cite{HM1990,Howie1978,Ho66,Aizenstat1962}, matrix semigroups \cite{Er67,La83}, endomorphism monoids of independence algebras \cite{FL92,Go95}, diagram semigroups \cite{EG17,EF2012,Ea11,MM2007,DEG2017,BDP2002},  and certain reductive linear algebraic monoids
\cite{Pu88a,Pu06}.

The concept of the free idempotent-generated semigroup $\IG(E)$ over a biordered set $E$ goes
back to Clifford \cite{Cl75} and the seminal work of Nambooripad on the structure theory of regular semigroups \cite{Na79}; see also \cite{Pa80}.  
Biordered sets are somewhat complicated algebraic/order-theoretic structures whose axioms were given in \cite{Na79}.  
It was later proved in \cite{Easdown1985} that every (abstract) biordered set is in fact the biordered set of a semigroup~$S$, i.e.~the set
\[
E=\sfE(S):=\set{e\in S}{e^2=e}
\]
of idempotents of $S$, with a partial multiplication that is defined by restricting the multiplication from $S$ to the set of so-called basic pairs:
\[
\Bp=\Bp(S):=\bigl\{ (e,f)\in E\times E\colon \{ef,fe\}\cap \{e,f\}\neq \es\bigr\}.
\]
Adopting this viewpoint, the free idempotent-generated semigroup $\IG(E)$ is defined by the presentation
\[
\bigl\langle X_E\mid x_ex_f=x_{ef}\ ((e,f)\in \Bp)\bigr\rangle,
\]
where $X_E:=\bigl\{ x_e\colon e\in E\}$ is an alphabet in one-one correspondence with $E$.  If we identify $e\in E$ with the equivalence class of $x_e$ modulo the relations, then $\IG(E)$ is the free-est semigroup generated by $E$, and having $E$ as its biordered set \cite{Easdown1985}. In particular, there is a natural homomorphism $\IG(E)\rightarrow S$, which restricts to the identity on $E$, and whose image is the subsemigroup of $S$ generated by its idempotents.

Following their discovery, it turned out that the free idempotent-generated semigroups are a rich source of interesting questions in combinatorial semigroup theory, as well as an intriguing interface to many other areas, notably combinatorial/geometric group theory and algebraic topology.
The above-mentioned link between subgroups and idempotents immediately points to the study of maximal subgroups of $\IG(E)$, which has turned out to be a key direction.
On the basis of early results \cite{El02,NP80,Pa78}, it seemed plausible for a while that the maximal subgroups may turn out to always be free. This was refuted in \cite{BM09}, where a concrete example was exhibited having the free abelian group of rank $2$ among its maximal subgroups.  The methods of \cite{BM09} were topological in nature, and involved realising subgroups of $\IG(E)$ as fundamental groups of natural $2$-complexes built from $E$.
Using Reidemeister--Schreier rewriting techniques, a general presentation for the maximal subgroups of $\IG(E)$ was established in \cite{GR12IJM}, and was deployed in the same paper to show that in fact \emph{any} group is the maximal subgroup of some $\IG(E)$. A number of variations and refinements followed, e.g.~\cite{DR13,DG17,DD19,Do21,Do22,GY14}.

In the course of this development it became clear that actually identifying the maximal subgroups of $\IG(E)$ for a specific given $E$ can be a tricky business. 
In \cite{GR12} it was shown that when $E=\sfE(\T_n)$ is the biordered set of the full transformation semigroup $\T_n$, the maximal subgroup of $\IG(E)$ containing an idempotent of rank $r\leq n-2$ is isomorphic to the symmetric group~$\S_r$, coinciding exactly with the corresponding maximal subgroup of $\T_n$ itself.
Analogous results were proved for the partial transformation monoid \cite{Do13} and for endomorphism monoids of free $G$-acts ($G$ a group) \cite{GY14}.
Furthermore, it was shown in \cite{DG14} that if one starts with the general linear monoid of dimension $n$ over a division ring~$Q$, the resulting maximal subgroup of $\IG(E)$ is the general linear group of dimension $r$ over~$Q$,
but only for $r<n/3$.
The nature of the  maximal subgroups for higher ranks is an intriguing open problem.

One important class of semigroups that have completely eluded computations of this kind are the so-called \emph{diagram monoids}, such as the partition monoid $\P_n$, Brauer monoid $\B_n$ and Temperley--Lieb monoid~$\TL_n$.  
Diagram monoids play an important role in many fields, including algebra, topology, category theory, physics and biology \cite{HR2005,Brauer1937,TL1971,Martin1994,Jones1994_2,Kauffman1987,FJ2022}.  Beyond their intrinsic interest, diagram monoids have important structural parallels with transformation and linear monoids.  Of particular relevance to the topic of the current paper is the fact that~$\P_n$,~$\B_n$ and~$\TL_n$ are `almost idempotent-generated' in the sense that every non-unit is a product of idempotents \cite{Ea11,MM2007,EG17}.  Thus, it is natural to ask whether maximal subgroups of these monoids 
are isomorphic to those of the associated free idempotent-generated semigroups~$\IG(E)$, as is the case for (linear) transformation monoids.  In fact, it follows from one of our main results, Theorem \ref{thm:mainIGPn}, that this is \emph{not} the case for $\P_n$.  Specifically, maximal subgroups corresponding to an idempotent of rank $r\leq n-2$ are (isomorphic to) symmetric groups $\S_r$ in $\P_n$, but are direct products $\Z\times\S_r$ in $\IG(E)$.

However, diagram monoids come equipped with some additional structure, in the form of an involution, making them into regular $*$-semigroups.  These are defined as semigroups with a unary operation $a\mapsto a^*$ that satisfies the following axioms:
\[
(a^*)^* = a = aa^*a \ANd  (ab)^* = b^*a^*.  
\]
They were introduced in \cite{NS78}, and investigated further over time, e.g.~in \cite{NP1985,Adair1982,Im83,Pe85,Jo12,Po01,Yamada1982},
coming into sharper focus in recent years; see \cite{EM24,EGMR}.

By way of motivation, let us dwell a bit more on the findings of \cite{EGMR}. 
Given a regular $*$-semigroup $S$, an idempotent $p\in \sfE(S)$ satisfying $p=p^*$ is called a \emph{projection}.
The set $P=\sfP(S)$ of all projections can be given the (unary) structure of a \emph{projection algebra}; these were introduced first in
 \cite{Im83}, and subsequently by several different authors in a variety of equivalent forms (under varying names),
a development that is explained in detail in \cite[Section 4]{EM24}. 
So every regular $*$-semigroup gives rise to its projection algebra $\sfP(S)$. One of the main findings of \cite{Im83} is that, conversely, for every projection algebra $P$ there exists a regular $*$-semigroup~$S$ such that $\sfP(S)=P$.
An alternative proof of this is given in \cite{EGMR}, where it was shown that every projection algebra $P$ gives rise to a free projection-generated regular $*$-semigroup
$\PG(P)$. This can be viewed as the analogue of $\IG(E)$ for the category of regular $*$-semigroups, as reflected in the following presentation for~$\PG(P)$, established in \cite[Theorem 7.2]{EGMR}:
\[
\bigl\langle X_P \mid 
x_p^2=x_p,\ (x_px_q)^2=x_px_q,\ x_px_qx_p=x_{pqp}\ (p,q\in P)\bigr\rangle,
\]
where $X_P=\{ x_p\colon p\in P\}$.  
This semigroup has projection algebra $P$, and for any other regular $*$-semigroup with $P=\sfP(S)$, the identity map on $P$ extends to a homomorphism ${\PG(P)\to S}$.

So, paralleling the development for $\IG(E)$, the following natural general question arises: 
\bit
\item 
\emph{Given a (projection-generated) regular $*$-semigroup $S$, with $P=\sfP(S)$, how is $\PG(P)$ related to~$S$?}
\eit
Certainly, $S$ is a natural homomorphic image of $\PG(P)$ by construction.
Furthermore, since the projection algebras of~$S$ and~$\PG(P)$ coincide, it follows from the basic theory of regular $*$-semigroups that there are natural bijections between their $\R$-, $\L$-, $\D$- and $\H$-classes (see Subsections \ref{ss:IGE}--\ref{ss:PGP} below for more detail).
Informally speaking, this means that the only potential difference between these two semigroups is in their maximal subgroups. 
The paper \cite{EGMR} gives examples of situations where these are different from each other, and some where they are the same. A striking example of the latter situation is provided by the Temperley--Lieb monoids $\TL_n$, which turn out to be isomorphic to their own free semigroups $\PG(\sfP(\TL_n))$; see \cite[Theorem~9.1]{EGMR}.

There is a strong connection between the projection structure of a regular $*$-semigroup and its idempotent structure, which we briefly outline;
for details and references see \cite[Subsection~3.1]{EM24} and \cite[Section 6]{EGMR}.
Let $S$ be a regular $*$-semigroup, let $P:=\sfP(S)$ be its projection algebra, and let $E:=\sfE(S)$ be
its biordered set of idempotents.
Every $e\in E$ is a product of two projections.
More precisely, letting
\[
\F:= \big\{ (p,q)\in P\times P\colon pqp=p,\ qpq=q\big\}
\]
be the \emph{friendliness relation}, every idempotent $e\in E$ is \emph{uniquely} expressed as a product $e=pq$, with $(p,q)\in \F$.
In \cite[Section 6]{EGMR} it is proved that the projection algebra $P$ uniquely determines the biordered set $E$, in the sense that for every regular $*$-semigroup  with projection algebra isomorphic to $P$ its biordered set is isomorphic to $E$.
Even more strongly, it is proved that there is a categorical equivalence between projection algebras and so-called regular $*$-biordered sets, as defined in \cite{NP1985}.
One consequence of all this is  that the semigroup $\PG(P)$ is a natural homomorphic image of~$\IG(E)$. 

So, in this paper we develop the theory of maximal subgroups of free projection-generated regular $*$-semigroups $\PG(P)$, and apply it to compare these with the corresponding subgroups of free idempotent-generated semigroups $\IG(E)$.  We do this on both a general level, and also concretely for the free semigroups arising from the partition monoid $\P_n$.
The paper thus has two distinctive parts. In the first, we obtain a number of general presentations (by generators and defining relations) for the maximal subgroups of free projection-generated regular $*$-semigroups.  See Theorem \ref{th:linkedpres} and Corollaries~\ref{co:lkid} and~\ref{co:lkdp1}, and also Theorem~\ref{thm:pgmaxq} and Corollary~\ref{cor:pgmaxq}; the latter two results exhibit maximal subgroups of~$\PG(P)$ as explicit quotients of corresponding maximal subgroups of $\IG(E)$.
Presentations such as these are important, because they provide concrete descriptions of the maximal subgroups, which can then be analysed with the tools of combinatorial group theory; they also provide a link to the topological viewpoint of \cite[Section~10]{EGMR}.
Then, in the second part of the paper, we apply our theory to determine the maximal subgroups of the free projection- and idempotent-generated semigroups arising from the partition monoid~$\P_n$, our main results being the following:

\begin{thm}
\label{thm:mainPGPn}
The maximal subgroup of $\PG(\sfP(\P_n))$ corresponding to any idempotent of rank ${0\leq r\leq n-2}$ is isomorphic to the symmetric group $\S_r$.
\end{thm}

\begin{thm}
\label{thm:mainIGPn}
The maximal subgroup of $\IG(\sfE(\P_n))$ corresponding to any idempotent of rank ${0\leq r\leq n-2}$ is isomorphic to $\Z\times\S_r$.
\end{thm}

The extra direct factor of $\Z$  in Theorem \ref{thm:mainIGPn} may appear mysterious at this stage, but we will see that it arises from a tight connection with a certain \emph{twisted} monoid $\P_n^\Phi$, which has the same biordered set as $\P_n$.  
The subgroups corresponding to idempotents with $r\in \{n-1,n\}$ in both theorems are actually easy to determine:
when $r=n-1$ they are free with ranks 
that can be explicitly computed (see Remark \ref{rem:noss} and Corollary \ref{cor:PGPnn-1}), and when $r=n$ they are trivial. Thus, we know precisely all maximal subgroups in both 
$\PG(\sfP(\P_n))$ and $\IG(\sfE(\P_n))$.

The knowledge of the maximal subgroups of $\PG(\sfP(\P_n))$ translates into a complete description of the semigroup $\PG(\sfP(\P_n))$.
It is `almost equal' to $\P_n$ itself: the only differences are that the $\D$-class consisting of partitions of rank $n-1$ is `inflated' so that the maximal subgroups are copies of the appropriate free group, and that non-identity bijections are absent.
This may be interesting to compare to the case of the free projection-generated semigroup arising from the Temperley--Lieb monoid $\TL_n$, which is proved to be \emph{exactly} $\TL_n$ itself in \cite[Theorem 9.1]{EGMR}.

The situation is different for $\IG(\sfE(\P_n))$: even though we know all its maximal subgroups, and hence its regular $\D$-classes, this monoid contains many non-regular elements (see Subsection~\ref{ss:rships}), and hence has a more complicated structure. 
However, for every free idempotent-generated semigroup $\IG(E)$, there exists a distinguished (maximal) regular homomorphic image $\RIG(E)$, whose structure is again completely determined by its maximal subgroups.
Thus, Theorem \ref{thm:mainIGPn} and the above discussion do entail a complete structural description of $\RIG(\sfE(\P_n))$. However, this semigroup is no longer `almost equal' to $\P_n$ but to its twisted counterpart $\P_n^\Phi$.
The authors find this quite remarkable: starting from the well-known semigroup $\P_n$, and applying the well-known construction $\IG$ to it, leads to another well-known semigroup $\P_n^\Phi$, which originally arose via a very different, combinatorial-geometric, route.

Theorems \ref{thm:mainPGPn} and \ref{thm:mainIGPn} can be viewed as analogues of the main result from \cite{GR12IJM}, which asserts that every  maximal subgroup of $\IG(\sfE(\T_n))$ corresponding to an idempotent of rank ${1\leq r\leq n-2}$ is the symmetric group $\S_r$.
In fact, the key idea behind the proofs of both our theorems is to establish links with maximal subgroups of 
$\IG(\sfE(\T_n))$.
However, several conceptual and technical difficulties present themselves in this reduction, especially in the case of $\IG(\sfE(\P_n))$.
The most notable is the fact that the `target group' is no longer $\S_r$, but $\Z\times \S_r$. Another is that while in~$\T_n$ every $2\times2$ rectangular band is a singular square, this is not true in $\P_n$; see Subsection~\ref{ss:wow}.

\bigskip

The paper is organised as follows.
Section \ref{sec:prelim} covers the general preliminaries concerning semigroups, including regular $*$-semigroups, presentations, projection algebras, biordered sets, and partition monoids.
In Section~\ref{sec:rs} we set up our main tool in establishing our general presentations, a Reidemeister--Schreier type rewriting process for maximal subgroups originally developed in \cite{Ru99}.
In Section \ref{sec:maxIG} we review the application of this process to the maximal subgroups of free idempotent-generated semigroups~$\IG(E)$ from~\cite{GR12IJM}, and then introduce a small modification to  the final presentation to better suit our needs.
In Section \ref{sec:IGETn} we briefly review the set-up and notation from the computation  of maximal subgroups of the free idempotent-generated semigroups $\IG(\sfE(\T_n))$ arising from the full transformation monoid $\T_n$ \cite{GR12},
which, as mentioned above, are a key ingredient in our overall proof strategy for Theorems \ref{thm:mainPGPn} and \ref{thm:mainIGPn}.

Section \ref{sec:pres-linked} contains our first two general presentation results -- Theorem \ref{th:linkedpres} and Corollary~\ref{co:lkdp1} -- and some related presentations including Corollary \ref{co:lkid}.
It also contains a number of applications, among them
 the proof of Theorem \ref{thm:mainPGPn} in the case $r=0$.
Section \ref{sec:pres-quot} gives the remaining general presentations, Theorem \ref{thm:pgmaxq} and Corollary~\ref{cor:pgmaxq}; these reflect the fact that the maximal subgroups of $\PG(P)$ are natural homomorphic images of the maximal subgroups of $\IG(E)$.  After proving these results, we compare the presentations of Sections \ref{sec:pres-linked} and \ref{sec:pres-quot} to each other, which involves analysing the relationships between two kinds of squares of idempotents, singular squares and linked squares.

This concludes the general, theory-building part of the paper, and the rest is devoted to
an analysis of the maximal subgroups of~$\PG(\sfP(\P_n))$ and~$\IG(\sfE(\P_n))$.
Before the arguments for $\PG(\sfP(\P_n))$ and~$\IG(\sfE(\P_n))$ diverge,
in Section \ref{sec:common} we introduce some further preliminaries
and results that are shared by the two.
After these preparations we are already able to complete the proof of Theorem \ref{thm:mainPGPn} for $r\geq1$ in Section \ref{sec:maxPGPn}.
  The  proof of Theorem \ref{thm:mainIGPn} for $r\geq 1$ is  the longest and most intricate argument of the paper.
  It is given in Section \ref{sec:maxIGPn}, which contains at the beginning an outline and summary of the main ideas.
The proofs of our main results are rounded off in Section \ref{sec:r0}, where we prove Theorem \ref{thm:mainIGPn} for $r=0$.

The paper concludes in Section~\ref{sec:conc} with some remarks and open questions for future research.

\subsection*{Acknowledgements}

This work was supported by the following grants:
Future Fellowship FT190100632 of the Australian Research Council;
EP/V032003/1, EP/S020616/1 and EP/V003224/1 of the Engineering and Physical Sciences Research Council.
The first author thanks the Heilbronn Institute for partially funding his visit to St Andrews in 2025.
The second author thanks the Sydney Mathematical Research Institute, the University of Sydney
and Western Sydney University for partially funding his visit to Sydney in 2023.

During the early stages of this research we were greatly assisted by computations with the Semigroups package for GAP \cite{Semigroups,GAP}.

We thank the referees for their valuable feedback.

\section{Preliminaries}
\label{sec:prelim}

\subsection{General notation}

We will use $\Z$ to denote both the set of integers and the corresponding additive group.  We will also use standard notation to denote intervals of integers, e.g.~${[m,n]:=\{ m,m+1,\ldots,n\}}$ and ${(m,n] := [m+1,n]}$. For an integer $n\geq0$ we will write $[n]$ to denote the interval~$[1,n]$.

\subsection{Presentations}

We will use two kinds of presentations in this paper: semigroup and group.
Let $A$ be an alphabet, and denote by $A^+$ the \emph{free semigroup} over $A$, consisting of all non-empty (finite) words over~$A$, under concatenation.
Denote by~$1$ the empty word, and
 let $A^\ast:=A^+\cup\{1\}$ denote the \emph{free monoid} over $A$.

 A \emph{semigroup presentation} is a pair $\langle A\mid \rels\rangle$, where $\rels\subseteq A^+\times A^+$.
 The elements of $A$ are called  \emph{generators} and
 the elements of $\rels$ (\emph{defining}) \emph{relations}; a relation $(u,v)\in\rels$ is typically displayed as an equality,~${u=v}$.
 The presentation $\langle A\mid \rels\rangle$ defines the semigroup $S:=A^+/\rho$, where $\rho$ is the congruence generated by $\rels$.
 The mapping $\pi\colon A^+\rightarrow S$, $w\mapsto w/\rho$ is a surmorphism.
 The semigroup defined by $\langle A\mid \rels\rangle$ is the free-est semigroup generated by an image of $A$, and satisfying the relations in $\rels$.
 More formally, if $T$ is a semigroup, and $\phi\colon A^+\rightarrow T$ a homomorphism such that $u\phi=v\phi$ for all $(u,v)\in \rels$, then there exists a unique homomorphism $\psi\colon S\rightarrow T$ such that $\phi=\pi\psi$.  
 
We will typically identify a word $u\in A^+$ and the element $u/\rho$ of $S=A^+/\rho$ it represents.
Thus we will often write $u=v$ to mean that $u/\rho=v/\rho$ in $S$.
When there is a danger of confusion, we will emphasise that the equality holds in $S$, or we will say that $u=v$ is a consequence of~$\rels$.

A \emph{group presentation} is a pair $\langle A\mid \rels\rangle$, where $\rels\subseteq (A\cup A^{-1})^\ast \times (A\cup A^{-1})^\ast$, with ${A^{-1}:=\{ a^{-1}\colon a\in A\}}$ a new alphabet consisting of formal inverses of letters from $A$.
It defines the group $(A\cup A^{-1})^*/\rho$ where $\rho$ is the congruence generated by $\rels\cup\big\{ (aa^{-1}=a^{-1}a=1)\colon a\in A\big\}$.
It is the free-est group generated by an image of $A$ that satisfies the defining relations from $\rels$.

\subsection{Green's relations, idempotents and regularity}
\label{ss:Green}

Green's relations (pre-orders and equivalences) are the standard tool for working with the ideal structure of a semigroup.
Here we give the basic definitions and facts, and refer the reader to any textbook on semigroup theory, such as
\cite{Ho95}, for a more systematic exposition.

Let $S$ be a semigroup, and denote by $S^1$ the monoid obtained by adjoining an identity element to $S$ if necessary.
For $a,b\in S$, define
\[
a \leq_\L b \iff S^1a\subseteq S^1b  ,\quad  a \leq_\R b \iff aS^1\subseteq bS^1 \ANd a \leq_\J b \iff S^1aS^1 \subseteq S^1bS^1.
\]
If $\leq_\mathscr{K}$ is any of these three pre-orders, it induces an equivalence relation $\mathscr{K}$ via
\[
a\mr\mathscr{K} b \iff a\leq_\mathscr{K} b \text{ and } b\leq_{\mathscr{K}} a.
\]
There are two further Green's equivalences:
\[
\H := \L\cap\R \ANd \D :=\L\vee\R,
\]
where $\vee$ denotes the join of equivalence relations.
It is well known that $\D=\L\circ\R=\R\circ\L$, where $\circ$ denotes relational composition.
For $a\in S$, denote by $L_a$, $R_a$, $J_a$, $H_a$ and $D_a$ its $\L$-, $\R$-, $\J$-, $\H$- and $\D$-classes, respectively.
If $\mathscr{K}$ is any of $\L$, $\R$ or~$\J$,  the set $S/\mathscr{K} = \{ K_a\colon a\in S\}$ of all $\mathscr{K}$-classes of $S$ has an induced partial order $\leq$, given by
\[
K_a \leq K_b \iff a\leq_\mathscr{K} b \qquad\text{for all $a,b\in S$.}
\]
One of the most important basic results concerning Green's relations is the following; for a proof, see for example \cite[Lemma 2.2.1]{Ho95}.

\begin{lemma}[Green's Lemma]\label{lem:GL}
Let $a$ and $b$ be $\R$-related elements of a semigroup $S$, so that $b=as$ and $a=bt$ for some $s,t\in S^1$.  Then the maps
\[
L_a\to L_b,\ u\mapsto us\quad\text{and}\quad L_b\to L_a,\ v\mapsto vt
\]
are mutually inverse bijections.  Moreover, these maps restrict to mutually inverse bijections $H_a\to H_b$ and  $H_b\to H_a$.
\end{lemma}

There is a left-right dual of the previous result, whose statement we omit, but throughout the article we will frequently refer to both formulations as Green's Lemma.  We note that Green's Lemma, and many of the other forthcoming results and arguments concerning Green's relations, can be conveniently visualised through the use of so-called \emph{egg-box diagrams}.  In such a diagram we place the elements of a $\D$-class in a rectangular array, with $\R$-related elements in the same row, $\L$-related elements in the same column, and $\H$-related elements in the same cell (intersection of a row and column).  For more details, see for example \cite[Chapter~2]{Ho95}.

An \emph{idempotent} of a semigroup $S$ is an element $e\in S$ such that $e=e^2$.  The set of all idempotents of $S$ is denoted
\[
\sfE(S) := \set{e\in S}{e^2=e}.
\]
If $e,f\in \sfE(S)$, then it is easy to check that
\begin{equation}\label{eq:eleqf}
e\leq_\L f \iff e=ef \AND e\leq_\R f \iff e=fe.
\end{equation}

An element $a$ of a semigroup $S$ is \emph{regular} if $a=axa$ for some $x\in S$.  Idempotents are obviously regular.  If $a=axa$, then $a \mr\R ax$ and $a\mr\L xa$, with the elements $ax$ and $xa$ being idempotents.  Moreover, the element $y:=xax$ satisfies $a=aya$ and $y=yay$; the elements~$a$ and~$y$ are said to be \emph{\emph{(}semigroup\emph{)} inverses} of each other.
We say that $S$  is a regular semigroup if every element of $S$ is regular.  Some basic facts include the following:
\bit
\item When~$a$ is regular, every element of the $\D$-class $D_a$ is regular, and this is equivalent to the $\D$-class containing an idempotent, and in fact to every $\R$- and $\L$-class in $D_a$ containing an idempotent.
\item If $e$ is an idempotent, then the $\H$-class $H_e$ is a group with identity $e$, and is a maximal subgroup of $S$.  
\item If $e$ and $f$ are two $\D$-related idempotents, then the groups $H_e$ and $H_f$ are isomorphic.
\eit
The following well-known result is typically attributed to FitzGerald \cite{Fi72} or Hall \cite{Ha73}, although this exact formulation is not found in either paper.  We give a simple (standard) proof for completeness and convenience, and since the result is of such fundamental importance to us.

\begin{prop}[{cf.~\cite[Lemma 1]{Fi72} or \cite[Lemma 1]{Ha73}}]\label{prop:DF}
If $a$ is a regular element of a semigroup~$S$, and if $a\in\la\sfE(S)\ra$, then there exist idempotents $e_1,\ldots,e_k\in\sfE(S)$ such that
\[
a = e_1\cdots e_k \AND a \mr\R e_1 \mr\L e_2 \mr\R e_3 \mr\L e_4 \mr\R \cdots \mr\L e_{k-1} \mr\R e_k \mr\L a.
\]
\end{prop}

\begin{proof}
By assumption we have $a = f_1\cdots f_n$ for some $f_1,\ldots,f_n\in\sfE(S)$, and also $a=axa$ and $x=xax$ for some $x\in S$.  We then define the elements
\[
g_i = f_i\cdots f_n x f_1\cdots f_i \text{ for $1\leq i\leq n$} \AND h_i = f_i\cdots f_n x f_1\cdots f_{i-1} \text{ for $1\leq i\leq n+1$.}
\]
Using $xf_1\cdots f_nx = xax = x$, one can easily check that each $g_i,h_i\in\sfE(S)$, and that 
\[
g_ih_i = h_i \COMMA h_ig_i = g_i \COMMA g_ih_{i+1} = g_i \AND h_{i+1}g_i = h_{i+1} \qquad\text{for all $1\leq i\leq n$,}
\]
which says that $h_i \mr\R g_i \mr\L h_{i+1}$ for all such $i$.  We also have
\[
a\mr\R ax = h_1 \mr\R g_1 \AND a\mr\L xa = h_{n+1} \mr\L g_n.
\]
A simple induction shows that $g_1\cdots g_i = (ax)^if_1\cdots f_i$ for each $1\leq i\leq n$, so that in particular 
\[
g_1\cdots g_n = (ax)^nf_1\cdots f_n = (ax)^na = a.
\]
Combining this with $g_i = h_ig_i$, it follows that
\[
a = g_1\cdots g_n = g_1h_2g_2h_3g_3\cdots h_ng_n,
\]
which is a factorisation of the desired form.
\end{proof}

\subsection{Graham--Houghton graphs}\label{ss:GHG}

The idempotents of a regular $\D$-class of a semigroup lead to the following important combinatorial structure:

\begin{defn}[{\textbf{Graham--Houghton graph}}]\label{defn:GHD}
Let $D$ be a $\D$-class in a semigroup $S$, and let~$R_i$~(${i\in I}$) and $L_j$ ($j\in J$) be the $\R$-
and $\L$-classes of $S$ contained in $D$, respectively.
The \emph{Graham--Houghton graph} $\GH(D)$ of $D$ is
the bipartite graph with vertices $I\sqcup J$, the disjoint union of $I$ and $J$, and with
 an edge connecting $i\in I$ and $j\in J$ if and only if the $\H$-class $H_{i,j} := R_i \cap L_j$ is a group.
We identify such an edge~$i-j$ with the unique idempotent $e_{i,j}$ contained in $H_{i,j}$, which is the identity of this group.
\end{defn}

The following observation will play a crucial role in our study.  It goes back at least to \cite{Graham1968} in the finite case; the argument of \cite{Graham1968} works essentially unchanged in the case that the principal factor of the $\D$-class is completely 0-simple; see also \cite[Theorem 2.1]{Ru94} for a slightly different formulation.

\begin{prop}
\label{pr:GHconn}
If $S$ is idempotent-generated, then $\GH(D)$ is connected for every regular $\D$-class $D$ of $S$.
\end{prop}

\begin{proof}
Keeping the notation of Definition \ref{defn:GHD}, it suffices to show that arbitrary vertices $i\in I$ and $j\in J$ are connected by a path in $\GH(D)$.  To do so, fix some such $i\in I$ and $j\in J$, and also fix an element~${a\in H_{i,j} = R_i\cap L_j}$.    
Now write $a=e_1\cdots e_k$ as in
Proposition \ref{prop:DF}.
Identifying each idempotent with an edge of $\GH(D)$ as explained in Definition \ref{defn:GHD},
the sequence $e_1,\ldots,e_k$ is the desired path from $i$ to $j$.
\end{proof}

\subsection{Biordered sets and singular squares}\label{ss:BS}

Let $S$ be a semigroup, and $E = \sfE(S)$ its set of idempotents.  A pair $(e,f)\in E\times E$ is said to be \emph{basic} if $\{ef,fe\}\cap\{e,f\} \not= \es$, i.e.~if at least one of the products $ef$ or $fe$ is equal to one of~$e$ or $f$.  In this case $ef$ and $fe$ are both idempotents, but they are not necessarily both equal to one of $e$ or $f$.  We write
\[
\Bp = \Bp(S) := \bigset{(e,f)\in E\times E}{\{ef,fe\}\cap\{e,f\} \not= \es}
\]
for the set of all such basic pairs.  Note that $\Bp$ consists of all pairs of idempotents that are comparable in one or both of the $\leq_\L$ or $\leq_\R$ orders; cf.~\eqref{eq:eleqf}. The \emph{biordered set of $S$} is the algebra with:
\bit
\item underlying set $E = \sfE(S)$, and 
\item a partial binary operation, which is the restriction to $\Bp$ of the product in $S$.
\eit
Biordered sets were also defined axiomatically by Nambooripad in \cite{Na79}, without reference to any containing semigroup.  We do not need to give this axiomatisation, however, as it was later shown by Easdown \cite{Easdown1985} that every (abstract) biordered set arises from a semigroup in the above way.

Given a biordered set $E = \sfE(S)$, and  idempotents $e,f\in E$, we define the \emph{sandwich set} 
\begin{equation}\label{eq:Sef}
\SS(e,f):=\{ h\in E\colon ehf=ef,\ fhe=h\}.
\end{equation}
Note that here the products $ehf$ and $ef$ are taken in the semigroup $S$ rather than in $E$, and these need not be idempotents. 
Nonetheless, $\SS(e,f)$ can be defined from the biordered set $E$ alone \cite[p2]{Na79}.
 The biordered set $E$ is said to be \emph{regular} if each sandwich set is non-empty.  Equivalently, $E$ is regular if and only if it is the biordered set of idempotents of some regular semigroup.  But note that a regular biordered set can also be the biordered set of a non-regular semigroup.

Fix a biordered set $E = \sfE(S)$, where $S$ is a semigroup.  A \emph{square of idempotents} in $E$ is an array~$\smat efgh$, where $e,f,g,h\in E$, and where
\[
e \mr\R f \COMMA g \mr\R h \COMMA e \mr\L g \AND f \mr\L h.
\]
Note that $\{e,f,g,h\}$ need not be a subsemigroup of $S$; specifically the `diagonal products', such as $eh$, need not be idempotents. The elements $e,f,g,h$ are not required to be distinct; if there are equalities among them,  the square is said to be \emph{degenerate}.  Because of the above $\R$- and $\L$-relationships, we have
\[
e=f \iff g=h \AND e=g \iff f=h.
\]

We say that the square $\smat efgh$ is a \emph{left-right} (\emph{LR}) \emph{singular square} if there exists $u\in E$ such that
\begin{equation}\label{eq:LR}
ue=e \COMMA ug=g \COMMA eu=f \AND gu = h.
\end{equation}
The idempotent $u$ is said to \emph{\emph{(}LR-\emph{)}singularise} the square.  In this case we will also say that~$\smat fehg$ is a \emph{right-left \emph{(}RL\emph{)} singular square, \emph{(}RL-\emph{)}singularised by $u$}.  We refer to both $\smat efgh$ and $\smat fehg$ as \emph{horizontal singular squares}, and we visually represent them as shown in Figure~\ref{fig:ssq}(a) and~(b).
The following facts are straightforward to check, and we will use them throughout without further reference:
\begin{itemize}
\item
If $\smat efgh$ is a square satisfying 
\eqref{eq:LR} then $uf=f=fu$ and $uh=h=hu$, which is to say that $f,h\leq u$ in the natural partial order on $E$.
\item
If $\smat efgh$ is a square satisfying the first two equalities of \eqref{eq:LR}, then the third and fourth equalities are equivalent to each other; thus in any specific situation it is sufficient to verify only one of them.
\item
If $e,g,u\in E$ satisfy $e\mr\L g$, and $ue=e$ and $ug=g$, then  $\smat e{eu}g{gu}$ is an LR singular square of idempotents (singularised by $u$). 
\end{itemize}

Dually, $\smat efgh$ is an \emph{up-down \emph{(}UD\emph{)} singular square}, with \emph{\emph{(}UD\emph{)}-singularising idempotent} $u$, if
\begin{equation}\label{eq:UD}
eu=e \COMMA fu=f \COMMA ue=g \AND uf = h.
\end{equation}
In this case we will also say that~$\smat ghef$ is a \emph{down-up \emph{(}DU\emph{)} singular square, \emph{(}DU-\emph{)}singularised by~$u$}.  These two squares are referred to as \emph{vertical}, and are visualised in Figure~\ref{fig:ssq}(c) and~(d).
Statements analogous to the above three bullet points hold for these squares.

\begin{figure}
\begin{center}

\setlength{\tabcolsep}{7mm}

\begin{tabular}{cccc}
\begin{tikzpicture}

\fill[lightgray!50] (-0.7,1.5)--(1.5,1.5)--(1.5,-0.5)--(-0.7,-0.5)--(-0.7,1.5);

\node (e) at (0,1) {$\mathrlap{e}{\phantom{f}}$};
\node (f) at (1,1) {$\mathrlap{f}{\phantom{f}}$};
\node (g) at (0,0) {$\mathrlap{g}{\phantom{f}}$};
\node (h) at (1,0) {$\mathrlap{h}{\phantom{f}}$};

\draw [->] (e) -- (f);
\draw [->] (g) -- (h);
\draw [->] (e) edge [loop left]  (e);
\draw [->] (g) edge [loop left] (g);

\node (u) at (0.5,1.8) {$u$};

\end{tikzpicture}
&
\begin{tikzpicture}

\fill[lightgray!50] (-0.5,1.5)--(1.7,1.5)--(1.7,-0.5)--(-0.5,-0.5)--(-0.5,1.5);

\node (f) at (0,1) {$\mathrlap{f}{\phantom{f}}$};
\node (e) at (1,1) {$\mathrlap{e}{\phantom{f}}$};
\node (h) at (0,0) {$\mathrlap{h}{\phantom{f}}$};
\node (g) at (1,0) {$\mathrlap{g}{\phantom{f}}$};

\draw [->] (e) -- (f);
\draw [->] (g) -- (h);
\draw [->] (e) edge [loop right]  (e);
\draw [->] (g) edge [loop right] (g);

\node (u) at (0.5,1.8) {$u$};

\end{tikzpicture}
&
\begin{tikzpicture}

\fill[lightgray!50] (-0.5,1.8)--(1.5,1.8)--(1.5,-0.5)--(-0.5,-0.5)--(-0.5,1.5);

\node [minimum size = 6mm] (e) at (0,1) {$\mathrlap{e}{\phantom{f}}$};
\node (f) at (1,1) {$\mathrlap{f}{\phantom{f}}$};
\node (g) at (0,0) {$\mathrlap{g}{\phantom{f}}$};
\node (h) at (1,0) {$\mathrlap{h}{\phantom{f}}$};

\draw [->] (e) -- (g);
\draw [->] (f) -- (h);
\draw [->] (e) edge [loop above]  (e);
\draw [->] (f) edge [loop above] (e);

\node (u) at (0.5,2.1) {$u$};

\end{tikzpicture}
&
\begin{tikzpicture}

\fill[lightgray!50] (-0.5,1.5)--(1.5,1.5)--(1.5,-0.8)--(-0.5,-0.8)--(-0.5,1.5);

\node (g) at (0,1) {$\mathrlap{g}{\phantom{f}}$};
\node (h) at (1,1) {$\mathrlap{h}{\phantom{f}}$};
\node (e) at (0,0) {$\mathrlap{e}{\phantom{f}}$};
\node (f) at (1,0) {$\mathrlap{f}{\phantom{f}}$};

\draw [->] (e) -- (g);
\draw [->] (f) -- (h);
\draw [->] (e) edge [loop below]  (e);
\draw [->] (f) edge [loop below] (f);

\node (u) at (0.5,1.8) {$u$};

\end{tikzpicture}
\\
(a) & (b) & (c) & (d)
\end{tabular}

\caption{Visual representations of singular squares:
(a) and (b) respectively represent the LR and~RL squares $\smat{e}{f}{g}{h}$ and $\smat{f}{e}{h}{g}$ satisfying \eqref{eq:LR};
(c) and (d) respectively represent the UD and DU squares $\smat{e}{f}{g}{h}$ and $\smat{g}{h}{e}{f}$ satisfying~\eqref{eq:UD}.  In all cases the singularising element is $u$.}
\label{fig:ssq}
\end{center}
\end{figure}
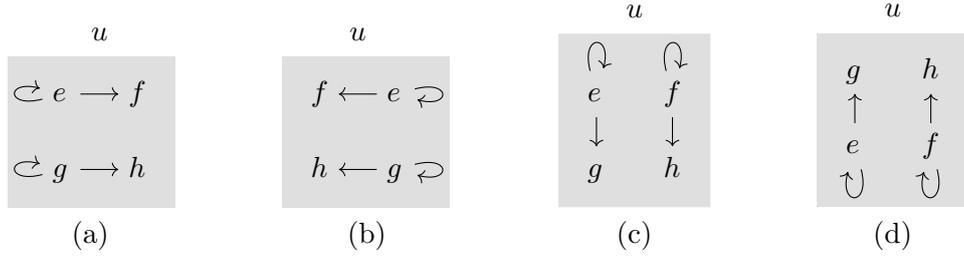

\subsection[Free (regular) idempotent-generated semigroups $\IG(E)$ and $\RIG(E)$]{\boldmath Free (regular) idempotent-generated semigroups $\IG(E)$ and $\RIG(E)$}
\label{ss:IGE}

Fix a semigroup $S$, and its biordered set $E = \sfE(S)$.  Recall from the Introduction that the \emph{free idempotent-generated semigroup}~$\IG(E)$ is defined by the presentation
\[
\bigl\langle X_E\mid x_ex_f=x_{ef}\ ((e,f)\in \Bp)\bigr\rangle .
\]
Here $X_E:=\{ x_e\colon e\in E\}$ is an abstract generating set in one-one correspondence with $E$, and 
$\Bp:=\bigl\{ (e,f)\in E\times E\colon \{ef,fe\}\cap \{e,f\}\neq \es\bigr\}$ is the set of basic pairs.

Let $\pi\colon \IG(E)\rightarrow S$ be the natural homomorphism induced by $x_e\mapsto e$, for $e\in E$.  (Here we recall that the one-letter word $x_e$ is identified with its equivalence class modulo the relations.)  The image of $\pi$ is $\la E\ra$, the idempotent-generated subsemigroup of $S$.
The following are some basic facts that will be used extensively throughout.
\begin{enumerate}[label=\textsf{(IG\arabic*)},leftmargin=10mm]
\item\label{it:IG1}
$\sfE(\IG(E))=X_E$, and $\pi$ induces an isomorphism of biordered sets of $\IG(E)$ and $S$.  
\item\label{it:IG2}
For $e\in E$, the $\D$-class $D_{x_e}$ of $\IG(E)$ is mapped onto the $\D$-class $D_e$ of $\la E\ra$ under $\pi$.
Furthermore, $\pi$ induces bijections between the sets of $\R$-, $\L$- and $\H$-classes in $D_{x_e}$ and 
the corresponding sets in $D_e$.
\end{enumerate}
Intuitively, \ref{it:IG2} is saying that the egg-box diagrams of the $\D$-classes $D_{x_e}$ (in $\IG(E)$) and $D_e$ (in~$\la E\ra$) look the same, except that the maximal subgroups (and hence the $\H$-class sizes) may be larger in $D_{x_e}$.  As noted in Subsection \ref{ss:Green}, the $\D$-classes $D_{x_e}$ and $D_e$ ($e\in E$) are precisely the regular $\D$-classes of $\IG(E)$ and~$\la E\ra$, respectively.
The statement \ref{it:IG1} follows from the main result of \cite{Easdown1985}, while~\ref{it:IG2} follows from~\ref{it:IG1} and~\cite{Fi72}, as discussed in \cite[p148]{GR12IJM}.

The semigroup $\IG(E)$ need not be regular, even when $E$ is a regular biordered set.  But $\IG(E)$ in this case does have a free-est \emph{regular} homomorphic image still with the same biordered set of idempotents.  
This is the \emph{free regular idempotent-generated semigroup} $\RIG(E)$, which is defined by the presentation
\[
\bigl\langle X_E\mid x_ex_f=x_{ef}\ ((e,f)\in \Bp),\ 
x_ex_hx_f=x_ex_f\ (e,f\in E,\ h\in \SS(e,f))\bigr\rangle .
\]
The second family of relations involve the sandwich set $\SS(e,f)$, defined  in \eqref{eq:Sef}.
This semigroup is regular by \cite[Proposition 4.9 and Corollary 5.10]{Pa80}; an alternative construction using categorical machinery can be found in \cite{Na79}.  It is clear that there are homomorphisms $\pi_\supRIG\colon \RIG(E)\rightarrow S$, $x_e\mapsto e$ and $\theta\colon \IG(E)\rightarrow \RIG(E)$, $x_e\mapsto x_e$, and that $\theta\pi_\supRIG=\pi$.  In particular, it follows that:
\begin{enumerate}[label=\textsf{(RIG\arabic*)},leftmargin=12mm]
\item\label{it:RIG1}
$\sfE(\RIG(E))=X_E$, and $\pi_\supRIG$ induces an isomorphism of biordered sets of $\RIG(E)$ and $S$.
\item\label{it:RIG2}
For $e\in E$, the $\D$-class $D_{x_e}$ of $\RIG(E)$ is mapped onto the $\D$-class $D_e$ in $\la E\ra$ under $\pi_\supRIG$.
Furthermore, $\pi_\supRIG$ induces bijections between the sets of $\R$-, $\L$- and $\H$-classes in $D_{x_e}$ and 
the corresponding sets in $D_e$.
\end{enumerate}
By \cite[Theorem 3.6]{BM09} we also have:
\begin{enumerate}[label=\textsf{(RIG\arabic*)},leftmargin=12mm]
\setcounter{enumi}{2}
\item\label{it:RIG3}
For any $e\in E$ the homomorphism $\theta$ induces an isomorphism between the maximal subgroups containing $x_e$ in $\IG(E)$ and $\RIG(E)$.
\end{enumerate}
This time, the entire $\RIG(E)$ has exactly the same Green's structure as $\langle E\rangle $, except that the maximal subgroups of $\RIG(E)$ are homomorphic preimages of those in $\langle E\rangle $.

\subsection{Regular $*$-semigroups}
\label{ss:rss}

Now we gather some basic definitions and properties concerning a regular $*$-semigroup $S$. For ease of reference we include some already encountered in the Introduction.
First recall that by definition,  $S$ has a unary operation $a\mapsto a^*$ such that the following hold:
\begin{equation}\label{eq:*}
(a^*)^* = a = aa^*a \ANd  (ab)^* = b^*a^*\qquad \text{for all } a,b\in S.
\end{equation}
Unless otherwise stated, the proofs of the following can be found in \cite[Subsection 2.1]{EGMR} or in earlier references cited therein.

\begin{enumerate}[label=\textup{\textsf{(RS\arabic*)}},leftmargin=13mm]
\item\label{it:RS1} 
A \emph{projection} is an idempotent $p$ satisfying $p^*=p$; the set of all projections of $S$ is denoted by $P = \sfP(S)$.
\item \label{it:RS2} 
For $a\in S$ the element $aa^*$ is the unique projection $\R$-related to $a$, and $a^*a$ is the unique projection $\L$-related to $a$.  Consequently, $a\mr\R b \iff aa^*=bb^*$ and $a\mr\L b \iff a^*a=b^*b$.
\item \label{it:RS3} The product of two projections is an idempotent, but need not be a projection.
\item\label{it:RS4} 
The \emph{friendliness relation} on $P$ is given by
\begin{align*}
\F &:= \big\{ (p,q)\in P\times P\colon pqp=p,\ qpq=q\big\} \\
&\phantom{:}= \big\{ (p,q)\in P\times P\colon \text{$R_p\cap L_q$ is a group (with identity~$pq$)} \big\}.
\end{align*}
\item \label{it:RS5} 
For any idempotent $e\in \sfE(S)$, the pair $(ee^*,e^*e)$ is the unique pair $(p,q)\in\F$ such that~${e=pq}$.
\item \label{it:RS6} The product of two idempotents need not be an idempotent.
\item \label{it:RS7}
For $a\in S$ and $p\in P$ we have $a^*pa\in P$; 
in particular, $qpq\in P$ for all $p,q\in P$.
\item\label{it:RS8} 
The set $P = \sfP(S)$ together with the unary operations $\theta_p\colon q\mapsto pqp$ (one operation for each $p\in P$) is called the \emph{projection algebra} of $S$.
\item\label{it:RS9} 
Projection algebras can be abstractly axiomatised, and every abstract projection algebra is the projection algebra of a regular $*$-semigroup \cite{Im83,Jo12,EGMR}.
\item\label{it:RS10}
Associated to $\F$, there is a reflexive relation $\leq_\F$ on $P$ defined by $p\leq_\F q \iff pqp=p$, and we have ${\F} = {\leq_\F}\cap{\geq_\F}$.
In fact, $p\leq_\F q \iff p \mr\F qpq$.
\item\label{it:RS11}
For $p,q\in P$ we have $p\leq_\F q$ if and only if $pq\mr\R p$ (see below), and dually $p\leq_\F q$ if and only if $qp\mr\L p$.  Consequently, $p\mr\F q \iff p \mr\R pq \mr\L q$.
\item\label{it:RS12} 
Every projection algebra $P$ uniquely determines an associated biordered set $E = \sfE(P)$ in the sense that for every regular $*$-semigroup $S$ with $\sfP(S)\cong P$ we have $\sfE(S)\cong E$; see \cite[Theorem 6.19]{EGMR}.
\end{enumerate}
As justification for \ref{it:RS11}, note that $p\leq_\F q \implies p = pqp \implies p \mr\R pq$.  Conversely, by \ref{it:RS2} we have $p \mr\R pq \implies p = pq(pq)^* = pqp \implies p\leq_\F q$.

\subsection[Free projection-generated semigroups $\PG(P)$]{\boldmath Free projection-generated semigroups $\PG(P)$}
\label{ss:PGP}

Let $P$ be a projection algebra, and let $S$ be a regular $*$-semigroup such that $P=\sfP(S)$.
The \emph{free projection-generated semigroup} $\PG(P)$ was constructed in \cite{EGMR} using categorical machinery.  By \cite[Theorem 7.2]{EGMR}, we can take $\PG(P)$ to be the semigroup defined by the presentation
\[
\bigl\langle X_P \mid 
x_p^2=x_p,\ (x_px_q)^2=x_px_q,\ x_px_qx_p=x_{pqp}\ (p,q\in P)\bigr\rangle,
\]
where $X_P:=\{ x_p\colon p\in P\}$, and where the product $pqp$ is taken in $S$.  Note that $pq$ itself need not belong to $P$, but $pqp$ always does; see~\ref{it:RS3} and \ref{it:RS7}.  
The following facts include some that are  proved in \cite{EGMR}, and some simple consequences:

\begin{enumerate}[label=\textup{\textsf{(PG\arabic*)}},leftmargin=13mm]
\item\label{it:PG1}
The mapping $x_p\mapsto p$ extends to a homomorphism $\pi\colon \PG(P)\rightarrow S$.
The image of $\pi$ is the (regular $*$-)subsemigroup $\langle P\rangle$ of $S$ generated by the projections, which coincides with the subsemigroup generated by the idempotents.
See \cite[Theorem 5.8]{EGMR}.
\item\label{it:PG2}
$\sfP(\PG(P))=X_P$, and the restriction of $\pi$ to $X_P$ is an isomorphism of projection algebras.
Again see \cite[Theorem 5.8]{EGMR}.
\item\label{it:PG3}
Consequently, the homomorphism $\pi$ induces bijections between the sets of $\D$-, $\R$-, $\L$- and $\H$-classes of $\PG(P)$ and $\langle P\rangle$.
It also induces an isomorphism between the biordered sets of idempotents of $\PG(P)$ and $\langle P\rangle$.
\item\label{it:PG4}
It also follows that for every projection $p\in \sfP(S)$, the homomorphism $\pi$ induces a surmorphism from the maximal subgroup of $\PG(P)$ containing $x_p$ onto the maximal subgroup of $\langle P\rangle $ containing~$p$.
\item\label{it:PG5}
Letting $E = \sfE(P)$ be the biordered set associated with $P$ as per \ref{it:RS12},
the mapping ${x_e\mapsto x_{ee^*}x_{e^*e}}$ ($e\in E$) induces a surmorphism $\theta\colon \IG(E)\rightarrow \PG(P)$.
Recall here that~$x_e$ is a typical generator of $\IG(E)$, and that $ee^*$ and $e^*e$ are both projections by \ref{it:RS2}.
See \cite[p45 and the proof of Theorem 7.10]{EGMR}.
\item\label{it:PG6}
Again, the surmorphism $\theta$ induces bijections between the sets of regular $\D$-, $\R$- and $\L$-classes of $\IG(E)$ and (all) the $\D$-, $\R$- and $\L$-classes of $\PG(P)$, an isomorphism between their biordered sets of idempotents, and surmorphisms from maximal subgroups of $\IG(E)$ onto the corresponding maximal subgroups of $\PG(P)$.
See \ref{it:IG2}, \ref{it:RS12} and \ref{it:PG2}.
\item\label{it:PG7} Analogous statements hold about the relationship between $\RIG(E)$ and $\PG(P)$.
See \ref{it:RIG2}, \ref{it:RS12} and \ref{it:PG2}.
\end{enumerate}

\subsection[Some remarks about the relationships between $S$, $\PG(P)$, $\IG(E)$ and $\RIG(E)$]{\boldmath Some remarks about the relationships between $S$, $\PG(P)$, $\IG(E)$ and $\RIG(E)$}
\label{ss:rships}

The observations that follow are technically not needed for the subsequent development, but the reader may find them useful because they reflect the viewpoint and motivation of this paper; unless otherwise stated, each assertion follows from various combinations of items in the preceeding subsections.
Let $S$ be a projection-generated regular $*$-semigroup, and let $P:=\sfP(S)$ and $E:=\sfE(S)$.
Then we have a chain of natural surmorphisms
\[
\IG(E)\stackrel{\theta_{\sfI}}{\longrightarrow} \RIG(E)\stackrel{\theta_{\sfR}}{\longrightarrow} \PG(P)\stackrel{\theta_{\sfP}}{\longrightarrow} S.
\]
Results of the current paper can be used to show that none of these mappings are isomorphisms when $S = \la\sfE(\P_n)\ra$ is the idempotent-generated submonoid of the partition monoid of degree $n\geq3$ (see the next subsection for definitions):
\bit
\item $\theta_{\sfI}$ is not an isomorphism because $\RIG(E)$ is regular but $\IG(E)$ is not.  Indeed, it follows quickly from \cite[Theorem 3.6]{DG17} that any word of the form $x_ex_f$ with $(ef,e)\not\in\R$ and $(ef,f)\not\in\L$ represents a non-regular element of $\IG(E)$.  Such elements $e$ and $f$ are easy to find in $\P_n$. 
\item $\theta_{\sfR}$ is not an isomorphism because maximal subgroups corresponding to projections of rank $0\leq r\leq n-2$ are finite in $\PG(P)$, but infinite in $\RIG(E)$.  See Theorems \ref{thm:mainPGPn} and \ref{thm:mainIGPn}, and recall that $\RIG(E)$ and $\IG(E)$ have the same subgroups up to isomorphism.
\item Finally, $\theta_{\sfP}$ is not an isomorphism because $S$ is finite, but $\PG(P)$ is infinite.  Indeed, it follows from Corollary \ref{cor:PGPnn-1} that in the second-top $\D$-class of $\PG(P)$, maximal subgroups are free of rank $\binom{n-1}2\geq1$.
\eit
However, the following strong links are present (for any $S$):
\bit
\item
Each of $\theta_{\sfI}$, $\theta_{\sfR}$ and $\theta_{\sfP}$ restricts to an isomorphism of biordered sets;
furthermore, $\theta_{\sfP}$ restricts to an isomorphism of projection algebras.
\item
Let $\theta$ be any of $\theta_{\sfI}$, $\theta_{\sfR}$ or $\theta_{\sfP}$, and let $D$ be a regular $\D$-class
in the domain of $\theta$ (which means \emph{any} $\D$-class in the case of $\theta_{\sfR}$ or $\theta_{\sfP}$).
Then $D\theta$ is a (regular) $\D$-class in the image of $\theta$, and~$\theta$ induces 
bijections between the $\R$-, $\L$- and $\H$-classes of $D$ and $D\theta$.
\item
Restricted to any group $\H$-class, $\theta_{\sfI}$ is a group isomorphism,
while $\theta_{\sfR}$ and $\theta_{\sfP}$ are group surmorphisms.
The latter two are not isomorphisms in general (for example when ${S=\la\sfE(\P_n)\ra}$, as above).
\item
As a consequence, $\IG(E)$ and $\RIG(E)$ do not in general admit a regular $*$-structure that would be natural in the sense of being compatible with the above scheme of surmorphisms.
\eit
The topic of this paper is to analyse the maximal subgroups of $\IG(E)$ (equivalently $\RIG(E)$),~$\PG(P)$ and $S$ in relation to each other.

\subsection[The partition monoid $\P_n$ and the full transformation monoid $\T_n$]{\boldmath The partition monoid $\P_n$ and the full transformation monoid $\T_n$}
\label{ss:Pn}

Let $X$ be a set, and $X' = \set{x'}{x\in X}$ a disjoint copy.  The \emph{partition monoid} $\P_X$ consists of all set partitions of $X\cup X'$.  Such a partition $a\in\P_X$ is identified with any graph on vertex set $X\cup X'$ whose connected components are the blocks of $a$; such a graph is typically drawn with vertices from $X$ and $X'$ on a top and bottom row, respectively.  Partitions $a,b\in\P_X$ are multiplied as follows.  First, let $X''=\set{x''}{x\in X}$ be another copy of $X$, and define three further graphs:
\bit
\item $a^\vee$, the graph on vertex set $X\cup X''$ obtained by changing each lower vertex $x'$ of $a$ to~$x''$,
\item $b^\wedge$, the graph on vertex set $X''\cup X'$ obtained by changing each upper vertex $x$ of $b$ to~$x''$,
\item $\Ga(a,b)$, the graph on vertex set $X\cup X''\cup X'$ whose edge set is the union of the edge sets of $a^\vee$ and $b^\wedge$.
\eit
We call $\Ga(a,b)$ the \emph{product graph} of $a$ and $b$.  The product $ab$ is the unique element of $\P_X$ such that $x,y\in X\cup X'$ belong to the same block of $ab$ if and only if they are in the same connected component of $\Ga(a,b)$.  When displaying product graphs, the vertices from $X''$ are shown on a middle row.
When $X = [n] = \{1,\ldots,n\}$ for some integer $n\geq0$, we write $\P_n$ for $\P_X$.  Unless otherwise specified, vertices are shown in the natural order, $1<\cdots<n$.  An example product is given in Figure \ref{fig:P6}, for
\begin{align*}
a &= \big\{\{1,4\},\{2,3,4',5'\},\{5,6\},\{1',2',6'\},\{3'\}\big\} ,\\
b &= \big\{\{1,2\},\{3,4,1'\},\{5,5',6'\},\{6\},\{2',3'\},\{4'\}\big\}, \\
\text{and}\quad ab &= \big\{\{1,4\},\{2,3,1',5',6'\},\{5,6\},\{2',3'\},\{4'\}\big\}.
\end{align*}
The partition monoids $\P_n$ contain many important diagram monoids as submonoids, including the Brauer, Temperley--Lieb and Motzkin.  For example, the Brauer monoid~$\B_n$ consists of all partitions whose blocks have size $2$.
The partition monoids $\P_n$ also contain (isomorphic copies of) the full transformation monoids~$\T_n$, as we will soon see.

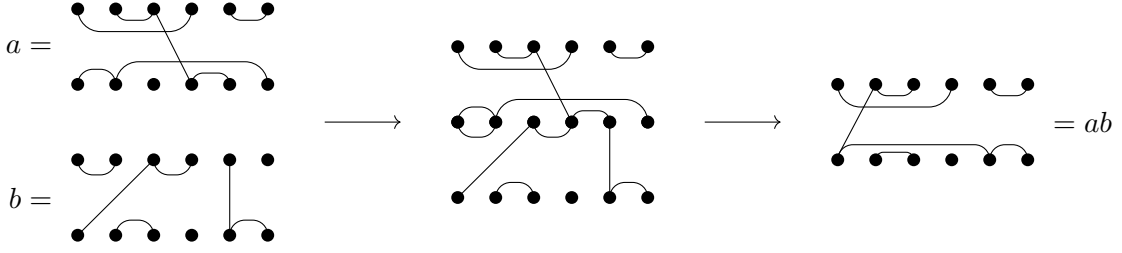
\begin{figure}[t]
\begin{center}
\begin{tikzpicture}[scale=.5]

\begin{scope}[shift={(0,0)}]	
\olduvs{1,...,6}
\oldlvs{1,...,6}
\olduarcx14{.6}
\olduarcx23{.3}
\olduarcx56{.3}
\olddarc12
\olddarcx26{.6}
\olddarcx45{.3}
\oldstline34
\draw(0.6,1)node[left]{$a=$};
\draw[->](7.5,-1)--(9.5,-1);
\end{scope}

\begin{scope}[shift={(0,-4)}]	
\olduvs{1,...,6}
\oldlvs{1,...,6}
\olduarc12
\olduarc34
\olddarc56
\olddarc23
\oldstline31
\oldstline55
\draw(0.6,1)node[left]{$b=$};
\end{scope}

\begin{scope}[shift={(10,-1)}]	
\olduvs{1,...,6}
\oldlvs{1,...,6}
\olduarcx14{.6}
\olduarcx23{.3}
\olduarcx56{.3}
\olddarc12
\olddarcx26{.6}
\olddarcx45{.3}
\oldstline34
\draw[->](7.5,0)--(9.5,0);
\end{scope}

\begin{scope}[shift={(10,-3)}]	
\olduvs{1,...,6}
\oldlvs{1,...,6}
\olduarc12
\olduarc34
\olddarc56
\oldstline31
\oldstline55
\olddarc23
\end{scope}

\begin{scope}[shift={(20,-2)}]	
\olduvs{1,...,6}
\oldlvs{1,...,6}
\olduarcx14{.6}
\olduarcx23{.3}
\olduarcx56{.3}
\olddarc15
\olddarc56
\oldstline2{.95}
\olddarcx23{.2}
\draw(6.4,1)node[right]{$=ab$};
\end{scope}

\end{tikzpicture}
\caption{Multiplication of partitions $a,b\in\P_6$, with the product graph $\Ga(a,b)$ shown in the middle.}
\label{fig:P6}
\end{center}
\end{figure}

For a set $X$, we call a non-empty subset $A$ of $X\cup X'$:
\bit
\item an \emph{upper non-transversal} (or \emph{upper block}) if $A\sub X$,
\item a \emph{lower non-transversal} (or \emph{lower block}) if $A\sub X'$,
\item a \emph{transversal} otherwise.
\eit
If $a\in\P_X$ for finite $X$, we write
\[
a = \begin{partn}{6} A_1&\cdots&A_r&C_1&\cdots&C_s\\ \hhline{~|~|~|-|-|-} B_1&\cdots&B_r&D_1&\cdots&D_t\end{partn}
\]
to indicate that $a$ has transversals $A_i\cup B_i'$ (for $1\leq i\leq r$), upper non-transversals $C_i$ (for $1\leq i\leq s$), and lower non-transversals $D_i'$ (for $1\leq i\leq t$).  For example, with $a\in\P_6$ from Figure~\ref{fig:P6}, we have $a = \begin{partn}{3} 2,3 & 1,4 & 5,6 \\ \hhline{~|-|-} 4,5 & 1,2,6 & 3 \end{partn}$.  
We sometimes adopt the convention of not listing singleton blocks, so for example, $b = \begin{partn}{3} 3,4 & 5 & 1,2 \\ \hhline{~|~|-} 1 & 5,6 & 2,3 \end{partn}$.

The partition monoid $\P_X$ has a natural involution $a\mt a^*$, obtained by swapping dashed and undashed vertices.  For finite $X$, if $a = \begin{partn}{6} A_1&\cdots&A_r&C_1&\cdots&C_s\\ \hhline{~|~|~|-|-|-} B_1&\cdots&B_r&D_1&\cdots&D_t\end{partn}$, then $a^* = \begin{partn}{6} B_1&\cdots&B_r&D_1&\cdots&D_t\\ \hhline{~|~|~|-|-|-} A_1&\cdots&A_r&C_1&\cdots&C_s\end{partn}$.  Graphically, this corresponds to a reflection in the horizontal axis.  This involution gives $\P_X$ the structure of a regular $*$-semigroup, as it is easy to check that the identities \eqref{eq:*} hold.
The above-mentioned submonoids (Brauer, Temperley--Lieb, etc) are closed under ${}^*$, and are hence themselves regular $*$-semigroups.

The \emph{domain} of a partition $a= \begin{partn}{6} A_1&\cdots&A_r&C_1&\cdots&C_s\\ \hhline{~|~|~|-|-|-} B_1&\cdots&B_r&D_1&\cdots&D_t\end{partn}\in\P_n$ is the set 
${\dom(a):=A_1\cup\cdots \cup A_r}$.
The \emph{kernel} of $a$ is the equivalence relation $\ker(a)$ on $[n]$ with classes
$A_1,\ldots,A_r,C_1,\ldots,C_s$.
The \emph{codomain} $\codom(a)$ and \emph{cokernel} $\coker(a)$ are defined dually, using the sets $B_1,\ldots, B_r,D_1,\ldots, D_t$.
The \emph{rank} of $a$ is the number of transversals, $\rank(a) = r$.
Green's equivalences on~$\P_n$ can be described using these parameters as follows; for proofs see \cite{Wi07,FL11}.

\begin{lemma}\label{la:Green_Pn}
For $a,b\in\P_n$, we have
\begin{thmenumerate}
\item $a \mr\R b \iff \dom(a)=\dom(b)$ and $\ker(a)=\ker(b)$;
\item $a \mr\L b \iff \codom(a)=\codom(b)$ and $\coker(a)=\coker(b)$;
\item $a \mr\D b \iff  \rank(a)=\rank(b)$;
\item every group $\H$-class of $\P_n$ consisting of partitions of rank $r$ is isomorphic to the symmetric group $\S_r$.
\end{thmenumerate}
\end{lemma}

The full transformation monoid $\T_n$ can be naturally viewed as the submonoid of $\P_n$ consisting of all partitions with full domain and trivial cokernel.
For such a partition
$a=\begin{partn}{6} A_1 & \cdots & A_r & \multicolumn{3}{c}{}\\ \hhline{~|~|~|-|-|-} b_1&\cdots & b_r& c_1&\cdots& c_{n-r}\end{partn}$ we will often use the more familiar transformation notation
$a =\trans{A_1 & \cdots & A_r}{b_1&\cdots&b_r}$.
It is well known that~$\T_n$ is a regular semigroup, and so it inherits Green's $\L$ and $\R$ equivalences from $\P_n$ by \cite[Proposition 2.4.2]{Ho95}.  The same happens to be true of the $\D$ relation, and one obtains the familiar description:

\newpage

\begin{lemma}\label{la:Green_Tn}
For $a,b\in\T_n$, we have
\begin{thmenumerate}
\item $a \mr\R b \iff \ker(a)=\ker(b)$;
\item $a \mr\L b \iff \codom(a)=\codom(b)$;
\item $a \mr\D b \iff  \rank(a)=\rank(b)$;
\item every group $\H$-class of $\T_n$ consisting of transformations of rank $r$ is isomorphic to the symmetric group $\S_r$.
\end{thmenumerate}
\end{lemma}

Note that $\T_n$ is not closed under the involution ${}^*$.  Rather, $\T_n^* = \set{a^*}{a\in\T_n}$ is an anti-isomorphic copy of $\T_n$.

\section{Reidemeister--Schreier rewriting for maximal subgroups}
\label{sec:rs}

In this section we outline a general method for finding presentations for maximal subgroups of semigroups given by presentations, which was introduced in \cite{Ru99}, and has many similarities with the Reidemeister--Schreier rewriting theory for groups \cite[Section II.4]{LS01}.
This method was used in \cite{GR12IJM} to establish a presentation for maximal subgroups of free idempotent-generated semigroups $\IG(E)$, of which we will present a brief review in Section \ref{sec:maxIG}.
We will deploy it twice more, 
in Sections \ref{sec:pres-linked} and \ref{sec:pres-quot}, to find two different presentations for maximal subgroups of free projection-generated semigroups $\PG(P)$.

Let $S$ be the semigroup defined by a (semigroup) presentation $\langle X\mid\rels\rangle$, and let 
$e_0\in X^+$ be a word representing an idempotent in $S$.
We want to exhibit a presentation for the maximal subgroup~$G$ of~$S$ containing $e_0$, i.e.~for the group $\H$-class $H_{e_0}$.

Let $D=D_{e_0}$ be the $\D$-class of $e_0$.  Index the $\R$- and $\L$-classes in $D$ as $R_i$ ($i\in I$), and~$L_j$~($j\in J$).  For $i\in I$ and $j\in J$, let $H_{i,j}:=R_i\cap L_j$ -- these are the $\H$-classes in $D$.  Let $i_0\in I$ and~${j_0\in J}$ be such that $e_0\in H_{i_0,j_0}$; thus $G=H_{i_0,j_0}$.  Let 
\begin{equation}\label{eq:K}
K:=\big\{ (i,j)\in I\times J\colon H_{i,j} \text{ is a group}\big\}.
\end{equation}
For $(i,j)\in K$, let $e_{i,j}\in X^+$ be any word that represents the identity of $H_{i,j}$, with $e_{i_0,j_0}=e_0$.

By Green's lemma there is an action of $S$ on the set $(R_{i_0}/\H)\cup\{0\} = \set{H_{i_0,j}}{j\in J} \cup \{0\}$, consisting of the $\H$-classes contained in $R_{i_0}$ plus an extra symbol $0$, given by right multiplication.
This is equivalent to an action of $S$ on the set $J\cup\{0\}$, determined by 
\begin{equation}\label{eq:j.s}
j\cdot s=
\begin{cases} 
k & \text{if } H_{i_0,j}s=H_{i_0,k}\\ 
0 & \text{otherwise}.
\end{cases}
\end{equation}

A set of \emph{Schreier representatives} for this action is a prefix-closed set $\set{r_j}{j\in J}$ of words from~$X^*$, such that
$j_0\cdot r_j=j$ for  every $j\in J$.
It is routine to see that such a set must exist.  One way is to use the action to define a directed graph on vertex set $J$, with edges labelled by the generators (from $X$), observe that the graph must be strongly connected (i.e.~there is a directed path from any vertex to any other vertex) by the definition of the $\R$ equivalence, and then take a directed spanning tree rooted at $j_0$.  By this we mean a subgraph whose underlying undirected graph is a tree, such that there exists a (unique) directed path from~$j_0$ to every other vertex.  The edge-labels on these paths determine the representatives.

Also, by the definition of $\R$, and Green's Lemma, there exist words $\set{r_j'}{j\in J}$ from $X^*$ such that
$ar_jr_j'=a$ for all $a\in H_{i_0,j_0}$ and $br_j'r_j=b$ for all $b\in H_{i_0,j}$.
We note that nothing further is required from the words $r_j'$; e.g.~they need not be prefix-closed.

By \cite[Theorem 2.7]{Ru99}, the maximal subgroup $G$ 
 is generated by the set
 \[
 \{ e_0r_jxr_{j\cdot x}'\colon j\in J,\ x\in X,\ j\cdot x\neq 0\}.
 \]
So we introduce a new alphabet
\[B:=\bigl\{ [j,x]\colon j\in J,\ x\in X,\ j\cdot x\neq 0\bigr\},
\]
where $[j,x]$ is a letter/symbol, which we think of as representing the generator $e_0r_jxr_{j\cdot x}'$. The presentation for~$G$ will be written in terms of the alphabet $B$.

The defining relations for $G$ are obtained by \emph{rewriting} in some sense the defining relations~$\rels$ of $S$. To make this formal, we need the
\emph{rewriting mapping}:
\[
\phi\colon \{ (j,w)\in J\times X^\ast\colon j\cdot w\neq 0\}\rightarrow B^\ast,
\]
defined inductively by
\begin{equation}
\label{eq:rewr}
\phi(j,1)=1 \ANd \phi(j,xw)=[j,x]\phi(j\cdot x,w) \text{ for } w\in X^* \text{ with } j\cdot xw \not= 0.
\end{equation}

\begin{thm}[{see \cite[Theorem 2.9]{Ru99}}]
\label{thm:RSrw}
With the notation as above, the maximal subgroup $G$ is defined by the group presentation
\begin{alignat}{3}
\label{eq:gpIG1}
\bigl\langle B \mid\: & [j,x]=1 &\quad& (j\in J,\ x\in X,\ r_jx=r_{j\cdot x}),\\
\label{eq:gpIG2}
& \phi(j,u)=\phi(j,v)&& (j\in J,\ (u=v)\in \rels,\ j\cdot u\neq 0)\bigr\rangle.
\end{alignat}
\end{thm}
 
\section{\boldmath Maximal subgroups of $\IG(E)$}
\label{sec:maxIG}

\subsection{Applying the general rewriting}
\label{ss:genrew}

The above theory was applied to the case of maximal subgroups of $\IG(E)$ in \cite{GR12IJM}.
Here we briefly review the steps, with a two-fold aim: 1) we want to make a small modification to the way the final presentation is written out; and 2) we will want to re-run this process in the special case where $E$ comes from a regular $*$-semigroup in Section \ref{sec:pres-quot}, which will yield a comparison between maximal subgroups of $\IG(E)$ and $\PG(P)$.

Let $E$ be a biordered set.  As previously discussed, we may take $E$ to be embedded in a semigroup~$S$ so that $E=\sfE(S)$.  There are two main reasons for adopting this viewpoint,
as opposed to working with $E$ as an abstract biordered set, as specified by the axioms. Firstly, while this axiomatics is natural and reflects the fundamental properties of idempotent structure, it is certainly more cumbersome than working within a single semigroup. And, secondly, our main results in the second half of the paper will follow this pattern: we will take a concrete semigroup, specifically the partition monoid $\P_n$, and investigate the maximal subgroups of the free idempotent- and projection-generated semigroups arising from its idempotents or projections.

In what follows, we will further assume that $S$ is idempotent-generated. If it were not, we would replace it with the subsemigroup $\langle E\rangle$ generated by the idempotents.

Recall the presentational definition of $\IG(E)$ given in Subsection \ref{ss:IGE}:
\begin{equation}
\label{eq:IGpres}
\IG(E) := \bigl\langle X_E\mid x_ex_f=x_{ef}\ ((e,f)\in \Bp)\bigr\rangle ,
\end{equation}
where $X_E:=\{ x_e\colon e\in E\}$, and 
$\Bp:=\bigl\{ (e,f)\in E\times E\colon \{ef,fe\}\cap \{e,f\}\neq \es\bigr\}$.
Since we assumed that $S$ is idempotent-generated, it follows that the natural homomorphism $\pi\colon\IG(E)\rightarrow S$, $x_e\mapsto e$, is surjective.
Suppose now that we are given a fixed idempotent $e_0\in E$, and we want to obtain a presentation for the maximal subgroup $\GI = H_{x_{e_0}}$  of $\IG(E)$ containing $x_{e_0}$.
We apply the method from the previous section, so we carry the notation over, including the sets 
\[
D = D_{e_0} \COMMA R_i\ (i\in I) \COMMA L_j\ (j\in J) \COMMA H_{i,j} = R_i\cap L_j \AND K.
\]
We will also write $E_D = E\cap D = \set{e_{i,j}}{(i,j)\in K}$ for the set of idempotents in $D$, and we let $i_0\in I$ and $j_0\in J$ be such that $e_0\in H_{i_0,j_0}$.

However, there are two important technical subtleties that the reader will need to keep in mind, as we now record:

\begin{rem}
A key aspect in Section \ref{sec:rs}
is the action of the semigroup under consideration on the $\H$-classes within a fixed $\R$-class.
So, for us here, that would be the action of $\IG(E)$ on the $\H$-classes in the $\R$-class of $x_{e_0}$ in $\IG(E)$. 
However, the natural surmorphism $\pi$ induces a bijection between the $\L$-classes in 
the $\D$-class of $x_{e_0}$ in $\IG(E)$ and those in the $\D$-class of $e_0$ in $S$ by \ref{it:IG2}.
In turn, this induces a bijection between the $\H$-classes in the $\R$-class of $x_{e_0}$ in~$\IG(E)$ and the $\H$-classes in the $\R$-class of $e_0$ in $S$.
Furthermore, 
each $x_e$ acts on the former in the same way as  $e$ acts on the latter.
A complete formal proof that these two actions are the same is given in \cite[Lemma 2.8]{DG17}.  
Since it is $S$ that is \emph{a priori} given to us, we will throughout use its actions in place of those of $\IG(E)$.
This point of view was also adopted and explained in \cite{GR12IJM}.
\end{rem}

\begin{rem}
\label{re:Schrestr}
\emph{A priori}, a set of Schreier representatives would be a set of words over the generating set $X_E$ of $\IG(E)$.
However, recalling that we are assuming that $S$ is idempotent-generated, it follows by Proposition \ref{prop:DF}  that 
$D\subseteq \langle E_D\rangle$. Repeating the reasoning for the existence of Schreier representatives from Section \ref{sec:rs}, but using only edges labelled by elements of $E_D$, yields a set of Schreier representatives over the alphabet $\{x_e\colon e\in E_D\}$.
\end{rem}

The presentation for $\GI$ given in Theorem \ref{thm:RSrw} has generators
\[
B:= \bigl\{ [j,x_e]\colon j\in J,\ e\in E,\ j\cdot e\neq 0\bigr\}
\]
and defining relations
\begin{alignat}{2}
\label{eq:IG1}
&[j,x_e]=1 &\quad& (j\in J,\ e\in E,\ r_jx_e = r_{j\cdot e}),\\
\label{eq:IG22}
&[j,x_e][j\cdot e, x_f]=[j,x_{ef}] && (j\in J,\ (e,f)\in\Bp,\ j\cdot ef\neq 0).
\end{alignat}

This presentation is transformed in \cite{GR12IJM} in the following way.
For every $i\in I$ pick any $\mapj(i)\in J$ such that $(i,\mapj(i))\in K$; this exists because $D$ is regular. 
Introduce new generating symbols via
\begin{equation}
\label{eq:IG4}
\fg{i,j} := [\mapj(i), x_{e_{i,j}}] \text{ for }  (i,j)\in K,
\quad\text{and set}\quad
A:=\bigl\{ \fg{i,j}\colon (i,j)\in K\bigr\}.
\end{equation}
The original generators can be eliminated as follows.
Consider an arbitrary $[t,x_e]\in B$.
Let $j:= t\cdot e$.
By \cite[Lemma 3.3]{GR12IJM}, there exists $i\in I$ such that $(i,t),(i,j)\in K$ and $es=s$ for all $s\in R_i$,
and then \cite[Lemma 3.1]{GR12IJM} yields
\begin{equation}
\label{eq:IG4a}
[t,x_e]=\fg{i,t}^{-1} \fg{i,j}. 
\end{equation}
Performing this elimination 
yields the following presentation for the maximal subgroup $\GI$ of $\IG(E)$ containing $x_{e_0}$ (see \cite[Theorem 5]{GR12IJM}):
\begin{alignat}{3}
\label{eq:IG5}
\bigl\langle A\mid\ && &\fg{i,j}=\fg{i,l} && ((i,j),(i,l)\in K,\ r_j x_{e_{i,l}} = r_l),\\
\label{eq:IG6}
&& &\fg{i,\mapj(i)}=1&& (i\in I),\\
\label{eq:IG7}
&& & \fg{i,j}^{-1} \fg{i,l} = \fg{k,j}^{-1} \fg{k,l} &\quad& (\smat{e_{i,j}}{e_{i,l}}{e_{k,j}}{e_{k,l}}\in  \Sq)\bigr\rangle,
\end{alignat}
where $\Sq$ stands for the set of all singular squares of idempotents of $D$, as defined in Subsection~\ref{ss:BS}.  Recall that if $\smat{e_{i,j}}{e_{i,l}}{e_{k,j}}{e_{k,l}}$ is an LR singular square, then $\smat{e_{i,l}}{e_{i,j}}{e_{k,l}}{e_{k,j}}$ is an RL-square; clearly these squares give equivalent relations of the form \eqref{eq:IG7}.  Similarly, matching pairs of UD- and DU-squares give equivalent relations.  A relation \eqref{eq:IG7} coming from a degenerate square is trivially true; i.e.~it holds in the free group over $A$.

Before we move on, we make a brief historical remark about the presentation \eqref{eq:IG5}--\eqref{eq:IG7}.  
A topological version of the above presentation, in the case of $\RIG(E)$  ($E$ regular), is given in \cite{BM09} and attributed to Nambooripad.  
In the same paper it is shown that  the maximal subgroups of $\RIG(E)$ and $\IG(E)$ corresponding to the same idempotent are isomorphic, with the consequence that \eqref{eq:IG5}--\eqref{eq:IG7} is indeed a presentation for $\GI$ when $E$ is regular.  The general case, where $E$ need not be regular, was established in \cite{GR12IJM}.

\subsection{Spanning trees}
\label{ss:span}

For our purposes here, we want to recast the presentation \eqref{eq:IG5}--\eqref{eq:IG7} in a slightly different way, using the language of spanning trees.  We keep the notation of Subsection \ref{ss:genrew}.

Since the $\D$-class $D$ is regular, its Graham--Houghton graph~$\GH(D)$ is connected by Proposition \ref{pr:GHconn}.  We can therefore fix a spanning tree $T$ of $\GH(D)$. We will identify $T$ with its set of edges, 
 i.e.~we consider it as a subset of $E_D = E\cap D$.
 We claim that $T$ yields a natural set of Schreier representatives $r_j$ and elements~$\mapj(i)$.
 Indeed, for $j\in J$, consider the unique path from~$j_0$ to~$j$ using edges from $T$:
 \begin{equation}
 \label{eq:path}
 j_0- i_1-j_1-\cdots- i_{k-1}-j_{k-1}-i_k-j_k=j.
 \end{equation}
This means that $(i_1,j_0),(i_1,j_1),\ldots, (i_k,j_{k-1}),(i_k,j_k)\in K$.  The path \eqref{eq:path} is empty if $j=j_0$.
We now define
\[
r_j:= x_{e_{i_1,j_1}} x_{e_{i_2,j_2}}\cdots x_{e_{i_k,j_k}},
\]
i.e.~the product of generators corresponding to every other edge in the path, starting with $i_1-j_1$.  Note that $r_{j_0}$ is necessarily the empty word.
We claim that $\{ r_j\colon j\in J\}$ is a set of Schreier representatives.
Note that we have $H_{i_0,j_{t-1}}e_{i_t,j_t}=H_{i_0,j_t}$ because $H_{i_t,j_{t-1}}$ is a group.
Therefore $j_{t-1}\cdot x_{e_{i_t,j_t}}=j_t$, and so
\[
j_0\cdot r_j=j_0\cdot x_{e_{i_1,j_1}} x_{e_{i_2,j_2}}\cdots x_{e_{i_k,j_k}}
=j_1x_{e_{i_2,j_2}}\cdots x_{e_{i_k,j_k}}
=\cdots=j_k=j.
\]
To show that the $r_j$ are prefix closed consider an arbitrary prefix $x_{e_{i_1,j_1}} \cdots x_{e_{i_t,j_t}}$ of $r_j$.
The sequence $ j_0- i_1-j_1-\cdots- i_t-j_t$ is a subpath of \eqref{eq:path}, and hence the unique path from $j_0$ to~$j_t$ using edges from $T$.
Hence $x_{e_{i_1,j_1}} \cdots x_{e_{i_t,j_t}} = r_{j_t}$, completing the proof that $\{ r_j\colon j\in J\}$ is a system of Schreier representatives.
Furthermore, in reference to Remark \ref{re:Schrestr}, we note that, by construction, each $r_j$ is a word over the alphabet $\{ x_e\colon e\in E_D\}$.

Moving on to selecting the elements $\mapj(i)$, suppose $i\in I$ is arbitrary, and let 
 \[
 j_0- i_1-j_1-\cdots- i_{k-1}-j_{k-1}-i_k=i
\]
be the unique path from $j_0$ to $i$ using edges from $T$.
Then define $\mapj(i):= j_{k-1}$. It is clear that $(i,\mapj(i))\in K$.

Consider the presentation \eqref{eq:IG5}--\eqref{eq:IG7} arising from this choice of the $r_j$ and $\mapj(i)$.
In particular, consider an arbitrary relation \eqref{eq:IG5}.
From the condition $r_jx_{e_{i,l}}=r_l$ (and keeping in mind \eqref{eq:path}), it follows that the final two edges in the unique $T$-path from $j_0$ to $l$ are $e_{i,j}$ and $e_{i,l}$. Therefore $j=\mapj(i)$, and so, in the presence of relations \eqref{eq:IG6}, relations \eqref{eq:IG5} can be simplified to
\begin{equation}
\label{eq:IG6a}
\fg{i,l}=1 \quad  \text{whenever } x_{e_{i,l}} \text{ is the last letter of } r_l.
\end{equation}
Now notice that for every $e_{i,j}$ in $T$ one of the following two conditions holds:
\bit
\item
$j$ immediately precedes $i$ in the unique $T$-path from $j_0$ to $i$, and so $j=\mapj(i)$; or
\item
$i$ immediately precedes $j$ in the unique $T$-path from $j_0$ to $j$, in which case $e_{i,j}$ is the final letter of~$r_j$. 
\eit
Therefore the relations \eqref{eq:IG6} and \eqref{eq:IG6a} can together be written as
\[
\fg{i,j}=1\quad (e_{i,j}\in T).
\]
Renaming the letters $\fg{i,j}$ into $a_{e_{i,j}}$, so that $A=\{ a_e\colon e\in E_D\}$, we have proved the following:

\begin{thm}[cf.~{\cite[Theorem 5]{GR12IJM}}]
\label{thm:spantreepres}
With the notation as above, and $T$ a spanning tree for the Graham--Houghton graph $\GH(D)$, the maximal subgroup of $\IG(E)$ containing $x_{e_0}$ is defined by the presentation
\begin{alignat}{3}
\nonumber\bigl\langle A \mid\: & a_e=1 && (e\in T),\\
 &
  a_e^{-1}a_f=a_g^{-1}a_h &\quad& (\smat{e}{f}{g}{h}\in  \Sq)\bigr\rangle.
 \tag*{\qed}
\end{alignat}
\end{thm}

Motivated by this, we introduce some notation for future use:

\begin{defn}
For a subset $F\subseteq E_D$, let:
\bit
\item
$\rels_1(F):=\big\{ (a_e=1)\colon e\in F\big\}$;
\item
$\rels_{\subSq}:=\big\{ (a_e^{-1}a_f=a_g^{-1}a_h)\colon \smat{e}{f}{g}{h}\in\Sq\big\}$;
\item
$\presn(F):=\big\langle A\mid \rels_1(F),\ \rels_{\subSq}\big\rangle$.
\eit
\end{defn}

With this notation, Theorem \ref{thm:spantreepres} asserts that $\GI$, the maximal subgroup of $\IG(E)$ containing~$x_{e_0}$, is defined by $\presn(T)$,
where $T$ is any spanning tree of the Graham--Houghton graph~$\GH(D)$.
One can arrive at the same result by using the fact that $\GI$ is the fundamental group of the 
\emph{Graham--Houghton complex} of $D$ \cite[Section 3]{BM09}. However, in what follows we will deploy 
variants of the above argument in situations where such a topological interpretation is not so readily available.

\begin{rem}
\label{rem:noss}
When there are no singular squares, 
the Graham--Houghton complex is simply a graph, and its fundamental group is free of rank 
$|E_D\setminus T|$, which is equal to $|E_D|-|I|-|J|+1$ when $E_D$ is finite. 
This is also explained in \cite{BM09} after Theorem 4.1, and
 can be seen as an immediate consequence of Theorem \ref{thm:spantreepres}.
Furthermore, if $S$ is a regular $*$-semigroup, we have $|I|=|J|=|P_D|$, the number of projections in $D$, and hence the rank is $|E_D|-2|P_D|+1$. 
When $E=\sfE(\P_n)$ and $D$ is the $\D$-class of partitions of rank $n-1$, there are no singular squares
in $D$ because the identity is the only idempotent $\J$-above $D$. 
In this case we have $|P_D| = n + \binom n2$, and $|E_D| = n + 5\binom n2$; see \cite{EG17}.
Hence, the maximal subgroups in $\IG(E)$ corresponding to idempotents of rank $n-1$ are free
of rank $|E_D|-2|P_D|+1=\frac{(n-1)(3n-2)}{2}$.
\end{rem}

\section{Maximal subgroups in \boldmath{$\IG(\sfE(\T_n))$}}
\label{sec:IGETn}

When we come to determining the maximal subgroups of the free projection-generated regular $*$-semigroup 
$\PG(\sfP(\P_n))$ and the free idempotent-generated semigroup $\IG(\sfE(\P_n))$ arising from the partition monoid $\P_n$,
the key idea will be to make a series of reductions towards the maximal subgroups of the free idempotent-generated semigroup $\IG(E(\T_n))$ arising from the full transformation monoid $\T_n$.
These latter maximal subgroups have been determined in \cite{GR12}. Specifically, the maximal subgroup containing an idempotent of rank $r\leq n-2$ is isomorphic to the symmetric group $\S_r$.
The starting presentation in \cite{GR12} arises from a specific choice of Schreier representatives and elements $\mapj(i)$. For our purposes here, we need a presentation that arises from a spanning tree.  It is important, however, that the two presentations are \emph{equivalent}, in the sense that they have the same generators, and the two sets of defining relations imply each other; this means that any consequence of the relations from \cite{GR12} are also consequences of ours.

Now, in the previous subsection we have seen that every spanning tree gives rise to a natural choice of Schreier representatives and elements $\mapj(i)$. Unfortunately, the Schreier representatives and the $\mapj(i)$ from \cite{GR12} do not arise in such a way. Nonetheless, we are able to identify a spanning tree that works, with a little bit of extra justification.

Let us briefly review the set-up from \cite{GR12}; for full details see Section 2 in that paper.
Let $1\leq r\leq n-2$, and take $e_0:= \bigl(\begin{smallmatrix} 1&\cdots &r & r+1&\cdots & n\\ 1&\cdots &r& r&\cdots &r\end{smallmatrix}\bigr)$, a fixed idempotent of rank $r$ in $\T_n$.
Denote by~$D(n,r)$ the $\D$-class of~$e_0$ in $\T_n$; it consists of all mappings of rank $r$.
Based on Lemma \ref{la:Green_Tn}, we can
index  the  $\R$-classes in $D(n,r)$ by the set $I$ of all partitions of $[n] = \{1,\ldots,n\}$ into $r$ blocks, 
and the $\L$-classes by the set $J$  of all subsets of~$[n]$ of size $r$.
For $C\in J$ and $V\in I$ we write $C\perp V$ if $C$ is a cross-section of $V$.
This is precisely the necessary and sufficient condition for the existence of an idempotent transformation with kernel~$V$ and image~$C$, which we denote by $e_{V,C}$; in other words, $K=\bigl\{ (V,C)\in I\times J\colon C\perp V\bigr\}$.  If $(V,C)\in K$, and if $V = \{V_1,\ldots,V_r\}$ and $C = \{c_1,\ldots,c_r\}$ with each $c_i\in V_i$, then $e_{V,C} = \trans{V_1&\cdots&V_r}{c_1&\cdots&c_r}$.

For every partition $V=\{ V_1,\ldots, V_r\}$ in $I$, define
\[
\mapC(V):=\{ \min(V_1),\ldots,\min(V_r)\}\in J.
\]
This is the lexicographically smallest $C\in J$ with $C\perp V$.
Note that subsets from $J$ are ordered lexicographically as follows.  For $C=\{c_1<\cdots<c_r\}$ and $D=\{d_1<\cdots<d_r\}$, we have $C<_\lex D$ if there exists $1\leq k\leq r$ such that $c_i=d_i$ for $1\leq i<k$ and $c_k< d_k$.
The sets~$\mapC(V)$ will serve as the $\R$-class representatives, which were denoted by $\mapj(i)$ in the previous two sections.

Analogously, for every set $C=\{c_1<\cdots<c_r\}$ in $J$, define
\[
\mapV(C):=\bigl\{ [1,c_1],(c_1,c_2],\ldots, (c_{r-2},c_{r-1}],(c_{r-1},n]\bigr\}\in I.
\]
This is the lexicographically smallest $V\in I$ with $C\perp V$.  Note that partitions from $I$ are ordered lexicographically as follows.  For ${V=\{V_1<_\lex\cdots<_\lex V_r\}}$ and $W=\{W_1<_\lex\cdots<_\lex W_r\}$, we have $V<_\lex W$ if there exists $1\leq k\leq r$ such that $V_i=W_i$ for $1\leq i<k$ and $V_k<_\lex W_k$.

In the presentation established in Subsection \ref{ss:genrew}, relations \eqref{eq:IG6} are precisely  $\fg{V,\mapC(V)}=1$
for $V\in I$, while the relations \eqref{eq:IG5} take the form $\fg{V,C}=\fg{V,D}$ for some $V\in I$ and some ${C,D\in J}$.
Which combinations are taken depends on the actual choice of Schreier representatives, but we do not need this specific information. Key for us is to note that, by the construction in \cite{GR12}, $V$ is a \emph{convex} partition in each of those relations, i.e.~a partition in which all the blocks are intervals.
In other words, there exist convex partitions $V^{(1)},\ldots,V^{(t)}$ in $I$ and sets $C_1,\ldots,C_t, D_1,\ldots, D_t$ in $J$ so that the maximal subgroup $\GI$ is defined by the presentation
\begin{alignat}{3}
\label{eq:GRpres1}
\bigl\langle A \mid\: & \fg{V^{(i)},C_i}=\fg{V^{(i)},D_i} && (i=1,\ldots ,t),\\
\label{eq:GRpres2}
& \fg{V,\mapC(V)}=1 && (V\in I),\\
\label{eq:GRpres3}
& \fg{V,C}^{-1}\fg{V,D}=\fg{W,C}^{-1}\fg{W,D}&\quad& 
(\smat{e_{V,C}}{e_{V,D}}{e_{W,C}}{e_{W,D}}\in \Sq)\bigr\rangle.
\end{alignat}

\begin{prop}[cf.~{\cite[Main Theorem]{GR12}}]
\label{pr:treeTn}
The set
\[
\Tlex=\Tlex(n,r):=\big\{ e_{V,\mapC(V)}\colon V\in I\big\} \cup \big\{ e_{\mapV(C),C}\colon C\in J\big\}.
\]
is a spanning tree for the Graham--Houghton graph of $D(n,r)$ in $\T_n$.
The resulting presentation~$\presn(\Tlex) = \langle A\mid \rels_1(\Tlex),\ \rels_{\subSq}\rangle$ is equivalent to the presentation \eqref{eq:GRpres1}--\eqref{eq:GRpres3}, and defines the symmetric group $\S_r$.
\end{prop}

\begin{proof}
Observe that $e_0$ belongs to both sets forming $\Tlex$, so that $|\Tlex|\leq |I|+|J|-1$. Hence to show that $\Tlex$ is a spanning tree, it is sufficient to show that it is connected.

Note that for an arbitrary $C=\{c_1<\cdots <c_r\}\in J$ we have 
\[
\mapC(\mapV(C))=\{ 1,c_1+1,c_2+1,\ldots, c_{r-1}+1\}.
\]
Thus $\mapC(\mapV(C))\leq_\lex C$, with equality holding if and only if $C=[r]$.
In $\Tlex$ we have a path $C-\mapV(C)-\mapC(\mapV(C))$. 
An inductive argument now shows that every $C\in J$ is connected to $[r]$ by a path in $\Tlex$.
Also, any $V\in I$ is connected to $\mapC(V)\in J$ by an edge in $\Tlex$, and so again there is a path from $V$ to $[r]$. Hence $\Tlex$ is indeed connected, and so a tree.

Now consider the presentation \eqref{eq:GRpres1}--\eqref{eq:GRpres3}.
By \cite[Lemma 4.2]{GR12}, the relations \eqref{eq:GRpres1}--\eqref{eq:GRpres3} imply
all the relations $\fg{V,C}=1$ where $V$ is a convex partition and $C\perp V$.
In particular, they imply $\fg{\mapV(C),C}=1$ for all $C\in J$, and we add these relations
to the presentation. Noting that these relations together with \eqref{eq:GRpres2} are precisely the relations $\rels_1(\Tlex)$, we now have the presentation
$\langle A\mid \eqref{eq:GRpres1},\ \rels_1(\Tlex),\ \rels_\subSq\rangle$ for $\GI$.

The proof of the proposition will be complete if we can show that 
\eqref{eq:GRpres1} is a consequence of $\rels_1(\Tlex)\cup \rels_\subSq$. In fact we prove the following stronger statement, which is the analogue of \cite[Lemma 4.2]{GR12} for the new presentation $\presn(\Tlex)$:
for every convex partition $V\in I$ and every $C\in J$ with $C\perp V$, the relation $\fg{V,C}=1$ is a consequence of the relations 
$\rels_1(\Tlex)\cup \rels_\subSq$.

The proof is inductive with respect to the lexicographic ordering on the set of pairs under consideration, where both the set of all convex partitions, and the set $J$ are taken to be themselves ordered lexicographically.
First note that if $V=\bigl\{\{ 1\},\ldots,\{r-1\},[r,n]\bigr\}$ then we must have $C=\{1,\ldots,r-1,q\}$ for some $r\leq q\leq n$; in this case, $V=\mapV(C)$, and the relation $\fg{V,C}=1$ is included in $\rels_1(\Tlex)$. 
This includes the case where $C=[r]$.
Now consider any pair $(V,C)$ where $C\perp V$,  $V\neq \bigl\{\{ 1\},\ldots,\{r-1\},[r,n]\bigr\}$ and $C\neq [r]$.

We know that $\mapC(V)\leq_\lex C$ and $\mapV(C)\leq_\lex V$.
If either of these is an equality, the relation $\fg{C,V}=1$ is in $\rels_1(\Tlex)$, and we are done.
So we may suppose that $D:=\mapC(V)<_\lex C$ and $W:=\mapV(C)<_\lex V$.

We claim that $D\perp W$.
If $C=\{ c_1<\cdots<c_r\}$ and $V=\bigl\{ [v_1,w_1]<\cdots<[v_r,w_r]\bigr\}$ then
$D=\{ v_1,\ldots, v_r\}$ and $W=\bigl\{[1,c_1],(c_1,c_2],\ldots, (c_{r-2},c_{r-1}], (c_{r-1},n]\bigr\}$.
From $C\perp V$ we have $v_i\leq c_i\leq w_i$ for all $i$.
Therefore $c_{i-1}\leq w_{i-1}<v_i\leq c_i$ for all $i\geq 2$. Additionally, 
$1= v_1\leq c_1$, and so indeed $D\perp W$, as claimed.

We therefore have a square of idempotents $\smat{e_{V,C}}{e_{V,D}}{e_{W,C}}{e_{W,D}}$.
One can verify that this is in fact a rectangular band, by computing the diagonal products,
e.g.~$e_{V,C}e_{W,D}=e_{V,D}$. By \cite[Remark~2.8]{GR12} every rectangular band in $\T_n$ is a singular square, so $\smat{e_{V,C}}{e_{V,D}}{e_{W,C}}{e_{W,D}}\in \Sq$.
Thus, 
$\rels_\subSq$ contains the relation $\fg{V,C}^{-1}\fg{V,D}=\fg{W,C}^{-1}\fg{W,D}$.
From $D<_\lex C$ and $W<_\lex V$ it follows by induction that $\fg{V,D}=\fg{W,C}=\fg{W,D}=1$ are consequences
of $\rels_1(\Tlex)\cup \rels_\subSq$. Therefore $\fg{V,C}=1$ is also a consequence
of $\rels_1(\Tlex)\cup \rels_\subSq$, completing the proof.
\end{proof}

\begin{rem}\label{rem:Tlex42}
In Figure \ref{fig:Tlex42} we picture the Graham--Houghton graph $\GH(D(4,2))$, and the spanning tree $\Tlex(4,2)$.
The vertices from $I$ and $J$ are shown on the upper and lower rows, respectively, and are ordered lexicographically, from left to right.  Partitions are indicated by an obvious notation, so for example $(1|234)$ represents $\big\{\{1\},\{2,3,4\}\big\}$.  
The idempotents~$e_{V,\mapC(V)}$~$(V\in I$) and $e_{\mapV(C),C}$ ($C\in J$) are indicated by red and blue edges, respectively.  As noted at the start of the above proof, the idempotent $e_0$ is the unique edge coloured both red and blue.
\end{rem}

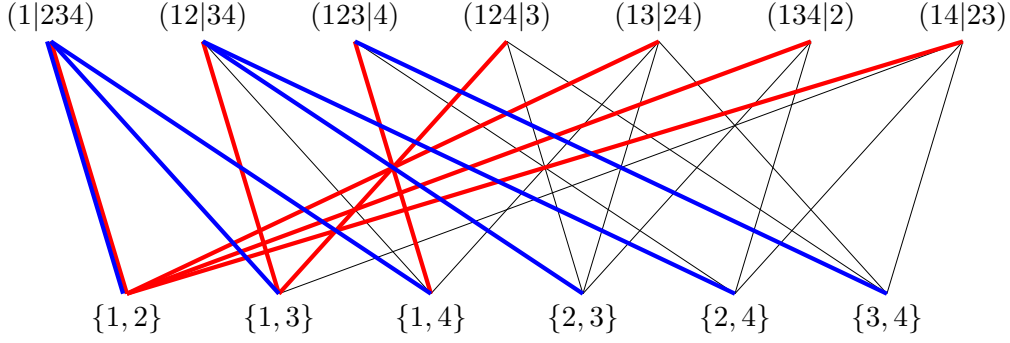
\begin{figure}[t!]
\begin{center}
\scalebox{1}{
\begin{tikzpicture}[scale=1]
\node (a) at (0,4) {$(1|234)$};
\node (a') at (-.05,4) {\phantom{$(1|234)$}};
\node (b) at (2,4) {$(12|34)$};
\node (c) at (4,4) {$(123|4)$};
\node (d) at (6,4) {$(124|3)$};
\node (e) at (8,4) {$(13|24)$};
\node (f) at (10,4) {$(134|2)$};
\node (g) at (12,4) {$(14|23)$};
\node (1) at (1,0) {$\{1,2\}$};
\node (1') at (.95,0) {\phantom{$\{1,2\}$}};
\node (2) at (3,0) {$\{1,3\}$};
\node (3) at (5,0) {$\{1,4\}$};
\node (4) at (7,0) {$\{2,3\}$};
\node (5) at (9,0) {$\{2,4\}$};
\node (6) at (11,0) {$\{3,4\}$};
\foreach \x in {1,2,3} {\draw[ultra thin] (a.south)--(\x.north);};
\foreach \x in {2,3,4,5} {\draw[ultra thin] (b.south)--(\x.north);};
\foreach \x in {3,5,6} {\draw[ultra thin] (c.south)--(\x.north);};
\foreach \x in {2,4,6} {\draw[ultra thin] (d.south)--(\x.north);};
\foreach \x in {1,3,4,6} {\draw[ultra thin] (e.south)--(\x.north);};
\foreach \x in {1,4,5} {\draw[ultra thin] (f.south)--(\x.north);};
\foreach \x in {1,2,5,6} {\draw[ultra thin] (g.south)--(\x.north);};
\foreach \x/\y in {a/1, b/2, c/3, d/2, e/1, f/1, g/1} {\draw[ultra thick, red] (\x.south)--(\y.north);};
\foreach \x/\y in {1'/a'} {\draw[ultra thick, blue] (\y.south)--(\x.north);};
\foreach \x/\y in {2/a, 3/a, 4/b, 5/b, 6/c} {\draw[ultra thick, blue] (\y.south)--(\x.north);};
\end{tikzpicture}
}
\caption{The Graham--Houghton graph $\GH(D(4,2))$, with the edges of the spanning tree $\Tlex(4,2)$ shown in colour.  See Remark~\ref{rem:Tlex42} for more details.}
\label{fig:Tlex42}
\end{center}
\end{figure}

\section{First family of presentations for maximal subgroups of \boldmath{$\PG(P)$}}
\label{sec:pres-linked}

In this section and the next we obtain several general presentations for maximal subgroups of the free projection-generated semigroup $\PG(P)$ associated to a projection algebra $P$.  The presentations in each section have their own `flavour', and are derived from two different presentations for $\PG(P)$ itself, by applying the Reidemeister--Schreier rewriting method described in Section~\ref{sec:rs}.  Here we start with the presentation in terms of the generating set $P$, and in the next section we use the generating set $E = \sfE(P)$.

\subsection{Presentations via linked diamonds}
\label{ss:linked}

Throughout this section, $P$ denotes a fixed projection algebra, and without loss of generality we take $P=\sfP(S)$ for some projection-generated regular $*$-semigroup $S$.
We start by recalling the standard presentation for~$\PG(P)$ stated in Subsection \ref{ss:PGP}.  It has generators $X_P=\{x_p\colon p\in P\}$, and defining relations
\begin{alignat}{2}
\label{eq:pg1}
& x_p^2=x_p&& (p\in P),
\\
\label{eq:pg2}
& (x_px_q)^2=x_px_q &\quad& (p,q\in P),
\\
\label{eq:pg3}
&x_px_qx_p=x_{pqp} && (p,q\in P).
\end{alignat}

By \ref{it:RS2}, every $\D$-class of a regular $*$-semigroup contains projections. Since maximal subgroups contained in a single $\D$-class are all isomorphic, every maximal subgroup is isomorphic to a maximal subgroup containing a projection.
By \ref{it:PG2}, the projections of $\PG(P)$ are precisely (the equivalence-classes of) the generators $x_p$, for $p\in P$. Thus, we will fix an arbitrary projection $p_0\in P$, and seek a presentation for the maximal subgroup $\GP$ containing $x_{p_0}$
(which is of course precisely the $\H$-class of $x_{p_0}$ in $\PG(P)$).

Following the methodology of Section \ref{sec:rs}, we need to index 
the $\R$- and $\L$-classes of the $\D$-class of $p_0$ and find a system of Schreier representatives.
Analogous to the discussion in Subsection~\ref{ss:genrew}, but this time based on \ref{it:PG3}, we can make these choices in $S$ instead of $\PG(P)$.

So, given $p_0\in P$, let $D=D_{p_0}$ be
its $\D$-class.
Let $P_D:=P\cap D$.
Then $\{R_p\colon p\in P_D\}$ and $\{L_p\colon p\in P_D\}$ are precisely the $\R$- and $\L$-classes contained in $\D$ by \ref{it:RS2}; note here that~$R_p$ and~$L_p$ simply denote the $\R$- and $\L$-classes of the projection $p\in P_D$. Thus, we will be using $P_D$ as a natural and convenient index set for both $\R$- and $\L$-classes.
Thus, in the notation of Section~\ref{sec:rs}, we are taking $I=J=P_D$.
For $p,q\in P_D$, let $H_{p,q}:=R_p\cap L_q$.
Note that, by \ref{it:RS4}, we have
\[
K=\big\{ (i,j)\in I\times J\colon H_{i,j}\text{ is a group}\big\}=\big\{(p,q)\in P_D\times P_D\colon pq\in E_D\big\}=\F_D,
\]
the restriction of the friendliness relation $\F$ to $D$.

From the above, the $\H$-classes in $R_{p_0}$ are  $\{ H_{p_0,p}\colon p\in P_D\}$.
 By \ref{it:RS2} and \ref{it:RS7}, the action~\eqref{eq:j.s} of $S$ on $P_D\cup\{0\}$ induced by the action of $S$ on these $\H$-classes is given by
 \begin{equation}
 \label{eq:sract}
 p\cdot s = \begin{cases} s^*ps &\text{if } ps\in R_{p}\\ 0 &\text{otherwise.}\end{cases}
 \end{equation}
 When $s=q \in P$ is a projection, using \ref{it:RS11}, this action becomes
  \begin{equation}
  \label{eq:qract}
 p\cdot q = \begin{cases} qpq &\text{if } p\leq_\F q \\ 0 &\text{otherwise.}\end{cases}
 \end{equation}

 We are now going to establish a set of Schreier representatives, which we additionally require to consist of words over the alphabet $\{ x_p\colon p\in P_D\}$.
 A routine modification of the argument in Remark \ref{re:Schrestr} shows that such a set exists.
 However, we are going to be a bit more specific, and adapt the construction from Subsection \ref{ss:span}.
 We begin by noting that there is a close connection between the Graham--Houghton graph $\GH(D)$, and the relation $\F_D$  considered as a (symmetric) digraph; these were called `projection graphs' in \cite{EG17}.
 Specifically, it follows from \ref{it:RS11} that $(p,q)\in \F_D$ if and only if the idempotent $pq$ (and also~$qp$) is an edge of $\GH(D)$.
 Since $S$ is projection-generated, Proposition~\ref{pr:GHconn} tells us that $\GH(D)$ is connected, and hence it follows that $\F_D$ is strongly connected.
 Let $T$ be a directed spanning tree of $\F_D$, rooted at $p_0$.

 Consider an arbitrary $q\in P_D$. Let $p_0\rightarrow p_1\rightarrow\cdots\rightarrow p_{k-1}\rightarrow p_k=q$ be the unique path from~$p_0$ to $q$ via edges from $T$.
 Define $r_q:=x_{p_1}\cdots x_{p_k}$. It is clear that $r_q$ is a word over the alphabet $\{ x_p\colon p\in P_D\}$, and that the set $\{ r_p\colon p\in P_D\}$ is prefix-closed. 
Since $(p_i,p_{i+1})\in\F$ for each $0\leq i<k$, it follows that $p_{i+1}p_ip_{i+1}=p_{i+1}$ for each $i$.  Thus, using \eqref{eq:qract}, we have
 \[
 p_0\cdot r_q=p_0\cdot x_{p_1}\cdots x_{p_k}=(p_k\cdots p_1)p_0(p_1\cdots p_k)=p_k=q.
 \]
 We conclude that $\{r_p\colon p\in P_D\}$ is a set of Schreier representatives as required.  For the words $\set{r_p'}{p\in P_D}$, we can take $r_p' = r_p^*x_{p_0}$.  Indeed, since $x_{p_0}r_p \mr\R x_{p_0}$, and since $\PG(P)$ is a regular $*$-semigroup, it follows from \ref{it:RS2} that $x_{p_0} = x_{p_0}r_p(x_{p_0}r_p)^* = x_{p_0}r_pr_p^*x_{p_0} = x_{p_0}r_pr_p'$.

Theorem \ref{thm:RSrw} now yields a presentation for $\GP$, the maximal subgroup of $\PG(P)$ containing~$x_{p_0}$.  The generators are
\[
B=\bigl\{ [t,p]\colon t\in P_D,\ p\in P,\ t\cdot p\neq 0\bigr\}=\bigl\{ [t,p]\colon t\in P_D,\ p\in P,\ t\leq_\F p\bigr\},
\]
where \eqref{eq:qract} is used for the second equality.
Actually, in the notation of Theorem \ref{thm:RSrw}, the generators from~$B$ were denoted $[t,x_p]$, but we have abbreviated these to $[t,p]$ in order to avoid subsequent calculations occurring in subscripts.
Keeping in mind the above choice of words $r_p'$, the generator $[t,p]$ stands for the element
\begin{equation}\label{eq:tp}
[t,p] = x_{p_0} r_tx_pr_{t\cdot p}' = x_{p_0} r_tx_pr_{ptp}^*x_{p_0}.
\end{equation}
Recalling the definition \eqref{eq:rewr} of the rewriting mapping~$\phi$, 
the general presentation \eqref{eq:gpIG1}--\eqref{eq:gpIG2} for $\GP$, applied to the presentation
\mbox{\eqref{eq:pg1}--\eqref{eq:pg3}} for $\PG(P)$, becomes:
\begin{align}
\label{eq:rwp0}
&[t,p]=1&&(t,p\in P_D,\ r_tx_p=r_{t\cdot p}),\\
\label{eq:rwp1}
&[t,p][ptp,p]=[t,p]&&(t\in P_D,\ p\in P,\ t\leq_\F p),\\
\label{eq:rwp2}
&[t,p][ptp,q][qptpq,p][pqptpqp,q]=[t,p][ptp,q] &&(t\in P_D,\ p,q\in P,\ t\leq_\F p,\ ptp\leq_\F q),\\
\label{eq:rwp3}
&[t,p][ptp,q][qptpq,p]=[t,pqp]&&(t\in P_D,\ p,q\in P,\ t\leq_\F p,\ ptp\leq_\F q).
\end{align}

We are now going to transform this presentation into a form analogous to Theorem~\ref{thm:spantreepres}. First, with regards to \eqref{eq:rwp0}, we note that
$t\cdot p=ptp$ by \eqref{eq:qract}. Furthermore, it follows from the construction of the Schreier representatives $r_q$ that
$r_tx_p=r_{ptp}$ holds if and only if the unique $T$-path from $p_0$ to $p(=ptp)$ visits $t$ just before $p$. In turn, this is the case if and only if~$tp$ is an edge in $T$.
Keeping this in mind, and performing  the obvious cancellations in~\eqref{eq:rwp1} and~\eqref{eq:rwp2}, as we may because we are working with a \emph{group} presentation, the relations \eqref{eq:rwp0}--\eqref{eq:rwp3} become:
\begin{alignat}{2}
\label{eq:rwp4}
&[t,p]=1&&(t,p\in P_D,\ \ tp\in T),\\
\label{eq:rwp5}
&[ptp,p]=1&&(t\in P_D,\ p\in P,\ t\leq_\F p),\\
\label{eq:rwp6}
&[qptpq,p][pqptpqp,q]=1 &&(t\in P_D,\ p,q\in P,\ t\leq_\F p,\ ptp\leq_\F q),\\
\label{eq:rwp7}
&[t,p][ptp,q][qptpq,p]=[t,pqp]&\quad&(t\in P_D,\ p,q\in P,\ t\leq_\F p,\ ptp\leq_\F q).
\end{alignat}
We note that relations \eqref{eq:rwp5} contain among them
\begin{equation}
\label{eq:rwp8}
[p,p]=1\quad (p\in P_D).
\end{equation}
Now consider an arbitrary generator $[t,p]\in B$, so $t\in P_D$, $p\in P$ and $t\leq_\F p$.
Define ${q:=ptp\in P_D}$.
Trivially, $ptp\leq_\F q$, and hence we have a relation from \eqref{eq:rwp7}, which (after some simplification of terms such as $pqp=pptpp=ptp$) reads
\begin{equation}
\label{eq:rwp9}
[t,p][ptp,ptp][ptp,p]=[t,ptp].
\end{equation}
The generators $[ptp,ptp]$ and $[ptp,p]$ equal $1$ as a consequence of 
\eqref{eq:rwp8} and~\eqref{eq:rwp5}, respectively, so the relation~\eqref{eq:rwp9} becomes
\begin{equation}
\label{eq:rwp10}
[t,p]=[t,ptp].
\end{equation}
Also note that $t\leq_\F p$ implies $(t,ptp)\in\F_D$ (by~\ref{it:RS10}).
Thus, \eqref{eq:rwp10}  enables us to eliminate all the generators from $B$ of the form $[t,p]$, except those where
$(t,p)\in\F_D$.  Transforming~\mbox{\eqref{eq:rwp4}--\eqref{eq:rwp7}} in this way (and remembering that $pq$ is an idempotent for $p,q\in P$), yields the following presentation:
\begin{align}
\label{eq:rwp11}
&[t,p]=1&&(t,p\in P_D,\ tp\in T),\\
\label{eq:rwp12}
&[ptp,ptp]=1&&(t\in P_D,\ p\in P,\ t\leq_\F p),\\
\label{eq:rwp13}
&[qptpq,pqptpqp] [pqptpqp,qptpq]=1 &&(t\in P_D,\ p,q\in P,\ t\leq_\F p,\ ptp\leq_\F q),\\
\label{eq:rwp14}
&[t,ptp] [ptp,qptpq] [qptpq,pqptpqp]= [t,pqptpqp]&&(t\in P_D,\ p,q\in P,\ t\leq_\F p,\ ptp\leq_\F q).
\end{align}
Now note that the relations \eqref{eq:rwp12} coincide precisely with the relations $[p,p]=1$ ($p\in P_D$).
Also, each relation \eqref{eq:rwp13} has the form $[u,v]=[v,u]^{-1}$ with ${(u,v)\in\F_D}$.
Conversely, if ${(t,p)\in\F_D}$, then letting $q=t$ in~\eqref{eq:rwp13} we obtain $[t,p]=[p,t]^{-1}$.
Thus we arrive at the following presentation:
\begin{align}
\label{eq:rwp15}
&[t,p]=1&&(t,p\in P_D,\ tp\in T),\\
\label{eq:rwp16}
&[p,p]=1&&(p\in P_D),\\
\label{eq:rwp17}
&[p,q]=[q,p]^{-1}  &&((p,q)\in\F_D),\\
\label{eq:rwp18}
&[t,ptp] [ptp,qptpq] [qptpq,pqptpqp]= [t,pqptpqp]&&(t\in P_D,\ p,q\in P,\ t\leq_\F p,\ ptp\leq_\F q).
\end{align}
We now turn our attention to the final family of relations \eqref{eq:rwp18}.
A straightforward manipulation, including one application of \eqref{eq:rwp17}, lets us write it in the following form:
\begin{equation}
\label{eq:rwp19}
 [t,ptp]^{-1} [t,pqptpqp]=[qptpq,ptp]^{-1} [qptpq,pqptpqp] \quad (t\in P_D,\ p,q\in P,\ t\leq_\F p,\ ptp\leq_\F q).
\end{equation}
Let us say that a quadruple $(s,u;v,w)$ of elements of $P_D$ is a \emph{linked diamond} if 
\bit
\item ${(s,v),(s,w),(u,v),(u,w)\in \F_D}$, and 
\item there exists $p\in P$ 
such
that $psp=v$ and $pup=w$.
\eit
In this case we say that the diamond is \emph{$p$-linked}, and we note that $p$ generally does not belong to~$P_D$,  but rather comes from a $\D$-class higher up in the $\leq_\J$ order.  
We now press on towards the main result of this subsection, but following it there will be a brief discussion of the name choice and visualisation for linked diamonds in Remark \ref{rem:suvw}.

Denote by $\Lk_D$ the set of all linked diamonds with entries from $P_D$.
Note that in~\eqref{eq:rwp19}, the quadruple $(t,qptpq;ptp,pqptpqp)$ is $p$-linked.
In other words, all relations in \eqref{eq:rwp19} are contained in
\begin{equation}
\label{eq:rwp19a}
 [s,v]^{-1} [s,w]=[u,v]^{-1} [u,w] \quad ((s,u;v,w)\in \Lk_D).
\end{equation}
Now consider an arbitrary $p$-linked diamond $(s,u;v,w)\in\Lk_D$.
Let $t:=s$ and $q:=u$. Since $sp$ is an idempotent and $s\mr\F v$, we have
$tpt = sps = s(psp)s = svs = s = t$, i.e.~$t\leq_\F p$.
Also, $ptp = psp = v \mr\F u = q$, so certainly $ptp\leq_\F q$.
Hence, we have a relation~\eqref{eq:rwp19} for these values of $t$, $p$ and $q$,
and a straightforward calculation shows that it is precisely the relation \eqref{eq:rwp19a} for the original values of $s$, $u$, $v$ and $w$.
Therefore, relations \eqref{eq:rwp19a} coincide with \eqref{eq:rwp19}, and hence are equivalent to
 \eqref{eq:rwp18}.
If we introduce a new, more `letter-like', notation $a_{p,q}$ for the generating symbol $[p,q]$, and re-index mildly, we obtain the following:

\begin{thm}
\label{th:linkedpres}
Let $P$ be a projection algebra, let $S$ be any projection-generated regular $*$-semigroup with $P=\sfP(S)$, and let $p_0\in P$ be arbitrary.
Denote by $D$ the $\D$-class of $S$ containing~$p_0$, and let $T$ be any $p_0$-rooted directed spanning tree for $\F_D$.  
With the rest of the notation as above, the maximal subgroup of $\PG(P)$ containing $x_{p_0}$ has a presentation
with generators $A=\{ a_{p,q}\colon (p,q)\in\F_D\}$ and relations
\begin{alignat}{2}
\label{eq:rwp20}
&a_{p,q}=1&&(p,q\in P_D,\ pq\in T),\\
\label{eq:rwp21}
&a_{p,p}=1&&(p\in P_D),\\
\label{eq:rwp22}
&a_{p,q}=a_{q,p}^{-1}  &&((p,q)\in\F_D),\\
\label{eq:rwp23}
& a_{s,v}^{-1} a_{s,w}=a_{u,v}^{-1} a_{u,w} &\quad& ((s,u;v,w)\in \Lk_D).\\
\tag*{\qed}
\end{alignat}
\end{thm}

\begin{rem}\label{rem:suvw}
A $p$-linked diamond $(s,u;v,w)$ can be visualised as in the left-hand diagram of Figure~\ref{fig:plinkeddiamond}.  In the diagram, the solid lines indicate $\F$-relationships, and the dotted arrows indicate the action of `conjugation' by $p$.
The dashed line between $v$ and $w$ indicates that we additionally have $v\mr\F w$.  This was not part of the defining properties of the linked diamond, but it quickly follows.  Indeed, using the assumed $\F$-relationships and $pwp = w$, we have
\[
vwv = pspwpsp = pswsp = psp = v \quad\text{and similarly}\quad wvw = w.
\]
On the other hand, $s$ and $u$ might not be $\F$-related. It is also worth noting that in the presence of \eqref{eq:rwp22}, the relation \eqref{eq:rwp23} can be rewritten as
\[
a_{s,v}a_{v,u}a_{u,w}a_{w,s} = 1.
\]
Intuitively, one can think of this as equating the path $s\to v\to u\to w\to s$ around the boundary of the diamond to the empty path.  This in turn links back to the topological interpretation of~$\PG(P)$ from \cite[Section 10]{EGMR} as the fundamental groupoid of a certain $2$-complex built from~$\F$.
\end{rem}

\begin{figure}[t]
\begin{center}
\scalebox{1}{
\begin{tikzpicture}[scale=1.2,>=Latex]
\nc\yy{0.6}
\nc\zz0
\node (v) at (0,0) {$v$};
\node (w) at (2,0) {$w$};
\node (s) at (1,1) {$s$};
\node (u) at (1,-1) {$u$};
\draw(s)--(v)--(u)--(w)--(s);
\draw[dashed](v)--(w);
\node[left] () at (0,-.03) {$psp={}$};
\node[right] () at (2,-.03) {${}=pup$};
\draw[dotted,-{latex}] (s) to [bend right = 30] (v);
\draw[dotted,-{latex}] (u) to [bend right = 30] (w);
\node () at (1.8,-.8) {\footnotesize$p$};
\node () at (.2,.8) {\footnotesize$p$};
\begin{scope}[shift = {(6,0)}]
\node (v) at (0,0+\yy) {$v$};
\node (w) at (2,0+\yy) {$w$};
\node (v') at (0,0-\zz) {$v$};
\node (w') at (2,0-\zz) {$w$};
\node (s) at (1,1+\yy) {$s$};
\node (u) at (1,-1-\zz) {$u$};
\draw(s)--(v)--(w)--(s)  (v')--(w')--(u)--(v');
\draw[dotted,-{latex}] (s) to [bend right = 30] (v);
\draw[dotted,-{latex}] (u) to [bend right = 30] (w');
\node () at (1.8,-.8-\zz) {\footnotesize$p$};
\node () at (.2,.8+\yy) {\footnotesize$p$};
\path [dotted, -{latex}] (w) edge  [loop right] node {\footnotesize$p$} ();
\path [dotted, -{latex}] (v') edge  [loop left] node {\footnotesize$p$} ();
\draw[transform canvas={xshift=-1pt}] (v) -- (v');
\draw[transform canvas={xshift=1pt}] (v') -- (v);
\draw[transform canvas={xshift=-1pt}] (w) -- (w');
\draw[transform canvas={xshift=1pt}] (w') -- (w);
\end{scope}
\end{tikzpicture}
}
\caption{Left: a $p$-linked diamond $(s,u;v,w)$.  Right: the corresponding $p$-linked triangles $(w,s,v)$ and $(v,u,w)$.  See Remarks \ref{rem:suvw} and \ref{rem:suv} for more details.}
\label{fig:plinkeddiamond}
\end{center}
\end{figure}
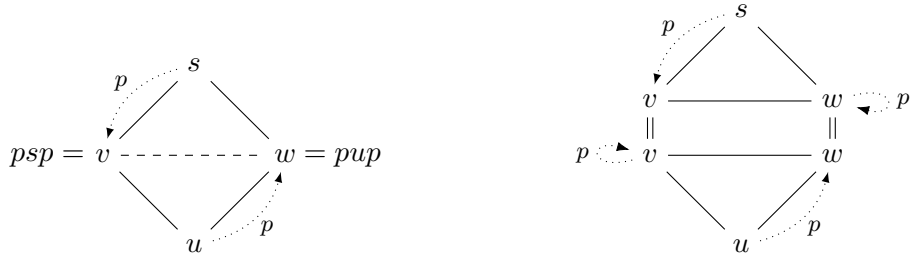

\subsection{Linked squares and linked pairs}\label{subsect:LSLP}

The generators $a_{p,q}$ ($(p,q)\in\F_D$) in the presentation we have just established are in one-one correspondence with the idempotents in $D$ by \ref{it:RS5}.
Thus, one may view the relations \eqref{eq:rwp23} as having the same form as the singular square relations for $\IG(E)$
that we saw in Section \ref{sec:maxIG}: they express equality of quotients of generators corresponding to four idempotents arranged in a `square' (and forming a rectangular band).
Which squares are taken, however, is different in the two situations, although some strong links will be established in Section \ref{sec:pres-quot}.

Let us be more specific.  Consider an arbitrary $p$-linked diamond $(s,u;v,w)\in\Lk_D$.  Then
\begin{equation}\label{eq:DtoS}
\smat efgh := \smat{sv}{sw}{uv}{uw}
\end{equation}
is a square of idempotents, indeed a rectangular band,
as follows from checking a diagonal product, e.g.~$eh = sv \sgap uw = spsp \sgap upup = sp \sgap up = s \sgap pup = sw = f$.
Additionally, we have $pee^*p=psp=v=e^*e$ and $phh^*p=pup=w=h^*h$.

Motivated by this, we call a square of idempotents $\smat efgh$ a \emph{$p$-linked square} (for some projection $p$) if
\begin{equation}\label{eq:peep}
pee^*p = e^*e \ANd phh^*p = h^*h.
\end{equation}
Note that we have similar conditions concerning the other idempotents in the square, since $e\mr\R f\mr\L h \implies ff^*=ee^*$ and $f^*f=h^*h$, and similarly $g^*g=e^*e$ and $gg^*=hh^*$.  Thus, if we define the projections
\begin{equation}\label{eq:StoD}
s := ee^* = ff^* \COMMa
u := gg^* = hh^* \COMMa
v := e^*e = g^*g \ANd
w := f^*f = h^*h,
\end{equation}
then condition \eqref{eq:peep} says that $psp = v$ and $pup = w$, which is one of the conditions for a linked diamond of projections.  The other condition involving $\F$-relationships comes from the very fact that the idempotents $e$, $f$, $g$ and $h$ exist in the first place; see \ref{it:RS4}.

Consider now a linked square $\smat efgh$, corresponding to the linked diamond $(s,u;v,w)$ as above.  
Writing $a_{pq}=a_{p,q}$ for $(p,q)\in\F_D$, the relation $ a_{s,v}^{-1} a_{s,w}=a_{u,v}^{-1} a_{u,w}$ from~\eqref{eq:rwp23} becomes $a_{e}^{-1} a_{f}=a_{g}^{-1} a_{h}$.
Transforming the remaining relations \eqref{eq:rwp20}--\eqref{eq:rwp22} in the presentation from Theorem \ref{th:linkedpres} we then obtain:

\begin{cor}\label{co:lkid}
Let $P$ be a projection algebra, let $S$ be any projection-generated regular $*$-semigroup with $P=\sfP(S)$, and let $p_0\in P$ be arbitrary.
Denote by $D$ the $\D$-class of $S$ containing~$p_0$, and let $T$ be any $p_0$-rooted directed spanning tree for $\F_D$.  
With the rest of the notation as above, the maximal subgroup of $\PG(P)$ containing $x_{p_0}$ has a presentation
with generators $A=\{ a_e\colon e\in E_D\}$ and relations
\begin{alignat}{2}
\label{eq:rwp40}
&a_e = 1&&(e\in T \cup P_D),\\
\label{eq:rwp42}
&a_e = a_{e^*}^{-1}  &&(e\in E_D),\\
\label{eq:rwp43}
& a_{e}^{-1} a_{f}=a_{g}^{-1} a_{h} &\quad& (\smat{e}{f}{g}{h} \text{ is a linked square of idempotents in $D$}).
\\
\tag*{\qed}
\end{alignat}
\end{cor}

\begin{rem}\label{rem:plinkedsquare}
We can visualise the $p$-linked square $\smat efgh$ in its egg-box as in Figure \ref{fig:plinkedsquare}.  In the diagram, $\R$- and $\L$-classes are indexed by projections, so for example $e,f\in R_s$ and $e,g\in L_v$, and we have $e=sv$.  As in Figure \ref{fig:plinkeddiamond}, dotted arrows indicate the action of conjugation by $p$.
\end{rem}

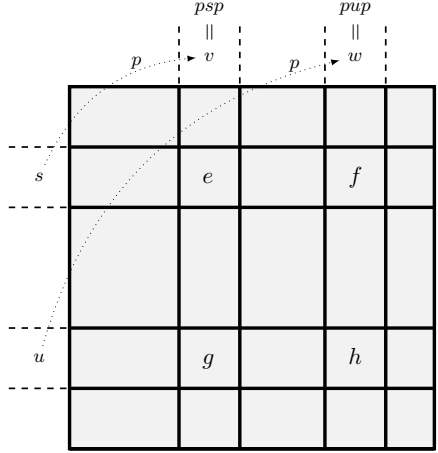
\begin{figure}[t]
\begin{center}
\scalebox{0.8}{
\begin{tikzpicture}[scale=1]

\fill[gray!10](0,0)--(6,0)--(6,6)--(0,6)--(0,0);

\foreach \x in {1,2,4,5} {\draw[dashed, thick] (-1,6-\x)--(0,6-\x); \draw[ultra thick] (0,6-\x)--(6,6-\x);}
\foreach \x in {1.5+.3,2.5+.3,3.5+.7,4.5+.7} {\draw[dashed, thick] (\x,7)--(\x,6); \draw[ultra thick] (\x,0)--(\x,6);}

\draw[ultra thick](0,0)--(6,0)--(6,6)--(0,6)--(0,0)--(6,0);

\node (v) at (2+.3,6.5) {\footnotesize $v$};
\node (w) at (4+.7,6.5) {\footnotesize $w$};
\node (s) at (-.5,4.5) {\footnotesize $s$};
\node (u) at (-.5,1.5) {\footnotesize $u$};

\node (vv) at (2+.3,6.5+.8) {\footnotesize $psp$};
\node (ww) at (4+.7,6.5+.8) {\footnotesize $pup$};

\node () at (2+.3+.03,6.5+.4) {$\rotatebox{90}{=}$};
\node () at (4+.7+.03,6.5+.4) {$\rotatebox{90}{=}$};

\node () at (2+.3,4.5) {$\mathrlap{e}{\phantom{f}}$};
\node () at (4+.7,4.5) {$\mathrlap{f}{\phantom{f}}$};
\node () at (2+.3,1.5) {$\mathrlap{g}{\phantom{f}}$};
\node () at (4+.7,1.5) {$\mathrlap{h}{\phantom{f}}$};

\draw[dotted,-{latex}] (s) to [bend left = 30] (v);
\draw[dotted,-{latex}] (u) to [bend left = 30] (w);
\node () at (1.1,6.4) {\footnotesize$p$};
\node () at (3.7,6.35) {\footnotesize$p$};

\end{tikzpicture}
}
\caption{A $p$-linked square $\smat efgh = \smat{sv}{sw}{uv}{uw}$, corresponding to the $p$-linked diamond $(s,u;v,w)$.  See Remark \ref{rem:plinkedsquare} for more details.}
\label{fig:plinkedsquare}
\end{center}
\end{figure}

Another equivalent form of the relation \eqref{eq:rwp23} is obtained by recalling from \cite{EGMR} the notion of a $p$-linked pair of projections.
It is easy to see that if $(s,u;v,w)$ is a $p$-linked diamond, then~$(s,u)$ is a $p$-linked pair, meaning that $s = spups$ and $u = upspu$. And conversely a $p$-linked pair~$(s,u)$ leads to the $p$-linked diamond $(s,u;psp,pup)$.  Hence we can write
 \eqref{eq:rwp23} as:
\[
a_{s,psp}^{-1} a_{s,pup}=a_{u,psp}^{-1} a_{u,pup} \quad (s,u\in P_D,\ (s,u)\text{ is a $p$-linked pair}).
\]

\subsection{Linked triangles}

Our final variation identifies a subset of \eqref{eq:rwp23} that implies the rest.
Let us call a triple $(s,u,w)$ of projections from $P_D$ a \emph{$p$-linked triangle} if $(s,u;s,w)$ is a $p$-linked diamond, i.e.~if $(s,w),(u,s),(u,w)\in\F_D$, $psp=s$ and $pup=w$. 
Let~$\Trs_D$ denote the set of all linked triangles in~$P_D$.
Now consider an arbitrary $p$-linked diamond $(s,u;v,w)$.  It is easy to check that the defining properties of this diamond (and the consequence $v\mr\F w$ from Remark \ref{rem:suvw}) imply that $(w,s,v),(v,u,w)\in\Trs_D$ (both are $p$-linked).  
The relations \eqref{eq:rwp23} arising from the diamonds $(w,s;w,v)$ and $(v,u;v,w)$ are
\[
a_{w,w}^{-1} a_{w,v}=a_{s,w}^{-1} a_{s,v}\quad\text{and}\quad 
a_{v,v}^{-1} a_{v,w}=a_{u,v}^{-1} a_{u,w}.
\]
These two relations, together with \eqref{eq:rwp21} and \eqref{eq:rwp22}, imply
\[
a_{s,v}^{-1}a_{s,w}=a_{w,v}^{-1} =a_{v,w}=a_{u,v}^{-1}a_{u,w},
\]
which is the relation \eqref{eq:rwp23} arising from the original diamond $(s,u;v,w)$.  Thus, we can replace the entire set of relations \eqref{eq:rwp23} by the subset corresponding to triangles $(s,u,w)\in\Trs_D$.  The relation \eqref{eq:rwp23} for the corresponding diamond $(s,u;s,w)\in\Lk_D$ is $a_{s,s}^{-1}a_{s,w} = a_{u,s}^{-1}a_{u,w}$.  Using~\eqref{eq:rwp21}, this simplifies to $a_{u,s}a_{s,w}= a_{u,w}$.
Hence we obtain the following:

\begin{cor}
\label{co:lkdp1}
With the notation as above, the maximal subgroup of $\PG(P)$ containing the projection $x_{p_0}$ has a presentation
with generators $A=\{ a_{p,q}\colon (p,q)\in\F_D\}$ and relations
\begin{align}
\nonumber
&a_{p,q}=1&&(p,q\in P_D,\ pq\in T),\\
\nonumber
&a_{p,p}=1&&(p\in P_D),\\
\nonumber
&a_{p,q}=a_{q,p}^{-1}  &&((p,q)\in\F_D),\\
&  a_{u,s}a_{s,w}= a_{u,w} && ((s,u,w)\in \Trs_D).
\tag*{\qed}
\end{align}
\end{cor}

\begin{rem}\label{rem:suv}
As in Remark \ref{rem:suvw}, one can visualise a linked triangle $(s,u,w) \equiv (s,u;s,w)$ as:
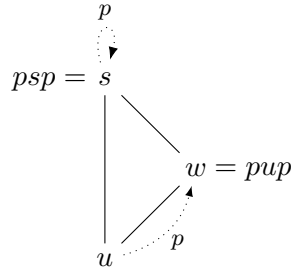
\[
\begin{tikzpicture}[scale=1.2,>=Latex]
\nc\yy{0.3}
\node (w) at (2,0) {$w$};
\node (s) at (1,1) {$s$};
\node (u) at (1,-1) {$u$};
\draw(s)--(u)--(w)--(s);
\node[left] () at (1,1-.03) {$psp={}$};
\node[right] () at (2,-.03) {${}=pup$};
\draw[dotted,-{latex}] (u) to [bend right = 30] (w);
\node () at (1.8,-.8) {\footnotesize$p$};
\path [dotted, -{latex}] (s) edge  [loop above] node {\footnotesize$p$} ();
\end{tikzpicture}
\]
We saw above that a linked diamond $(s,u;v,w)$ induces two linked triangles $(w,s,v)$ and $(v,u,w)$.  These triangles can be seen in Figure \ref{fig:plinkeddiamond} (right), and this can again be interpreted topologically; see \cite[Section~10]{EGMR}.  
\end{rem}

One can also formulate Corollary \ref{co:lkdp1} in terms of the alphabet $\set{a_e}{e\in E_D}$.  For this, note that a $p$-linked triangle $(s,u,v)$ corresponds to the $p$-linked diamond $(s,u;s,w)$ by definition, which in turn corresponds (as in \eqref{eq:DtoS}) to the $p$-linked square $\smat efgh = \smat{ss}{sw}{us}{uw}$, in which $e=ss=s$ is a projection.  Conversely, given a $p$-linked square $\smat efgh$ with $e$ a projection, the associated $p$-linked diamond $(s,u;v,w)$ in \eqref{eq:StoD} satisfies $v=e=s$, and hence leads to the $p$-linked triangle $(s,u,w)$.  The next result now follows from Corollary \ref{co:lkdp1} in exactly the same way that Corollary \ref{co:lkid} followed from Theorem \ref{th:linkedpres}.

\begin{cor}\label{co:lkidp}
Let $P$ be a projection algebra, let $S$ be any projection-generated regular $*$-semigroup with $P=\sfP(S)$, and let $p_0\in P$ be arbitrary.
Denote by $D$ the $\D$-class of $S$ containing~$p_0$, and let $T$ be any $p_0$-rooted directed spanning tree for $\F_D$.  
With the rest of the notation as above, the maximal subgroup of $\PG(P)$ containing $x_{p_0}$ has a presentation
with generators $A=\{ a_e\colon e\in E_D\}$ and relations
\begin{alignat}{2}
&a_e = 1&&(e\in T \cup P_D),\\
&a_e = a_{e^*}^{-1}  &&(e\in E_D),\\
& a_{g} a_{f}= a_{h} &\quad& (\smat{e}{f}{g}{h} \text{ is a linked square of idempotents in $D$ with $e\in P_D$}).
\\
\tag*{\qed}
\end{alignat}
\end{cor}

\subsection{First application: degenerate diamonds and free groups}

We conclude this section by giving some applications of Theorem \ref{th:linkedpres}.  The first family of applications is based on the observation that in the absence of linked diamonds, the presentation \eqref{eq:rwp20}--\eqref{eq:rwp23} defines a free group.  We can actually do a bit better than that.  Let us call a linked diamond $( s,u;v,w)$ \emph{degenerate} if at least one of the following conditions is satisfied:
\begin{enumerate}[label=\textsf{(D\arabic*)},leftmargin=10mm]
\begin{multicols}{3}
\item \label{D1} $s=u$,
\item \label{D2} $v=w$,
\item \label{D3} $s=v$ and $u=w$.
\end{multicols}
\end{enumerate}
We note in passing that in fact \ref{D1} implies \ref{D2}.
In each of these three situations, the relation~\eqref{eq:rwp23} reads as follows:
\begin{enumerate}[label=\textsf{(D\arabic*)},leftmargin=10mm]
\begin{multicols}{3}
\item $a_{s,v}^{-1} a_{s,w}=a_{s,v}^{-1} a_{s,w}$,
\item $a_{s,v}^{-1} a_{s,v}=a_{u,v}^{-1} a_{u,v}$,
\item $a_{s,s}^{-1} a_{s,u}=a_{u,s}^{-1} a_{u,u}$.
\end{multicols}
\end{enumerate}
Thus in cases \ref{D1} and \ref{D2} the relation is trivially true, while in case \ref{D3} it is a consequence of 
\eqref{eq:rwp21} and \eqref{eq:rwp22}.  It is also worth noting that when $(s,u;v,w)$ is degenerate, the diamond pictured in Figure \ref{fig:plinkeddiamond} collapses to a line; the specific type of collapse depends on the type of degeneracy \ref{D1}--\ref{D3}.

In the next statement recall that $P_D = P\cap D$ and $E_D = E\cap D$, where $D$ is the $\D$-class of~$p_0$.  We also fix a subset $E_D'$ of $E_D\setminus P_D$ such that:
\bit
\item $E_D'$ contains exactly one of $e$ or $e^*$ for each $e\in E_D\setminus P_D$,
\item $E_D'$ contains the spanning tree $T$ from Theorem \ref{th:linkedpres}.
\eit

\begin{prop}
\label{pr:free}
With the notation of Theorem \ref{th:linkedpres}, if there are no non-degenerate linked diamonds in $D$, then the maximal subgroup of $\PG(P)$ containing $x_{p_0}$ is free of rank $|E_D'\setminus T|$.  In particular, when $P_D$ is finite, this rank is equal to
\[
\frac{|E_D|-3|P_D|}2 + 1.
\]
\end{prop}

\begin{proof}
In the absence of non-degenerate linked diamonds, 
the presentation from Theorem \ref{th:linkedpres} consists of generators $A = \set{a_{p,q}}{(p,q)\in\F_D}$ and relations \eqref{eq:rwp20}--\eqref{eq:rwp22}, and we recall that there is a bijection $\F_D\to E_D$ given by $(p,q)\mt pq$.
The generators of the form $a_{p,p}$ are equal to~$1$ by \eqref{eq:rwp21}, and the remaining generators
are mutual inverses in pairs according to \eqref{eq:rwp22}.  We can therefore eliminate all generators but those corresponding to $E_D'$.  Finally, the generators corresponding to~$T$ can be eliminated, being equal to $1$ by \eqref{eq:rwp20}.  When this elimination has been carried out, no relations remain.  This gives the first claim.

For the second claim, we note that when $P_D$ is finite, so too is $E_D$ and $T$, and so we have $|E_D'\setminus T| = |E_D'|-|T|$.  The formula then follows from $|E_D'| = \frac{|E_D|-|P_D|}2$ and $|T| = |P_D|-1$.
\end{proof}

\begin{rem}
Proposition \ref{pr:free} can also be interpreted and proved topologically.  By \cite[Theorem~10.7]{EGMR}, and since there are no non-degenerate linked diamonds in $D$, the maximal subgroup in question is (isomorphic to) the fundamental group of the graph with vertex set $P_D$, and an edge $\{p,q\}$ for distinct $p,q\in P_D$ with $p\mr\F q$.  It is well known (see \cite[Example 1.22]{Hatcher2002}) that the fundamental group of a connected graph with edge set $E$ is free of rank $|E\setminus T|$ for any spanning tree $T$.
\end{rem}

We now deduce a number of corollaries of this result.  The first concerns the \emph{adjacency semigroups} of Jackson and Volkov \cite{JV10}, which were treated as a key example in \cite{EGMR}.  Let $\Ga$ be a connected, symmetric, reflexive digraph with vertex set $P$ and edge set $E$.  The \emph{adjacency semigroup} $\sfA(\Ga)$ has underlying set $(P\times P)\cup\{0\}$, and product
\[
0^2=0=0(p,q)=(p,q)0
\AND
(p,q)(r,s) = \begin{cases}
(p,s) &\text{if $(q,r)\in E$}\\
0 &\text{otherwise.}
\end{cases}
\]
The involution $0^*=0$ and $(p,q)^*=(q,p)$ gives it the structure of a regular $*$-semigroup.  We have $\sfE(\sfA(\Ga)) = E\cup\{0\}$, and if we identify $p\in P$ with $(p,p)\in\sfA(\Ga)$, then $\sfP(\sfA(\Ga)) = P\cup\{0\}$.  It follows quickly from connectivity of $\Ga$ that $\sfA(\Ga)$ is projection-generated.  Indeed, for any $p,q\in P$ we have $(p,q) = p_1\cdots p_k$ for any path $p = p_1\to\cdots\to p_k=q$ in $\Ga$.  It follows that $D = P\times P$ is the unique non-zero $\D$-class in $\sfA(\Ga)$, and that there are no non-degenerate linked diamonds; specifically, it is easy to check that if $(s,u;v,w)$ is $p$-linked, then $v=w=p$.  Proposition \ref{pr:free} therefore applies, and we obtain the following:

\begin{cor}
\label{co:0rb}
Let $S = \sfA(\Ga)$ be the adjacency semigroup of a finite, connected, symmetric, reflexive digraph $\Ga = (P,E)$.  Suppose also that $\Ga$ has $|P| = n$ vertices, and that the simple (undirected, loop-free) reduct of $\Ga$ has $k$ edges.  Then the maximal subgroup of $\PG(\sfP(S))$ containing any non-zero projection is free of rank $k-n+1$.
\end{cor}

\begin{proof}
Writing $D:=P\times P$, we have $|P_D| = n$ and $|E_D| = 2k+n$.  Proposition \ref{pr:free} and a simple calculation gives the stated rank.
\end{proof}

We can also give an $r=n-1$ version of Theorem \ref{thm:mainPGPn}:

\begin{cor}\label{cor:PGPnn-1}
Let $n\geq 1$, and let $p_0$ be a projection of rank $n-1$ in the partition monoid $\P_n$.
Then the maximal subgroup of the free projection-generated semigroup $\PG(\sfP(\P_n))$ containing~$x_{p_0}$ is free of rank $\binom{n-1}2$.
\end{cor}

\begin{proof}
Let $D := D_{p_0}$.  As in Remark \ref{rem:noss} we have $|P_D| = n + \binom n2$, and $|E_D| = n + 5\binom n2$.  Since the only idempotent in the $\D$-class above $D$ is the identity element, there are no non-degenerate linked diamonds in $D$.  The proof again concludes with an application of Proposition~\ref{pr:free}, and a simple calculation.
\end{proof}

As a final non-obvious application of Proposition \ref{pr:free}, we can also determine the maximal subgroup in the bottom $\D$-class of $\PG(\sfP(\B_4))$, where $\B_4$ is the Brauer monoid of degree $4$, consisting of partitions from $\P_4$ into blocks all of size $2$.  

\begin{cor}\label{cor:PGB4}
Let $p_0$ be a projection of rank $0$ in the Brauer monoid $\B_4$.
Then the maximal subgroup of the free projection-generated semigroup $\PG(\sfP(\B_4))$ containing $x_{p_0}$ is free cyclic.
\end{cor}

\begin{proof}
Note that $D:=D_{p_0}$ is a $3\times 3$ square band, because there are three equivalence relations on the set $[4]$ with classes of size~$2$.
We claim that there are no non-degenerate linked diamonds in $D$, and therefore Proposition~\ref{pr:free} applies.  Since $|P_D|=3$ and $|E_D|=9$, the proposition will tell us the group is free of rank $1$.

To prove the claim, consider a $p$-linked diamond $(s,u;v,w)$ in $D$.  If $p=1$, then $v=psp=s$ and $w=pup=u$, so the diamond is degenerate of type \ref{D3}.  Now suppose $p\not=1$, so that $p$ has an upper block $A$, and hence also the lower block $A'$.  We see then that $v=psp$ and $w=pup$ also contain the blocks $A$ and $A'$.  Since $v$ and $w$ have rank $0$, they also contain the blocks $[4]\setminus A$ and $[4]'\setminus A'$, and hence $v=w$.  Thus, the diamond is degenerate of type \ref{D2}.
\end{proof}

\subsection[Second application: maximal subgroups in $\PG(\sfP(\P_n))$, the case of rank~$0$]{\boldmath Second application: maximal subgroups in $\PG(\sfP(\P_n))$, the case of rank~$0$}

Our second application of Theorem \ref{th:linkedpres} is to determine the maximal subgroups of the free projection-generated semigroups over partition monoids corresponding to idempotents in the bottom $\D$-class, thereby establishing the~${r=0}$ case of Theorem \ref{thm:mainPGPn}.  This time there are \emph{many} non-degenerate linked diamonds, and these play an important role in the proof.

\begin{prop}
\label{p:PnD0}
Let $n\geq 1$, and let $p_0$ be a projection of rank $0$ in the partition monoid $\P_n$.
Then the maximal subgroup of the free projection-generated semigroup $\PG(\sfP(\P_n))$ containing $x_{p_0}$ is trivial.
\end{prop}

\begin{proof}
The projections in $D:=D_{p_0}$ are naturally indexed by the set $\Eq[n]$ of equivalence relations on $[n]$.
For $\sigma\in\Eq[n]$, let us write $\overline{\sigma}$ and $\widetilde{\sigma}$ respectively for the projection in $D$
and  the full-domain projection determined by $\sigma$. More precisely, if $A_1,\ldots,A_k$ are the $\sigma$-classes, then 
\[
\overline\sigma = \begin{partn}{3}A_1&\cdots&A_k\\\hhline{-|-|-}A_1&\cdots&A_k\end{partn}
\AND
\widetilde\sigma = \begin{partn}{3}A_1&\cdots&A_k\\\hhline{~|~|~}A_1&\cdots&A_k\end{partn}.
\]
The projections of $D$ are then precisely $P_D=\bigl\{ \overline{\sigma}\colon \sigma\in\Eq[n]\bigr\}$.

Let us denote by $\Delta=\Delta_{[n]}$ and $\nabla = \nabla_{[n]}$ the equality and full relations on $[n]$, respectively.  Without loss of generality we can assume that $p_0= \overline{\nabla}$.  For the abstract generators 
$a_{\overline{\sigma},\overline{\tau}}$ appearing in Theorem~\ref{th:linkedpres}, we omit the overlines, and also temporarily revert to the square bracket notation to enhance readability: $[\sigma,\tau]:=a_{\overline{\sigma},\overline{\tau}}$.  Since~$D$ is a rectangular band, $\F_D = \nabla_{P_D}$ is the universal relation on $P_D$, so for the spanning tree $T$ we can pick the star graph with edges $\overline{\nabla}\to\overline{\sigma}$ for $\sigma\in \Eq[n]\setminus\{\nabla\}$.  The presentation \mbox{\eqref{eq:rwp20}--\eqref{eq:rwp23}} becomes:
\begin{alignat}{2}
\label{eq:rwp31}
&[\nabla,\sigma]=1 && (\sigma\in\Eq[n]\setminus\{\nabla\}), \\
\label{eq:rwp32}
&[\sigma,\sigma]=1 && (\sigma\in\Eq[n]), \\
\label{eq:rwp33}
&[\sigma,\tau]=[\tau,\sigma]^{-1} && (\sigma,\tau\in\Eq[n]), \\
\label{eq:rwp34}
&[\sigma,\mu]^{-1} [\sigma,\nu]= [\tau,\mu]^{-1} [\tau,\nu] &\quad& ((\ol\sigma,\ol\tau;\ol\mu,\ol\nu)\in \Lk_D).
\end{alignat}
We aim to show that $[\sigma,\tau]=1$ for all $\sigma,\tau\in\Eq[n]$ as a consequence of 
\eqref{eq:rwp31}--\eqref{eq:rwp34}. We first prove a special instance:

\begin{claim}
\label{cl:pnD0-1}
If $\sigma,\tau\in\Eq[n]$ satisfy $\sigma\subseteq \tau$, then $[\sigma,\tau]=1$ as a consequence of \eqref{eq:rwp31}--\eqref{eq:rwp34}.
\end{claim}

\begin{proof}
Writing $\cdot$ for the action in \eqref{eq:sract}, notice that
\[
\ol\nabla \cdot \widetilde{\tau}=\ol\nabla \quad\text{and}\quad \ol\sigma\cdot\widetilde{\tau}=\ol\tau,
\]
the second equality being a consequence of the assumption $\sigma\subseteq\tau$.
Hence $(\ol\nabla,\ol\sigma;\ol\nabla,\ol\tau)$ is a~$\widetilde\tau$-linked diamond, and \eqref{eq:rwp34} contains the relation
\[
[\nabla,\nabla]^{-1}\: [\nabla,\tau]=[\sigma,\nabla]^{-1}\:[\sigma,\tau].
\]
But $[\nabla,\nabla]=[\nabla,\tau]=[\sigma,\nabla]=1$ from \eqref{eq:rwp31}-- \eqref{eq:rwp33}, so indeed $[\sigma,\tau]=1$.
\end{proof}

Now consider arbitrary $\sigma,\tau\in\Eq[n]$.
Denote by $\sigma\wedge\tau$ and $\sigma \vee\tau$ their meet (i.e.~intersection) and join in the lattice $\Eq[n]$.
This time we have
\[
\ol\sigma\cdot \widetilde{\tau}=\ol{\sigma\vee\tau}\quad\text{and}\quad
\ol{\sigma\wedge\tau}\cdot \widetilde{\tau}=\ol\tau,
\]
so that $(\ol\sigma,\ol{\sigma\wedge\tau};\ol{\sigma\vee\tau},\ol\tau)$ is $\widetilde\tau$-linked, and so \eqref{eq:rwp34} contains the relation
\[
[\sigma,\sigma\vee\tau]^{-1} \: [\sigma,\tau]=[\sigma\wedge\tau,\sigma\vee\tau]^{-1} \: [\sigma\wedge\tau,\tau].
\]
By Claim \ref{cl:pnD0-1} we have $[\sigma,\sigma\vee\tau] = [\sigma\wedge\tau,\sigma\vee\tau] = [\sigma\wedge\tau,\tau] = 1$, and it again follows that $[\sigma,\tau]=1$, 
completing the proof of the proposition.
\end{proof} 
\setcounter{claim}{0}

\begin{rem}
The previous result again has a topological interpretation, saying that a natural simplicial $2$-complex corresponding to the projections of rank $0$ is simply connected; cf.~\cite[Section 10]{EGMR}.  The presence of the linked diamonds means that this complex has many (non-degenerate) $2$-cells, and hence Proposition \ref{pr:free} does not apply.
One may wonder whether it is possible to gain a more intrinsic understanding of this complex, so as to naturally deduce its simple connectedness via a purely topological route.
\end{rem}

\section{Second family of presentations for maximal subgroups of \boldmath{$\PG(P)$}}
\label{sec:pres-quot}

We now aim to obtain a presentation for a maximal subgroup of $\PG(P)$ that would reflect the fact that such a group is a homomorphic image of the corresponding maximal subgroup of $\IG(E)$.
Here $P$ is a projection algebra, and $E = \sfE(P)$ is the associated biordered set, and, as 
in Section \ref{sec:pres-linked}, we can assume that $P = \sfP(S)$ and $E = \sfE(S)$ for some projection-generated regular $*$-semigroup~$S$.  We fix such an $S$ for the rest of this section.  We also fix a projection~${p_0\in P(\sub E)}$, and we continue to write $\GI$ and $\GP$ for the maximal subgroups of $\IG(E)$ and~$\PG(P)$ containing~$x_{p_0}$, respectively.

\subsection{Presentations via singular squares}
\label{ss:pres-quot}

Our starting point is \cite[Theorem 7.9]{EGMR}, which gives the following presentation for $\PG(P)$ as a homomorphic image of $\IG(E)$:

\begin{thm}
\label{thm:prespgig}
Let $P$ be a projection algebra,  let $E = \sfE(P)$ be the associated biordered set, and let $X_E=\{ x_e\colon e\in E\}$ be an alphabet in one-one correspondence with $E$.
Then the free projection-generated semigroup $\PG(P)$ is defined by the presentation
\begin{align}
\label{eq:pge1}
\langle X_E\ | \ & x_e x_f=x_{ef} \quad ((e,f)\in \Bp),
\\
\label{eq:pge2}
& x_px_q=x_{pq}\quad (p,q\in P)\rangle.
\end{align}
\end{thm}

We will now apply  the Reidemeister--Schreier rewriting method to this presentation.
Roughly speaking, rewriting the defining relations \eqref{eq:pge1} for $\IG(E)$  will give us a presentation for its maximal subgroup $\GI$ as in Section \ref{sec:maxIG}, while rewriting the additional relations \eqref{eq:pge2} will yield  some further relations, giving us the quotient $\GP$.

Once again we need to set up the Reidemeister--Schreier framework.
As explained in Subsections \ref{ss:genrew} and \ref{ss:linked}, we can index the $\R$- and $\L$-classes in the $\D$-class $D$ of $p_0$ in $S$, and use this indexing for the $\D$-classes of $x_{p_0}$ in both $\PG(P)$ and $\IG(E)$.
And, as in Subsection \ref{ss:linked}, 
we take $I=J=P_D=P\cap D$, and $K=\F_D$.
Moreover, for $(p,q)\in\F_D$, the idempotent $e_{p,q} \in H_{p,q} = R_p\cap L_q$ is $e_{p,q} = pq$.
 Following the approach from Subsection \ref{ss:span}, we will fix an arbitrary spanning tree $T$ of the Graham--Houghton graph $\GH(D)$; this leads to a natural system of Schreier representatives~$\{{r_p\colon p\in P_D\}}$, as explained in Subsection \ref{ss:span}.
 These representatives are words over the alphabet~${\{ x_e\colon e\in E_D\}}$, where $E_D:=E\cap D$.

Applying  the Reidemeister--Schreier rewriting to the presentation for $\PG(P)$ given in Theorem \ref{thm:prespgig}, we obtain a presentation for $\GP$ over the alphabet
\[
B:= \bigl\{ [t,x_e]\colon t\in P_D,\ e\in E,\ t\cdot e\neq 0\bigr\},
\]
corresponding to the generating set
\[
\{ x_{p_0}r_t x_e r_{t\cdot e}'\colon t\in P_D,\ e\in E,\ t\cdot e\neq 0\},
\]
and with the following defining relations over $B$:
\begin{alignat}{2}
\label{eq:pq1}
&[t,x_e]=1 && (t\in P_D,\ e\in E_D,\ r_tx_e=r_{t\cdot e}),\\
\label{eq:pq2}
&[t,x_e][t\cdot e, x_f]=[t,x_{ef}] &\quad& (t\in P_D,\ (e,f)\in\Bp,\ t\cdot ef\neq 0),\\
\label{eq:pq3}
&[t,x_p][t\cdot p,x_q]=[t,x_{pq}]&& (t\in P_D,\ p,q\in P,\ t\cdot pq\neq 0).
\end{alignat}
The first two families of relations constitute the presentation for $\GI$ we considered in Section \ref{sec:maxIG}, where it was transformed into the presentation \eqref{eq:IG5}--\eqref{eq:IG7}.
Here we follow this same process, but in our new notation.
The key step is the introduction of new generators \eqref{eq:IG4}, which now becomes:
\begin{equation}
\label{eq:pq4}
\fg{p,q} = [\mapq(p), x_{pq}]\quad (p,q\in\F_D).
\end{equation}
Here, $\mapq(p)$ replaces $\mapj(i)$, and is obtained in reference to the spanning tree $T$, as explained in Subsection \ref{ss:span}; its salient property is that $H_{p,\mapq(p)}$ is a group, i.e.~$(p,\mapq(p))\in \F_D (= K)$.  Then the original generators were eliminated using~\eqref{eq:IG4a}, which now becomes
\begin{equation}
\label{eq:pq4a}
[t,x_e]=\fg{p,t}^{-1} \fg{p,q}, \qquad \text{where}\qquad p:=ee^*tee^*\ \text{and}\ q:=e^*te.
\end{equation}
In this elimination, note that the `$j$' from Subsection \ref{ss:genrew} is $j = t\cdot e = e^*te = q$, and the `$i$' is any element of $I$ such that ${(i,t),(i,j)\in K}$ and $es=s$ for all $s\in R_i$.  One can check that $i = ee^*tee^* = p$ is a valid choice.  For this check, note that a generator $[t,x_e]$ existing means that $t\cdot e\not=0$, i.e.~$te\mr\R t$, which gives $t = te(te)^* = tee^*t$.
The final presentation \eqref{eq:IG5}--\eqref{eq:IG7}, written in our current notation, becomes:
\begin{alignat}{2}
\label{eq:pq5}
&\fg{p,t}=\fg{p,q} && ((p,t),(p,q)\in\F_D,\ r_t x_{pq} = r_q),\\
\label{eq:pq6}
&\fg{p,\mapq(p)}=1&& (p\in P_D),\\
\label{eq:pq7}
&\fg{p,s}^{-1} \fg{p,t} = \fg{q,s}^{-1} \fg{q,t} &\quad& (\smat{ps}{pt}{qs}{qt}\in \Sq).
\end{alignat}
To obtain our desired presentation for $\GP$  we need to perform the substitution \eqref{eq:pq4a} in \eqref{eq:pq3} as well:
\begin{multline}
\label{eq:pq8}
\fg{ptp,t}^{-1} \fg{ptp,ptp} \fg{qptpq,ptp}^{-1} \fg{qptpq,qptpq}=\fg{pqptpqp,t}^{-1}\fg{pqptpqp,qptpq}\\
(t\in P_D,\ p,q\in P,\ t\cdot pq\neq 0).
\end{multline}
 A subset of these relations can be obtained by taking $(p,q)\in\F_D$ and $t=p$:
 \begin{equation}
\label{eq:pq9}
\fg{p,p}^{-1} \fg{p,p} \fg{qpq,p}^{-1} \fg{qpq,qpq}=\fg{pqpqp,p}^{-1}\fg{pqpqp,qpq}\quad
((p,q)\in \F_D).
\end{equation}
Using $(p,q)\in\F$, and swapping the left- and right-hand sides, this simplifies to:
 \begin{equation}
\label{eq:pq10}
 \fg{p,p}^{-1}\fg{p,q}=\fg{q,p}^{-1} \fg{q,q}\quad 
((p,q)\in \F_D).
\end{equation}

We will now show that in fact relations \eqref{eq:pq10}, together with the $\IG(E)$ relations
\eqref{eq:pq5}--\eqref{eq:pq7},
 imply all relations \eqref{eq:pq8}.
To do so, let $t\in P_D$ and $p,q\in P$ be such that $t\cdot pq \not= 0$, as in \eqref{eq:pq8}.
It follows from $t\cdot pq\not=0$ (and $pq=pqpq$) that $t \mr\R tpq \mr\R tpqp$, which (using the involution) implies $pqpt \mr\L t$.
But $t\cdot pq\not=0$ also implies $t\cdot p\not=0$, which gives $t\mr\R tp$, and hence $pt \mr\L t$.  
The previous two conclusions give $pt \mr\L pqpt$.  The elements $pt$ and $pqpt = pqp\sgap t$ are idempotents (as products of two projections), and since $p$ is a left identity for them both, we have an LR singular square
\[
\lmat{pt}{pt\sgap p}{pqpt}{pqpt\sgap p} = \lmat{ptp\sgap t}{ptp\sgap ptp}{pqptpqp\sgap t}{pqptpqp\sgap ptp}.
\] 
The corresponding relation from \eqref{eq:pq7} is:
\begin{equation}
\label{eq:pq11}
\fg{ptp,t}^{-1} \fg{ptp,ptp}=\fg{pqptpqp,t}^{-1} \fg{pqptpqp,ptp}.
\end{equation}
We saw above that $tp \mr\R t \mr\R tpqp$.  Since $\R$ is a left congruence, it follows that $qptp \mr\R qptpqp$.  These last two elements are again idempotents (products of two projections), and have $p$ as a common right identity, so we have a UD singular square
\[
\lmat{qptp}{qptpqp}{p\sgap qptp}{p\sgap qptpqp} = \lmat{qptpq\sgap ptp}{qptpq\sgap pqptpqp}{pqptpqp\sgap ptp}{pqptpqp\sgap pqptpqp}.
\]
This gives us the following relation from \eqref{eq:pq7}:
\begin{equation}
\label{eq:pq12}
\fg{qptpq,ptp}^{-1} \fg{qptpq,pqptpqp}=\fg{pqptpqp,ptp}^{-1} \fg{pqptpqp,pqptpqp}.
\end{equation}
And we observe that \eqref{eq:pq10} contain the following relation:
\begin{equation}
\label{eq:pq13}
\fg{qptpq,qptpq}^{-1} \fg{qptpq,pqptpqp}=\fg{pqptpqp,qptpq}^{-1} \fg{pqptpqp,pqptpqp}.
\end{equation}
Now we have
\begin{align*}
\fg{ptp,t}^{-1} & \fg{ptp,ptp} \fg{qptpq,ptp}^{-1} \fg{qptpq,qptpq}&&\\
&=\fg{pqptpqp,t}^{-1} \fg{pqptpqp,ptp} \fg{qptpq,ptp}^{-1} \fg{qptpq,qptpq}
&&\text{by \eqref{eq:pq11}}
\\
&=\fg{pqptpqp,t}^{-1} \fg{pqptpqp,pqptpqp} \fg{qptpq,pqptpqp}^{-1}  \fg{qptpq,qptpq}
&&\text{by \eqref{eq:pq12}}
\\
&=\fg{pqptpqp,t}^{-1} \fg{pqptpqp,qptpq}
&&\text{by \eqref{eq:pq13}},
\end{align*}
which is precisely the relation \eqref{eq:pq8}.
So, a presentation for the maximal subgroup $\GP$ consists of relations \eqref{eq:pq5}--\eqref{eq:pq7} and \eqref{eq:pq10}.

As a final transformation, recall 
that we derived the Schreier representatives $r_p$ and the elements $\mapq(p)$ from the spanning tree $T$. Therefore, following the argument from Subsection~\ref{ss:span}, the relations \eqref{eq:pq5} and \eqref{eq:pq6} can be replaced by the spanning tree relations $\rels_1(T)$.
And once again, we rename the generators $\fg{p,q}$ into $a_e$ where $e=pq\in E_D$, noting then that $p = ee^*$ and $q = e^*e$.
Thus we obtain the following, in which we note that relations \eqref{eq:pq18} and \eqref{eq:pq17} come from \eqref{eq:pq10} and \eqref{eq:pq7}, respectively:

\begin{thm}
\label{thm:pgmaxq}
Let $P$ be a projection algebra, let $S$ be any projection-generated regular $*$-semigroup with $P=\sfP(S)$, and let $p_0\in P$ be arbitrary.  Denote by $D$ the $\D$-class of $S$ containing~$p_0$, and let $T$ be any spanning tree for the Graham--Houghton graph $\GH(D)$.  With the rest of the notation as above, the maximal subgroup of $\PG(P)$ containing $x_{p_0}$ is defined by the presentation with generators $A = \{ a_e\colon e\in E_D\}$ and relations
\begin{alignat}{2}
\label{eq:pq15}
&a_e=1 && (e\in T),\\
\label{eq:pq18}
&a_{ee^*}^{-1}a_{e}=a_{e^*}^{-1} a_{e^*e} &\quad&(e\in E_D),\\
\label{eq:pq17}
&a_e^{-1} a_f = a_g^{-1} a_h && (\smat{e}{f}{g}{h}\in \Sq).
\\
\tag*{\qed}
\end{alignat}
\end{thm}

The above presentation simplifies somewhat when one makes 
what seems a natural requirement on the spanning tree, namely that it contains all the projections from $D$.
This can always be done, by, for example, taking a spanning tree for the friendship digraph $\F_D$, as we did in Subsection~\ref{ss:linked}, and extending it by the set of projections.
Under this assumption,
 relations \eqref{eq:pq15} contain all $a_{p}=1$ ($p\in P_D$), and
relations \eqref{eq:pq18} then become $a_{e}=a_{e^*}^{-1}$.

\begin{cor}
\label{cor:pgmaxq}
With the notation as above, and assuming that the spanning tree $T$ contains~$P_D$, the maximal subgroup of $\PG(P)$ containing $x_{p_0}$
is defined by the presentation with generators $A = \{ a_e\colon e\in E_D\}$ and relations
\begin{alignat}{2}
\label{eq:pq19}
&a_e=1 && (e\in T),\\
\label{eq:pq22}
&a_e=a_{e^*}^{-1} &&(e\in E_D),\\
\label{eq:pq21}
&a_e^{-1} a_f = a_g^{-1} a_h &\quad& (\smat{e}{f}{g}{h}\in \Sq).\\
\tag*{\qed}
\end{alignat}
\end{cor}

\subsection{The relationship between singular squares and linked squares}

The presentations for the group $\GP$ in Corollaries \ref{co:lkid} and \ref{cor:pgmaxq} are strikingly similar.  The first two sets of relations in each are the same (bearing in mind the connections between the two types of trees, as discussed before Corollary \ref{cor:pgmaxq}, and assuming we take the same tree in both presentations), while the third sets of relations in each presentation are of the form $a_e^{-1} a_f = a_g^{-1} a_h$ for squares $\smat efgh$ of two different kinds: linked squares or singular squares.
Tracing back the proofs, we claim that the letter $a_e$ ($e\in E_D$) in both presentations corresponds to the same element of $\PG(P)$, as long as one chooses the set-up carefully.  
Indeed, keeping in mind that $X_P\sub X_E$, we can make the same choice of words $\set{r_p}{p\in P_D}$ and $\set{r_p'}{p\in P_D}$ 
for both presentations, as in Subsection~\ref{ss:linked}, and we choose $\mapq(p) = p$ for all $p\in P_D$ for the second presentation.  Now consider a letter $a_e$, where $e\in E_D$, and write $t = ee^*$ and $p = e^*e$, so that~$(t,p)\in\F$ and $e=tp$.  Then in the second presentation, $a_e$ represents
\[
a_e = a_{tp} = \fg{t,p} = [\mapq(t),x_{tp}] = [t,x_e] = x_{p_0}r_tx_er_{t\cdot e}' 
= x_{p_0} r_t \sgap x_tx_p \sgap r_{e^*te}^*x_{p_0} 
= x_{p_0} r_t x_p  r_{ptp}^*x_{p_0},
\]
where in the second-last step we used \eqref{eq:pge2}, and in the last step $e^*te = pt\sgap t\sgap tp = ptp$, and the fact that the last letter of $r_t$ is $x_t$ by definition (so that $r_tx_t=r_t$ by \eqref{eq:pge2}).  Comparing this to~\eqref{eq:tp}, the claim is proved.

Now that we know the generators in the two presentations for $\GP$ correspond to the same elements of the group, we know \emph{a priori} that the sets of relations in the two presentations are equivalent.  One might then wonder if this can be shown directly, i.e.~show explicitly that each set of relations implies the other set.  It turns out that this is actually a worthwhile exercise, as it ultimately leads to a deeper understanding of the relationship between (and among) the two kinds of squares.  The next three lemmas explore this relationship, showing that any square implies the existence of other squares of various kinds. The squares under consideration in each of them are illustrated in  Figure \ref{fig:SStoLS}.  After we prove these lemmas, we use them to explicitly derive the presentations from Corollaries~\ref{co:lkid} and~\ref{cor:pgmaxq} from each other.
As an additional consequence of this derivation, we will see that some further simplification in the relations is possible, in the sense that certain `special' square relations imply the others; see Remark \ref{rem:LSSS}.

\begin{lemma}\label{la:LSSS}
Let $\smat efgh = \smat{sv}{sw}{uv}{uw}$ be a $p$-linked square of idempotents, corresponding to the $p$-linked diamond of projections $(s,u;v,w)$.  Then the squares 
\[
\smat efv{vw} = \smat{sv}{sw}v{vw} \ANd \smat gh{wv}w = \smat{uv}{uw}{wv}w
\]
are both UD-singularised by $p$. 
\end{lemma}

\begin{proof}
Because of the $\F$-relationships among the projections $s,u,v,w$ (including $v\mr\F w$ as explained in Remark~\ref{rem:suvw}), we have the following idempotents:
\[
\begin{pmatrix}
s & sv & & sw \\
vs & v & vu & vw \\
 & uv & u & uw \\
 ws & wv & wu & w
\end{pmatrix}.
\]
In the above, entries in the same row or column are $\R$- or $\L$-related, respectively, 
but it is possible for some of these $\R$- or $\L$-classes to coincide. We only show that the square $\smat efv{vw} = \smat{sv}{sw}v{vw}$ is UD-singularised by $p$, as the other is analogous.  Since $v=psp$ and $w=pup$ (as $(s,u;v,w)$ is $p$-linked), it follows that $p$ is a right identity for both $sv$ and $sw$.  We also have $p\sgap sv = pspsp = psp = v$.
\end{proof}

\begin{figure}[t]
\begin{center}
\scalebox{.88}{
\begin{tikzpicture}[scale=1]

\nc\colsq[3]{\fill[#3](#1,#2)--(#1+1,#2)--(#1+1,#2+1)--(#1,#2+1)--(#1,#2);}

\foreach \x/\y in {0/0,0/2,1/1,1/3,2/0,2/2,3/1,3/3,4/3,4/1} {\colsq\x\y{gray!30}}
\foreach \x/\y in {0/3,1/2,2/1,3/0} {\colsq\x\y{blue!20}}

\foreach \x in {0,1,2,3,4} {\draw[dashed, thick] (-1,4-\x)--(0,4-\x); \draw[ultra thick] (0,4-\x)--(5,4-\x);}
\foreach \x in {0,1,2,3,4,5} {\draw[dashed, thick] (\x,5)--(\x,4); \draw[ultra thick] (\x,0)--(\x,4);}

\draw[ultra thick] (0,0)--(5,0)--(5,4)--(0,4)--(0,0)--(4,0);

\node () at (-.5,3.5) {\footnotesize $s$};
\node () at (-.5,2.5) {\footnotesize $v$};
\node () at (-.5,1.5) {\footnotesize $u$};
\node () at (-.5,0.5) {\footnotesize $w$};

\node () at (0.5,4.5) {\footnotesize $s$};
\node () at (1.5,4.5) {\footnotesize $v$};
\node () at (2.5,4.5) {\footnotesize $u$};
\node () at (3.5,4.5) {\footnotesize $w$};

\node[inner sep=0pt] (e) at (1.5,3.5) {$\mathrlap{e}{\phantom{f}}$};
\node[inner sep=0pt] (f) at (3.5,3.5) {$\mathrlap{f}{\phantom{f}}$};
\node[inner sep=0pt] (g) at (1.5,1.5) {$\mathrlap{g}{\phantom{f}}$};
\node[inner sep=0pt] (h) at (3.5,1.5) {$\mathrlap{h}{\phantom{f}}$};
\node (ep) at (4.5,3.5) {$\mathrlap{ep}{\phantom{ff}}$};
\node (gp) at (4.5,1.5) {$\mathrlap{gp}{\phantom{ff}}$};

\draw[very thick] (e)--(f)--(h)--(g)--(e);

\draw[very thick, red] (3.75,3.25)--(4.5,3.25)--(4.5,1.75)--(3.75,1.75)--(3.75,3.25)--(4.5,3.25);
\draw[very thick, red] (1.25,3.75)--(4.75,3.75)--(4.75,1.25)--(1.25,1.25)--(1.25,3.75)--(4.75,3.75);

\begin{scope}[shift={(7,0)}]

\foreach \x/\y in {0/0,0/2,1/0,1/1,1/3,2/0,2/2,3/1,3/2,3/3} {\colsq\x\y{gray!30}}
\foreach \x/\y in {0/3,1/2,2/1,3/0} {\colsq\x\y{blue!20}}

\foreach \x in {0,1,2,3,4} {\draw[dashed, thick] (-1,4-\x)--(0,4-\x); \draw[ultra thick] (0,4-\x)--(4,4-\x);}
\foreach \x in {0,1,2,3,4} {\draw[dashed, thick] (\x,5)--(\x,4); \draw[ultra thick] (\x,0)--(\x,4);}

\draw[ultra thick] (0,0)--(4,0)--(4,4)--(0,4)--(0,0)--(4,0);

\node () at (-.5,3.5) {\footnotesize $s$};
\node () at (-.5,2.5) {\footnotesize $v$};
\node () at (-.5,1.5) {\footnotesize $u$};
\node () at (-.5,0.5) {\footnotesize $w$};

\node () at (0.5,4.5) {\footnotesize $s$};
\node () at (1.5,4.5) {\footnotesize $v$};
\node () at (2.5,4.5) {\footnotesize $u$};
\node () at (3.5,4.5) {\footnotesize $w$};

\node[inner sep=0pt] (e) at (1.5,3.5) {$\mathrlap{e}{\phantom{f}}$};
\node[inner sep=0pt] (f) at (3.5,3.5) {$\mathrlap{f}{\phantom{f}}$};
\node[inner sep=0pt] (g) at (1.5,1.5) {$\mathrlap{g}{\phantom{f}}$};
\node[inner sep=0pt] (h) at (3.5,1.5) {$\mathrlap{h}{\phantom{f}}$};

\draw[very thick] (e)--(f)--(h)--(g)--(e);

\draw[very thick, red] (.25,3.75)--(3.75,3.75)--(3.75,2.5)--(.25,2.5)--(.25,3.75)--(3.75,3.75);
\draw[very thick, red] (2.5,2.25)--(3.25,2.25)--(3.25,1.75)--(2.5,1.75)--(2.5,2.25)--(3.25,2.25);

\end{scope}

\begin{scope}[shift={(-6,0)}]

\foreach \x/\y in {0/0,0/2,1/0,1/1,1/3,2/0,2/2,3/1,3/2,3/3} {\colsq\x\y{gray!30}}
\foreach \x/\y in {0/3,1/2,2/1,3/0} {\colsq\x\y{blue!20}}

\foreach \x in {0,1,2,3,4} {\draw[dashed, thick] (-1,4-\x)--(0,4-\x); \draw[ultra thick] (0,4-\x)--(4,4-\x);}
\foreach \x in {0,1,2,3,4} {\draw[dashed, thick] (\x,5)--(\x,4); \draw[ultra thick] (\x,0)--(\x,4);}

\draw[ultra thick] (0,0)--(4,0)--(4,4)--(0,4)--(0,0)--(4,0);

\node () at (-.5,3.5) {\footnotesize $s$};
\node () at (-.5,2.5) {\footnotesize $v$};
\node () at (-.5,1.5) {\footnotesize $u$};
\node () at (-.5,0.5) {\footnotesize $w$};

\node () at (0.5,4.5) {\footnotesize $s$};
\node () at (1.5,4.5) {\footnotesize $v$};
\node () at (2.5,4.5) {\footnotesize $u$};
\node () at (3.5,4.5) {\footnotesize $w$};

\node[inner sep=0pt] (e) at (1.5,3.5) {$\mathrlap{e}{\phantom{f}}$};
\node[inner sep=0pt] (f) at (3.5,3.5) {$\mathrlap{f}{\phantom{f}}$};
\node[inner sep=0pt] (g) at (1.5,1.5) {$\mathrlap{g}{\phantom{f}}$};
\node[inner sep=0pt] (h) at (3.5,1.5) {$\mathrlap{h}{\phantom{f}}$};

\draw[very thick] (e)--(f)--(h)--(g)--(e);

\draw[very thick, red] (1.75,3.25)--(3.25,3.25)--(3.25,2.5)--(1.75,2.5)--(1.75,3.25)--(3.25,3.25);
\draw[very thick, red] (1.75,3.25-2)--(3.25,3.25-2)--(3.25,2.5-2)--(1.75,2.5-2)--(1.75,3.25-2)--(3.25,3.25-2);

\end{scope}

\end{tikzpicture}
}
\caption{Left: a $p$-linked square (black), and the two resulting singular squares (red) from Lemma~\ref{la:LSSS}.
Middle: a singular square (black), and the two resulting projection-singularised squares (red) from Lemma \ref{la:SSpSS}.
Right: a projection-singularised square (black), and the two resulting linked squares (red) from Lemma \ref{la:SSLS}.
In each diagram, blue $\H$-classes contain projections, and grey $\H$-classes contain idempotents.}
\label{fig:SStoLS}
\end{center}
\end{figure}
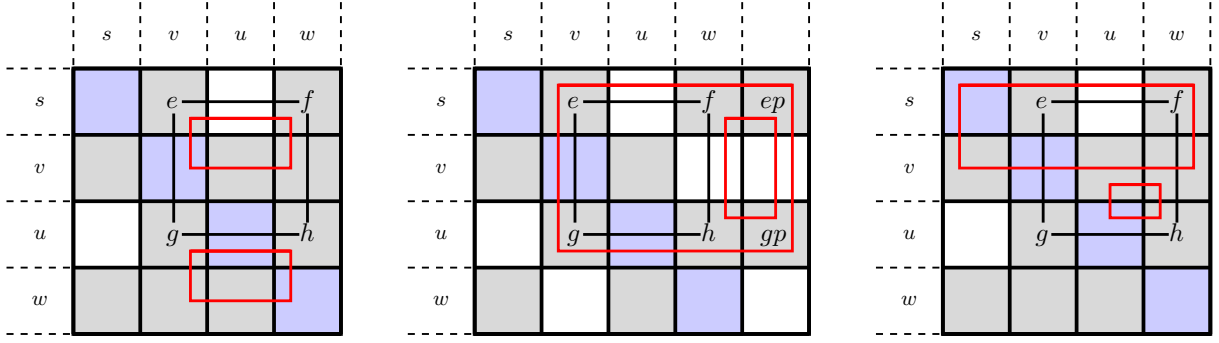

Of course applying the involution to the vertical squares in Lemma \ref{la:LSSS} yields a pair of horizontal squares.  Similarly, the next two results have duals, which we will not state explicitly.

\begin{lemma}\label{la:SSpSS}
Let $\smat efgh$ be a singular square that is LR-singularised by the idempotent $pq$, where $(p,q)\in\F$.  Then the squares 
\[
\smat e{ep}g{gp} \ANd \smat f{ep}h{gp}
\]
are both LR-singularised by $p$.  
\end{lemma}

\begin{proof}
Since $pq$ is a left identity for each of $e$, $f$, $g$ and $h$, it follows that $p$ is a left identity for the same elements. Hence $\smat e{ep}g{gp}$ and $\smat f{fp}h{hp}$ are singularised by $p$. 
Since $epq=f$ and  $p\mr\F q$ it follows that
$fp=(epq)p=ep$; analogously, $hp=gp$, completing the proof.
\end{proof}

\begin{lemma}\label{la:SSLS}
Let $\smat efgh = \smat{sv}{sw}{uv}{uw}$ be a singular square that is LR-singularised by the projection~$p$, and where $(s,v),(s,w),(u,v),(u,w)\in\F$.  Then $(s,v;s,w)$ and $(u,v;u,w)$ are $p$-linked diamonds of projections (i.e.~$(s,v,w)$ and $(u,v,w)$ are $p$-linked triangles), yielding the $p$-linked squares of idempotents
\[
\smat sf{e^*}{vw} = \smat s{sw}{vs}{vw} \ANd \smat uh{g^*}{vw} = \smat u{uw}{vu}{vw}.
\]
\end{lemma}

\begin{proof}
We prove the statement for $(s,v;s,w)$; $(u,v;u,w)$ is analogous.
Since, by assumption, $(s,v),(s,w)\in\F$, and $(s,s)\in\F$ trivially,
we just need to verify that $(v,w)\in\F$, $psp=s$ and $pvp=w$.
Keeping in mind \ref{it:RS5}, we have $s=ee^*=ff^*$, $v=e^*e$ and $w=f^*f$.
Now we have:
\[
psp =pee^*p=(pe)(pe)^*=ee^*=s \AND
pvp=pe^*ep=(ep)^*(ep)=f^*f=w.
\]
With the help of the second equality above, and remembering that the product of two projections is an idempotent, we also have $wvw = pvp\sgap v\sgap pvp=pvp=w$, and 
\begin{align*}
vwv &= v\sgap pvp\sgap v=vpv=e^*eppe^*e=e^* (ep)(ep)^*e=e^*ff^*e=e^*ee^*e=e^*e=v,
\end{align*}
completing the proof.
\end{proof}

We now have:

\begin{prop}
\label{pr:coreq}
The presentations given in Corollaries \ref{co:lkid} and \ref{cor:pgmaxq} are equivalent.
\end{prop}

\begin{proof}
($\Leftarrow$)
We show that  relations \mbox{\eqref{eq:pq19}--\eqref{eq:pq21}} imply~\eqref{eq:rwp43}.
Suppose $\smat efgh$ is a linked square of idempotents in $E_D$, so that ${\smat efgh = \smat{sv}{sw}{uv}{uw}}$ for a corresponding linked diamond of projections $(s,u;v,w)$; cf.~Remark~\ref{rem:plinkedsquare}.  Using the singular squares $\smat efv{vw}$ and $\smat gh{wv}w$ from Lemma \ref{la:LSSS}, we calculate
\[
a_e^{-1}a_f 
\stackrel{\text{\tiny\eqref{eq:pq21}}}{=} a_v^{-1}a_{vw} 
\stackrel{\text{\tiny\eqref{eq:pq19}}}{=} a_{vw} 
\stackrel{\text{\tiny\eqref{eq:pq22}}}{=} a_{wv}^{-1} 
\stackrel{\text{\tiny\eqref{eq:pq19}}}{=} a_{wv}^{-1} a_w
\stackrel{\text{\tiny\eqref{eq:pq21}}}{=} 
a_g^{-1}a_h,
\]
which is exactly relation \eqref{eq:rwp43} for our original square $\smat efgh$.

($\Rightarrow$)
Now we show that  relations \eqref{eq:rwp40}--\eqref{eq:rwp43} imply the singular square relations
arising from $\smat efgh = \smat{sv}{sw}{uv}{uw}$.
First we do this in the case when the square is LR-singularised by a projection $p$.  By Lemma \ref{la:SSLS} we have two $p$-linked squares of idempotents
$\smat sf{e^*}{vw}$ and $ \smat uh{g^*}{vw}$.
Using these, we calculate
\[
a_f
\stackrel{\text{\tiny\eqref{eq:rwp40}}}{=} a_s^{-1}a_f
\stackrel{\text{\tiny\eqref{eq:rwp43}}}{=} a_{e^*}^{-1}a_{vw}
\stackrel{\text{\tiny\eqref{eq:rwp42}}}{=} a_ea_{vw}
\ANd
a_h
\stackrel{\text{\tiny\eqref{eq:rwp40}}}{=} a_u^{-1}a_h
\stackrel{\text{\tiny\eqref{eq:rwp43}}}{=} a_{g^*}^{-1}a_{vw}
\stackrel{\text{\tiny\eqref{eq:rwp42}}}{=} a_ga_{vw}.
\]
These in turn yield $a_e^{-1}a_f = a_{vw}$ and $a_g^{-1}a_h = a_{vw}$, which combine to give $a_e^{-1}a_f = a_g^{-1}a_h$, i.e.~the relation~\eqref{eq:pq21} for the square $\smat efgh$.

Now consider a square $\smat efgh$ that is LR-singularised by an arbitrary idempotent $pq$, where $(p,q)\in\F$. By Lemma \ref{la:SSpSS}, the squares $\smat e{ep}g{gp}$ and $\smat {ep}f{gp}h$ are LR-singularised by the projection~$p$. By the previous paragraph, \eqref{eq:rwp40}--\eqref{eq:rwp43} imply
\[
a_e^{-1}a_{ep} = a_g^{-1}a_{gp} \ANd a_{ep}^{-1}a_f = a_{gp}^{-1}a_h,
\]
which, in turn, imply
$a_e^{-1}a_f = a_e^{-1}a_{ep} \sgap a_{ep}^{-1}a_f = a_g^{-1}a_{gp} \sgap a_{gp}^{-1}a_h = a_g^{-1}a_h$, i.e.~relation~\eqref{eq:pq21} for the LR square $\smat efgh$.

The proof for UD squares is dual, and the proposition follows.
\end{proof}

\begin{rem}\label{rem:LSSS}
It is worth drawing out the fact (which is apparent in the proof of Proposition~\ref{pr:coreq}) that in the presentation from Corollary \ref{cor:pgmaxq}, we only need to include singular square relations~\eqref{eq:pq21} for squares singularised by projections.
In fact, looking at the form of the squares in Lemma \ref{la:LSSS} (which were used in the proof of Proposition \ref{pr:coreq}), we only need to include such relations for projection-singularised squares in which one of the entries (of the square) is a projection.

It is also worth noting that the exact form of the spanning tree $T$ appearing in the presentations from Corollaries \ref{co:lkid} and \ref{cor:pgmaxq} did not play a part in the above deductions, only that it contains the projections $P_D$.
\end{rem}

\section{Maximal subgroups in \boldmath{$\PG(\sfP(\P_n))$} and \boldmath{$\IG(\sfE(\P_n))$}:
the common set-up}
\label{sec:common}

At this point we turn our attention to our main application of the theory developed in the first half of the paper: computing the maximal subgroups in the free projection-generated regular $*$-semigroup $\PG(P)$ and in the free idempotent-generated semigroup $\IG(E)$ arising from the partition monoid~$\P_n$. 
Following the first preliminaries on $\P_n$ from Subsection \ref{ss:Pn},
here we establish the common notation and framework to be used in the subsequent sections.

\subsection{Relations and presentations}

Throughout what follows, we will consider the partition monoid $\P_n$ for some $n\geq 2$, and a projection $p_0\in \P_n$ of rank $r\leq n-2$. We let $P:=\sfP(\P_n)$ be the projection algebra of $\P_n$, and $E:=\sfE(\P_n)$ the biordered set of idempotents of $\P_n$.
We will again denote by $\GP$ and $\GI$ the maximal subgroups of $\PG(P)$ and $\IG(E)$ containing $x_{p_0}$, respectively.  
In $\P_n$, we will write:
\bit
\item $D(n,r)$ for the $\D$-class of all partitions of rank $r$,
\item $P(n,r)$ for the set of all projections in $D(n,r)$, 
\item $E(n,r)$ for the set of all idempotents in $D(n,r)$,
\item $\Sq(n,r)$ for the set of all singular squares  involving idempotents from~$E(n,r)$, and
\item $A(n,r):=\big\{ a_e\colon e\in E(n,r)\big\}$ for the alphabet featuring in the presentations from Theorems~\ref{thm:spantreepres} and~\ref{thm:pgmaxq} and Corollaries \ref{co:lkid} and \ref{cor:pgmaxq}.
\eit
We also write
\[
A_F:=\{ a_e\colon e\in F\}\subseteq A(n,r) \quad\text{for any $F\subseteq E(n,r)$.}
\]
The relations featuring in the presentations established in Sections \ref{sec:maxIG}--\ref{sec:pres-quot} come in families, and we will use different notations for these.  As before, we have the sets
\[
\rels_1(F):=\big\{ (a_e=1)\colon e\in F\big\} \quad \text{for any $F\sub E(n,r)$},
\]
and
\[
\rels_\subSq(n,r):=\big\{ (a_e^{-1}a_f=a_g^{-1}a_h)\colon \smat{e}{f}{g}{h}\in \Sq(n,r)\big\}.
\]
Finally, we gather the relations that say that $a_e$ is the inverse of $a_{e^*}$, which are present only in  Corollary \ref{cor:pgmaxq}:
\[
\rels_\subinv(n,r):=\big\{ (a_e=a_{e^*}^{-1})\colon e\in E(n,r)\big\}.
\]
With this notation, Theorem \ref{thm:spantreepres} and Corollary \ref{cor:pgmaxq} respectively say that:
\bit
\item
$\GI$ has presentation $\langle A(n,r) \mid \rels_1(T),\ \rels_\subSq(n,r)\rangle$, where $T$ is any spanning tree of the Graham--Houghton graph of $D(n,r)$.
\item
$\GP$ has presentation $\langle A(n,r) \mid \rels_1(T),\ \rels_\subSq(n,r),\ \rels_\subinv(n,r)\rangle$, where $T$ is any spanning tree of the Graham--Houghton graph of $D(n,r)$ that contains $P(n,r)$.
\eit

\subsection{Projections and idempotents}

In the coming sections we will of course work a lot with projections and idempotents of $\P_n$, so we need a good understanding of what these look like.  Projections are very simple: it is readily checked (and well known) that they are the partitions of the form
\[
\begin{partn}{6} A_1&\cdots&A_r&B_1&\cdots&B_s \\ \hhline{~|~|~|-|-|-} A_1&\cdots&A_r&B_1&\cdots&B_s \end{partn}.
\]
Arbitrary idempotents of $\P_n$ are more complicated, and were described in \cite{DEEFHHL15}.  To review this, we need some notation.  First, given partitions $a\in\P_X$ and $b\in\P_Y$ over disjoint sets $X$ and $Y$, their (disjoint) union is a partition in $\P_{X\cup Y}$; we will denote it by $a\oplus b$.
Next, we define the \emph{super-kernel} of a partition $a\in \P_X$ to be the relation $\KER(a):=\ker(a)\vee\coker(a)$.  If $\KER(a)=\nabla_X$ (the universal relation on $X$) we say that $a$ is \emph{connected}.  This is equivalent to the middle row~$X''$ being contained in a single component of the product graph $\Ga(a,a)$.

\begin{prop}[see {\cite[Theorem 5]{DEEFHHL15}}]\label{prop:EPn}
The idempotents of $\P_n$ are precisely the partitions of the form $e = e_1\oplus\cdots\oplus e_k$, where  each component $e_i$ is connected and has rank $\leq 1$.
\end{prop}

For example, one can check that the partitions $a,b\in\P_6$ pictured in Figure \ref{fig:P6} are both idempotents.  Moreover, $a$ is connected, and hence a single component of rank $1$, while $b$ is the sum of two such components, the $\KER(b)$-classes being $\{1,2,3,4\}$ and $\{5,6\}$.  On the other hand,~$ab$ is not an idempotent; the $\KER(ab)$-classes are $\{1,4,5,6\}$ and $\{2,3\}$, but $ab$ is not the union of partitions from $\P_{\{1,4,5,6\}}$ and $\P_{\{2,3\}}$.

The subsemigroup of $\P_n$ generated by its idempotents was proved in \cite{Ea11} to consist of the identity partition, plus all partitions of rank $<n$; in other words, $\langle \sfE(\P_n)\rangle$ is $\P_n$ with the non-identity permutations from $\S_n$ removed. 
In terms of Green's relations this semigroup is identical to $\P_n$, except for the very top $\D$-class, which is trivial. Therefore, in what follows we will work with $\P_n$ instead of $\langle \sfE(\P_n)\rangle$ as we would be required to, strictly-speaking, by the theory from Sections~\ref{sec:IGETn}--\ref{sec:pres-quot}.

\subsection{A note on rectangular bands and singular squares in $\P_n$}
\label{ss:wow}

As well as the idempotents of $\P_n$, we will also have to deal with singular squares formed by them. 
We begin with a word of warning concerning the nature of singular squares in $\P_n$, in contrast with the situation for $\T_n$.

Towards the end of the  proof of Proposition \ref{pr:treeTn}, we recalled from \cite{GR12} that every $2\times2$ rectangular band in~$\T_n$ is a singular square (and hence the $2\times2$ rectangular bands are precisely the non-degenerate singular squares in $\T_n$).  It was also noted in \cite[Remark~2.7]{GR12} that every singular square in~$\T_n$ is horizontal (and  possibly also  vertical).  It turns out that both of these special properties of $\T_n$ fail in the partition monoid $\P_n$.  

For example, the elements $\smat {e_1}{f_1}{g_1}{h_1}$ of $\sfE(\P_2)$ pictured in Figure \ref{fig:ssP2P3} (left) form a rectangular band, but this square is not singular.  Indeed, it is easy to check that the only common left identity for these elements is $1$ (cf.~\eqref{eq:LR}), and that the only common right identity is also~$1$ (cf.~\eqref{eq:UD}).  And clearly $1$ cannot singularise any non-degenerate square.

In fact, the square $\smat {e_1}{f_1}{g_1}{h_1}$ has the form $\smat p{pq}{qp}q$ for the friendly pair of projections ${(p,q) = (e_1,h_1)}$ in~$\sfP(\P_2)$.  Such a square $\smat p{pq}{qp}q$ exists for any friendly pair $(p,q)\in\F$ in a regular $*$-semigroup~$S$, but can only be singular if it is totally degenerate, meaning that $p=q=pq=qp$.  Indeed, suppose this square $\smat p{pq}{qp}q$ is LR-singularised by $u\in\sfE(S)$, so that:
\begin{multicols}{4}
\begin{thmenumerate}
\item \label{prosq1} $up = p$,
\item \label{prosq2} $uqp = qp$,
\item \label{prosq3} $pu = pq$, 
\item \label{prosq4} $qpu = q$.
\end{thmenumerate}
\end{multicols}
\noindent We then have $p \stackrel{\text{\tiny\ref{prosq1}}}{=} up = uu^*up \stackrel{\text{\tiny\ref{prosq1}}}{=} uu^*p = u(pu)^* \stackrel{\text{\tiny\ref{prosq3}}}{=} u(pq)^* = uqp \stackrel{\text{\tiny\ref{prosq2}}}{=} qp$.  It then follows that $p = p^* = (qp)^* = pq$, and since $p\mr\F q$ that $q = qpq = pq = p$.

It is still possible to find non-singular $2\times2$ rectangular bands in $\P_n$ involving no projections.  For example, this is the case for the elements $\smat {e_2}{f_2}{g_2}{h_2}$ of $\sfE(\P_3)$ pictured in Figure \ref{fig:ssP2P3} (middle).  Again, the only common left or right identity for these elements is $1$.

Regarding the failure in $\P_n$ of the second special property of $\T_n$, i.e.~that all singular squares are horizontal, Figure \ref{fig:ssP2P3} (right) shows a square $\smat{e_3}{f_3}{g_3}{h_3}$ in $\sfE(\P_3)$, and its UD-singularising element $u\in\sfE(\P_3)$.  This is not a horizontal square, however, as the only common left identity is again~$1$.  Of course, applying the involution yields a horizontal singular square that is not vertical.

The upshot of all of this is that when we want to establish singularity of a square in $\P_n$ we must specify the singularising idempotent, in addition to the four idempotents forming the square.

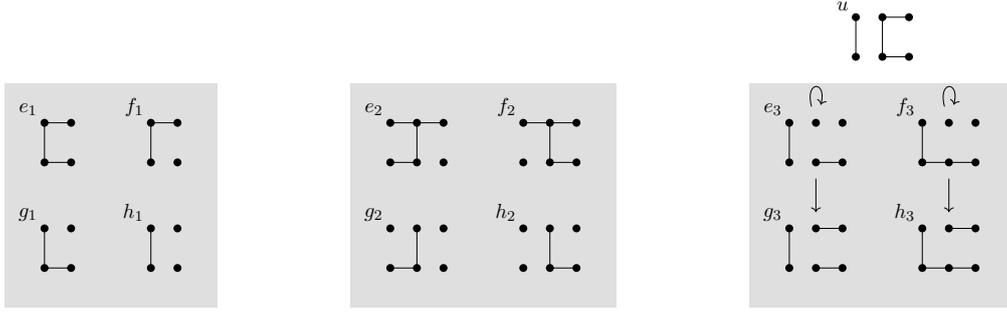
\begin{figure}[t!]
\begin{center}
\scalebox{0.7}{
\begin{tikzpicture}[scale=.5]
\begin{scope}[shift={(0,0)}]
\diagramshading2
\ediagram{e_1}{2}{\udline11 \uuline12 \ddline12}
\fdiagram{f_1}{2}{\udline11 \uuline12}
\gdiagram{g_1}{2}{\udline11 \ddline12}
\hdiagram{h_1}{2}{\udline11}
\end{scope}
\begin{scope}[shift={(13,0)}]
\diagramshading3
\ediagram{e_2}{3}{\udline22 \uuline13 \ddline12}
\fdiagram{f_2}{3}{\udline22 \uuline13 \ddline23}
\gdiagram{g_2}{3}{\udline22 \ddline12}
\hdiagram{h_2}{3}{\udline22 \ddline23}
\end{scope}
\begin{scope}[shift={(28,0)}]
\diagramshading3
\ediagram{e_3}{3}{\udline11 \ddline23}
\fdiagram{f_3}{3}{\ddline13 \udline11}
\gdiagram{g_3}{3}{\udline11 \uuline23 \ddline23}
\hdiagram{h_3}{3}{\ddline13 \udline11 \uuline23}
\udiagram{u}{3}{\udline11 \udline22 \ddline23 \uuline23}
\UDarrows
\end{scope}
\end{tikzpicture}
}
\caption{Certain squares of idempotents in $\P_2$ and $\P_3$.  See Subsection \ref{ss:wow} for more details.}
\label{fig:ssP2P3}
\end{center}
\end{figure}

\subsection{NT-reducing squares and generator elimination}
\label{ss:geneli}

We now establish two types of singular squares, one general (Lemma \ref{la:ehrsq}) and the other fairly special (Lemma \ref{la:type7}), which then let us eliminate all generators from the presentations for $\PG(P)$ and $\IG(E)$, except for those of three particular kinds.
The squares of the first type  will be used extensively in the subsequent sections as well.

For a partition $a\in \P_n$, let $\NTu(a)$ and $\NTd(a)$ respectively denote the numbers of upper and lower non-transversals of $a$, and let $\NT(a):=\NTu(a)+\NTd(a)$ be the total number of non-transversals of~$a$.  If we write $\Vert\sigma\Vert$ for the number of $\sigma$-classes for an equivalence relation $\sigma$, then
\begin{equation}\label{eq:Vert}
\NTu(a) = \Vert{\ker(a)}\Vert - \rank(a) \ANd \NTd(a) = \Vert{\coker(a)}\Vert - \rank(a).
\end{equation}
Note also that $0\leq\NTu(a)\leq n-r$ for $a\in D(n,r)$, and that
\[
\NTu(a) = 0 \iff \dom(a) = [n] \ANd \NTu(a) = n-r\iff \ker(a) = \Delta_{[n]}.
\]
Analogous statements hold for $\NTd(a)$ and codomains/cokernels.  Also note that $\NTu(a),\NTd(a)\geq1$ when $r=0$.
Since Green's $\R$- and $\L$-relations are governed by equality of (co)domains and (co)kernels (see Lemma~\ref{la:Green_Pn}), it follows that 
\begin{equation}\label{eq:NT}
a\mr\R b \implies \NTu(a) = \NTu(b) \ANd a\mr\L b \implies \NTd(a) = \NTd(b).
\end{equation}
If $p$ is a projection then $\NTu(p)=\NTd(p)$.
We stratify $P(n,r)$ according to the number of upper (equivalently, lower) blocks:
\[
P_k=P_k(n,r):=\{ p\in P(n,r)\colon \NTu(p)=k\} \quad \text{for } 0\leq k\leq n-r.
\]

\begin{rem}\label{rem:eggbox1}
We can also use the $\NTu$ and $\NTd$ parameters to stratify the egg-box diagram of the entire $\D$-class $D(n,r)$, as shown in Figure \ref{fig:eggbox1} (left).  For each $0\leq k,l\leq n-r$, it follows from~\eqref{eq:NT} that the sets
\[
D^k = D^k(n,r) := \set{a\in D(n,r)}{\NTu(a)=k} \ANd D_l = D_l(n,r) := \set{a\in D(n,r)}{\NTd(a)=l}
\]
are unions of $\R$- and $\L$-classes, respectively.  We call the sets $D^k$ and $D_l$ horizontal and vertical \emph{strips}, respectively.  The intersection of a horizontal and vertical strip,
\[
D^k_l = D^k_l(n,r) := D^k(n,r) \cap D_l(n,r) = \set{a\in D(n,r)}{\NTu(a)=k,\ \NTd(a)=l},
\]
is a union of $\H$-classes.  In Figure \ref{fig:eggbox1}, we have coloured certain sets of the above kinds.  For example, $D_0^0$ and $D_1^1$ are blue and green, respectively.  The rest of $D^0$ and $D_0$ are pink.  The set $D^0_{n-r} = D(n,r)\cap\T_n$ has been outlined in red.  The further significance of these shaded regions will be explained in Remarks \ref{rem:eggbox2} and \ref{rem:eggbox3}.

Note that this stratification of $D(n,r)$ induces one of $E(n,r)$ as well.  Moreover,
\[
E^k_l(n,r) := E\cap D^k_l(n,r)
\]
is non-empty for each $0\leq k,l\leq n-r$ when $r\geq1$, or for each $1\leq k,l\leq n-r$ when $r=0$.  For example, if $(n,r,k,l) = (10,5,4,2)$, then we have
\[
\begin{tikzpicture}[scale=.3]
\olduvs{1,...,10}
\oldlvs{1,...,10}
\draw(5,2)--(6,2) (5,0)--(8,0);
\oldstline11
\oldstline22
\oldstline33
\oldstline44
\oldstline55
\oldstline6{8.1}
\draw(10.4,1)node[right]{$\in E^k_l(n,r)$.};
\end{tikzpicture}
\]
\end{rem}

\begin{figure}[t]
\begin{center}
\scalebox{0.8}{
\begin{tikzpicture}[scale=1]

\fill[gray!10](0,0)--(6,0)--(6,6)--(0,6)--(0,0);

\fill[blue!20] (0,6)--(1,6)--(1,5)--(0,5)--(0,6);
\fill[green!20] (1,5)--(2,5)--(2,4)--(1,4)--(1,5);
\fill[red!20] (1,6)--(6,6)--(6,5)--(1,5)--(1,6);
\fill[red!20] (0,5)--(0,0)--(1,0)--(1,5)--(0,5);

\foreach \x in {0,1,2,3,5,6} {\draw[ultra thick] (\x,0)--(\x,6) (0,6-\x)--(6,6-\x);}

\draw[ultra thick](0,0)--(6,0)--(6,6)--(0,6)--(0,0)--(6,0);

\draw[thick, dashed] (0,6)--(-1,7);

\foreach \x in {0,1,2,3,6} {\node () at (\x-.5,2+.1) {$\vdots$};}
\foreach \x in {0,1,2,3,6} {\node () at (4,6.5-\x) {$\cdots$};}
\node () at (4,2+.1) {$\ddots$};

\node () at (-.6,6.22) {\footnotesize $\NTu$};
\node () at (-.22,6.6) {\footnotesize $\NTd$};

\foreach \x in {0,1,2,3,5,6} {\draw[thick, dashed] (\x,6)--(\x,7) (0,6-\x)--(-1,6-\x);}

\node () at (0.5,6.5) {\footnotesize $0$};
\node () at (1.5,6.5) {\footnotesize $1$};
\node () at (2.5,6.5) {\footnotesize $2$};
\node () at (5.5,6.5) {\footnotesize $n-r$};
\node () at (-.5,5.5) {\footnotesize $0$};
\node () at (-.5,4.5) {\footnotesize $1$};
\node () at (-.5,3.5) {\footnotesize $2$};
\node () at (-.5,.5) {\footnotesize $n-r$};

\draw[line width=1.3mm,red] (5,5)--(6,5)--(6,6)--(5,6)--(5,5)--(6,5);

\begin{scope}[shift={(11,0)}]
\fill[gray!10](0,0)--(6,0)--(6,6)--(0,6)--(0,0);

\foreach \x in {1,2,4,5} {\draw[dashed, thick] (-1,6-\x)--(0,6-\x); \draw[ultra thick] (0,6-\x)--(6,6-\x);}
\foreach \x in {1.5+.3,2.5+.3,3.5+.7,4.5+.7} {\draw[dashed, thick] (\x,7)--(\x,6); \draw[ultra thick] (\x,0)--(\x,6);}

\draw[ultra thick](0,0)--(6,0)--(6,6)--(0,6)--(0,0)--(6,0);

\draw[thick, dashed] (0,6)--(-1,7);

\node () at (-.6,6.22) {\footnotesize $\NTu$};
\node () at (-.22,6.6) {\footnotesize $\NTd$};

\node () at (2+.3,6.5) {\footnotesize $l'$};
\node () at (4+.7,6.5) {\footnotesize $l$};
\node () at (-.5,4.5) {\footnotesize $k'$};
\node () at (-.5,1.5) {\footnotesize $k$};

\node () at (2+.3,4.5) {$e_1$};
\node () at (4+.7,4.5) {$e_2$};
\node () at (2+.3,1.5) {$e_3$};
\node () at (4+.7,1.5) {$e$};
\end{scope}

\end{tikzpicture}
}
\caption{Left:  the $\D$-class $D(n,r)$ for $r\geq1$, stratified by the $\NTu$ and $\NTd$ parameters; when $r=0$, the indexing starts at $1$; see Remarks \ref{rem:eggbox1}, \ref{rem:eggbox2} and~\ref{rem:eggbox3} for more details.
Right: an NT-reducing singular square $\smat{e_1}{e_2}{e_3}e$ with base $e$; see Remark \ref{rem:NTred} for more details.}
\label{fig:eggbox1}
\end{center}
\end{figure}
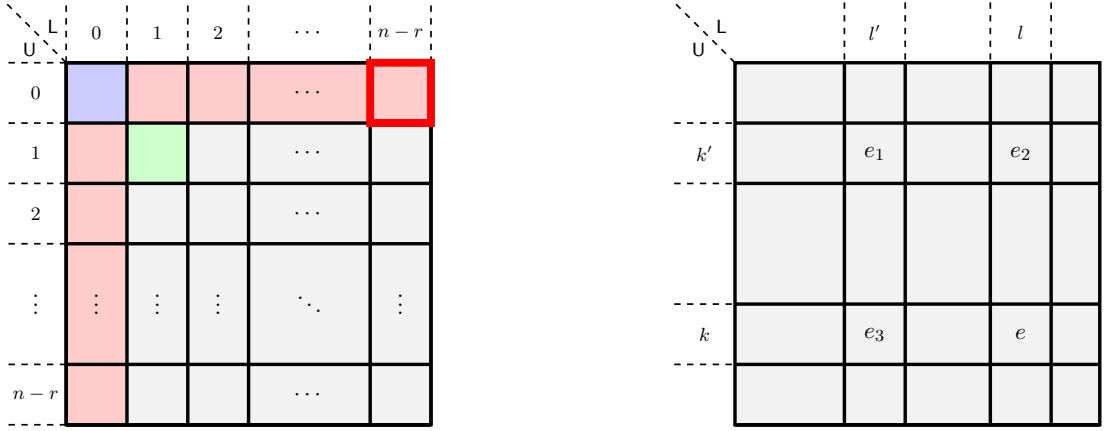

Our first goal is to show that our groups $\GI$ and $\GP$ are generated by letters $a_e$ corresponding to idempotents $e\in E(n,r)$ with small $\NTu$ and $\NTd$ parameters; see Lemma \ref{la:geneli}.  The key idea behind this is the following:

\begin{defn}\label{defn:NT}
We say that a singular square $\smat{e_1}{e_2}{e_3}{e}\in\Sq(n,r)$ is \emph{NT-reducing} with \emph{base} $e$ if
\[
\NTu(e_2)<\NTu(e) \quad \text{and} \quad \NTd(e_3)<\NTd(e).
\]
\end{defn}

Combined with \eqref{eq:NT}, the above inequalities in fact yield
\[
\NTu(e_1)=\NTu(e_2)<\NTu(e_3)=\NTu(e) \quad\textup{and}\quad
\NTd(e_1)=\NTd(e_3)<\NTd(e_2)=\NTd(e),
\]
which in turn yields
\begin{equation}\label{eq:NTei}
\NT(e_i)< \NT(e)\quad \text{for } i=1,2,3.
\end{equation}
Consequently, every NT-reducing square is non-degenerate.

\begin{rem}\label{rem:NTred}
One can visualise an NT-reducing square $\smat{e_1}{e_2}{e_3}e$ using the stratification of the $\D$-class $D(n,r)$, as explained in Remark \ref{rem:eggbox1}.  In the notation of that remark, we have
\[
e \in E^k_l \ANd e_1 \in E^{k'}_{l'} \qquad\text{for some $0\leq k'<k\leq n-r$ and $0\leq l'<l\leq n-r$,}
\]
and then we also have $e_2\in E^{k'}_l$ and $e_3\in E^k_{l'}$.  In particular, $e_1$, $e_2$ and $e_3$ are closer than $e$ to the top-left corner of the $\D$-class~$D(n,r)$.  This is all shown in Figure \ref{fig:eggbox1} (right).
\end{rem}

For an equivalence relation $\sigma$ on a set $X$, with equivalence classes $A_1,\ldots,A_k$,
let $\id_\sigma$ denote the full-(co)domain projection  $\begin{partn}{3} A_1 &\cdots & A_k\\ \hhline{~|~|~} A_1& \cdots 
&A_k\end{partn}$ in $\P_X$.
For any $a\in\P_n$, define the projections ${\ED(a):= \id_{\ker(a)}}$ and $\ER(a):= \id_{\coker(a)}$.
Observe that
$\ED(a)a=a=a\ER(a)$.
In fact, the $\ED$ and $\ER$ operations give $\P_n$ the structure of an \emph{Ehresmann semigroup}; see \cite{EG21}.

\begin{lemma}
\label{la:ehrsq}
Let $e,f\in E(n,r)$ be such that $e\mr\L f$ and $\ker(e)\subseteq \ker(f)$.
Then 
\[
\begin{pmatrix} e_1 & e_2\\ e_3& e\end{pmatrix}:=
\begin{pmatrix} f\ED(e) & f \\ e\ED(e)&e\end{pmatrix}
\] 
is a singular square. Furthermore, if $\ker(e)\neq \ker(f)$ and $\ker(e)\nsubseteq\coker(e)$ then
the square is NT-reducing with base $e$.
\end{lemma}

\begin{proof}
We have already noted that $\ED(e) e=e$; also, $\ED(e)f=f$, as $\ker(e)\subseteq \ker(f)$.
It therefore follows that $\smat{e_1}{e_2}{e_3}{e}$ is indeed an RL-square singularised by $\ED(e)$. This square is illustrated in Figure \ref{fig:ehr}.

When $\ker(e)\neq \ker(e_2)$, we have $\ker(e)\subsetneq\ker(e_2)$, so there are more $\ker(e)$-classes than $\ker(e_2)$-classes.  Since $\rank(e) = \rank(e_2)$, we then deduce $\NTu(e)>\NTu(e_2)$ from \eqref{eq:Vert}.

If $\ker(e)\nsubseteq\coker(e)$, then
\[
\coker(e_3)=\coker(eD(e))=\coker(e)\vee \ker(e)\supsetneq \coker(e),
\]
so it again follows from \eqref{eq:Vert} that $\NTd(e)>\NTd(e_3)$.
\end{proof}

The previous lemma has a dual, but we do not need to state it explicitly.

\begin{rem}
As a special instance of the square from Lemma \ref{la:ehrsq}, we can take $f = R(e)e$.  It follows from $eR(e)=e$ that $e\mr\L f$, and we have $\ker(f) = \ker(R(e)e) = \coker(e)\vee\ker(e) \supseteq \ker(e)$.  So we indeed have the square
\[
\lmat {R(e)eD(e)}{R(e)e}{eD(e)}{e},
\]
which is in fact RL- and DU-singularised by $D(e)$ and $R(e)$, respectively.  (This square is present in any Ehresmann semigroup.)
If $\ker(e)\not\sub\coker(e)$ and ${\coker(e)\not\sub\ker(e)}$ both hold, then the square is NT-reducing.
\end{rem}

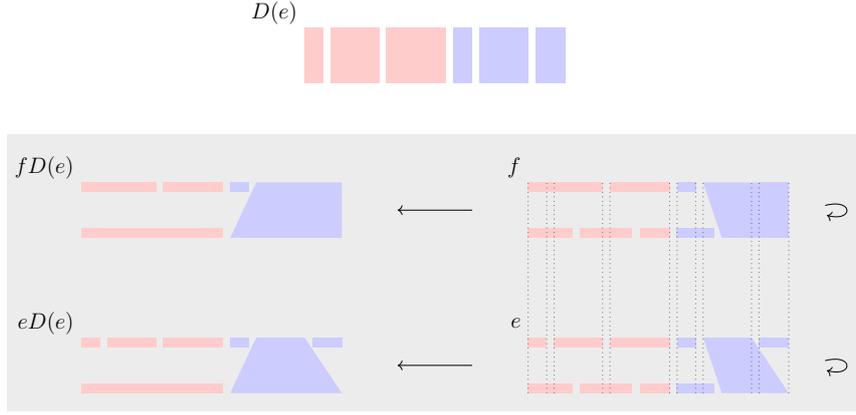
\begin{figure}[t!]
\begin{center}
\scalebox{0.7}{
\begin{tikzpicture}[scale=.7]

\nc\yyy{4.2}

\nc\XL{-1}
\nc\XR{22}
\nc\YL{-.5}
\nc\YR{7}
\fill[lightgray!30] (\XL,\YL)--(\XR,\YL)--(\XR,\YR)--(\XL,\YR)--(\XL,\YL);

\begin{scope}[shift={(0,\yyy)}]
\Eredupper13
\Eredupper{3.2}{4.8}
\Eblueupper5{5.5}
\Ebluetrans{5.7}858
\Eredlower1{4.8}
\Ebluepartlabel{fD(e)}
\draw[<-] (8+1.5,.75)--(8+3.5,.75);
\end{scope}

\begin{scope}[shift={(12,\yyy)}]
\Eredupper13
\Eredupper{3.2}{4.8}
\Eblueupper5{5.5}
\Eblueupper{7.2}8
\Ebluetrans{5.7}8{6.2}8
\Eredlower1{2.2}
\Eredlower{2.4}{3.8}
\Eredlower{4}{4.8}
\Ebluelower56
\Ebluepartlabel{f}
\node (R) at (8.5,.75) {\phantom{\Large$E$}};
\draw[->] (R) edge [loop right] ();
\end{scope}

\begin{scope}[shift={(0,0)}]	
\Eredupper1{1.5}
\Eredupper{1.7}3
\Eredupper{3.2}{4.8}
\Eblueupper5{5.5}
\Eblueupper{7.2}8
\Ebluetrans{5.7}758
\Eredlower1{4.8}
\Ebluepartlabel{eD(e)}
\draw[<-] (8+1.5,.75)--(8+3.5,.75);
\end{scope}

\begin{scope}[shift={(12,0)}]	
\Eredupper1{1.5}
\Eredupper{1.7}3
\Eredupper{3.2}{4.8}
\Eblueupper5{5.5}
\Eblueupper{7.2}8
\Ebluetrans{5.7}7{6.2}8
\Eredlower1{2.2}
\Eredlower{2.4}{3.8}
\Eredlower{4}{4.8}
\Ebluelower56
\Ebluepartlabel{e}
\node (R) at (8.5,.75) {\phantom{\Large$E$}};
\draw[->] (R) edge [loop right] ();
\foreach \x in {1,1.5,1.7,3,3.2,4.8,5,5.5,7.2,8,5.7,7} {\draw[black!70,dotted](\x,0) -- (\x,1.5+\yyy);}
\end{scope}

\begin{scope}[shift={(6,2*\yyy)}]
\Eredtrans1{1.5}1{1.5}
\Eredtrans{1.7}3{1.7}3
\Eredtrans{3.2}{4.8}{3.2}{4.8}
\Ebluetrans5{5.5}5{5.5}
\Ebluetrans{7.2}8{7.2}8
\Ebluetrans{5.7}7{5.7}7
\Ebluepartlabel{D(e)}
\end{scope}

\end{tikzpicture}
}
\caption{An example of the singular square featuring in Lemma \ref{la:ehrsq}.}
\label{fig:ehr}
\end{center}
\end{figure}

The rest of the results of this section are valid for $r\geq1$.  Analogues for $r=0$ will be given when we eventually consider this case in Section \ref{sec:r0}.

\begin{lemma}
\label{la:type7}
If $e=\begin{partn}{3} A & B& C \\ \hhline{~|-|-} A& B & C\end{partn}\oplus f\in E(n,r)$, then $\smat{e_1}{e_2}{e_3}{e}$ is an NT-reducing singular square with base $e$, where
\[
e_1:= \begin{partn}{2} A\cup C & B  \\ \hhline{~|-} A\cup B & C\end{partn}\oplus f \COMMa
e_2 := \begin{partn}{3} A\cup C & \multicolumn{2}{c}{B}  \\ \hhline{~|-|-} A& B & C\end{partn}\oplus f \ANd
e_3  := \begin{partn}{3} A& B & C  \\ \hhline{~|-|-} A\cup B & \multicolumn{2}{c}{C}\end{partn}\oplus f.
\]
\end{lemma}

\begin{proof}
The square is RL-singularised by $u = \begin{partn}{3} A&B & C \\ \hhline{~|-|~} A\cup  B & &C\end{partn}\oplus f$, and is clearly NT-reducing; see Figure~\ref{fig:type7}. 
\end{proof}

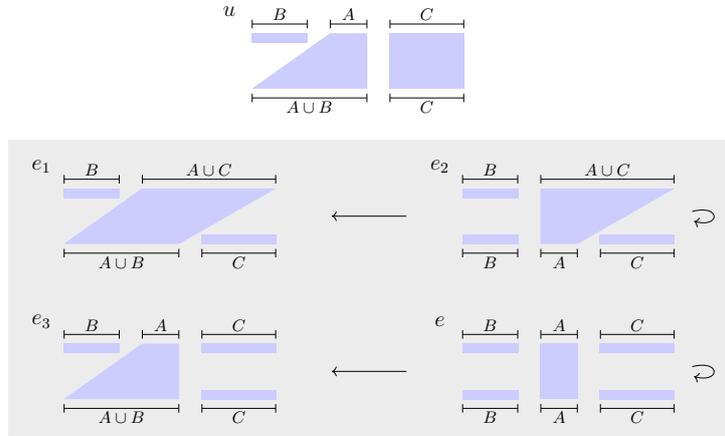
\begin{figure}[t!]
\begin{center}
\scalebox{0.7}{
\begin{tikzpicture}[scale=.7]

\nc\BBB{1.5}
\nc\AAA{1}
\nc\CCC{2}

\nc\DDD{0.6}

\pgfmathsetmacro\BL{1}
\pgfmathsetmacro\BR{1+\BBB}
\pgfmathsetmacro\AL{1+\BBB+\DDD}
\pgfmathsetmacro\AR{1+\AAA+\BBB+\DDD}
\pgfmathsetmacro\CL{1+\AAA+\BBB+2*\DDD}
\pgfmathsetmacro\CR{1+\AAA+\BBB+\CCC+2*\DDD}

\nc\yyy{4.2}

\nc\XL{-.5}
\nc\XR{19}
\nc\YL{-1}
\nc\YR{7}
\fill[lightgray!30] (\XL,\YL)--(\XR,\YL)--(\XR,\YR)--(\XL,\YR)--(\XL,\YL);

\begin{scope}[shift={(0,\yyy)}]
\bluetrans\AL\CR\BL\AR{A\cup C}{A\cup B}
\blueupper\BL\BR{B}
\bluelower\CL\CR{C}
\draw[<-] (\CR+1.5,.75)--(\CR+3.5,.75);
\bluepartlabel{e_1}
\end{scope}

\begin{scope}[shift={(\CR+4,\yyy)}]
\bluetrans\AL\CR\AL\AR{A\cup C}{A}
\blueupper\BL\BR{B}
\bluelower\BL\BR{B}
\bluelower\CL\CR{C}
\bluepartlabel{e_2}
\node (R) at (\CR,.75) {\phantom{\Large$E$}};
\draw[->] (R) edge [loop right] ();
\end{scope}

\begin{scope}[shift={(0,0)}]	
\bluetrans\AL\AR\BL\AR{A}{A\cup B}
\blueupper\BL\BR{B}
\blueupper\CL\CR{C}
\bluelower\CL\CR{C}
\draw[<-] (\CR+1.5,.75)--(\CR+3.5,.75);
\bluepartlabel{e_3}
\end{scope}

\begin{scope}[shift={(\CR+4,0)}]	
\bluetrans\AL\AR\AL\AR{A}{A}
\blueupper\BL\BR{B}
\blueupper\CL\CR{C}
\bluelower\BL\BR{B}
\bluelower\CL\CR{C}
\bluepartlabel{e}
\node (R) at (\CR,.75) {\phantom{\Large$E$}};
\draw[->] (R) edge [loop right] ();
\end{scope}

\begin{scope}[shift={(\AR/2+3,2*\yyy)}]
\bluetrans\AL\AR\BL\AR{A}{A\cup B}
\blueupper\BL\BR{B}
\bluetrans\CL\CR\CL\CR{C}{C}
\bluepartlabel{u}
\end{scope}

\end{tikzpicture}
}
\caption{The NT-reducing singular square featuring in Lemma \ref{la:type7}; in each partition, the `${}\oplus f$' part has been omitted.}
\label{fig:type7}
\end{center}
\end{figure}

We now define a very special set of idempotents, still assuming that $r\geq1$.  Specifically, let $\Esub(n,r)$ be the set of all idempotents from $E(n,r)$ that are full-domain or full-codomain or projections with precisely one upper (equivalently, lower) block:
\begin{equation}\label{eq:Fnr}
\Esub(n,r):= \bigl\{ e\in E(n,r)\colon \NTu(e)=0\textup{ or }
\NTd(e)=0 \bigr\} \cup P_1(n,r).
\end{equation}

\begin{rem}\label{rem:eggbox2}
The elements of $F = F(n,r)$ are contained in the coloured regions in Figure~\ref{fig:eggbox1}~(left).  The coloured horizontal and vertical strips contain the idempotents
\[
E^0 = E^0(n,r) := \bigl\{ e\in E(n,r)\colon \NTu(e)=0 \bigr\} 
\ANd
E_0 = E_0(n,r) := \bigl\{ e\in E(n,r)\colon \NTd(e)=0 \bigr\} ,
\]
respectively.  Their intersection $E^0\cap E_0$ is the set $P_0$ of all full-(co)domain projections; this is contained in the blue region, $D_0^0$.  The set $P_1$ is contained in the green region, $D_1^1$, which typically contains non-projection idempotents as well.
In fact, all $D_i^i$ contain non-projection idempotents except when $i=0$ or $i=n-r$.
\end{rem}

The next proof refers to the components of an idempotent, as in Proposition \ref{prop:EPn}.

\begin{lemma}
\label{la:decomp}
If $1\leq r\leq n-2$, then every element of  $E(n,r)\setminus F(n,r)$
is the base of an NT-reducing singular square.
\end{lemma}

\begin{proof}
Let $e\in E(n,r)\setminus F(n,r)$.
If $e$ is a projection, then it must have at least two upper (and lower) blocks, and Lemma \ref{la:type7} applies.
So now suppose $e$ is not a projection, i.e.~${\ker(e)\neq\coker(e)}$.
Without loss of generality assume that $\ker(e)\nsubseteq\coker(e)$.
By Lemma  \ref{la:ehrsq} it suffices to show that
there exists an idempotent $f\in E(n,r)$ with $e\mr\L f$ and $\ker(e)\subsetneq\ker(f)$. 
As $\NTu(e)\neq 0$, there must exist an upper block $A$ in $e$. 
Let $B\cup C'$ be a transversal of $e$, chosen so that it belongs to the same component as $A$ if possible (and is arbitrary otherwise).
We can now take $f$ to be obtained from $e$, by replacing the blocks $A$ and $B\cup C'$ by a single block $A\cup B\cup C'$.
Proposition~\ref{prop:EPn} can be used to show that $f$ is indeed an idempotent.  We have $e\mr\L f$ (cf.~Lemma~\ref{la:Green_Pn}) and $\ker(e) \subsetneq\ker(f)$ by construction.
\end{proof}

\begin{lemma}
\label{la:geneli}
Let $1\leq r\leq n-2$, and let $G$ be a group defined by a presentation with generators~$A(n,r)$, and defining relations that include all the square relations $\rels_{\subSq}(n,r)$.
Then $G$ is in fact generated by the set $A_{F(n,r)}$.
\end{lemma}

\begin{proof}
Throughout the proof we write $F=F(n,r)$, and we must show that $a_e\in \langle A_F\rangle $ for all $e\in E(n,r)$.
We prove this by induction on $\NT(e)$.
If $\NTu(e)=0$ or $\NTd(e)=0$  then $e\in F$, and there is nothing to prove.
This in particular includes the induction anchor $\NT(e)=0$.

Now suppose $\NTu(e)>0$ and $\NTd(e)>0$, and assume without loss of generality that $e\not\in F$.
By Lemma \ref{la:decomp} there exists an NT-reducing square $\smat{e_1}{e_2}{e_3}{e}$ with base $e$.
In particular, \eqref{eq:NTei} gives $\NT(e_i)<\NT(e)$ for each $i$, so by induction each $a_{e_i}\in \langle A_F\rangle$.
By assumption, the singular square relation $a_{e_1}^{-1}a_{e_2}=a_{e_3}^{-1} a_e$  holds in $G$. But then $a_e=a_{e_3}a_{e_1}^{-1}a_{e_2}\in \langle A_{F}\rangle$, and we are done.
\end{proof}

\subsection{Labels}
\label{ss:labs}

Now we introduce an auxiliary device of \emph{labelling} the idempotents from~$E(n,r)$
by permutations from $\S_r$. It actually turns out that the label of $e\in E(n,r)$ is the image of the generator $a_e$ under the natural homomorphism from $\GP$ or $\GI$ to~$\P_n$, provided a suitable spanning tree is chosen, and the corresponding maximal subgroup in~$\P_n$ is appropriately identified with~$\S_r$.
However, at this stage the reader need not concern themselves with this, and can regard the labelling as formalising a certain permutational aspect of each $e\in E(n,r)$.

All we say here is vacuous for $r=0$, so we assume that $r\geq1$ for the rest of this section.
Consider an arbitrary  $e\in E(n,r)$.  Suppose the transversals of $e$ are $A_1\cup B_1',\ldots, A_r\cup B_r'$, and let $a_i:=\min(A_i)$ and $b_i:=\min(B_i)$ for each $i\in [r]$.  Let $\lamp(e)$ be the partial bijection $\trans{a_1&\cdots&a_r}{b_1&\cdots&b_r}$.  And then let $\lam(e)$, the \emph{label} of $e$,  be the result of `scaling down' $\lamp(e)$ to a permutation in $\S_r$, i.e.~mapping each of the sets $\{ a_1,\ldots,a_r\}$ and $\{ b_1,\ldots,b_r\}$ in an order-preserving fashion onto~$[r]$.

Diagrammatically, one obtains $\lamp(e)$ by connecting the minimum elements in the top and bottom parts of each transversal of $e$.  The label $\lam(e)$ is then obtained by deleting points that do not participate in transversals of $\lamp(e)$, and compressing the remaining diagram to the left.  Here is an example for an idempotent $e\in E(5,3)$:
\[
\begin{tikzpicture}[scale=.4]
\olduvs{1,...,5}
\oldlvs{1,...,5}
\oldstline22
\oldstline44
\oldstline55
\olduarc14
\olddarc35
\draw(0.6,1)node[left]{$e=$};
\draw[-{latex}] (6.5,1)--(7.5,1);
\begin{scope}[shift={(11.5,0)}]
\olduvs{1,...,5}
\oldlvs{1,...,5}
\oldstline14
\oldstline22
\oldstline53
\draw(0.6,1)node[left]{$\lamp(e)=$};
\draw[-{latex}] (11,1)--(12,1);
\draw(5.4,1)node[right]{$=\trans{1&2&5}{4&2&3}$};
\end{scope}
\begin{scope}[shift={(27.5,0)}]
\olduvs{1,...,3}
\oldlvs{1,...,3}
\oldstline13
\oldstline21
\oldstline32
\draw(0.6,1)node[left]{$\lam(e)=$};
\draw(3.4,1)node[right]{$=\trans{1&2&3}{3&1&2}$.};
\end{scope}
\end{tikzpicture}
\]
The idempotents $a,b\in\sfE(\P_6)$ in Figure \ref{fig:P6} both have trivial label: $\lam(a)=1\in\S_1$ and~${\lam(b)=1\in\S_2}$.

The main purpose of the  labels for us is that they will enable us to `import' certain technical results concerning the full transformation monoid $\T_n$ from \cite{GR12}.
Throughout we will identify~$\T_n$ with its natural copy inside $\P_n$, i.e.~with the set of all partitions
with full domain and trivial cokernel; see Subsection \ref{ss:Pn}. 
In \cite[Subsection 2.5]{GR12} a labelling function was defined for idempotents in $\T_n$, and it is easy to check that this is the restriction to $E(\T_n)$ of our labelling of~$E(\P_n)$.
The following facts are also clear.

\begin{lemma}
\label{la:lamprops}
\begin{thmenumerate}
\item\label{it:lp1}
For every projection $p\in \sfP(\P_n)$ we have $\lam(p)=1$.
\item\label{it:lp2}
For every idempotent $e\in\sfE(\P_n)$ we have $\lam(e^*)=\lam(e)^{-1}$.\qed
\end{thmenumerate}
\end{lemma}

Next we list the results from \cite{GR12} that will be used subsequently.  We do so with slightly changed notation and terminology, to make them suitable for our purposes.  Specifically, we will write:
\bit
\item
$D^\T(n,r):= \T_n\cap D(n,r)$ for the $\D$-class of $\T_n$ consisting of mappings of rank $r$,
\item
$E^\T(n,r):= \T_n\cap E(n,r)$ for the set of idempotents in $D^\T(n,r)$,
\item 
$A^\T(n,r):= {\{a_e\colon e\in E^\T(n,r)\}}$ for the set of generators from $A(n,r)$ corresponding to idempotents in~$\T_n$,
\item
$\rels_\subSq^\T(n,r)$ for the set of singular square relations involving only idempotents from $E^\T(n,r)$.
\eit
We also continue to write $T_\lex$ for the spanning tree for the Graham--Houghton graph $\GH(D^\T(n,r))$, which was constructed in Section \ref{sec:IGETn}.
We call an idempotent $e\in E(n,r)$ a \emph{Coxeter idempotent} if $\lam(e) = (i,i+1)$ for some $1\leq i\leq r-1$, i.e.~if $\lam(e)$ is a Coxeter transposition.

\begin{lemma}
\label{la:GRtransl}
Let $G$ be the group defined by the presentation ${\langle A^\T(n,r)\mid \rels_1(T_\lex),\ \rels_\subSq^\T(n,r)\rangle}$, where $1\leq r\leq n-2$, and let $e,f\in E^\T(n,r)$.
\begin{thmenumerate}
\item\label{it:GRt1}
If $\lam(e)=1$ then $a_e=1$ in $G$.
\item\label{it:GRt2}
If $\lam(e)=\lam(f)$ then $a_e=a_f$ in $G$.
\item\label{it:GRt3}
If $e$ is a Coxeter idempotent, then $a_e^2=1$ in $G$.
\item\label{it:GRt4}
If $r\geq2$, then ${a_e = a_{e_1}\cdots a_{e_k}}$ in $G$ for some Coxeter idempotents $e_1,\ldots,e_k\in E^\T(n,r)$.
\end{thmenumerate}
\end{lemma}

\begin{proof}
Bearing Proposition \ref{pr:treeTn} in mind, parts \ref{it:GRt1} and \ref{it:GRt3} are precisely \cite[Lemma~4.3]{GR12} and \cite[Lemma~8.1]{GR12}, respectively.
Part  \ref{it:GRt2} is proved in \cite[Lemma 6.2]{GR12} for consecutive cycles, and the general statement follows by an easy induction based on \cite[Lemma 7.1]{GR12}.
Likewise, part 
\ref{it:GRt4} can be obtained by induction based on
\cite[Lemmas 6.5 and 7.1]{GR12}.
\end{proof}

In what follows we often need to check that $\lam(e) = \lam(f)$ for idempotents $e,f\in E(n,r)$.  One obvious way for this to happen is for $\lamp(e)$ and $\lamp(f)$ to be equal.  It also happens when $\lamp(e)$ and $\lamp(f)$ differ in a small number of entries, by a small amount, so that the relative positions of entries in relation to each other do not change when passing from $\lam'$ to $\lam$.  Here is one instance of this:

\begin{lemma}\label{la:consec}
If $e,f\in E(n,r)$ are such that
\[
\lamp(e) = \trans{i_1&\cdots&i_{k-1}&i_k&i_{k+1}&\cdots&i_r}{j_1&\cdots&j_{k-1}&j_k&j_{k+1}&\cdots&j_r} \ANd \lamp(f) = \trans{i_1&\cdots&i_{k-1}&i&i_{k+1}&\cdots&i_r}{j_1&\cdots&j_{k-1}&j&j_{k+1}&\cdots&j_r},
\]
where $i\in\{i_k,i_k\pm1\}$ and $j\in\{j_k,j_k\pm1\}$, then $\lam(e) = \lam(f)$.
\qed
\end{lemma}

\subsection[The tree $\Tfd$]{\boldmath The tree $\Tfd$}\label{ss:Tfd}

Our ability to link up with the results from \cite{GR12} cited in Lemma \ref{la:GRtransl} will depend on a `good' choice of a spanning tree~$T$ of the Graham--Houghton graph $\GH(D(n,r))$
for $r\geq 1$.
Different spanning trees will be used in Sections \ref{sec:maxPGPn} and \ref{sec:maxIGPn} below, but they will have a certain subset in common, which we specify now.  
When we come to deal with the case $r=0$ in Section~\ref{sec:r0} we will make use of an entirely different spanning tree.

In what follows, we will take the vertices of $\GH(D(n,r))$ to be $I \sqcup J$, where
\[
I:=\{ i_p\colon p\in P(n,r)\} \ANd J:=\{ j_p\colon p\in P(n,r)\}
\]
are two disjoint copies of $P(n,r)$, respectively indexing the $\R$- and $\L$-classes in $D(n,r)$.  This is simply to avoid the notational difficulties arising from the fact that normally we use the same set for both, namely $P(n,r)$ itself. Furthermore, we partition these sets as
\[
I=I_0\cup \cdots\cup I_{n-r} \AND J=J_0\cup \cdots\cup J_{n-r},
\]
where 
\[
I_k:=\big\{ i_p\colon p\in P_k(n,r)\big\} \AND J_k:=\big\{ j_p\colon p\in P_k(n,r)\big\}.
\]
(Recall that $P_k = P_k(n,r) = \set{p\in P(n,r)}{\NTu(p)=k}$.)

Let us start with the spanning tree $\Tlex$ for the Graham--Houghton graph $\GH(D^\T(n,r))$, introduced in Section \ref{sec:IGETn}; see Proposition  \ref{pr:treeTn} and Figure \ref{fig:Tlex42}.  According to our interpretation of~$\T_n$ as partitions of full domain and trivial cokernel, $\Tlex$ is a spanning tree for the subgraph of~$\GH(D(n,r))$ induced on $I_0\cup J_{n-r}$.

We now add some edges to $\Tlex$.
For a projection ${p=\begin{partn}{6} A_1 & \cdots & A_r & B_1 & \cdots & B_k\\ \hhline{~|~|~|-|-|-} A_1 & \cdots & A_r & B_1 & \cdots & B_k\end{partn}\in P(n,r)}$ with $\min(A_1)<\cdots<\min(A_r)$, define
\begin{equation}
\label{eq:ep}
e_p:= \begin{partn}{7} A_1\cup B_1\cup\cdots \cup B_k &A_2& \cdots & A_r & \multicolumn{3}{c}{} \\ \hhline{~|~|~|~|-|-|-}
A_1 & A_2 & \cdots & A_r & B_1 & \cdots & B_k\end{partn}.
\end{equation}
Note that $e_p\in E(n,r)$ has full-domain; considered as an edge in $\GH(D(n,r))$, it connects some $i\in I_0$ to
$j_p\in J_k$. 
Hence the set 
\begin{equation}
\label{eq:Tfd}
\Tfd=\Tfd(n,r):=\Tlex\cup \{ e_p\colon p\in P_1\cup \cdots \cup P_{n-r-1}\}
\end{equation}
is a spanning tree for the subgraph of $\GH(D(n,r))$ induced on $I_0\cup (J_1\cup J_2\cup\cdots \cup J_{n-r})$.  A schematic diagram of the tree $\Tfd$ is shown in Figure \ref{fig:Tfd}.

\begin{figure}[t]
\begin{center}
\scalebox{1}{
\begin{tikzpicture}[scale=1]

\nc\yy{0.2}
\nc\zz{0.4}

\foreach \x/\y in {0/2,2/4,4/5,12/10,10/8,8/7} {\draw[|-] (\x,0-\yy)--(\y,0-\yy); \draw[|-] (\x,4+\yy)--(\y,4+\yy);}
\draw[dotted] (5.5,-\yy)--(6.5,-\yy);
\draw[dotted] (5.5,4+\yy)--(6.5,4+\yy);

\node () at (1,4+\yy+\zz) {$I_0$};
\node () at (3,4+\yy+\zz) {$I_1$};
\node () at (9,4+\yy+\zz) {$I_{n-r-1}$};
\node () at (11,4+\yy+\zz) {$I_{n-r}$};
\node () at (1,0-\yy-\zz) {$J_0$};
\node () at (3,0-\yy-\zz) {$J_1$};
\node () at (9,0-\yy-\zz) {$J_{n-r-1}$};
\node () at (11,0-\yy-\zz) {$J_{n-r}$};

\fill[blue!30,opacity=0.5] (0,4)--(2,4)--(10,0)--(2,0)--(0,4);
\fill[red!30,opacity=0.5] (0,4)--(2,4)--(12,0)--(10,0)--(0,4);

{\nc\zzz2 \foreach \x/\y in {1/1,2/1,3/5,4/5,5/5,6/1,7/3} {\draw[gray!70](\y*.25,4)--(\x*.25+\zzz,0);}; }
{\nc\zzz4 \foreach \x/\y in {1/1,2/1,3/5} {\draw[gray!70](\y*.25,4)--(\x*.25+\zzz,0);}; }
{\nc\zzz6 \foreach \x/\y in {5/1,6/1,7/5} {\draw[gray!70](\y*.25,4)--(\x*.25+\zzz,0);}; }
{\nc\zzz8 \foreach \x/\y in {1/1,2/1,3/5,4/5,5/5,6/1,7/3} {\draw[gray!70](\y*.25,4)--(\x*.25+\zzz,0);}; }
{\nc\zzz{10} \foreach \x/\y in {1/1,1/3,1/5,2/4,3/4,4/2,4/6,5/4,5/7,6/7,7/5,7/6} {\draw[gray!70](\x*.25,4)--(\y*.25+\zzz,0);}; }

\foreach \x in {.25,.5,.75,1,1.25,1.5,1.75} {\fill (\x,4)circle(.05); \fill (\x,0)circle(.05);}
\foreach \x in {.25,.5,.75,1,1.25,1.5,1.75} {\fill (\x+2,4)circle(.05); \fill (\x+2,0)circle(.05);}
\foreach \x in {.25,.5,.75} {\fill (\x+4,4)circle(.05); \fill (\x+4,0)circle(.05);}
\foreach \x in {1.25,1.5,1.75} {\fill (\x+6,4)circle(.05); \fill (\x+6,0)circle(.05);}
\foreach \x in {.25,.5,.75,1,1.25,1.5,1.75} {\fill (\x+8,4)circle(.05); \fill (\x+8,0)circle(.05);}
\foreach \x in {.25,.5,.75,1,1.25,1.5,1.75} {\fill (\x+10,4)circle(.05); \fill (\x+10,0)circle(.05);}

\end{tikzpicture}
}
\caption{Schematic diagram of the tree $\Tfd$ from \eqref{eq:Tfd}, which is a spanning tree for the induced subgraph of the Graham--Houghton graph $\GH(D(n,r))$ on vertices $I_0\cup (J_1\cup J_2\cup\cdots \cup J_{n-r})$, in the case $r\geq1$.  The edges from $\Tlex$ are contained in the pink region (with vertices $I_0\cup J_{n-r}$), and the remaining edges are in the blue region.}
\label{fig:Tfd}
\end{center}
\end{figure}
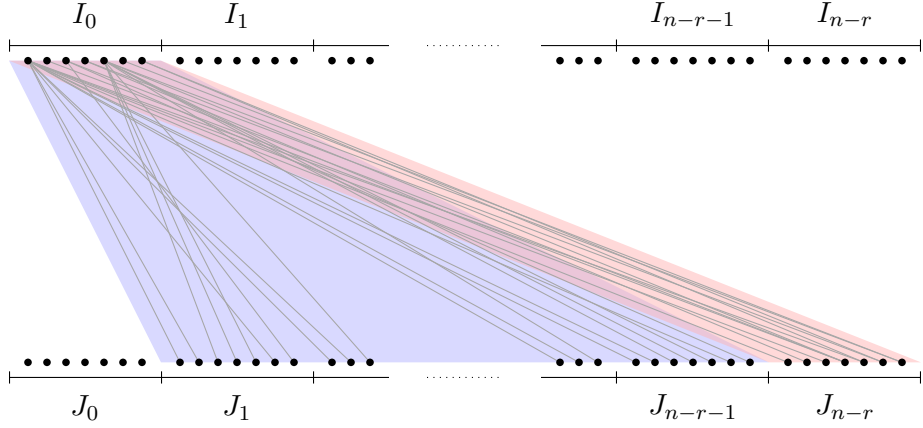

The following result is key for our link between $\IG(\sfE(\T_n))$ on the one hand, and $\PG(\sfP(\P_n))$ and $\IG(\sfE(\P_n))$ on the other:

\begin{lemma}
\label{la:fdtoTn}
For every non-projection idempotent $e\in E(n,r)$ with full domain, there exists an idempotent transformation $f\in E^\T(n,r)$ such that
$\lam(e)=\lam(f)$.
Furthermore, if ${\Tfd\subseteq T\subseteq E(n,r)}$, then the idempotent $f$ can be chosen so that $a_e=a_f$ is a consequence of~$\presn(T)$.
\end{lemma}

\begin{proof}
When $e\in \T_n$ there is nothing to prove, so suppose instead that $e\not\in\T_n$, and write $e=\begin{partn}{6} A_1&\cdots&A_r&\multicolumn{3}{c}{} \\ \hhline{~|~|~|-|-|-} B_1&\cdots&B_r&C_1&\cdots&C_k\end{partn}$.  Since $e$ is an idempotent, we have $B_i\sub A_i$ for each $i\in[r]$.  Let $b_i=\min(B_i)$ for each $i$, assuming that $b_1<\cdots<b_r$, and also pick arbitrary $c_i\in C_i$ for $i\in [k]$.  Define another three idempotents:
\[
f:=\begin{partn}{3} A_1&\cdots&A_r \\ \hhline{~|~|~}
b_1&\cdots&b_r\end{partn}
\COMMa
g:= \begin{partn}{7} B_1\cup C_1\cup\cdots \cup C_k&B_2&\cdots&B_r&\multicolumn{3}{c}{} \\ \hhline{~|~|~|~|-|-|-}
B_1&B_2&\cdots&B_r&C_1&\cdots&C_k\end{partn}
\ANd
h:=\begin{partn}{4} B_1\cup C_1\cup\cdots\cup C_k &B_2&\cdots&B_r \\ \hhline{~|~|~|~}
b_1&b_2&\cdots&b_r\end{partn}.
\]
Note that we are using the convention of omitting lower singletons in $f$ and $h$.  Note also that $g = e_q$ for the projection $q = e^*e = \begin{partn}{6} B_1&\cdots&B_r&C_1&\cdots&C_k\\ \hhline{~|~|~|-|-|-} B_1&\cdots&B_r&C_1&\cdots&C_k\end{partn}$.
We claim that $\smat{e}{f}{g}{h}$ is a singular square, RL-singularised by
\[
u := R(e) = \id_{\coker(e)} = \begin{partn}{6} B_1 &\cdots & B_r& C_1&\cdots& C_k\\ \hhline{~|~|~|~|~|~}
B_1 &\cdots & B_r& C_1&\cdots& C_k\end{partn}.
\]
For an illustration see Figure \ref{fig:fdtoTn}.
Indeed, it is straightforward to verify that $f$, $g$ and $h$ are idempotents, and that the appropriate $\R$- and $\L$-relationships hold.  The equality $uh=h$ is easily checked, while $uf=f$ follows from the observation that the typical component of $e$ (see~Proposition \ref{prop:EPn}) has the form $\begin{partn}{4} A_i & \multicolumn{3}{c}{} \\ \hhline{~|-|-|-} B_i & C_{j_1}&\cdots & C_{j_l}\end{partn}$ for some $l\geq 0$ and $j_1,\ldots, j_l\in [k]$.  The equalities $fu=e$ and $hu=g$ hold because $b_i\in B_i$ for each $i\in[r]$.

Since $e$ is neither a projection nor a transformation, it follows that $1\leq k\leq n-r-1$, and so $q$ (as above) belongs to $P_1\cup\cdots\cup P_{n-r-1}$, and hence $g = e_q \in\Tfd\sub T$,
so the relation $a_g=1$ is in~$\presn(T)$.
Also notice that $f,h\in \T_n$, with $\lam(f)=\lam(e)$, and $\lam(h)=1$.  (For the latter, note that $\min(B_1\cup C_1\cup\cdots\cup C_k) = 1$.)
Recalling that $\Tfd$, and hence $T$, contains the tree~$\Tlex$, Lemma~\ref{la:GRtransl}\ref{it:GRt1}
 implies that
$a_h=1$ is a consequence of $\presn(T)$. The square $\smat{e}{f}{g}{h}$ yields the relation
$a_e^{-1}a_f=a_g^{-1}a_h$ from $\rels_\subSq$, which is present in $\presn(T)$,
and hence $a_e=a_f$ as required.
\end{proof}

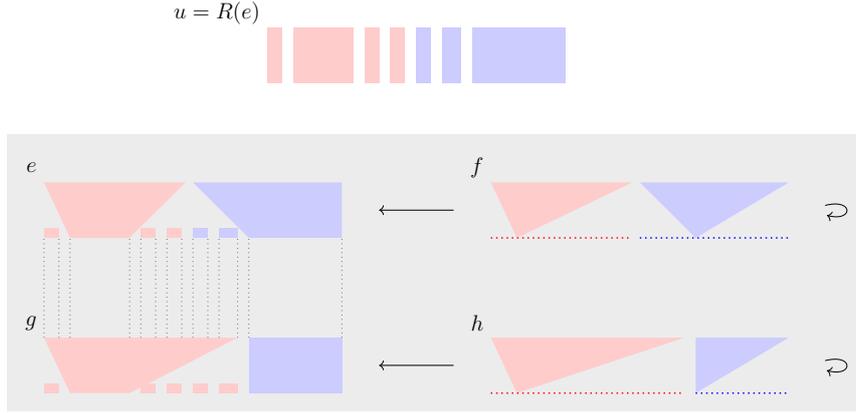
\begin{figure}[t]
\begin{center}
\scalebox{0.7}{
\begin{tikzpicture}[scale=.7]

\nc\yyy{4.2}

\nc\XL{0}
\nc\XR{23}
\nc\YL{-.5}
\nc\YR{7}
\fill[lightgray!30] (\XL,\YL)--(\XR,\YL)--(\XR,\YR)--(\XL,\YR)--(\XL,\YL);

\begin{scope}[shift={(0,\yyy)}]
\Eredtrans1{4.8}{1.7}{3.3}
\Eredlower1{1.4}
\Eredlower{3.6}{4}
\Eredlower{4.3}{4.7}
\Ebluetrans59{6.5}9
\Ebluelower5{5.4}
\Ebluelower{5.7}{6.2}
\Ebluepartlabel{e}
\draw[<-] (8.5+1.5,.75)--(8.5+3.5,.75);
\end{scope}

\begin{scope}[shift={(12,\yyy)}]
\Eredtrans1{4.8}{1.7}{1.7}
\Ebluetrans59{6.5}{6.5}
\Ebluepartlabel{f}
\node (R) at (9.5,.75) {\phantom{\Large$E$}};
\draw[->] (R) edge [loop right] ();
\draw[red,thick,dotted](1,0)--(4.7,0);
\draw[blue,thick,dotted](5,0)--(9,0);
\end{scope}

\begin{scope}[shift={(0,0)}]	
\Eredtrans1{6.2}{1.7}{3.3}
\Eredlower1{1.4}
\Eredlower{3.6}{4}
\Eredlower{4.3}{4.7}
\Ebluetrans{6.5}9{6.5}9
\Eredlower5{5.4}
\Eredlower{5.7}{6.2}
\Ebluepartlabel{g}
\draw[<-] (8.5+1.5,.75)--(8.5+3.5,.75);
\foreach \x in {1,1.4,1.7,3.3,3.6,4,4.3,4.7,5,5.4,5.7,6.2,6.5,9} {\draw[black!70,dotted](\x,1.5) -- (\x,\yyy);}
\end{scope}

\begin{scope}[shift={(12,0)}]	
\Eredtrans1{6.2}{1.7}{1.7}
\Ebluetrans{6.5}9{6.5}{6.5}
\Ebluepartlabel{h}
\node (R) at (9.5,.75) {\phantom{\Large$E$}};
\draw[->] (R) edge [loop right] ();
\draw[red,thick,dotted](1,0)--(6.2,0);
\draw[blue,thick,dotted](6.5,0)--(9,0);
\end{scope}

\begin{scope}[shift={(6,2*\yyy)}]
\Eredtrans{1.7}{3.3}{1.7}{3.3}
\Eredtrans1{1.4}1{1.4}
\Eredtrans{3.6}{4}{3.6}{4}
\Eredtrans{4.3}{4.7}{4.3}{4.7}
\Ebluetrans{6.5}9{6.5}9
\Ebluetrans5{5.4}5{5.4}
\Ebluetrans{5.7}{6.2}{5.7}{6.2}
\Ebluepartlabel{u=R(e)}
\end{scope}

\end{tikzpicture}
}
\caption{The singular square featuring in the proof of Lemma \ref{la:fdtoTn}.}
\label{fig:fdtoTn}
\end{center}
\end{figure}

\section{Maximal subgroups of \boldmath{$\PG(\sfP(\P_n))$}, the case of rank \boldmath{$\geq 1$}}
\label{sec:maxPGPn}

We have now reached the point in our treatment of $\P_n$ where we can identify the maximal subgroups of $\PG(\sfP(\P_n))$ corresponding to projections of rank $1\leq r\leq n-2$, and thus prove the main part of Theorem \ref{thm:mainPGPn} (the case $r=0$ was Proposition~\ref{p:PnD0}):

\begin{thm}
\label{thm:maxPGPn}
Let $n\geq 3$, and let $p_0\in \P_n$ be a projection of rank $1\leq r\leq n-2$.
Then the maximal subgroup of $\PG(\sfP(\P_n))$ containing $x_{p_0}$ is isomorphic to the symmetric group $\S_r$.
\end{thm}

\begin{proof}
We will work with the presentation given in Corollary \ref{cor:pgmaxq}.
For this we need to specify a spanning tree for the Graham--Houghton graph $\GH(D(n,r))$.
We start with the set $\Tfd$, which was constructed in Subsection \ref{ss:Tfd} as an extension of $\Tlex$ to $I_0\cup (J_1\cup\cdots\cup J_{n-r})$; see \eqref{eq:Tfd}.
 We expand $\Tfd$ to
$T:=\Tfd \cup P(n,r)$.
This is easily seen to be a spanning tree for $\GH(D(n,r))$; see Figure \ref{fig:T} for a schematic diagram.
By
Corollary \ref{cor:pgmaxq}, the maximal subgroup $\GP$ of $\PG(P(\P_n))$ containing $x_{p_0}$ is defined by the presentation
\[
\langle A(n,r)\mid \rels_1(T),\ \rels_{\subSq}(n,r),\ \rels_\subinv(n,r)\rangle.
\]
By Lemma \ref{la:geneli}, $\GP$ is generated by the set $A_F = \set{a_e}{e\in F}$, where $F = F(n,r)$ is as in~\eqref{eq:Fnr}.  Note that $F$ decomposes as $F=F_1\cup F_2\cup P_0\cup P_1$, where
\begin{equation}\label{eq:F1F2}
F_1:= \{ e\in E(n,r)\colon \NTu(e)=0<\NTd(e)\} \ANd
F_2:= \{ e\in E(n,r)\colon \NTd(e)=0<\NTu(e)\} .
\end{equation}
For $e\in P_0\cup P_1$, the relation $a_e=1$ is contained in $\rels_1(T)$,
while for $e\in F_2$ the relation  $a_e=a_{e^*}^{-1}$ is in $\rels_\subinv$, and we have $e^*\in F_1$.
Thus all generators $a_e$ with $e\in E(n,r)\setminus F_1$ 
can be eliminated in terms of generators $a_f$ with $f\in F_1$.
Furthermore, since $\Tfd\subseteq T$, 
for every such $f$ we have  $a_f=a_{f'}$ for some $f'\in E^\T(n,r)$ by Lemma \ref{la:fdtoTn}.
Denoting by $A'$ the alphabet $\big\{ a_h\colon h\in E^\T(n,r)\big\}$,
we see that for every $a_e\in A(n,r)\setminus A'$ there exists a word $w_e\in (A')^+$
such that $a_e=w_e$.  We also define $w_e=a_e$ for $e\in A'$. 
Performing the generator elimination then yields
the presentation $\langle A'\mid \rels\rangle$ for~$\GP$, where $\rels$ is obtained by substituting every instance of every $a_e$ by $w_e$ in all relations from
$\rels_1(T)\cup \rels_\subSq(n,r)\cup \rels_{\subinv}(n,r)$.
This set~$\rels$ of defining relations contains $\rels_1(\Tlex)$, as well as the
set $\rels_\subSq'(n,r)$ consisting of all square relations from $\rels_\subSq(n,r)$ involving solely idempotents from~$\T_n$.
In other words, $\GP$ is defined by $\langle A'\mid \rels_1(\Tlex),\ \rels_\subSq'(n,r),\ \rels''\rangle$, for some set of relations~$\rels''$.
By Proposition \ref{pr:treeTn}, the presentation $\langle A'\mid \rels_1(\Tlex),\ \rels_\subSq'(n,r)\rangle$ defines the symmetric group~$\S_r$.
Thus $\GP$ is a homomorphic image of~$\S_r$.
On the other hand, $\GP$ has the maximal subgroup of $\P_n$ containing $p_0$ as a homomorphic image, and this latter group is also isomorphic to $\S_r$ (cf.~Lemma \ref{la:Green_Pn}). Therefore $\GP\cong \S_r$, and the theorem is proved.
\end{proof}

\begin{rem}\label{rem:eggbox3}
The sets $F_1$ and $F_2$ from \eqref{eq:F1F2} are contained in the horizontal and vertical pink regions in Figure \ref{fig:eggbox1}, respectively.  The set $\T_n\cap E(n,r)$ -- corresponding to the final generating set~$A'$ -- consists of the idempotents in the region~$D^0_{n-r}(n,r)$, outlined in red in Figure~\ref{fig:eggbox1}.
\end{rem}

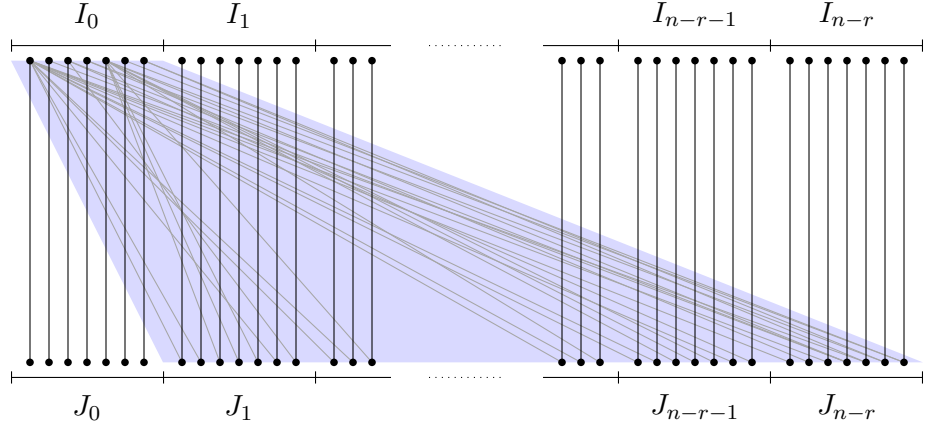
\begin{figure}[t]
\begin{center}
\scalebox{1}{
\begin{tikzpicture}[scale=1]

\nc\yy{0.2}
\nc\zz{0.4}

\foreach \x/\y in {0/2,2/4,4/5,12/10,10/8,8/7} {\draw[|-] (\x,0-\yy)--(\y,0-\yy); \draw[|-] (\x,4+\yy)--(\y,4+\yy);}
\draw[dotted] (5.5,-\yy)--(6.5,-\yy);
\draw[dotted] (5.5,4+\yy)--(6.5,4+\yy);

\node () at (1,4+\yy+\zz) {$I_0$};
\node () at (3,4+\yy+\zz) {$I_1$};
\node () at (9,4+\yy+\zz) {$I_{n-r-1}$};
\node () at (11,4+\yy+\zz) {$I_{n-r}$};
\node () at (1,0-\yy-\zz) {$J_0$};
\node () at (3,0-\yy-\zz) {$J_1$};
\node () at (9,0-\yy-\zz) {$J_{n-r-1}$};
\node () at (11,0-\yy-\zz) {$J_{n-r}$};

\fill[blue!30,opacity=0.5] (0,4)--(2,4)--(12,0)--(2,0)--(0,4);

{\nc\zzz2 \foreach \x/\y in {1/1,2/1,3/5,4/5,5/5,6/1,7/3} {\draw[gray!70](\y*.25,4)--(\x*.25+\zzz,0);}; }
{\nc\zzz4 \foreach \x/\y in {1/1,2/1,3/5} {\draw[gray!70](\y*.25,4)--(\x*.25+\zzz,0);}; }
{\nc\zzz6 \foreach \x/\y in {5/1,6/1,7/5} {\draw[gray!70](\y*.25,4)--(\x*.25+\zzz,0);}; }
{\nc\zzz8 \foreach \x/\y in {1/1,2/1,3/5,4/5,5/5,6/1,7/3} {\draw[gray!70](\y*.25,4)--(\x*.25+\zzz,0);}; }
{\nc\zzz{10} \foreach \x/\y in {1/1,1/3,1/5,2/4,3/4,4/2,4/6,5/4,5/7,6/7,7/5,7/6} {\draw[gray!70](\x*.25,4)--(\y*.25+\zzz,0);}; }

\foreach \x in {.25,.5,.75,1,1.25,1.5,1.75} { \draw (\x,4)--(\x,0);\fill (\x,4)circle(.05); \fill (\x,0)circle(.05);}
\foreach \x in {.25,.5,.75,1,1.25,1.5,1.75} { \draw (\x+2,4)--(\x+2,0);\fill (\x+2,4)circle(.05); \fill (\x+2,0)circle(.05);}
\foreach \x in {.25,.5,.75} { \draw (\x+4,4)--(\x+4,0);\fill (\x+4,4)circle(.05); \fill (\x+4,0)circle(.05);}
\foreach \x in {1.25,1.5,1.75} { \draw (\x+6,4)--(\x+6,0);\fill (\x+6,4)circle(.05); \fill (\x+6,0)circle(.05);}
\foreach \x in {.25,.5,.75,1,1.25,1.5,1.75} { \draw (\x+8,4)--(\x+8,0);\fill (\x+8,4)circle(.05); \fill (\x+8,0)circle(.05);}
\foreach \x in {.25,.5,.75,1,1.25,1.5,1.75} { \draw (\x+10,4)--(\x+10,0);\fill (\x+10,4)circle(.05); \fill (\x+10,0)circle(.05);}

\end{tikzpicture}
}
\caption{Schematic diagram of the tree $T = \Tfd\cup P(n,r)$ from the proof of Theorem \ref{thm:maxPGPn}, which is a spanning tree for the Graham--Houghton graph $\GH(D(n,r))$, in the case $r\geq1$.  The edges from~$\Tfd$ are indicated in grey, and are contained in the blue region (with vertices $I_0\cup (J_1\cup J_2\cup\cdots \cup J_{n-r})$), while the edges from $P(n,r)$ are indicated in black.}
\label{fig:T}
\end{center}
\end{figure}

\section{Maximal subgroups in \boldmath{$\IG(\sfE(\P_n))$}, the case of rank \boldmath{$\geq 1$}}
\label{sec:maxIGPn}

This section comprises the most difficult argument in the paper, namely the computation of the maximal subgroups of $\IG(E)$, where $E:=\sfE(\P_n)$.
Specifically, we prove the following, with the $r=0$ case being deferred to Section \ref{sec:r0}:

\begin{thm}
\label{thm:IGPn}
Let $n\geq 3$, and let $e_0\in \P_n$ be an idempotent of rank $1\leq r\leq n-2$.  Then the maximal subgroup of $\IG(\sfE(\P_n))$ containing $x_{e_0}$ is isomorphic to $\Z\times \S_r$.
\end{thm}

The basic idea of the proof is the same as in the previous section, i.e.~a reduction to~$\IG(\sfE(\T_n))$.
However, a major difference here is that the end result is \emph{not} $\S_r$, the corresponding maximal subgroup of $\IG(\sfE(\T_n))$, but rather $\Z\times\S_r$.  Consequently, we also need to account for the direct factor $\Z$.  Since the argument to follow is rather involved, we provide the following `road map' for the convenience of the reader.
\bit
\item 
We begin in Subsection \ref{ss:st} by identifying a suitable spanning tree $T(s)$, from which to build a presentation $\presn(T(s))$ for our group $\GI \cong \sfG(T(s))$.  This tree contains certain non-projection full-domain or -codomain idempotents with label $1$, plus an arbitrary single projection~${s\in P_0}$.  
\item 
In Subsection \ref{ss:lab1} we identify a large set $\labU = \labU(n,r)$ of idempotents $e$ of label $1$ for which the relation $a_e=1$ holds in our group $\sfG(T(s))$.  It follows that $\GI$ is isomorphic to $\sfG(\labU\cup\{s\})$, the group defined by the larger presentation $\presn(\labU\cup\{s\})$.  Removing the relation $a_s=1$ leaves us with a group $\Ghat(n,r) = \sfG(\labU)$, which (when $r\leq n-3$) contains many natural copies of~$\Ghat(n-1,r)$.  This is crucial, as it enables inductive proofs of the technical results of the next two subsections.
\item 
By Lemma \ref{la:geneli}, the group $\sfG(\labU)$ is generated by the set $A_F$, where $F = P_0\cup P_1 \cup F_1\cup F_2$ is as in \eqref{eq:Fnr}; cf.~\eqref{eq:F1F2}.  In Subsection \ref{ss:eP1} we show that all generators corresponding to projections from $P_1$ are equal in $\sfG(\labU)$.  Subsection \ref{ss:eP0} does the same for $P_0$.  Bringing the relation $a_s=1$ back in, it follows that our group is generated by $A_{F_1\cup F_2\cup\{t\}}$ for any fixed $t\in P_1$.
\item
In Subsection \ref{ss:glt} we show that certain generators from $A_{F_2}$ are equal to corresponding generators from $A_{F_1}$.  Combining this with earlier results will ultimately allow us to eliminate \emph{all} generators from $A_{F_2}$, so that our group $\sfG(\labU\cup\{s\})$ is generated by $A_{F_1\cup\{t\}}$.  Bringing in additional earlier results allows us to replace $F_1\cup\{t\}$ with the set $E' = E^\T(n,r)\cup\{t\}$ consisting of $t$ and all idempotent transformations of rank $r$.  By results of \cite{GR12}, the generators corresponding to $E^\T(n,r)$ satisfy the defining relations of the symmetric group $\S_r$.
\item
In Subsection \ref{ss:comm} we show that every generator $a_e$ with $e\in E^\T(n,r)$ commutes with $a_t$ (for our fixed $t\in P_1$).  Ultimately this will allow us to deduce that our group $\GI$ is a homomorphic image of $\Z\times\S_r$: the subgroup $\la A_{E^\T(n,r)}\ra$ is a quotient of $\S_r$ (as per the previous point), while the subgroup~$\la a_t\ra$ is cyclic, i.e.~a quotient of $\Z$, and the generators of these two subgroups commute.
\item
In Subsection \ref{ss:tw} we connect with the ($\Z$-)twisted partition monoid $\Ptw_n$, noting that results of \cite{EGMR2} -- including a biordered set isomorphism $\sfE(\Ptw_n) \cong \sfE(\P_n)$ -- imply that our group $\GI$ is also a homomorphic \emph{pre}-image of $\Z\times\S_r$.
\item
The proof is rounded off in Subsection \ref{ss:theend}, where we note that Hopficity of $\Z\times\S_r$ \cite{Hi69} implies that indeed $\GI\cong\Z\times\S_r$.
\eit

\subsection{A spanning tree}
\label{ss:st}

In what follows we will work with different presentations, involving generators that are subsets of $E(n,r)$ and relations involving the square relations and some relations of the form $a_e=1$.  As before, for any subset $F\sub E(n,r)$, we write
\begin{nitemize}
\item
$\presn(F)$ for the presentation $\langle A(n,r)\mid \rels_1(F),\ \rels_\subSq(n,r)\rangle$, and
\item
$\sfG(F)$ for the group defined by $\presn(F)$.
\end{nitemize}
By Theorem \ref{thm:spantreepres}, the maximal subgroup $\GI$ of $\IG(\sfE(\P_n))$ containing 
$x_{e_0}$  is isomorphic to the group $\sfG(T)$ for any spanning tree $T$ of the Graham--Houghton graph $\GH(D(n,r))$.
We are going to pick a specific spanning tree below, but we will also work with sets of idempotents $F$ that do not correspond to spanning trees.

As in the previous section, we start from the set $\Tfd$ constructed in Subsection \ref{ss:Tfd}, which is a spanning tree for the subgraph of $\GH(D(n,r))$ induced on ${I_0\cup (J_1\cup J_2\cup\cdots \cup J_{n-r})}$; see \eqref{eq:Tfd}.  Recall that $\Tfd$ contains the spanning tree $\Tlex$ for the Graham--Houghton graph $\GH(D^\T(n,r))$.  Applying the involution~${}^*$, the set 
 \begin{equation}
 \label{eq:Tfc}
\Tfc=\Tfc(n,r):= \Tfd^*= \Tlex^*\cup \{ e_p^*\colon p\in P_1\cup \cdots \cup P_{n-r-1}\}
\end{equation}
is a spanning tree for the subgraph of $\GH(D(n,r))$ induced on $(I_1\cup\cdots \cup I_{n-r}) \cup J_0$.  Consequently, $\Tfd\cup\Tfc$ has two connected components, so adding \emph{any} single edge between the two components will yield a spanning tree.  
We will specifically choose this additional edge $s$ to be a full-(co)domain projection, and denote the resulting tree by:
\begin{equation}\label{eq:Ts}
T(s):=\Tfd\cup\Tfc\cup\{s\} \qquad\text{for $s\in P_0 = P_0(n,r)$.}
\end{equation}
A schematic diagram of this tree is shown in Figure \ref{fig:Ts}.

\begin{lemma}
\label{la:Tlexpres}
For any full-domain projection $s\in P_0(n,r)$, the group $\GI$  is defined by the presentation
\[
\presn(T(s))=\bigl\langle A(n,r)\mid \rels_1 (T(s)), \rels_\subSq(n,r)\bigr\rangle.
\]
In other words, $\GI\cong \sfG(T(s))$.
\end{lemma}

\begin{proof}
This is just the presentation given by Theorem \ref{thm:spantreepres} stated in terms of the specific spanning tree $T(s)$ we have just built.
\end{proof}

\begin{figure}[t]
\begin{center}
\scalebox{1}{
\begin{tikzpicture}[scale=1]

\nc\yy{0.2}
\nc\zz{0.4}

\foreach \x/\y in {0/2,2/4,4/5,12/10,10/8,8/7} {\draw[|-] (\x,0-\yy)--(\y,0-\yy); \draw[|-] (\x,4+\yy)--(\y,4+\yy);}
\draw[dotted] (5.5,-\yy)--(6.5,-\yy);
\draw[dotted] (5.5,4+\yy)--(6.5,4+\yy);

\node () at (1,4+\yy+\zz) {$I_0$};
\node () at (3,4+\yy+\zz) {$I_1$};
\node () at (9,4+\yy+\zz) {$I_{n-r-1}$};
\node () at (11,4+\yy+\zz) {$I_{n-r}$};
\node () at (1,0-\yy-\zz) {$J_0$};
\node () at (3,0-\yy-\zz) {$J_1$};
\node () at (9,0-\yy-\zz) {$J_{n-r-1}$};
\node () at (11,0-\yy-\zz) {$J_{n-r}$};

\fill[blue!30,opacity=0.5] (0,4)--(1.9,4)--(12,0)--(2.1,0)--(0,4);
\fill[red!30,opacity=0.5] (0,0)--(1.9,0)--(12,4)--(2.1,4)--(0,0);

{\nc\zzz2 \foreach \x/\y in {1/1,2/1,3/5,4/5,5/5,6/1,7/3} {\draw[gray!70](\y*.25,4)--(\x*.25+\zzz,0); \draw[gray!70](\y*.25,0)--(\x*.25+\zzz,4);}; }
{\nc\zzz4 \foreach \x/\y in {1/1,2/1,3/5} {\draw[gray!70](\y*.25,4)--(\x*.25+\zzz,0); \draw[gray!70](\y*.25,0)--(\x*.25+\zzz,4);}; }
{\nc\zzz6 \foreach \x/\y in {5/1,6/1,7/5} {\draw[gray!70](\y*.25,4)--(\x*.25+\zzz,0); \draw[gray!70](\y*.25,0)--(\x*.25+\zzz,4);}; }
{\nc\zzz8 \foreach \x/\y in {1/1,2/1,3/5,4/5,5/5,6/1,7/3} {\draw[gray!70](\y*.25,4)--(\x*.25+\zzz,0); \draw[gray!70](\y*.25,0)--(\x*.25+\zzz,4);}; }
{\nc\zzz{10} \foreach \x/\y in {1/1,1/3,1/5,2/4,3/4,4/2,4/6,5/4,5/7,6/7,7/5,7/6} {\draw[gray!70](\x*.25,4)--(\y*.25+\zzz,0); \draw[gray!70](\y*.25,0)--(\x*.25+\zzz,4);}; }

\draw[red,thick](.5,4)--(.5,0);
\node[red] () at (.35,2) {$s$};
\node (is) at (-.5,3) {};
\node () at (-.5,3) {$i_s$};
\node (js) at (-.5,1) {};
\node () at (-.5,1) {$j_s$};
\draw[-{latex}] (is)--(.4,3.9);
\draw[-{latex}] (js)--(.4,0.1);

\foreach \x in {.25,.5,.75,1,1.25,1.5,1.75} {\fill (\x,4)circle(.05); \fill (\x,0)circle(.05);}
\foreach \x in {.25,.5,.75,1,1.25,1.5,1.75} {\fill (\x+2,4)circle(.05); \fill (\x+2,0)circle(.05);}
\foreach \x in {.25,.5,.75} {\fill (\x+4,4)circle(.05); \fill (\x+4,0)circle(.05);}
\foreach \x in {1.25,1.5,1.75} {\fill (\x+6,4)circle(.05); \fill (\x+6,0)circle(.05);}
\foreach \x in {.25,.5,.75,1,1.25,1.5,1.75} {\fill (\x+8,4)circle(.05); \fill (\x+8,0)circle(.05);}
\foreach \x in {.25,.5,.75,1,1.25,1.5,1.75} {\fill (\x+10,4)circle(.05); \fill (\x+10,0)circle(.05);}

\end{tikzpicture}
}
\caption{Schematic diagram of the tree $T(s) = \Tfd\cup\Tfc\cup\{s\}$ from \eqref{eq:Ts}, which is a spanning tree for the Graham--Houghton graph $\GH(D(n,r))$ in the case $r\geq1$.  The edges from $\Tfd$ and $\Tfc$ are indicated in grey, and are contained in the blue and pink regions, respectively, and the edge $s$ is indicated in red.}
\label{fig:Ts}
\end{center}
\end{figure}
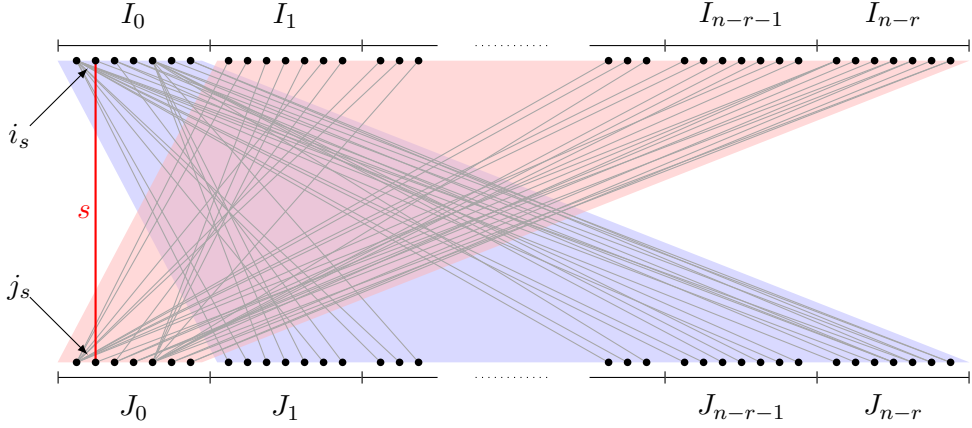

\subsection[Generators with label $1$]{Generators with label \boldmath{$1$}}
\label{ss:lab1}

We now prove a sequence of results involving the labels introduced in Subsection \ref{ss:labs}.  We start with Lemma \ref{la:Tlab1}, which observes that all elements of the spanning tree $T(s)$ in \eqref{eq:Ts} have label $1$, and then build to Lemma~\ref{la:Upres}, which says that we can add the relations $a_e=1$ to the presentation $\presn(T(s))$ for \emph{many} more idempotents $e$ with label~$\lam(e)=1$.

\begin{lemma}
\label{la:Tlab1}
All the idempotents in the spanning tree $T(s)$ have label $1$.
\end{lemma}

\begin{proof}
Consider an arbitrary $e\in T(s)$. We split the proof into several cases, depending on the form of $e$,
as indicated in \eqref{eq:Ts}.
\medskip

\noindent
\textit{Case 1: $e=e_{V, \mapC(V)}\in \Tlex$, where $V=\{V_1,\ldots,V_r\}$ is a partition of $[n]$ with $r$ blocks.} Letting $v_k=\min(V_k)$ for $k\in [r]$, we have 
$e
=\begin{partn}{3} V_1 & \cdots & V_r \\ \hhline{~|~|~} v_1 & \cdots & v_r\end{partn}$
(omitting the trivial blocks)
so that
$\lamp(e)=\trans{v_1&\cdots&v_r}{v_1&\cdots & v_r}$, and  hence
$\lam(e)=\trans{1&\cdots&r}{1&\cdots & r}$.
\medskip

\noindent
\textit{Case 2: $e=e_{\mapV(C), C} \in \Tlex$, where $C=\{c_1<\cdots<c_r\}\subseteq [n]$.}
In this case we have
\[
e=\begin{partn}{5} [1,c_1] & (c_1,c_2] & \cdots & (c_{r-2},c_{r-1}] & (c_{r-1},n] \\ \hhline{~|~|~|~|~} c_1 & c_2&\cdots &c_{r-1} &c_r\end{partn},
\]
so that
$\lamp(e)=\trans{1&c_1+1&\cdots&c_{r-2}+1&c_{r-1}+1}{c_1&c_2&\cdots&c_{r-1} & c_r}$, and  hence
$\lam(e)=\trans{1&\cdots&r}{1&\cdots & r}$.
\medskip

\noindent
\textit{Case 3: $e=e_p$, where $p\in P_k$ for some $1\leq k\leq n-r-1$.} With notation as in \eqref{eq:ep}, let $a_k:=\min(A_k)$ for $1\leq k\leq r$, and note that $\min(A_1\cup B_1\cup\cdots\cup B_k)=1$.  It follows that $\lamp(e)=\trans{1&a_2&\cdots&a_{r}}{a_1&a_2&\cdots& a_r}$, and hence $\lam(e)=\trans{1&\cdots&r}{1&\cdots & r}$.
\medskip

\noindent
\textit{Case 4: $e\in \Tfc$.}
This case follows from Cases 1--3, using Lemma \ref{la:lamprops}\ref{it:lp2}.
\medskip

\noindent
\textit{Case 5: $e=s$.}
This follows from Lemma \ref{la:lamprops}\ref{it:lp1}.
\end{proof}

In what follows we will in fact want to work with a set of relations that is larger than~$\rels_1(T(s))$,
and we now work towards enabling that. We begin by identifying certain sets of idempotents:
\begin{nitemize}
\item
$\labU_\fd=\labU_\fd (n,r):= \bigl\{e\in E(n,r)\colon \lambda(e)=1 \text{ and }
\NTu(e)=0<\NTd(e)\bigr\}$ -- the set of full-domain, non-projection idempotents with label $1$;
\item
$\labU_\fc=\labU_\fc (n,r):= \bigl\{e\in E(n,r)\colon \lambda(e)=1 \text{ and }
\NTd(e)=0<\NTu(e)\bigr\}$ -- the set of full-codomain, non-projection idempotents with label $1$;
\item
$\labU=\labU(n,r):=\labU_\fd(n,r)\cup\labU_\fc(n,r)$.
\end{nitemize}
Lemma \ref{la:Tlexpres} says that the maximal subgroup $\GI$ under consideration is defined by the presentation $\presn(T(s))$.
Shortly we are going to see that the presentation $\presn(T(s))$ is equivalent to $\presn(\labU\cup\{s\})$, so that $\GI$ also has presentation $\presn(\labU\cup\{s\})$.
In fact, much of what we will prove does not use the defining relation $a_s=1$, and 
we will prefer to work with the slightly smaller presentation $\presn(\labU)$.  This is to facilitate the inductive proofs of Lemmas \ref{la:P1nr} and \ref{la:P0nr} below; see Remark \ref{rem:presnUlab}.
The group defined by $\presn(\labU)$ is typically not isomorphic to $\GI$, but only a homomorphic preimage of it.  
It is worth noting that~$\labU$ contains $T_\fd\cup T_\fc$; indeed, the elements of the latter set are non-projections by definition, and they have label $1$ by Lemma \ref{la:Tlab1}.

\begin{lemma}
\label{la:eqlabel}
Let $e,f\in E(n,r)$ be any non-projection idempotents with full domain.
\begin{thmenumerate}
\item
\label{it:eql1}
If $\lam(e)=\lam(f)$ then $a_e=a_f$ is a consequence of $\presn(\Tfd)$.
\item
\label{it:eql2}
If $\lambda(e)=1$ then $a_e=1$ is a consequence of $\presn(\Tfd)$.
\end{thmenumerate}
\end{lemma}

\begin{proof}
\ref{it:eql1} 
By Lemma \ref{la:fdtoTn}, there exist idempotent transformations $e',f'\in E^\T(n,r)$ such that $\lam(e')=\lam(e)$ and $\lam(f')=\lam(f)$,
and such that the relations $a_e=a_{e'}$ and $a_f=a_{f'}$ are consequences of $\presn(\Tfd)$.
The relation $a_{e'}=a_{f'}$ is a consequence of $\presn(\Tfd)$ by Lemma \ref{la:GRtransl}\ref{it:GRt2}.

\ref{it:eql2} 
This follows from \ref{it:eql1}, keeping in mind Lemma \ref{la:Tlab1}.
\end{proof}

Lemma \ref{la:eqlabel} has an obvious dual, which we state without proof:

\begin{lemma}
\label{la:eqlabeldual}
Let $e,f\in E(n,r)$ be any non-projection idempotents with full codomain.
\begin{thmenumerate}
\item
\label{it:eql1*}
If $\lam(e)=\lam(f)$ then $a_e=a_f$ is a consequence of $\presn(\Tfc)$.
\item
\label{it:eql2*}
If $\lambda(e)=1$ then $a_e=1$ is a consequence of $\presn(\Tfc)$. \qed
\end{thmenumerate}
\end{lemma}

We can now achieve our above-mentioned goal:

\begin{lemma}
\label{la:Upres}
For any $1\leq r\leq n-2$, and any $s\in P_0(n,r)$, the group $\GI$ is defined by the presentation
\[
\presn(\labU\cup\{s\})=\bigl\langle A(n,r)\mid \rels_1(\labU\cup\{s\}),\ \rels_\subSq(n,r)\bigr\rangle.
\]
In other words $\GI\cong \sfG(\labU\cup\{s\})$.
\end{lemma}

\begin{proof}
By Lemma \ref{la:Tlexpres}, the group $\GI$ is defined by the presentation $\presn(T(s))$.
By Lemmas~\ref{la:eqlabel}\ref{it:eql2} and~\ref{la:eqlabeldual}\ref{it:eql2*}, all the relations~$\rels_1(\labU) = \rels_1(\labU_\fd \cup \labU_\fc)$ can be added to $\presn(T(s))$.
Since ${T(s)\subseteq \labU\cup \{s\}}$ by Lemma \ref{la:Tlab1}, it follows that the resulting presentation
is precisely ${\presn(\labU\cup \{s\})}$.
\end{proof}

\begin{rem}
In the group $\sfG(\labU\cup\{s\})$, we of course have $a_s=1$, for the fixed projection~${s\in P_0}$.  For the other projections $p\in P_0\setminus\{s\}$, it is not clear at this stage that we necessarily have $a_p=1$ in $\sfG(\labU\cup\{s\})$, but this is indeed the case.  It takes a lot more work to prove this, but we eventually will in Lemma \ref{la:P0nr1}.
\end{rem}

\subsection[Generators corresponding to projections from $P_1(n,r)$]{Generators corresponding to projections from \boldmath{$P_1(n,r)$}}
\label{ss:eP1}

Our next task is to show that all the generators corresponding to projections from $P_1(n,r)$ are equal to each other as a consequence of our defining relations.  This is achieved in Lemma \ref{la:P1nr} below.  The proof of this lemma is inductive, based on the case $r=n-2$, which is dealt with in Lemma \ref{la:P1rr2}.  By the nature of the induction (see Remark \ref{rem:presnUlab}), we are forced to omit the relation $a_s=1$ (for the fixed projection $s\in P_0$), and work with the group~$\sfG(\labU)$, where $\labU = \labU(n,r)$ was constructed in the previous subsection.  In order that we can conveniently vary our parameters, we denote this group by $\Ghat(n,r) = \sfG(\labU) = \sfG(\labU(n,r))$.

To aid our considerations we remark that, when $r=n-2$, Figure \ref{fig:eggbox1}  (discussed in Remark~\ref{rem:eggbox1}) simplifies to a $3\times 3$ grid, as shown in  Figure \ref{fig:eggbox4}.
The set of projections in $D(n,n-2)$ decomposes as $P(n,n-2) = P_0 \cup P_1\cup P_2$, with each subset $P_k$ ($k=0,1,2$) contained in the region~$D_k^k$.  As we noted in Remark~\ref{rem:eggbox2},~$P_0$ (resp.\ $P_2$) consists of \emph{all} idempotents from~$D_0^0$
(resp.\ $D_2^2$), but~$D_1^1$ contains non-projection idempotents.

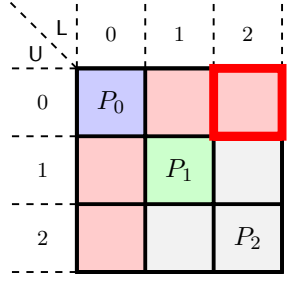
\begin{figure}[t]
\begin{center}
\scalebox{0.9}{
\begin{tikzpicture}[scale=1]

\fill[gray!10](0,0)--(3,0)--(3,3)--(0,3)--(0,0);

\fill[blue!20] (0,3)--(1,3)--(1,2)--(0,2)--(0,3);
\fill[green!20] (1,2)--(2,2)--(2,1)--(1,1)--(1,2);
\fill[red!20] (1,3)--(3,3)--(3,2)--(1,2)--(1,3);
\fill[red!20] (0,2)--(0,0)--(1,0)--(1,2)--(0,2);

\foreach \x in {0,1,2,3} {\draw[ultra thick] (\x,0)--(\x,3) (0,3-\x)--(3,3-\x);}

\draw[ultra thick](0,0)--(3,0)--(3,3)--(0,3)--(0,0)--(3,0);

\draw[thick, dashed] (0,3)--(-1,4);

\node () at (-.6,3.22) {\footnotesize $\NTu$};
\node () at (-.22,3.6) {\footnotesize $\NTd$};

\foreach \x in {0,1,2,3} {\draw[thick, dashed] (\x,3)--(\x,4) (0,3-\x)--(-1,3-\x);}

\node () at (0.5,3.5) {\footnotesize $0$};
\node () at (1.5,3.5) {\footnotesize $1$};
\node () at (2.5,3.5) {\footnotesize $2$};
\node () at (-.5,2.5) {\footnotesize $0$};
\node () at (-.5,1.5) {\footnotesize $1$};
\node () at (-.5,0.5) {\footnotesize $2$};

\node () at (.5,2.5) {$P_0$};
\node () at (1.5,1.5) {$P_1$};
\node () at (2.5,.5) {$P_2$};

\draw[line width=1.3mm,red] (2,2)--(3,2)--(3,3)--(2,3)--(2,2)--(3,2);

\end{tikzpicture}
}
\caption{The simplification of Figure \ref{fig:eggbox1} for the case $r=n-2$.}
\label{fig:eggbox4}
\end{center}
\end{figure}

We begin with some quite technical lemmas.

\begin{lemma}
\label{la:P1ep1}
Suppose $1\leq r= n-2$, let $i,j,k\in [n]$ be distinct, let $Y:=[n]\setminus \{i,j,k\}$, and let
\[
e:=\begin{partn}{3} i&j&k\\ \hhline{~|-|-}  i,j & \multicolumn{2}{c}{k}\end{partn}\oplus \id_Y \ANd
p:=\begin{partn}{2} i,j &k\\ \hhline{~|-}  i,j & k\end{partn}\oplus \id_Y.
\]
If $i\leq j+1$ then $a_e=a_p$ in $\Ghat(n,r)$.
\end{lemma}

\begin{proof}
Consider the three further idempotents:
\[
e_1:= \begin{partn}{3} i&j&k\\ \hhline{~|-|-} i,j,k&\multicolumn{2}{c}{}\end{partn}\oplus\id_Y \COMMa
e_2:= \begin{partn}{2} i,j&k\\ \hhline{~|-} i,j,k&\end{partn}\oplus\id_Y \ANd
u:= \begin{partn}{3} i&j&k\\ \hhline{~|~|-} i&j,k&\end{partn}\oplus\id_Y.
\]
These idempotents are shown pictorially in Figure \ref{fig:P1ep1}.
Routine computations show that $\smat{e}{e_1}{p}{e_2}$ is an LR square, singularised by $u$, i.e.~that
\[
e\mr\R e_1 \COMMA p\mr\R e_2 \COMMA e\mr\L p \AND e_1\mr\L e_2,
\]
and also
\[
ue=e \COMMA up=p \COMMA eu=e_1 \AND pu=e_2.
\]
The $\R$- and $\L$-relationships can be inferred by checking that (co)domains and (co)kernels match as appropriate, and applying Lemma \ref{la:Green_Pn}.  The required equalities are verified in Figure \ref{fig:P1ep1}. 

Next we note that the idempotents~$e_1$ and~$e_2$ are full-codomain non-projections,
and we claim that $\lam(e_1)=\lam(e_2)$.
To see this, write $x\wedge y$ for $\min(x,y)$, and  note that
\[
\lamp(e_1) = \trans{\phantom{j}i\phantom{j}}{i\wedge j\wedge k} \oplus \id_Y \ANd \lamp(e_2) = \trans{i\wedge j}{i\wedge j\wedge k} \oplus \id_Y.
\]
Since $i\leq j+1$, the values $i$ and $i\wedge j$ are either equal or consecutive,  and the claim follows by Lemma~\ref{la:consec}.

Now Lemma \ref{la:eqlabeldual}  gives $a_{e_1}=a_{e_2}$ in $\sfG(T_\fc)$.  Since $T_\fc\sub O$, it follows that the equality $a_{e_1}=a_{e_2}$ also holds in $\sfG(O) = \Ghat(n,r)$.
Combining this with the square relation $a_e^{-1}a_{e_1}=a_p^{-1}a_{e_2}$ yields $a_e=a_p$, as required.
 \end{proof}

While in the above lemma and the accompanying Figure \ref{fig:P1ep1} we have shown all the calculations needed to establish the required singular squares, from now on we will simply present the participating idempotents, and leave the actual checks to the reader. 

\begin{rem}\label{rem:eggbox5}
It may be interesting to observe how  the square $\smat e{e_1}p{e_2}$
from the proof of Lemma~\ref{la:P1ep1}
 straddles different regions in the stratified diagram of the $\D$-class $D(n,n-2)$.
 The `target equality' $a_p=a_e$ involves idempotents from $P_1\sub E_1^1$  and $E_1^2$, but the square also features idempotents from $E_0^1$ and $E_0^2$.
 This is illustrated in Figure \ref{fig:eggbox5} (left).
\end{rem}

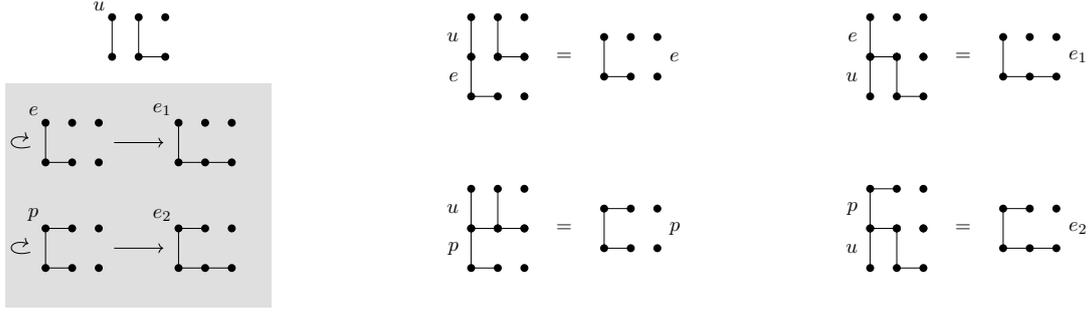
\begin{figure}[t!]
\begin{center}
\scalebox{0.7}{\begin{tikzpicture}[scale=.5]
\diagramshading3
\ediagram{e}{3}{\udline11 \ddline12}
\fdiagram{e_1}{3}{\udline11 \ddline12\ddline23}
\gdiagram{p}{3}{\udline11 \uuline12 \ddline12}
\hdiagram{e_2}{3}{\udline11 \uuline12 \ddline12\ddline23}
\udiagram{u}{3}{\udline11 \udline22 \ddline23}
\LRarrows
\begin{scope}[shift={(16,8)}]
\uvs{1,2,3}\lvs{1,2,3}\udline11 \udline22 \ddline23 
\node[left] () at (0.8,0.75) {$u$};
\end{scope}
\begin{scope}[shift={(16,6.5)}]
\uvs{1,2,3}\lvs{1,2,3}\udline11 \ddline12 
\node[left] () at (0.8,0.75) {$e$};
\end{scope}
\begin{scope}[shift={(21,7.25)}]
\node () at (-.5,0.75) {$=$};
\uvs{1,2,3}\lvs{1,2,3}\udline11 \ddline12 
\node[right] () at (3.2,0.75) {$e$};
\end{scope}
\begin{scope}[shift={(16,1.5)}]
\uvs{1,2,3}\lvs{1,2,3}\udline11 \udline22 \ddline23 
\node[left] () at (0.8,0.75) {$u$};
\end{scope}
\begin{scope}[shift={(16,0)}]
\uvs{1,2,3}\lvs{1,2,3}\udline11 \uuline12 \ddline12
\node[left] () at (0.8,0.75) {$p$};
\end{scope}
\begin{scope}[shift={(21,.75)}]
\node () at (-.5,0.75) {$=$};
\uvs{1,2,3}\lvs{1,2,3}\udline11 \uuline12 \ddline12
\node[right] () at (3.2,0.75) {$p$};
\end{scope}
\begin{scope}[shift={(31,8)}]
\uvs{1,2,3}\lvs{1,2,3}\udline11 \ddline12 
\node[left] () at (0.8,0.75) {$e$};
\end{scope}
\begin{scope}[shift={(31,6.5)}]
\uvs{1,2,3}\lvs{1,2,3}\udline11 \udline22 \ddline23 
\node[left] () at (0.8,0.75) {$u$};
\end{scope}
\begin{scope}[shift={(36,7.25)}]
\node () at (-.5,0.75) {$=$};
\uvs{1,2,3}\lvs{1,2,3}\udline11 \ddline12\ddline23
\node[right] () at (3.2,0.75) {$e_1$};
\end{scope}
\begin{scope}[shift={(31,1.5)}]
\uvs{1,2,3}\lvs{1,2,3}\udline11 \uuline12 \ddline12
\node[left] () at (0.8,0.75) {$p$};
\end{scope}
\begin{scope}[shift={(31,0)}]
\uvs{1,2,3}\lvs{1,2,3}\udline11 \udline22 \ddline23 
\node[left] () at (0.8,0.75) {$u$};
\end{scope}
\begin{scope}[shift={(36,.75)}]
\node () at (-.5,0.75) {$=$};
\uvs{1,2,3}\lvs{1,2,3}\udline11 \uuline12 \ddline12\ddline23
\node[right] () at (3.2,0.75) {$e_2$};
\end{scope}
\end{tikzpicture}
}

\caption{Left: the singular square featuring in the proof of Lemma \ref{la:P1ep1}.
Middle and right:  the calculations involved in checking its singularity.  In each partition, only the points $i,j,k$ and their dashed counterparts are shown (left to right).}
\label{fig:P1ep1}
\end{center}
\end{figure}

\begin{lemma}
\label{la:P1pq1}
Suppose $1\leq r= n-2$, let $i,j,k\in [n]$ be distinct, let $Y:=[n]\setminus \{i,j,k\}$, and let
\[
p:=\begin{partn}{2} i &j,k\\ \hhline{~|-}  i & j,k\end{partn}\oplus \id_Y \ANd
q:=\begin{partn}{2} i,j&k\\ \hhline{~|-}  i,j & k\end{partn}\oplus \id_Y.
\]
If $i\leq j+1$ then $a_p=a_q$  in $\Ghat(n,r)$.
\end{lemma}

\begin{proof}
Consider the following further idempotents:
\begin{align*}
e& := \begin{partn}{3} i & j&k \\ \hhline{~|-|-} i & \multicolumn{2}{c}{j,k}\end{partn}\oplus \id_Y, &
e_1&:=\begin{partn}{3} i&j&k\\ \hhline{~|-|-} i,j&\multicolumn{2}{c}{k}\end{partn}\oplus \id_Y, &
f&= \begin{partn}{2} i,k&j\\ \hhline{~|-} i&j,k\end{partn}\oplus \id_Y,&
f_1&= \begin{partn}{2} i,k&j\\ \hhline{~|-} i,j&k\end{partn}\oplus \id_Y,
\\
g& := \begin{partn}{2} i,j,k  & \\ \hhline{~|-} i & j,k\end{partn}\oplus \id_Y, &
g_1&:=\begin{partn}{2} i,j,k&\\ \hhline{~|-} i,j&k\end{partn}\oplus \id_Y, &
h&= \begin{partn}{3} i&j&k\\ \hhline{~|-|-} i,j,k&\multicolumn{2}{c}{}\end{partn}\oplus \id_Y,&
h_1&= \begin{partn}{2} i&j,k\\ \hhline{~|-} i,j,k&\end{partn}\oplus \id_Y.
\end{align*}
It is routine to verify that:
\begin{alignat}{5}
\label{eq:P1pq11}
&\smat{e}{e_1}{f}{f_1} &\ &\text{is LR-singularised by}
&\ &u=\begin{partn}{3} i&j&k\\ \hhline{~|-|~} i,j& &k\end{partn}\oplus\id_Y, &\ \text{ yielding}
&\quad&a_e^{-1}a_{e_1}=a_f^{-1}a_{f_1};
\\
\label{eq:P1pq12}
&\smat{f}{f_1}{g}{g_1} &&\text{is LR-singularised by}
&&v=\begin{partn}{3} i,k&j&\\ \hhline{~|~|-} i&j& k\end{partn}\oplus\id_Y,&\text{ yielding}
&&a_f^{-1}a_{f_1}=a_g^{-1}a_{g_1};
\\
\label{eq:P1pq13}
&\smat{e}{h}{p}{h_1} && \text{is UD-singularised by}
&&w= \begin{partn}{2} i&j,k\\ \hhline{~|~} i& j,k\end{partn}\oplus\id_Y,&\text{ yielding}
&&a_e^{-1}a_h=a_p^{-1}a_{h_1}.
\end{alignat}
These squares are illustrated in Figure \ref{fig:P1pq1}.  Using Lemma \ref{la:consec} and $i\leq j+1$, we see that the full-domain idempotents $g$ and $g_1$ have equal labels, and hence we have
\begin{equation}
\label{eq:P1pq14}
a_g=a_{g_1}
\end{equation}
by Lemma \ref{la:eqlabel}. Analogously, Lemma \ref{la:eqlabeldual} gives
\begin{equation}
\label{eq:P1pq15}
a_h=a_{h_1}.
\end{equation}
Finally, by Lemma \ref{la:P1ep1} we have
\begin{equation}
\label{eq:P1pq16}
a_{e_1}=a_q.
\end{equation}
Putting all this together we obtain
\[
a_p   \stackrel{\text{\tiny\eqref{eq:P1pq13}}}{=}    a_{h_1}a_h^{-1}a_e 
\stackrel{\text{\tiny\eqref{eq:P1pq15}}}{=}a_e 
\stackrel{\text{\tiny\eqref{eq:P1pq11}}}{=} a_{e_1}a_{f_1}^{-1}a_f
\stackrel{\text{\tiny\eqref{eq:P1pq16}}}{=}a_q a_{f_1}^{-1}a_f
\stackrel{\text{\tiny\eqref{eq:P1pq12}}}{=}a_q a_{g_1}^{-1}a_g  
\stackrel{\text{\tiny\eqref{eq:P1pq14}}}{=}a_q,
\]
as required.
\end{proof}

\begin{rem}\label{rem:eggbox6}
Figure \ref{fig:eggbox5} (right) shows the location in $D(n,n-2)$ of the idempotents from the singular squares $\smat{e}{e_1}{f}{f_1}$, $\smat{f}{f_1}{g}{g_1}$ and $\smat{e}{h}{p}{h_1}$ used in the proof of Lemma~\ref{la:P1pq1}.  
The dashed line between $q$ and $e_1$ indicates the equality \eqref{eq:P1pq16} of the generators corresponding to these idempotents, which itself follows from Lemma \ref{la:P1ep1}.
\end{rem}

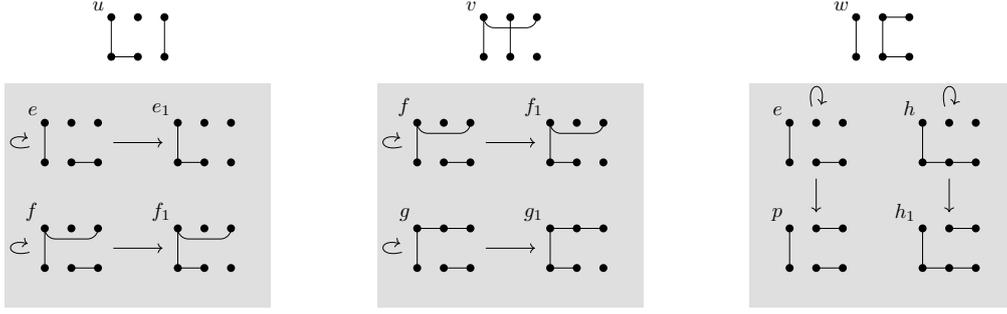
\begin{figure}[t!]
\begin{center}
\scalebox{0.7}{
\begin{tikzpicture}[scale=.5]
\begin{scope}[shift={(0,0)}]
\diagramshading3
\ediagram{e}{3}{\udline11 \ddline23}
\fdiagram{e_1}{3}{\udline11 \ddline12}
\gdiagram{f}{3}{\udline11 \ddline23 \uarc13}
\hdiagram{f_1}{3}{\udline11 \ddline12 \uarc13}
\udiagram{u}{3}{\udline11 \udline33 \ddline12}
\LRarrows
\end{scope}
\begin{scope}[shift={(14,0)}]
\diagramshading3
\ediagram{f}{3}{\udline11 \ddline23 \uarc13}
\fdiagram{f_1}{3}{\udline11 \ddline12 \uarc13}
\gdiagram{g}{3}{\udline11 \uuline13 \ddline23}
\hdiagram{g_1}{3}{\udline11 \uuline13 \ddline12}
\udiagram{v}{3}{\udline11 \udline22 \uarc13}
\LRarrows
\end{scope}
\begin{scope}[shift={(28,0)}]
\diagramshading3
\ediagram{e}{3}{\udline11 \ddline23}
\fdiagram{h}{3}{\udline11 \ddline13}
\gdiagram{p}{3}{\udline11 \uuline23 \ddline23}
\hdiagram{h_1}{3}{\udline11 \ddline13 \uuline23}
\udiagram{w}{3}{\udline11 \udline22 \ddline23 \uuline23}
\UDarrows
\end{scope}
\end{tikzpicture}
}
\caption{The singular squares featuring in the proof of Lemma \ref{la:P1pq1}.
In each partition, only the points $i,j,k$ and their dashed counterparts are shown (left to right).}
\label{fig:P1pq1}
\end{center}
\end{figure}

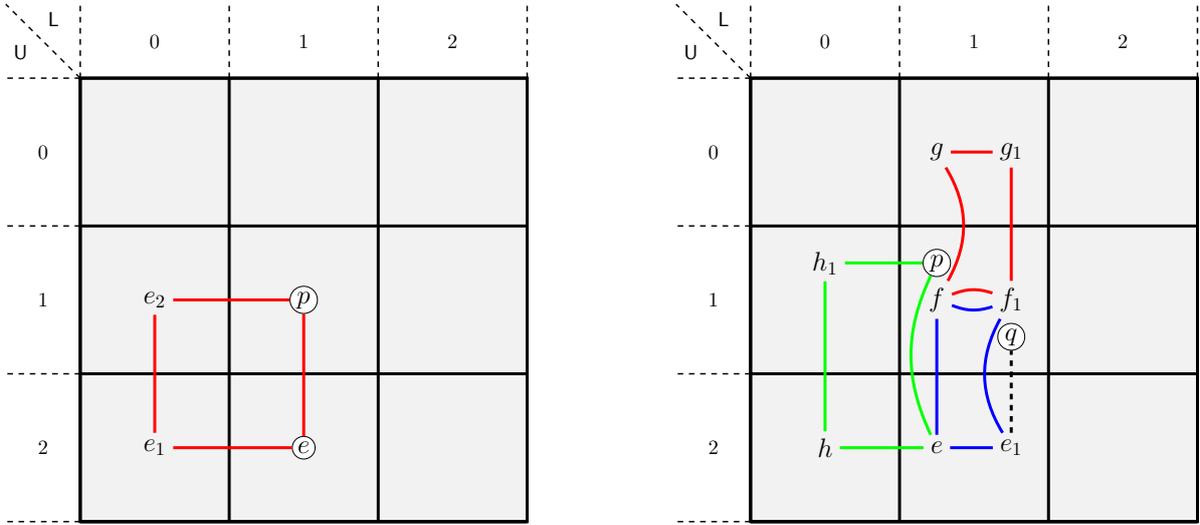
\begin{figure}[t]
\begin{center}
\scalebox{.7}{
\begin{tikzpicture}[scale=.7]

\begin{scope}
\fill[gray!10](0,0)--(12,0)--(12,12)--(0,12)--(0,0);
\foreach \x in {0,4,8,12} {\draw[ultra thick] (\x,0)--(\x,12) (0,12-\x)--(12,12-\x);}
\draw[ultra thick](0,0)--(12,0)--(12,12)--(0,12)--(0,0)--(12,0);
\draw[thick, dashed] (0,12)--(-2,14);
\node () at (-.6-1,12.22+.5) { $\NTu$};
\node () at (-.22-.5,12.6+1) { $\NTd$};
\foreach \x in {0,4,8,12} {\draw[thick, dashed] (\x,12)--(\x,14) (0,12-\x)--(-2,12-\x);}
\node () at (2,13) { $0$};
\node () at (6,13) { $1$};
\node () at (10,13) { $2$};
\node () at (-1,10) { $0$};
\node () at (-1,6) { $1$};
\node () at (-1,2) { $2$};
\node (e2) at (2,6) {\Large $e_2$};
\node[circle,draw=black, fill=white, inner sep = 0.04cm] (p) at (6,6) {\Large $p$};
\node (e1) at (2,2) {\Large $e_1$};
\node[circle,draw=black, fill=white, inner sep = 0.04cm] (e) at (6,2) {\Large $e$};
\draw[ultra thick, red] (e2)--(p)--(e)--(e1)--(e2);
\end{scope}

\begin{scope}[shift={(18,0)}]
\fill[gray!10](0,0)--(12,0)--(12,12)--(0,12)--(0,0);
\foreach \x in {0,4,8,12} {\draw[ultra thick] (\x,0)--(\x,12) (0,12-\x)--(12,12-\x);}
\draw[ultra thick](0,0)--(12,0)--(12,12)--(0,12)--(0,0)--(12,0);
\draw[thick, dashed] (0,12)--(-2,14);
\node () at (-.6-1,12.22+.5) { $\NTu$};
\node () at (-.22-.5,12.6+1) { $\NTd$};
\foreach \x in {0,4,8,12} {\draw[thick, dashed] (\x,12)--(\x,14) (0,12-\x)--(-2,12-\x);}
\node () at (2,13) { $0$};
\node () at (6,13) { $1$};
\node () at (10,13) { $2$};
\node () at (-1,10) { $0$};
\node () at (-1,6) { $1$};
\node () at (-1,2) { $2$};
\node[circle,draw=black, fill=white, inner sep = 0.04cm] (p) at (5,7) {\Large $p$};
\node[circle,draw=black, fill=white, inner sep = 0.04cm] (q) at (7,5) {\Large $q$};
\node (f) at (5,6) {\Large $f$};
\node (f1) at (7,6) {\Large $f_1$};
\node (e) at (5,2) {\Large $e$};
\node (e1) at (7,2) {\Large $e_1$};
\node (g) at (5,10) {\Large $g$};
\node (g1) at (7,10) {\Large $g_1$};
\node (h1) at (2,7) {\Large $h_1$};
\node (h) at (2,2) {\Large $h$};
\draw[ultra thick, red] (g)--(g1)--(f1);
\draw[ultra thick, red] (f1) to[bend right=20] (f);
\draw[ultra thick, red] (f) to[bend right=30] (g);
\draw[ultra thick, blue] (f)--(e)--(e1);
\draw[ultra thick, blue] (f) to[bend right=20] (f1);
\draw[ultra thick, blue] (f1) to[bend right=30] (e1);
\draw[ultra thick, green] (p)--(h1)--(h)--(e);
\draw[ultra thick, green] (p) to[bend right=25] (e);
\draw[ultra thick, dashed] (e1)--(q);
\end{scope}

\end{tikzpicture}
}
\caption{Stratified egg-box diagrams of the idempotents in the singular square(s) featuring in the proofs of Lemmas \ref{la:P1ep1} (left) and \ref{la:P1pq1} (right).  See Remarks \ref{rem:eggbox5} and \ref{rem:eggbox6} for more details.}
\label{fig:eggbox5}
\end{center}
\end{figure}

Our next lemma establishes equalities between generators corresponding to natural families of projections.
Still supposing $r=n-2$, for any $i,j\in [n]$ with $i<j$ denote by $\sfp(i,j)$ the unique projection in $P_1 = P_1(n,n-2)$ with the upper block $\{i,j\}$.

\begin{lemma}
\label{la:P1pq23}
Suppose $1\leq r= n-2$, and let $i,j\in [n]$ with $i<j$. 
\begin{thmenumerate}
\item
\label{it:pq231}
If $i+1<j$ then $a_{\sfp(i,j)}=a_{\sfp(i+1,j)}$  in $\Ghat(n,r)$.
\item
\label{it:pq232}
If $j<n$ then $a_{\sfp(i,j)}=a_{\sfp(i,j+1)}$  in $\Ghat(n,r)$.
\end{thmenumerate}
\end{lemma}

\begin{proof}
\ref{it:pq231}
Let $q:=\begin{partn}{2} i,i+1 & j \\ \hhline{~|-} i,i+1 & j\end{partn}\oplus \id_Y$, with $Y:=[n]\setminus\{i,i+1,j\}$.
Applying Lemma \ref{la:P1pq1}, with the role of the triple $(i,j,k)$ played by $(i+1,i,j)$, yields the relation $a_{\sfp(i,j)}=a_q$.
Likewise, the same lemma, but with the role of $(i,j,k)$ played by $(i,i+1,j)$, yields
$a_{\sfp(i+1,j)}=a_q$. Combining these two relations yields $a_{\sfp(i,j)}=a_{\sfp(i+1,j)}$.

\ref{it:pq232}
Analogous reasoning to part \ref{it:pq231}, with $q:= \begin{partn}{2} j,j+1 & i\\ \hhline{~|-} j,j+1 & i\end{partn}\oplus \id_Y$ for $Y:= [n]\setminus \{ i,j,j+1\}$, yields $a_{\sfp(i,j)}=a_q=a_{\sfp(i,j+1)}$.
\end{proof}

We can now complete our objective for $r=n-2$:

\begin{lemma}
\label{la:P1rr2}
 Suppose $1\leq r= n-2$. 
For arbitrary $p,q\in P_1(n,r)$,
we have $a_p=a_q$  in~$\Ghat(n,r)$. 
\end{lemma}

\begin{proof}
Any projection in $P_1$ satisfies precisely one of the following:
 either (1) all its transversals have size $2$; or else (2) there exists a single transversal of size $4$. 
 By Lemma \ref{la:P1pq1} every generator corresponding to a projection satisfying (2) is equal to a generator corresponding to a projection satisfying (1). So we may assume without loss of generality that both $p$ and $q$ satisfy (1).
 We can therefore write $p=\sfp(i,j)$ for some $i<j$.  
 By Lemma \ref{la:P1pq23} we have
 \[
a_p=a_{\sfp(i,j)}=a_{\sfp(i,j+1)}=\cdots=a_{\sfp(i,n)}=a_{\sfp(i+1,n)}=\cdots=a_{\sfp(n-1,n)}.
 \]
 Analogously, $a_q=a_{\sfp(n-1,n)}$, and hence $a_p=a_q$, as required.
\end{proof}

Now that we have dealt with $r=n-2$, we move on to the general case:

\begin{lemma}
\label{la:P1nr}
 Suppose $1\leq r\leq n-2$. 
For arbitrary $p,q\in P_1(n,r)$,
we have $a_p=a_q$ in~$\Ghat(n,r)$.  
\end{lemma}

\begin{proof}
We fix $r\geq1$, and prove the lemma by induction on $n\geq r+2$.  When $n=r+2$, the result is true by Lemma \ref{la:P1rr2}, so for the rest of the proof we assume that $n\geq r+3$.

For $i,j\in [n]$ with $i\neq j$ denote by $\P_n^{i,j}$ the set of all partitions $a\in\P_n$ such that $(i,j)\in \ker (a)$
and $(i,j)\in\coker(a)$.
This is a subsemigroup of $\P_n$ isomorphic to $\P_{n-1}$.

Recall that the group $\Ghat(n,r)$ is defined by the generators $A(n,r)$ and the defining relations 
$\rels_1(\labU(n,r))$ and $\rels_\subSq(n,r)$.
Furthermore, recall that $\labU(n,r)$ consists of all non-projection idempotents in $D(n,r)$ with full domain or full codomain, and with label $1$. In particular, $\labU(n,r)$ contains all the idempotents belonging to $\P_n^{i,j}$ with the above properties.
Likewise, $\rels_\subSq(n,r)$ contains all the relations arising from singular squares that lie entirely in $\P_n^{i,j}$.
In this way, we see that the presentation $\langle A(n,r)\mid \rels_1(\labU(n,r)),\ \rels_\subSq(n,r)\rangle$ contains as a sub-presentation a natural copy of the presentation
$\langle A(n-1,r)\mid \rels_1(\labU(n-1,r)),\ \rels_\subSq(n-1,r)\rangle$ corresponding to~$\P_n^{i,j}$.
Thus, if the projections $p$ and $q$ belong to some common $\P_n^{i,j}$ we conclude that $a_p=a_q$ in~$\Ghat(n,r)$ by induction.

Now consider arbitrary $p,q\in P_1$, and fix $(i,j)\in\ker(p)$ and $(k,l)\in\ker(q)$ with $i\neq j$ and~$k\neq l$.
Note that all elements of $P_1$ have non-trivial kernel, and also that we do not require $\{i,j\}$ and $\{k,l\}$ to be disjoint.
Since $r\leq n-3$ there exists $u\in P_1$ such that $(i,j),(k,l)\in \ker(u)$.
But then $p,u\in \P_n^{i,j}$ and $u,q\in\P_n^{k,l}$. Using the observation from the previous paragraph we have
$a_p=a_u=a_q$, completing the proof.
\end{proof}

\begin{rem}\label{rem:presnUlab}
The fact that we were working with the groups $\Ghat(n,r)$, which do not depend on the choice of projection $s\in P_0$, is crucial for the validity of the inductive argument in our proof of Lemma \ref{la:P1nr}. Indeed, if we worked with the set $O\cup\{s\}$, we would not be able to perform the inductive step, since $s$ certainly would not be contained in all $\P_n^{i,j}$.  The same is true of the proof of Lemma \ref{la:P0nr} below.
\end{rem}

\subsection[Generators corresponding to projections from $P_0(n,r)$]{Generators corresponding to projections from \boldmath{$P_0(n,r)$}}
\label{ss:eP0}

We now want to do the same for generators corresponding to projections from $P_0(n,r)$, i.e.~show they are all equal to each other as a consequence of the defining relations. We will again need to work within the group $\Ghat = \Ghat(n,r) = \sfG(\labU)$; however, once the relation $a_s=1$ is brought back in (for a fixed projection $s\in P_0(n,r)$), the conclusion will be that all these generators are in fact equal to $1$.  As before, we begin with the case $r=n-2$.
If $n=3$ (and $r=1=n-2$) then $|P_0|=1$, and Lemmas \ref{la:P0pq1}, \ref{la:P0pq23} and \ref{la:P0rr2} are vacuously true in this case.

\begin{lemma}
\label{la:P0pq1}
Suppose $1\leq r= n-2$, let $i,j,k,l\in [n]$ be distinct, let $Y:=[n]\setminus \{i,j,k,l\}$, and let
\[
p:=\begin{partn}{2} i,j,k&l\\ \hhline{~|~}  i,j,k & l\end{partn}\oplus \id_Y \ANd
q:=\begin{partn}{2} i,j &k,l\\ \hhline{~|~}  i,j & k,l\end{partn}\oplus \id_Y.
\]
If $i,l\leq k+1$ then $a_p=a_q$  in $\Ghat(n,r)$.
\end{lemma}

\begin{proof}
Consider the following further idempotents:
\begin{align*}
e& := \begin{partn}{4} i & j&k&l \\ \hhline{~|-|-|~} i,k & \multicolumn{2}{c|}{j} & l\end{partn}\oplus \id_Y, &
f&:=\begin{partn}{4} i&j&k&l\\ \hhline{~|-|-|~} i & \multicolumn{2}{c|}{j} & k,l \end{partn}\oplus \id_Y, &
g_1& := \begin{partn}{3} i,j&k&l   \\ \hhline{~|-|~} i,j,k& &l\end{partn}\oplus \id_Y, &
h_1&= \begin{partn}{3} i,j&k&l\\ \hhline{~|-|~} i,j&&k,l\end{partn}\oplus \id_Y,&
\\
g&= \begin{partn}{3} i,j&k&l\\ \hhline{~|-|~} i,k&j&l\end{partn}\oplus \id_Y,&
h& := \begin{partn}{3} i,j&k&l   \\ \hhline{~|-|~} i & j&k,l\end{partn}\oplus \id_Y, &
g_2&:=\begin{partn}{3} i,j,k&&l\\ \hhline{~|-|~} i,k&j&l\end{partn}\oplus \id_Y, &
h_2&= \begin{partn}{3} i,j&&k,l\\ \hhline{~|-|~} i&j&k,l\end{partn}\oplus \id_Y.
\end{align*}
Then:
\begin{alignat}{5}
\label{eq:P0pq11}
&\smat{e}{f}{g}{h} &\ &\text{is LR-singularised by}
&\ &u=\begin{partn}{4} i&j&k&l\\ \hhline{~|~|-|~} i&j&& k,l\end{partn}\oplus\id_Y,&\ \text{ yielding}
&\quad&a_e^{-1}a_{f}=a_g^{-1}a_{h};
\\
\label{eq:P0pq12}
&\smat{g}{g_1}{g_2}{p} &&\text{is LR-singularised by}
&&v=\begin{partn}{3} i,j&k&l\\ \hhline{~|~|~} i,j&k& l\end{partn}\oplus\id_Y,&\text{ yielding}
&&a_g^{-1}a_{g_1}=a_{g_2}^{-1}a_{p};
\\
\label{eq:P0pq13}
&\smat{h}{h_1}{h_2}{q} &&\text{is LR-singularised by}
&&w=\begin{partn}{3} i,j&k&l\\ \hhline{~|~|~} i,j&k& l\end{partn}\oplus\id_Y,&\text{ yielding}
&&a_h^{-1}a_{h_1}=a_{h_2}^{-1}a_{q}.
\end{alignat}
These squares are illustrated in Figure \ref{fig:P0pq1}.
By Lemma \ref{la:P1ep1}, we have $a_e=a_{e'}$, where
$e'=\begin{partn}{3} i,k&j&l\\ \hhline{~|-|~} i,k&j&l\end{partn}\oplus\id_Y$,
and also 
$a_f=a_{f'}$, where
$f'=\begin{partn}{3} i&j&k,l\\ \hhline{~|-|~} i&j&k,l\end{partn}\oplus\id_Y$.
Note that $e',f'\in P_1$, so that $a_{e'}=a_{f'}$ by Lemma \ref{la:P1nr}.
Hence $a_e=a_f$, and combining with \eqref{eq:P0pq11} yields $a_g=a_h$.

Combining this last conclusion with \eqref{eq:P0pq12} and \eqref{eq:P0pq13} yields
\[
a_{g_1}a_p^{-1}a_{g_2}=a_g=a_h=a_{h_1}a_q^{-1} a_{h_2}.
\]
The desired relation $a_p=a_q$ will follow if we manage to show that
$a_{g_1}=a_{h_1}$ and $a_{g_2}=a_{h_2}$.
In fact, by Lemmas  \ref{la:eqlabeldual} and \ref{la:eqlabel}, it suffices to show that $\lambda(g_i)=\lambda(h_i)$ for $i=1,2$.

Again writing $x\wedge y$ for $\min(x,y)$, we have
\begin{alignat*}{2}
\lamp (g_1)&= \trans{i\wedge j & l}{i\wedge j\wedge k&l}\oplus\id_Y,&\quad
\lamp(g_2) & = \trans{i\wedge j\wedge k & l}{i\wedge k&l}\oplus\id_Y,
\\
\lamp(h_1) &=  \trans{i\wedge j & l}{i\wedge j &k\wedge l}\oplus\id_Y,&
\lamp(h_2) & = \trans{i\wedge j & k\wedge l}{i&k\wedge l}\oplus\id_Y.
\end{alignat*}
Since $i\leq k+1$, the values $i\wedge j\wedge k$ and $i\wedge j$ are either equal or consecutive.
Likewise, $l$ and $k\wedge l$ are either equal or consecutive.
Furthermore, not both pairs can be consecutive, as this would imply that
$i=k+1=l$.
It again follows from Lemma \ref{la:consec} that $\lam(g_1)=\lam(h_1)$.
We obtain $\lam(g_2)=\lam(h_2)$ analogously, and this completes the proof of the lemma.
\end{proof}

\begin{rem}\label{rem:eggbox7}
Figure \ref{fig:eggbox7} shows the location in $D(n,n-2)$ of the idempotents from the singular squares $\smat{e}{f}{g}{h}$, $\smat{g}{g_1}{g_2}{p}$ and $\smat{h}{h_1}{h_2}{q}$ used in the proof of Lemma~\ref{la:P0pq1}.  The dashed lines between $e$, $e'$, $f'$ and $f$ indicate equality of the generators corresponding to these idempotents, as explained in the proof.
\end{rem}

\begin{figure}[t!]
\begin{center}
\scalebox{0.7}{
\begin{tikzpicture}[scale=.5]
\begin{scope}[shift={(0,0)}]
\diagramshading4
\ediagram{e}{4}{\udline11 \udline44 \darc13}
\fdiagram{f}{4}{\udline11 \udline44 \ddline34}
\gdiagram{g}{4}{\udline11 \udline44 \uuline12 \darc13}
\hdiagram{h}{4}{\udline11 \udline44 \uuline12 \ddline34}
\udiagram{u}{4}{\udline11 \udline22 \udline44 \ddline34}
\LRarrows
\end{scope}
\begin{scope}[shift={(14,0)}]
\diagramshading4
\ediagram{g}{4}{\udline11 \udline44 \uuline12 \darc13}
\fdiagram{g_1}{4}{\udline11 \udline44 \uuline12 \ddline12\ddline23}
\gdiagram{g_2}{4}{\udline11 \udline44 \uuline12 \uuline23 \darc13}
\hdiagram{p}{4}{\udline11 \udline44 \uuline12 \uuline23 \ddline12\ddline23}
\udiagram{v}{4}{\udline11 \udline33 \udline44 \uuline12 \ddline12}
\LRarrows
\end{scope}
\begin{scope}[shift={(28,0)}]
\diagramshading4
\ediagram{h}{4}{\udline11 \udline44 \uuline12 \ddline34}
\fdiagram{h_1}{4}{\udline11 \udline44 \uuline12 \ddline34 \ddline12}
\gdiagram{h_2}{4}{\udline11 \udline44 \uuline12 \ddline34 \uuline34}
\hdiagram{q}{4}{\udline11 \udline44 \uuline12 \ddline34 \ddline12 \uuline34}
\udiagram{w}{4}{\udline11 \udline33 \udline44 \uuline12 \ddline12}
\LRarrows
\end{scope}
\end{tikzpicture}
}
\caption{The singular squares featuring in the proof of Lemma \ref{la:P0pq1}.
In each partition, only the points $i,j,k,l$ and their dashed counterparts are shown (left to right).}
\label{fig:P0pq1}
\end{center}
\end{figure}
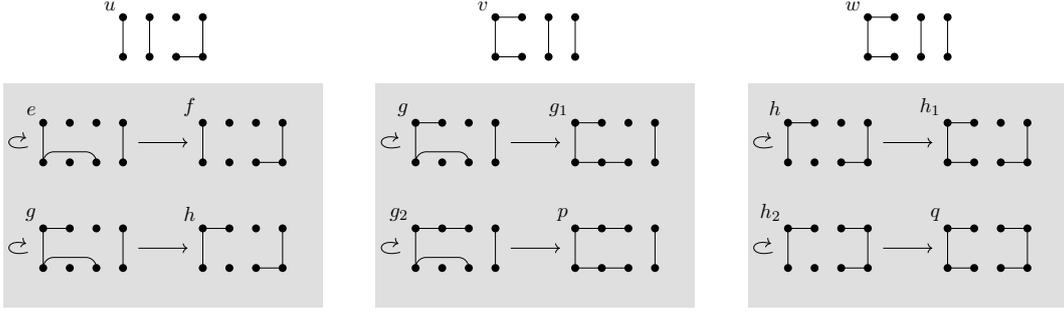

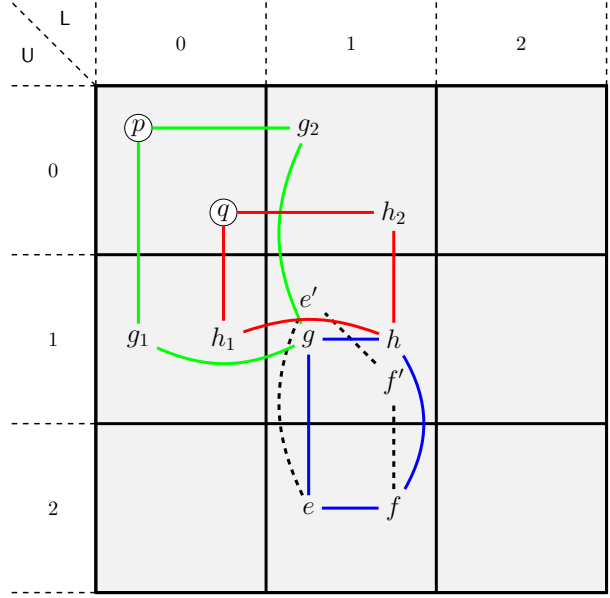
\begin{figure}[t]
\begin{center}
\scalebox{.7}{
\begin{tikzpicture}[scale=.8]

\begin{scope}[shift={(18,0)}]
\fill[gray!10](0,0)--(12,0)--(12,12)--(0,12)--(0,0);
\foreach \x in {0,4,8,12} {\draw[ultra thick] (\x,0)--(\x,12) (0,12-\x)--(12,12-\x);}
\draw[ultra thick](0,0)--(12,0)--(12,12)--(0,12)--(0,0)--(12,0);
\draw[thick, dashed] (0,12)--(-2,14);
\node () at (-.6-1,12.22+.5) { $\NTu$};
\node () at (-.22-.5,12.6+1) { $\NTd$};
\foreach \x in {0,4,8,12} {\draw[thick, dashed] (\x,12)--(\x,14) (0,12-\x)--(-2,12-\x);}
\node () at (2,13) { $0$};
\node () at (6,13) { $1$};
\node () at (10,13) { $2$};
\node () at (-1,10) { $0$};
\node () at (-1,6) { $1$};
\node () at (-1,2) { $2$};
\node[circle,draw=black, fill=white, inner sep = 0.04cm] (p) at (1,11) {\Large $p$};
\node[circle,draw=black, fill=white, inner sep = 0.04cm] (q) at (3,9) {\Large $q$};
\node (g2) at (5,11) {\Large $g_2$};
\node (h2) at (7,9) {\Large$h_2$};
\node (g1) at (1,6) {\Large$g_1$};
\node (h1) at (3,6) {\Large$h_1$};
\node (g) at (5,6) {\Large$g$};
\node (h) at (7,6) {\Large$h$};
\node (e) at (5,2) {\Large$e$};
\node (f) at (7,2) {\Large$f$};
\node (e') at (5,7) {\Large$e'$};
\node (f') at (7,5) {\Large$f'$};
\draw[ultra thick, dashed] (f)--(f')--(e');
\draw[ultra thick, dashed] (e) to[bend left=25] (e');
\draw[ultra thick, green] (g2)--(p)--(g1);
\draw[ultra thick, green] (g1) to[bend right=25] (g);
\draw[ultra thick, green] (g) to[bend left=25] (g2);
\draw[ultra thick, red] (h1)--(q)--(h2)--(h);
\draw[ultra thick, red] (h) to[bend right=20] (h1);
\draw[ultra thick, blue] (h)--(g)--(e)--(f);
\draw[ultra thick, blue] (f) to[bend right=30] (h);
\end{scope}

\end{tikzpicture}
}
\caption{Stratified egg-box diagram of the idempotents in the singular squares featuring in the proof of Lemma \ref{la:P0pq1}.  See Remark \ref{rem:eggbox7} for more details.}
\label{fig:eggbox7}
\end{center}
\end{figure}

Still supposing $r=n-2$, for any $i,j,k\in [n]$ with $i<j<k$ denote by $\sfp(i,j,k)$ the unique projection in $P_0$ with transversal $\{i,j,k,i',j',k'\}$.

\begin{lemma}
\label{la:P0pq23}
Suppose $1\leq r= n-2$, and let $i,j,k\in [n]$ with $i<j<k$. 
\begin{thmenumerate}
\item
\label{it:pq0231}
If $i+1<j$ then $a_{\sfp(i,j,k)}=a_{\sfp(i+1,j,k)}$  in $\Ghat(n,r)$.
\item
\label{it:pq0232}
If $j+1<k$ then $a_{\sfp(i,j,k)}=a_{\sfp(i,j+1,k)}$  in $\Ghat(n,r)$.
\item
\label{it:pq0233}
If $k <n$ then $a_{\sfp(i,j,k)}=a_{\sfp(i,j,k+1)}$  in $\Ghat(n,r)$.
\end{thmenumerate}
\end{lemma}

\begin{proof}
The proof is analogous to the proof of Lemma \ref{la:P1pq23}, but using Lemma \ref{la:P0pq1} instead of Lemma \ref{la:P1pq1}.
Thus for \ref{it:pq0231} we let 
$q:=\begin{partn}{2} i,k & i+1,j\\ \hhline{~|~} i,k & i+1,j    \end{partn}\oplus\id_Y$,
where $Y:=[n]\setminus\{i,i+1,j,k\}$,
and then apply Lemma \ref{la:P0pq1} twice to obtain $a_{\sfp(i,j,k)}=a_q=a_{\sfp(i+1,j,k)}$.
The intermediate partitions to use in proving \ref{it:pq0232} and \ref{it:pq0233}
are
$\begin{partn}{2} i,k & j,j+1\\ \hhline{~|~} i,k & j,j+1    \end{partn}\oplus\id_Y$ and
$\begin{partn}{2} i,j & k,k+1\\ \hhline{~|~} i,j & k,k+1    \end{partn}\oplus\id_Y$
(for appropriate $Y$), respectively.
\end{proof}

\begin{lemma}\label{la:P0rr2}
Suppose $1\leq r= n-2$.  For arbitrary $p,q\in P_0(n,r)$, we have $a_p=a_q$ in~$\Ghat(n,r)$. 
\end{lemma}

\begin{proof}
The proof is analogous to the proof of Lemma \ref{la:P1rr2}.
Lemma \ref{la:P0pq1} enables us to assume without loss of generality that $p$ and $q$ have a transversal of size $6$.
We can therefore write $p=\sfp(i,j,k)$ for some $i<j<k$.  
By Lemma \ref{la:P0pq23} we have
\begin{align*}
a_p=a_{\sfp(i,j,k)}=a_{\sfp(i,j,k+1)}=\cdots=a_{\sfp(i,j,n)}=a_{\sfp(i,j+1,n)}=\cdots&=a_{\sfp(i,n-1,n)}\\
&=a_{\sfp(i+1,n-1,n)}=\cdots=a_{\sfp(n-2,n-1,n)}.
\end{align*}
Analogously, $a_q=a_{\sfp(n-2,n-1,n)}$, and hence $a_p=a_q$.
\end{proof}

\begin{lemma}\label{la:P0nr}
Suppose $1\leq r\leq n-2$.  For arbitrary $p,q\in P_0(n,r)$, we have $a_p=a_q$  in~$\Ghat(n,r)$.  
\end{lemma}

\begin{proof}
This is exactly the same as the proof of Lemma \ref{la:P1nr}, except that Lemma \ref{la:P0rr2} is used in place of Lemma \ref{la:P1rr2} when dealing with the case $r= n-2$.
\end{proof}

Thus, as soon as we add to $\Ghat(n,r) = \sfG(\labU)$ a relation equating any generator corresponding to a projection from $P_0$ to $1$, all such generators are equal to $1$:

\begin{lemma}\label{la:P0nr1}
Suppose $1\leq r\leq n-2$, and let $s\in P_0(n,r)$ be arbitrary.  Then for every $p\in P_0(n,r)$ we have $a_p=1$  in~$\sfG(O\cup\{s\})$.  \qed
\end{lemma}

\subsection{Coxeter idempotents}
\label{ss:glt}

As our next step, we now move on to the generators~$a_e$ corresponding to full-domain/codomain Coxeter idempotents $e\in E(n,r)$, i.e.~those for which $\lam(e) = (i,i+1)$ is a Coxeter transposition.  The significance of these elements is that:
\bit
\item any generator of $\GI$ corresponding to a full-domain non-projection idempotent is a product of generators corresponding to full-domain Coxeter idempotents (see Lemmas \ref{la:GRtransl} and \ref{la:fdtoTn}), and
\item the generators corresponding to full-domain Coxeter idempotents satisfy the defining (Coxeter) relations of $\S_r$ (see \cite[Lemmas 8.1--8.3]{GR12IJM}).
\eit
Together with the dual statements concerning full-\emph{co}domain Coxeter idempotents (and additional arguments provided later in this section), it will follow that the full-domain and full-codomain idempotents lead to two copies of $\S_r$ in our group $\GI$.  The next lemma is pivotal, and will be used to show that in fact these copies of $\S_r$ coincide.  We also bring the relation $a_s=1$ back in, where $s$ is our fixed projection from~$P_0(n,r)$.

\begin{lemma}
\label{la:lab12}
Suppose $1\leq r\leq n-2$, and let $s\in P_0(n,r)$ be arbitrary.
Let $e,f\in E(n,r)$ be such that $e$ is full-domain, $f$ is full-codomain, and $\lambda(e)=\lambda(f)=(i ,  i+1)$ for some $1\leq i\leq r-1$.  Then $a_e=a_f$  in $\sfG(\labU\cup\{s\})$.
\end{lemma}

\begin{proof}
The result is vacuous for $r= 1$, so we assume that $r\geq2$, which also implies $n\geq4$.
During the proof, for a partition $g \in \P_{[i,i+3]}$, we write $g^+\in\P_n$ for the following:
\bit
\item If $i\leq r-2$, then $g^+ = g \oplus \id_\sigma$, where $\sigma$ is the equivalence on $[n]\setminus [i,i+3]$ with classes
\[
\{1\},\ldots,\{i-1\}, \{i+4\},\ldots,\{r+1\},[r+2,n].
\]
\item If $i=r-1$, then $g^+$ is obtained from $g\oplus\id_{[1,i-1]}$ by adding $i+4,\ldots,n$ to the block of $g$ containing $i+3$, and adding $(i+4)',\ldots,n'$ to the block of $g$ containing $(i+3)'$.
\eit
In each case, $g$ will be constructed so as to have rank $2$, and $g^+$ will consequently have rank $r$.
With this notation in place, it follows from Lemmas \ref{la:eqlabel} and \ref{la:eqlabeldual} that we can prove the current lemma for the specific choices
\[
e= \begin{partn}{3} i,i+2,i+3 && i+1 \\ \hhline{~|-|~} i+2,i+3 & i & i+1\end{partn}^{\!\!+} \ANd
f= \begin{partn}{3} i,i+2 & i+3 & i+1  \\ \hhline{-|~|~} & i,i+2,i+3 &i+1 \end{partn}^{\!\!+}.
\]
We proceed via a series of claims as follows.  Each involves certain idempotents $g_i$, whose meanings are fixed throughout.  The elements constructed during the proofs are temporary, however, and we will re-use their names ($h_i$ and $u$).  Also, $t$ will denote some arbitrary but fixed projection from $P_1=P_1(n,r)$.
The singular squares we use in the proofs of the claims are shown in Figure~\ref{fig:lab12}.  Some of these are instances of the square from Lemma \ref{la:ehrsq} (or its dual), but for convenience we will provide a singularising element in each case.

\begin{figure}[t!]
\begin{center}
\scalebox{0.7}{
\begin{tikzpicture}[scale=.5]
\begin{scope}[shift={(0,0)}]
\diagramshading4
\ediagram{g_1}{4}{\udline22 \udline44 \uarc13}
\fdiagram{h_1}{4}{\udline22 \udline44 \uarc13 \darc13}
\gdiagram{h_2}{4}{\udline22 \udline44 \uuline13 }
\hdiagram{h_3}{4}{\udline22 \udline44 \uuline13 \darc13}
\udiagram{u}{4}{\udline11 \udline22 \udline33 \udline44 \uarc13}
\LRarrows
\node () at (5.5,-2.1) {Claim 1};
\end{scope}
\begin{scope}[shift={(14,0)}]
\diagramshading4
\ediagram{g_2}{4}{\udline22 \udline44 \ddline34}
\fdiagram{h_1}{4}{\udline22 \udline44 \darc13 \ddline34}
\gdiagram{h_2}{4}{\udline22 \udline44 \uuline34 \ddline34}
\hdiagram{h_3}{4}{\udline22 \udline44 \uuline34 \darc13 \ddline34}
\udiagram{u}{4}{\udline22 \udline33 \udline44 \darc13}
\LRarrows
\node () at (5.5,-2.1) {Claim 2};
\end{scope}
\begin{scope}[shift={(28,0)}]
\diagramshading4
\ediagram{g_3}{4}{\udline22 \udline44 \uuline12}
\fdiagram{h_1}{4}{\udline22 \udline44 \uuline12 \ddline12}
\gdiagram{h_2}{4}{\udline22 \udline44 \uuline13}
\hdiagram{h_3}{4}{\udline22 \udline44 \uuline13 \ddline12}
\udiagram{u}{4}{\udline11 \udline22 \udline44 \uuline23}
\UDarrows
\node () at (5.5,-2.1) {Claim 3};
\end{scope}
\begin{scope}[shift={(0,-14)}]
\diagramshading4
\ediagram{g_4}{4}{\udline22 \udline44 \uuline12 \ddline34}
\fdiagram{h_1}{4}{\udline22 \udline44 \uuline12 \ddline12 \ddline34}
\gdiagram{h_2}{4}{\udline22 \udline44 \uuline12 \uuline34 \ddline34}
\hdiagram{h_3}{4}{\udline22 \udline44 \uuline12 \uuline34 \ddline12 \ddline34}
\udiagram{u}{4}{\udline22 \udline33 \udline44 \uuline12 \ddline12}
\LRarrows
\node () at (5.5,-2.1) {Claim 4};
\end{scope}
\begin{scope}[shift={(14,-14)}]
\diagramshading4
\ediagram{g_5}{4}{\udline22 \udline44}
\fdiagram{g_2}{4}{\udline22 \udline44 \ddline34}
\gdiagram{g_3}{4}{\udline22 \udline44 \uuline12}
\hdiagram{g_4}{4}{\udline22 \udline44 \uuline12 \ddline34}
\udiagram{u}{4}{\udline11 \udline22 \udline44 \ddline34}
\LRarrows
\node () at (5.5,-2.1) {Claim 5};
\end{scope}
\begin{scope}[shift={(28,-14)}]
\diagramshading4
\ediagram{g_6}{4}{\udline22 \udline44 \uarc13 \ddline34}
\fdiagram{g_1}{4}{\udline22 \udline44 \uarc13}
\gdiagram{g_2}{4}{\udline22 \udline44 \ddline34}
\hdiagram{g_5}{4}{\udline22 \udline44}
\udiagram{u}{4}{\udline22 \udline33 \udline44}
\UDarrows
\node () at (5.5,-2.1) {Claim 6};
\end{scope}
\begin{scope}[shift={(14,-28)}]
\diagramshading4
\ediagram{e}{4}{\udline44 \udline22 \uarc13 \uuline34 \ddline34}
\fdiagram{h_1}{4}{\udline22 \udline44 \uarc13 \uuline34 \darc13 \ddline34}
\gdiagram{g_6}{4}{\udline22 \udline44 \uarc13 \ddline34}
\hdiagram{f}{4}{\udline22 \udline44 \uarc13 \darc13 \ddline34}
\udiagram{u}{4}{\udline22 \udline33 \udline44 \uarc13 \darc13}
\LRarrows
\node () at (5.5,-2.1) {Claim 7};
\end{scope}
\end{tikzpicture}
}
\caption{The singular squares featuring in the proof of Lemma \ref{la:lab12}.
In each partition, only the points $i,i+1,i+2,i+3$ and their dashed counterparts are shown (left to right).}
\label{fig:lab12}
\end{center}
\end{figure}
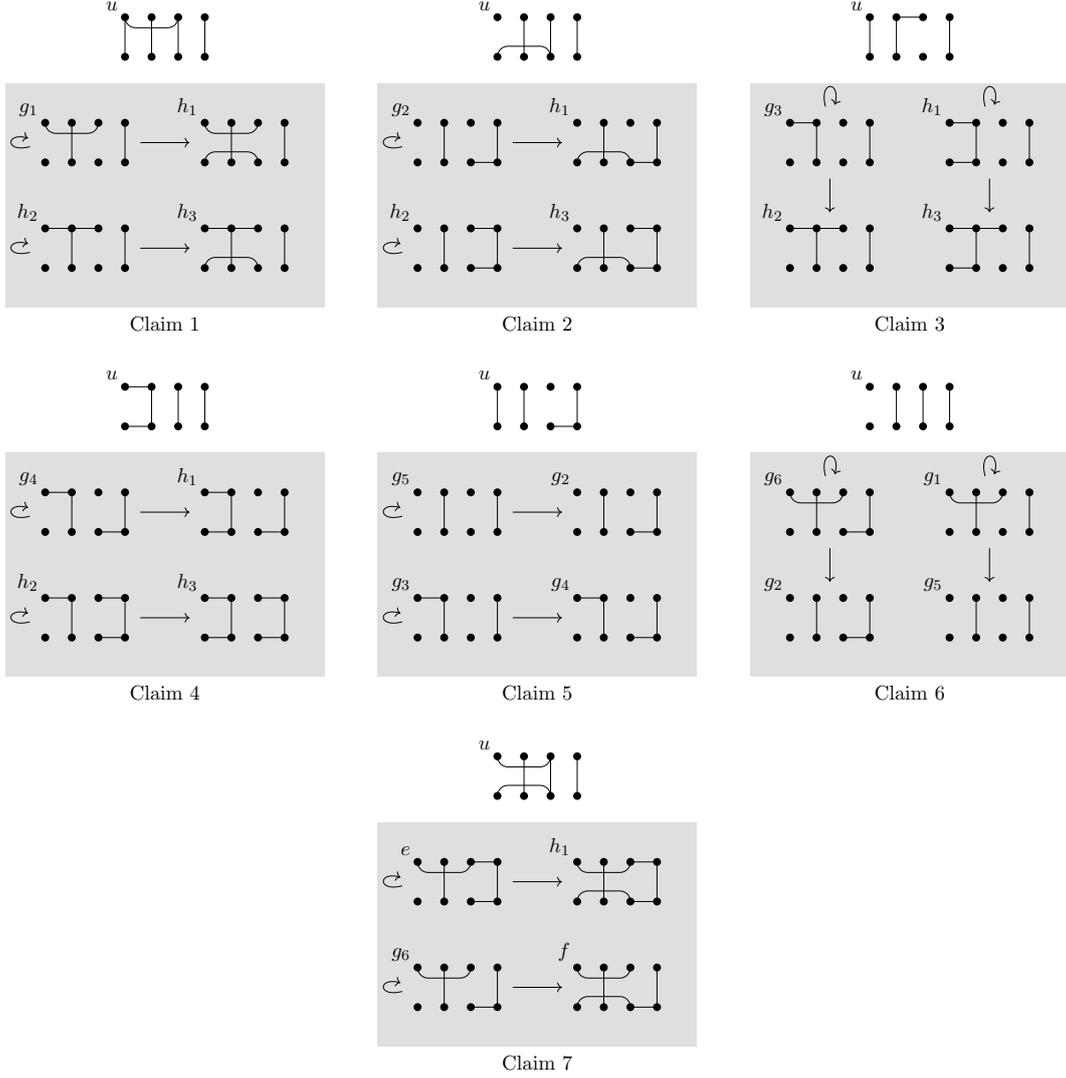

\begin{claim}
\label{cl:lab121}
For 
$g_1:=\begin{partn}{4} \multicolumn{2}{c|}{i,i+2} & i+1 & i+3  \\ \hhline{-|-|~|~} i & i+2 & i+1 & i+3 \end{partn}^{\!\!+}$
we have $a_{g_1}=a_t$.
\end{claim}

\begin{proof}
Consider the square $\smat{g_1}{h_1}{h_2}{h_3}$, where
\[
h_1:=\begin{partn}{3} i,i+2&i+1&i+3 \\ \hhline{-|~|~} i,i+2&i+1&i+3\end{partn}^{\!\!+},\quad
h_2:=\begin{partn}{4} i,i+1,i+2&\multicolumn{2}{c|}{}&i+3 \\ \hhline{~|-|-|~} i+1&i&i+2&i+3 \end{partn}^{\!\!+}\ANd
h_3:=\begin{partn}{3} i,i+1,i+2&&i+3 \\ \hhline{~|-|~} i+1  &i,i+2& i+3\end{partn}^{\!\!+}.
\]
It is easy to verify that it is LR-singularised by
$u=\begin{partn}{3} i,i+2&i+1&i+3\\ \hhline{~|~|~} i,i+2&i+1&i+3\end{partn}^{\!\!+}$.
As $h_1\in P_1$ we have $a_{h_1}=a_t$ by Lemma \ref{la:P1nr}.
The idempotents $h_2$ and $h_3$ have full domain and label $1$, and hence $a_{h_2}=a_{h_3}=1$ by Lemma \ref{la:eqlabel}.
The square relation $a_{g_1}^{-1}a_{h_1}=a_{h_2}^{-1}a_{h_3}$ now implies
$a_{g_1}=a_{h_1}=a_t$, as required.
\end{proof}

\begin{claim}
\label{cl:lab122}
For 
$g_2:=\begin{partn}{4} i&i+1&i+2&i+3 \\ \hhline{-|~|-|~} i & i+1 & &i+2, i+3\end{partn}^{\!\!+}$
we have $a_{g_2}=a_t$.
\end{claim}

\begin{proof}
The square $\smat{g_2}{h_1}{h_2}{h_3}$, where
\[
h_1:=\begin{partn}{4} i&i+1&i+2&i+3 \\ \hhline{-|~|-|~} &i+1&&i,i+2,i+3\end{partn}^{\!\!+},\quad
h_2:=\begin{partn}{3} i&i+1&i+2,i+3\\ \hhline{-|~|~}  i&i+1&i+2,i+3 \end{partn}^{\!\!+}\ANd
h_3:=\begin{partn}{3} i&i+1&i+2,i+3 \\ \hhline{-|~|~}  & i+1 &i,i+2,i+3\end{partn}^{\!\!+},
\]
is LR-singularised by
$u=\begin{partn}{4} i&i+1&i+2&i+3\\ \hhline{-|~|~|~} &i+1&i,i+2&i+3 \end{partn}^{\!\!+}$.
We have $a_{h_2}=a_t$ by Lemma \ref{la:P1nr} since $h_2\in P_1$.
The idempotents $h_1$ and $h_3$ are full-codomain, and both have label $(i ,  i+1)$; hence $a_{h_1}=a_{h_3}$ by Lemma \ref{la:eqlabeldual}.
The square relation now implies
$a_{g_2}=a_{h_2}=a_t$, as required.
\end{proof}

\begin{claim}
\label{cl:lab123}
For 
$g_3:=\begin{partn}{4} i,i+1&\multicolumn{2}{c|}{i+2}&i+3 \\ \hhline{~|-|-|~} i +1& i & i+2& i+3\end{partn}^{\!\!+}$
we have $a_{g_3}=a_t$.
\end{claim}

\begin{proof}
The proof is analogous to the previous two claims.
The square  is $\smat{g_3}{h_1}{h_2}{h_3}$, where
\[
h_1:=\begin{partn}{3} i,i+1&i+2&i+3 \\ \hhline{~|-|~} i,i+1&i+2&i+3\end{partn}^{\!\!+},\quad
h_2:=\begin{partn}{4} i,i+1,i+2&\multicolumn{2}{c|}{}&i+3\\ \hhline{~|-|-|~}  i+1&i&i+2&i+3 \end{partn}^{\!\!+}\ANd
h_3:=\begin{partn}{3} i,i+1,i+2&&i+3 \\ \hhline{~|-|~}  i,i+1&i+2&i+3\end{partn}^{\!\!+}.
\]
The UD-singularising idempotent is
$u=\begin{partn}{4} i&i+1,i+2&&i+3\\ \hhline{~|~|-|~} i&i+1&i+2&i+3 \end{partn}^{\!\!+}$.
This time $h_1$ belongs to $P_1$, while~$h_2$ and $h_3$ have full domain and label $1$.
\end{proof}

\begin{claim}
\label{cl:lab124}
For 
$g_4:=\begin{partn}{3} i,i+1&i+2&i+3 \\ \hhline{~|-|~}  i+1 &i& i+2, i+3\end{partn}^{\!\!+}$
we have $a_{g_4}=1$.
\end{claim}

\begin{proof}
This time the square is  $\smat{g_4}{h_1}{h_2}{h_3}$, where
\[
h_1:=\begin{partn}{3} i,i+1&i+2&i+3 \\ \hhline{~|-|~} i,i+1&&i+2,i+3\end{partn}^{\!\!+},\quad
h_2:=\begin{partn}{3} i,i+1&&i+2,i+3\\ \hhline{~|-|~}  i+1&i&i+2,i+3 \end{partn}^{\!\!+}\ANd
h_3:=\begin{partn}{2} i,i+1&i+2,i+3 \\ \hhline{~|~}  i,i+1&i+2,i+3\end{partn}^{\!\!+},
\]
and the LR-singularising idempotent is
$u=\begin{partn}{3} i,i+1&i+2&i+3\\ \hhline{~|~|~} i,i+1&i+2&i+3 \end{partn}^{\!\!+}$.
The idempotents $h_1$ and $h_2$ have full domain or codomain, and label $1$, so $a_{h_1}=a_{h_2}=1$.
The idempotent $h_3$ is a projection from $P_0$, so $a_{h_3}=1$ by Lemma \ref{la:P0nr1}.
Thus, the square relation gives $a_{g_4}=1$.
\end{proof}

\begin{claim}
\label{cl:lab125}
For 
$g_5:=\begin{partn}{4} i&i+1&i+2&i+3 \\ \hhline{-|~|-|~} i&i+1&i+2&i+3\end{partn}^{\!\!+}$
we have $a_{g_5}=a_t^2$.
\end{claim}

\begin{proof}
This time we have the square $\smat{g_5}{g_2}{g_3}{g_4}$,
where $g_2$, $g_3$ and $g_4$ were defined in Claims \ref{cl:lab122}--\ref{cl:lab124}.
The square is LR-singularised by 
$u=\begin{partn}{4} i&i+1&i+2&i+3\\ \hhline{~|~|-|~} i&i+1&&i+2,i+3 \end{partn}^{\!\!+}$.
Substituting the relations obtained in Claims \ref{cl:lab122}--\ref{cl:lab124}
into the square relation
$a_{g_5}^{-1}a_{g_2}=a_{g_3}^{-1}a_{g_4}$ yields $a_{g_5}=a_t^2$.
\end{proof}

\begin{claim}
\label{cl:lab126}
For 
$g_6:=\begin{partn}{3} i,i+2&i+1&i+3 \\ \hhline{-|~|~} i&i+1&i+2,i+3\end{partn}^{\!\!+}$
we have $a_{g_6}=1$.
\end{claim}

\begin{proof}
This time we have the square $\smat{g_6}{g_1}{g_2}{g_5}$,
where $g_1$, $g_2$ and $g_5$ were defined in Claims \ref{cl:lab121}, \ref{cl:lab122} and~\ref{cl:lab125}.
The square is UD-singularised by 
$u=\begin{partn}{4} i&i+1&i+2&i+3\\ \hhline{-|~|~|~} i&i+1&i+2&i+3\end{partn}^{\!\!+}$.
Substituting Claims \ref{cl:lab121}, \ref{cl:lab122} and \ref{cl:lab125}
into the square relation
$a_{g_6}^{-1}a_{g_1}=a_{g_2}^{-1}a_{g_5}$ yields $a_{g_6}=1$.
\end{proof}

\begin{claim}
\label{cl:lab127}
We have $a_e^{-1}=a_f$.
\end{claim}

\begin{proof}
This time we have the square $\smat e{h_1}{g_6}f$,
where
$h_1:=\begin{partn}{2} i,i+2,i+3&i+1 \\ \hhline{~|~} i,i+2,i+3&i+1\end{partn}^{\!\!+}$.
The square is LR-singularised by $u = \begin{partn}{3} i,i+2&i+1&i+3\\ \hhline{~|~|~} i,i+2&i+1&i+3\end{partn}^{\!\!+}$.
Since  $h_1\in P_0$ we have $a_{h_1}=1$ by Lemma \ref{la:P0nr1}, and
also $a_{g_6}=1$ by Claim~\ref{cl:lab126}.
The desired relation follows from the square relation $a_e^{-1}a_{h_1} = a_{g_6}^{-1}a_f$.
\end{proof}

To complete the proof of the lemma,
 let $e'\in E(n,r)$ be any idempotent full transformation with label $(i ,  i+1)$.
By Lemma \ref{la:eqlabel}\ref{it:eql1} we have $a_e=a_{e'}$. On the other hand, by Lemma \ref{la:GRtransl}\ref{it:GRt3} we have $a_{e'}^2=1$, and so $a_e^2=1$. Combining this with Claim \ref{cl:lab127} we deduce $a_e=a_e^{-1}=a_f$, as required.
\end{proof}
\setcounter{claim}{0}

\subsection{Commuting relations}
\label{ss:comm}

We now start moving towards identifying the group $\Z\times\S_r$.
It will turn out that the generators $a_e$ corresponding to full-domain or full-codomain idempotents $e$ will generate the symmetric group~$\S_r$, whereas those corresponding to $e\in P_1(n,r)$ -- which are mutually equal by Lemma~\ref{la:P1nr} -- will generate $\Z$. The key point in concluding that we indeed have a direct product is to establish commutation relations between the two families of generators.

\begin{lemma}
\label{la:peep}
Suppose $1\leq r\leq n-2$, and let $s\in P_0(n,r)$ be arbitrary.  Let $e\in E(n,r)$ be any idempotent with full domain and label $(i, i+1)$ for some $1\leq i\leq r-1$, and let $t\in P_1(n,r)$. Then $a_ea_t=a_ta_e$  in $\sfG(\labU\cup\{s\})$.
\end{lemma}

\begin{proof}
Again the result is vacuous for $r=1$, so we assume that $r\geq2$.  We also use the $g^+$ notation from the proof of Lemma \ref{la:lab12} for idempotents $g$ of $\P_{[i,i+3]}$.
By Lemma \ref{la:eqlabel} it is sufficient to show that $a_ea_t=a_ta_f$ holds for 
\[
e:= \begin{partn}{3}  i,i+3&&i+1,i+2\\ \hhline{~|-|~} i+3&i&i+1,i+2\end{partn}^{\!\!+} \ANd
f:= \begin{partn}{4} i,i+2,i+3&&i+1&\\ \hhline{~|-|~|-} i+3&i&i+1&i+2\end{partn}^{\!\!+}.
\]
Define some further idempotents:
\begin{alignat*}{4}
e_1&:= \begin{partn}{3} i,i+3&i+2&i+1 \\ \hhline{~|-|~} i+3&i&i+1,i+2\end{partn}^{\!\!+},\quad&
f_1&:= \begin{partn}{4} i,i+3&&i+1&i+2\\ \hhline{~|-|~|-} i+3&i&i+1&i+2\end{partn}^{\!\!+},\quad&
g_1&:= \begin{partn}{4} i&i+1&i+2&i+3\\ \hhline{-|~|-|~} i&i+1,i+2&&i+3\end{partn}^{\!\!+},
\\
e_2&:= \begin{partn}{3} i,i+3&i+1&i+2\\ \hhline{~|~|-} i,i+3&i+1,i+2&\end{partn}^{\!\!+},\quad&
f_2&:= \begin{partn}{3} i,i+3&i+1&i+2\\ \hhline{~|~|-} i,i+3&i+1&i+2\end{partn}^{\!\!+},\quad&
g_2&:= \begin{partn}{4} i&i+1&i+2&i+3\\ \hhline{-|~|-|~} i&i+1&i+2&i+3\end{partn}^{\!\!+}.
\\
e_3&:= \begin{partn}{2} i,i+3 & i+1,i+2\\ \hhline{~|~} i,i+3 & i+1,i+2\end{partn}^{\!\!+},\quad&
f_3&:= \begin{partn}{3} i,i+2,i+3&i+1&\\ \hhline{~|~|-} i,i+3&i+1&i+2\end{partn}^{\!\!+},\quad
&
\end{alignat*}
These idempotents participate in several singular squares, which will be exhibited below, and are illustrated in Figure \ref{fig:peep1}.

The square $\smat{e_1}{e_2}{e}{e_3}$ is UD-singularised by
$u = \begin{partn}{3} i&i+1,i+2&i+3\\ \hhline{~|~|~} i&i+1,i+2&i+3\end{partn}^{\!\!+}$.
In it, $e_3$ is a projection from~$P_0$, so $a_{e_3}=1$ by Lemma \ref{la:P0nr1};
and $e_2$ has full codomain and label $1$, so $a_{e_2}=1$ by Lemma \ref{la:eqlabeldual}.
Therefore, the square relation $a_{e_1}^{-1}a_{e_2}=a_e^{-1}a_{e_3}$ implies
\begin{equation}
\label{eq:peep1}
a_e=a_{e_1}.
\end{equation}

The square $\smat{f_1}{f_2}{f}{f_3}$ is UD-singularised
by 
$v=\begin{partn}{4} i&i+1&i+2,i+3&\\ \hhline{~|~|~|-} i&i+1&i+3&i+2\end{partn}^{\!\!+}$.
Note that $f_2\in P_1$, so by Lemma \ref{la:P1nr} we have $a_{f_2}=a_t$;
and $f_3$ has full domain and label $1$, so $a_{f_3}=1$ by Lemma~\ref{la:eqlabel}.
Combining this with the square relation $a_{f_1}^{-1}a_{f_2}=a_f^{-1}a_{f_3}$ yields 
\begin{equation}
\label{eq:peep2}
a_ta_f=a_{f_1}.
\end{equation}

Now consider $\smat{e_1}{f_1}{g_1}{g_2}$, which is LR-singularised by
$w=\begin{partn}{4} i&i+1&i+2&i+3\\ \hhline{~|~|-|~} i&i+1&i+2&i+3\end{partn}^{\!\!+}$.
In Claim \ref{cl:lab125} from the proof of Lemma \ref{la:lab12} we showed that $a_{g_2}=a_t^2$.
Also, we have $a_{g_1}=a_t$, which is shown analogously to Claim \ref{cl:lab123} in the proof of Lemma \ref{la:lab12}.  Consequently, $a_{g_1}^{-1}a_{g_2} = a_t$.
Combining this with the square relation $a_{e_1}^{-1}a_{f_1}=a_{g_1}^{-1}a_{g_2}$,
and with \eqref{eq:peep1} and \eqref{eq:peep2}, we have
\[
a_ea_t=a_ea_{g_1}^{-1}a_{g_2}=a_ea_{e_1}^{-1}a_{f_1}=a_ta_f,
\]
completing the proof.
\end{proof}

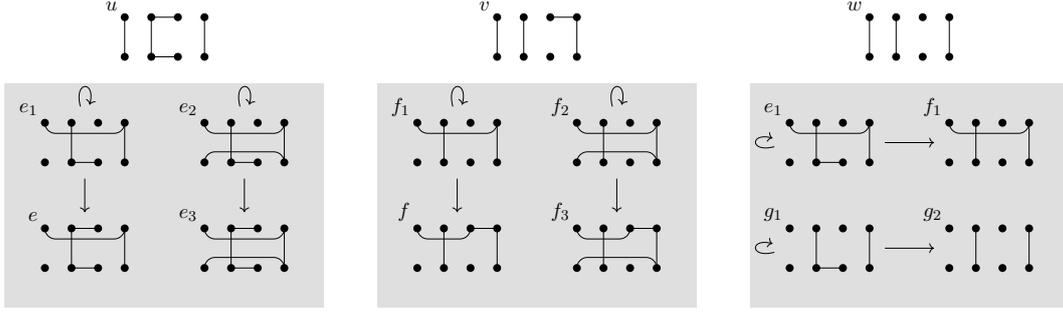
\begin{figure}[t!]
\begin{center}
\scalebox{0.7}{
\begin{tikzpicture}[scale=.5]
\begin{scope}[shift={(0,0)}]
\diagramshading4
\ediagram{e_1}{4}{\partitionedges{2/2,4/4}{}{1/4}{2/3}{}}
\fdiagram{e_2}{4}{\partitionedges{2/2,4/4}{}{1/4}{2/3}{1/4}}
\gdiagram{e}{4}{\partitionedges{2/2,4/4}{2/3}{1/4}{2/3}{}}
\hdiagram{e_3}{4}{\partitionedges{2/2,4/4}{2/3}{1/4}{2/3}{1/4}}
\udiagram{u}{4}{\partitionedges{1/1,2/2,4/4}{2/3}{}{2/3}{}}
\UDarrows
\end{scope}
\begin{scope}[shift={(14,0)}]
\diagramshading4
\ediagram{f_1}{4}{\partitionedges{4/4,2/2}{}{1/4}{}{}}
\fdiagram{f_2}{4}{\partitionedges{4/4,2/2}{}{1/4}{}{1/4}}
\gdiagram{f}{4}{\partitionedges{2/2,4/4}{3/4}{1/3}{}{}}
\hdiagram{f_3}{4}{\partitionedges{2/2,4/4}{3/4}{1/3}{}{1/4}}
\udiagram{v}{4}{\partitionedges{1/1,2/2,4/4}{3/4}{}{}{}}
\UDarrows
\end{scope}
\begin{scope}[shift={(28,0)}]
\diagramshading4
\ediagram{e_1}{4}{\partitionedges{2/2,4/4}{}{1/4}{2/3}{}}
\fdiagram{f_1}{4}{\partitionedges{4/4,2/2}{}{1/4}{}{}}
\gdiagram{g_1}{4}{\partitionedges{2/2,4/4}{}{}{2/3}{}}
\hdiagram{g_2}{4}{\partitionedges{2/2,4/4}{}{}{}{}}
\udiagram{w}{4}{\partitionedges{1/1,2/2,4/4}{}{}{}{}}
\LRarrows
\end{scope}
\end{tikzpicture}
}
\caption{The singular squares featuring in the proof of Lemma \ref{la:peep}.
In each partition, only the points $i,i+1,i+2,i+3$ and their dashed counterparts are shown (left to right).}
\label{fig:peep1}
\end{center}
\end{figure}

\begin{cor}
\label{cor:peep1}
Let $e\in E(n,r)$ be any non-projection idempotent with full domain, and let $t\in P_1$.  Then $a_ta_e=a_ea_t$ in $\sfG(\labU\cup\{s\})$.
\end{cor}

\begin{proof}
By Lemma \ref{la:fdtoTn}, we may assume without loss of generality that $e\in\T_n$.
Since $\labU\cup\{s\}$ contains $\Tlex$, 
it follows from Lemma \ref{la:GRtransl}\ref{it:GRt4} that $a_e=a_{e_1}\cdots a_{e_k}$, with each $e_i$ a Coxeter idempotent.
The result now follows from Lemma~\ref{la:peep}. 
\end{proof}

\subsection[The twisted partition monoid $\Ptw_n$]{The twisted partition monoid \boldmath{$\Ptw_n$}}
\label{ss:tw}

As we will soon see, the results obtained so far will allow us to conclude that our group $\GI$ is a homomorphic image of $\Z\times\S_r$.  We now provide a means to show that, conversely, $\GI$ is also a homomorphic \emph{preimage} of $\Z\times\S_r$.

Consider partitions $a,b\in\P_n$, and the product graph $\Ga(a,b)$.  A component of this graph is called \emph{floating} if it is contained entirely within $[n]''$, the middle row of vertices.   We let $\Phi(a,b)$ denote the number of such floating components.  For example, with $a,b\in\P_6$ from Figure \ref{fig:P6}, we have $\Phi(a,b) = 1$, with the unique floating component of $\Ga(a,b)$ being $\{1'',2'',6''\}$.

The ($\Z$-)\emph{twisted partition monoid} $\Ptw_n$ has underlying set $\Z\times\P_n$, and product
\[
(i,a)(j,b) = (i+j+\Phi(a,b),ab) \qquad\text{for $i,j\in\Z$ and $a,b\in\P_n$.}
\]
The monoid $\Ptw_n$ is a special case of a \emph{tight twisted product}, as studied in \cite{EGMR2}.  The analogous twisted product with underlying set $\mathbb N\times\P_n$ is more commonly studied; see \cite{ER22a,ER22b,FL11,LF2006,BDP2002,KV2023,ACHLV2015}.  The following is a list of facts about $\Ptw_n$, which are consequences of the stated results from \cite{EGMR2}: 
\bit
\item $\Ptw_n$ is regular \cite[Theorem 7.4]{EGMR2}.
\item The $\D$-classes of $\Ptw_n$ are the sets $\Z \times D(n,r)$, for $0\leq r\leq n$, and every maximal subgroup contained in $\Z\times D(n,r)$ is isomorphic to the direct product $\Z\times\S_r$
\cite[Corollary 5.2 and Theorem 6.2]{EGMR2}.
\item Writing $\ol e = (-\Phi(e,e),e)$ for $e\in\sfE(\P_n)$, we have $\sfE(\Ptw_n) = \set{\ol e}{e\in \sfE(\P_n)}$ \cite[Proposition~6.4 and Remark~6.5]{EGMR2}.
\item The map $e\mt\ol e$ determines an isomorphism of biordered sets $\sfE(\P_n) \to \sfE(\Ptw_n)$
\cite[Corollary~6.15 and Remark 6.16]{EGMR2}.
\item The idempotent-generated submonoid $\langle\sfE(\Ptw_n)\rangle$ consists of the identity element $(0,1)$, and all $\D$-classes $\Z\times D(n,r)$ for $0\leq r\leq n-1$
\cite[Theorem 8.7]{EGMR2}.
\eit
Consequently, we have a surmorphism $\IG(\sfE(\P_n)) \cong \IG(\sfE(\Ptw_n)) \to \langle\sfE(\Ptw_n)\rangle$ given by ${x_e\mt\ol e}$ for $e\in\sfE(\P_n)$.
Moreover, for any fixed $e_0\in \sfE(\P_n)$ of rank $0\leq r\leq n-1$, this surmorphism restricts to a surmorphism from $\GI$, the maximal subgroup of $\IG(\sfE(\P_n))$ containing $x_{e_0}$, to the maximal subgroup of $\Ptw_n$ containing $\ol e_0$, which is isomorphic to $\Z\times\S_r$.  We record this final conclusion here for future reference, and we note that this includes the case $r=0$.

\begin{prop}\label{prop:preimage}
For any $0\leq r\leq n-1$, the group $\GI$ is a homomorphic preimage of~${\Z\times\S_r}$.
\qed
\end{prop}

\subsection{Completing the proof}
\label{ss:theend}

Recall from Lemma \ref{la:Upres} that our maximal subgroup $\GI \cong \sfG(\labU\cup\{s\})$ is defined by the presentation
\[
\presn(\labU\cup\{s\})=\bigl\langle A(n,r)\mid \rels_1(\labU\cup \{s\}),\ \rels_{\subSq}(n,r)\bigr\rangle,
\]
where $s$ is a fixed projection from $P_0(n,r)$.

Fix any $t\in P_1$, and define
\[
E':=E(n,r)\cap\T_n = \{ e\in E(n,r)\colon \NTu(e)=0,\ \NTd(e)=n-r\}.
\]
Let~$\rels_{\subSq}' (n,r)$ be the set of square relations in
$\rels_{\subSq}(n,r)$ that
involve only idempotents from~$E'$, and let $\rels_{\subSq}'' (n,r):=\rels_{\subSq}(n,r)\setminus
\rels_{\subSq}'(n,r)$.
Adding the commuting relations established in Corollary~\ref{cor:peep1} for $e\in E'$ to the presentation $\presn(\labU\cup\{s\})$ yields the equivalent presentation:
\begin{equation}\label{eq:bigpres}
\bigl\langle A(n,r)\mid  \rels_1(\labU\cup \{s\}),\ \rels_{\subSq}'(n,r),\ 
\rels_{\subSq}'' (n,r),\ a_ea_t=a_ta_e\ (e\in E' )\bigr\rangle.
\end{equation}

Now let $A'(n,r):=\big\{ a_e\colon e\in E' \cup\{t\}\big\}$.  By Lemma \ref{la:geneli}, we have $\GI = \la A_F\ra$, where $F=F(n,r)$ is defined in \eqref{eq:Fnr}, and we again decompose $F = P_0\cup P_1\cup F_1\cup F_2$, where~$F_1$ and~$F_2$ are as in~\eqref{eq:F1F2}.  By Lemmas \ref{la:P1nr} and \ref{la:P0nr1}, we have $\GI = \la A_{F_1\cup F_2\cup\{t\}}\ra$.  Next, by Lemma~\ref{la:fdtoTn} and its dual, we obtain $\GI = \la A_{E'\cup E''\cup\{t\}}\ra$, where $E'':=\set{e^*}{e\in E'}$.  Combining the dual of Lemma \ref{la:GRtransl}\ref{it:GRt4} with Lemma \ref{la:lab12}, each letter from $A_{E''}$ is a product of letters from~$A_{E'}$.  So we conclude that $\GI = \la A_{E'\cup\{t\}}\ra = \la A'(n,r)\ra$.  
Thus, we can eliminate all the generators from $A(n,r)\setminus A'(n,r)$ in the presentation \eqref{eq:bigpres}.
Notice that the relations $\rels_{\subSq}'(n,r)$ are unaffected by this elimination, since they involve only generators from $A_{E'}\sub A'(n,r)$.

Recall that $\labU$ contains the tree $\Tlex$, and so
$\rels_1(\labU\cup \{s\})$ contains $\rels_1(\Tlex)$.
Hence we now see that $\GI$ is defined by the presentation
\[
\langle A'(n,r) \mid \rels_1(\Tlex),\ \rels_{\subSq}'(n,r),\ a_ea_t=a_ta_e\ (e\in E' ),\ \rels'''\rangle
\]
for some set of relations $\rels'''$.
By Proposition \ref{pr:treeTn}, the presentation
\[
\langle A'(n,r)\setminus\{t\} \mid \rels_1(\Tlex),\ \rels_{\subSq}'(n,r)\rangle
\]
defines the symmetric group $\S_r$.
Therefore $\GI$ is a homomorphic image of $\Z\times\S_r$.
But by Proposition \ref{prop:preimage}, $\Z\times\S_r$ is also a homomorphic image of $\GI$.  Since $\Z$ is abelian and Hopfian, and~$\S_r$ is finite, it follows by \cite[Corollary, p233]{Hi69} that $\Z\times\S_r$ is Hopfian.  Therefore $\GI\cong \Z\times\S_r$, as required.

\section{Maximal subgroups in \boldmath{$\IG(\sfE(\P_n))$}, the case of rank \boldmath{$0$}}
\label{sec:r0}

Having completed the proof of Theorem \ref{thm:mainIGPn} for $1\leq r\leq n-2$, we now deal with the final case, where $r=0$.  Here we cannot `reduce to $\T_n$', as $D(n,0)$ consists entirely of partitions of rank~$0$, and hence contains no transformations at all.

\begin{prop}
\label{pr:r0}
Let $n\geq 2$, and let $e_0$ be an idempotent of rank $0$ in $\P_n$.
Then the maximal subgroup of $\IG(\sfE(\P_n))$ containing $x_{e_0}$ is isomorphic to $\Z$.
\end{prop}

\begin{proof}
The proof is in many ways similar to, yet much easier than,  the proof for the $r\geq 1$ case, but with a couple of variations. We cover it rapidly,  but deal with the differences in full detail.

Note that \emph{all} the elements of $D(n,0)$ are idempotents, so we can pick a very straightforward spanning tree for $\GH(D(n,0))$:
\[
T=T(n,0):=\big\{ e\in E(n,0)\colon \ker(e)=\nabla_{[n]}\text{ or } \coker(e)=\nabla_{[n]}\big\}.
\]
We remark that every element of $E(n,0)$ must have at least one upper block and at least one lower block. Those that attain at least one of those lower bounds are precisely the elements of~$T$.

By Theorem \ref{thm:spantreepres}, the maximal subgroup $\GI$ is isomorphic to $\sfG(T)$,
i.e.~it is defined by the presentation 
\[
\presn(T)= \big\langle A(n,0)\colon \rels_1(T),\ \rels_\subSq(n,0)\big\rangle.
\]

We can eliminate the majority of the generators in $A(n,r)$ by using Lemma \ref{la:ehrsq}, and the following analogue of Lemma \ref{la:type7} concerning NT-reducing squares (cf.~Definition \ref{defn:NT}).

\begin{claim}
\label{cl:r01}
If $p=\begin{partn}{3} A&B&C\\ \hhline{-|-|-} A&B&C\end{partn}\oplus q\in P(n,0)$, then $\smat{e_1}{e_2}{e_3}{p}$ is an NT-reducing square with base $p$, where
\[
e_1:=\begin{partn}{2} A\cup C&B\\ \hhline{-|-} A\cup B&C\end{partn}\oplus q,\quad
e_2:=\begin{partn}{3} \multicolumn{2}{c|}{A\cup C}&B\\ \hhline{-|-|-} A&B&C\end{partn}\oplus q \ANd
e_3:=\begin{partn}{3} A&B&C\\ \hhline{-|-|-} \multicolumn{2}{c|}{A\cup B}&C\end{partn}\oplus q.
\] 
\end{claim}

\begin{proof}
An RL-singularising element is $u=\begin{partn}{3} A& B&C \\ \hhline{~|-|~} A\cup B&&C\end{partn}\oplus q$; see Figure \ref{fig:r01}.
\end{proof}

\begin{figure}[t!]
\begin{center}
\scalebox{0.9}{
\begin{tikzpicture}[scale=.7]

\nc\BBB{1.5}
\nc\AAA{1}
\nc\CCC{2}

\nc\DDD{0.6}

\pgfmathsetmacro\BL{1}
\pgfmathsetmacro\BR{1+\BBB}
\pgfmathsetmacro\AL{1+\BBB+\DDD}
\pgfmathsetmacro\AR{1+\AAA+\BBB+\DDD}
\pgfmathsetmacro\CL{1+\AAA+\BBB+2*\DDD}
\pgfmathsetmacro\CR{1+\AAA+\BBB+\CCC+2*\DDD}

\nc\yyy{4.2}

\nc\XL{-.5}
\nc\XR{19}
\nc\YL{-1}
\nc\YR{7}
\fill[lightgray!30] (\XL,\YL)--(\XR,\YL)--(\XR,\YR)--(\XL,\YR)--(\XL,\YL);

\begin{scope}[shift={(0,\yyy)}]
\brokenbluetrans\AL\CR\BL\AR{A\cup C}{A\cup B}
\blueupper\BL\BR{B}
\bluelower\CL\CR{C}
\draw[<-] (\CR+1.5,.75)--(\CR+3.5,.75);
\bluepartlabel{e_1}
\end{scope}

\begin{scope}[shift={(\CR+4,\yyy)}]
\brokenbluetrans\AL\CR\AL\AR{A\cup C}{A}
\blueupper\BL\BR{B}
\bluelower\BL\BR{B}
\bluelower\CL\CR{C}
\bluepartlabel{e_2}
\node (R) at (\CR,.75) {\phantom{\Large$E$}};
\draw[->] (R) edge [loop right] ();
\end{scope}

\begin{scope}[shift={(0,0)}]	
\brokenbluetrans\AL\AR\BL\AR{A}{A\cup B}
\blueupper\BL\BR{B}
\blueupper\CL\CR{C}
\bluelower\CL\CR{C}
\draw[<-] (\CR+1.5,.75)--(\CR+3.5,.75);
\bluepartlabel{e_3}
\end{scope}

\begin{scope}[shift={(\CR+4,0)}]	
\brokenbluetrans\AL\AR\AL\AR{A}{A}
\blueupper\BL\BR{B}
\blueupper\CL\CR{C}
\bluelower\BL\BR{B}
\bluelower\CL\CR{C}
\bluepartlabel{p}
\node (R) at (\CR,.75) {\phantom{\Large$E$}};
\draw[->] (R) edge [loop right] ();
\end{scope}

\begin{scope}[shift={(\AR/2+3,2*\yyy)}]
\bluetrans\AL\AR\BL\AR{A}{A\cup B}
\blueupper\BL\BR{B}
\bluetrans\CL\CR\CL\CR{C}{C}
\bluepartlabel{u}
\end{scope}

\end{tikzpicture}
}
\caption{The NT-reducing singular square featuring in  Claim \ref{cl:r01} from the proof of Proposition~\ref{pr:r0}; in each partition, the `${}\oplus q$' part has been omitted.}
\label{fig:r01}
\end{center}
\end{figure}
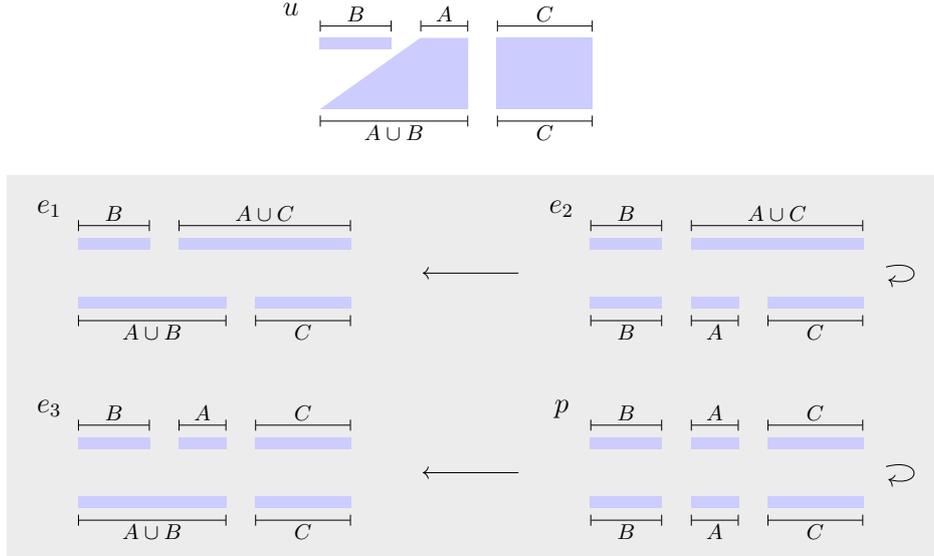

We now define $F = F(n,0) = T\cup P_2$, but we note that this is different from the set denoted~$F$ in Section \ref{sec:maxIGPn}.

\begin{claim}\label{cl:1.5}
Every element of $E(n,0) \setminus F(n,0)$ is the base of an NT-reducing singular square.
\end{claim}

\begin{proof}
Let $e\in E(n,0) \setminus F(n,0)$.  If $e$ is a projection, then Claim \ref{cl:r01} applies.  So now suppose~$e$ is not a projection, and assume without loss of generality that $\ker(e) \not\sub \coker(e)$.  Since $e$ does not belong to $T$, it has at least two upper blocks, say $A$ and $B$.  Now let $f\in E(n,0)$ be obtained from $e$ by merging these into the single block $A\cup B$.  Then $e\mr\L f$ (cf.~Lemma \ref{la:Green_Pn}) and $\ker(e) \subsetneq\ker(f)$.  Thus, Lemma \ref{la:ehrsq} applies.
\end{proof}

An inductive argument as in Lemma \ref{la:geneli} now gives $\GI = \la A_F\ra$.  Recalling that $a_e=1$ for $e\in T$, it follows that all the generators from $A(n,0)$ can be eliminated except those in $A_{P_2} = \{a_e\colon e\in P_2\}$.

\begin{claim}
\label{cl:r0pqP2}
If $p,q\in P_2(n,0)$ then $a_p=a_q$  in $\sfG(T)$.
\end{claim}

\begin{proof}
Analogously to Lemma \ref{la:P1nr}, we prove this by induction on $n\geq2(=r+2)$.
The first case is $n=2$, when $|P_2|=1$, and there is nothing to prove.
However, unlike the proof of Lemma~\ref{la:P1nr}, here we must also check the case $n=3$, which we do below. The singular squares used in the proof are illustrated in Figure \ref{fig:r01a}.  Although each square is singular by Lemma \ref{la:ehrsq} or Claim~\ref{cl:r01}, we still picture a singularising element for convenience.

For distinct $i,j\in [3]$, define
$\sfp(i,j):=\begin{partn}{2} i,j&k\\ \hhline{-|-} i,j&k\end{partn}$,
where $\{i,j,k\}=[3]$.
Then $\sfp(i,j)=\sfp(j,i)$ and $P_2=\{\sfp(1,2),\sfp(1,3),\sfp(2,3)\}$.

By Lemma \ref{la:ehrsq} we have a singular square $\smat{f_1}{f_2}{\sfp(i,j)}{e_1}$, where
\[
f_1:=\begin{partn}{2} \multicolumn{2}{c}{i,j,k}\\ \hhline{-|-} i,j&k\end{partn},\quad
f_2:=\begin{partn}{3} \multicolumn{3}{c}{i,j,k}\\ \hhline{-|-|-} i&j&k\end{partn} \ANd
e_1:=\begin{partn}{3} \multicolumn{2}{c|}{i,j} & k\\ \hhline{-|-|-} i&j&k\end{partn}.
\]
Furthermore $f_1,f_2\in T$, so $a_{f_1}=a_{f_2}=1$. The square relation now yields
\begin{equation}
\label{eq:r01}
a_{\sfp(i,j)}=a_{e_1}\quad \text{ for } e_1= \begin{partn}{3} \multicolumn{2}{c|}{i,j} & k\\ \hhline{-|-|-} i&j&k\end{partn}.
\end{equation}
A dual argument shows that
\begin{equation}
\label{eq:r02}
a_{\sfp(j,k)}=a_{e_2}\quad \text{ for } e_2=\begin{partn}{3} i&j&k\\ \hhline{-|-|-} i& \multicolumn{2}{c}{j,k}\end{partn}.
\end{equation}
Lemma \ref{la:ehrsq} again gives us a singular square $\smat{g_1}{g_2}{g_3}{e_3}$, where
\[
g_1:=\begin{partn}{1} i,j,k\\ \hhline{-} i,j,k\end{partn},\quad
g_2:=\begin{partn}{2} \multicolumn{2}{c}{i,j,k}\\ \hhline{-|-} i&j,k\end{partn},\quad
g_3:=\begin{partn}{2} i,j&k\\ \hhline{-|-} \multicolumn{2}{c}{i,j,k}\end{partn}\ANd
e_3:=\begin{partn}{2} i,j&k\\ \hhline{-|-} i&j,k\end{partn}.
\]
Since $g_1,g_2,g_3\in T$, we have $a_{g_1}=a_{g_2}=a_{g_3}=1$, and the square relation yields
\begin{equation}
\label{eq:r03}
a_{e_3}=1\quad \text{ for } e_3=\begin{partn}{2} i,j&k\\ \hhline{-|-} i&j,k\end{partn}.
\end{equation}
Finally, let $z:=\begin{partn}{3} i&j&k\\ \hhline{-|-|-} i&j&k\end{partn}
=\begin{partn}{3} 1&2&3\\ \hhline{-|-|-} 1&2&3\end{partn}$.
The square $\smat{e_3}{e_1}{e_2}{z}$ is singular by Claim \ref{cl:r01}.
Substituting \eqref{eq:r01}--\eqref{eq:r03} into the resulting square relation $a_{e_3}^{-1}a_{e_1} = a_{e_2}^{-1}a_z$ yields
\[
a_{\sfp(j,k)}a_{\sfp(i,j)}=a_z \quad \text{ for } \{i,j,k\}=\{1,2,3\}.
\]
Using this twice we have
\[
a_{\sfp(j,k)}a_{\sfp(i,j)}=a_z=a_{\sfp(k,j)}a_{\sfp(i,k)}=a_{\sfp(j,k)}a_{\sfp(i,k)},
\]
from which it follows that $a_{\sfp(i,j)}=a_{\sfp(i,k)}$ if $\{i,j,k\}=\{1,2,3\}$.
And now
\[
a_{\sfp(1,2)}=a_{\sfp(1,3)}=a_{\sfp(3,1)}=a_{\sfp(3,2)}=a_{\sfp(2,3)}.
\]
This completes the proof of Claim \ref{cl:r0pqP2} when $n=3$.

For $n\geq 4$ the inductive step now proceeds as in the proof of Lemma \ref{la:P1nr}.
Let $p,q\in P_2$ be arbitrary.
Recall the copies $\P_n^{i,j}$ of $\P_{n-1}$ in $\P_n$, for arbitrary distinct $i,j\in [n]$.
Note that our spanning tree $T(n,0)$ contains a natural copy of the tree $T(n-1,0)$ corresponding to each~$\P_n^{i,j}$. 
Thus, if $p$ and $q$ belong to the same $\P_n^{i,j}$ we obtain $a_p=a_q$ by induction.
Otherwise, we let $(i,j)\in\ker(p)$ and $(k,l)\in \ker(q)$ be arbitrary with $i\neq j$ and $k\neq l$.
Since $n\geq 4$, there exists~${u\in P_2}$ such that $(i,j),(k,l)\in\ker(u)$.
But then $a_p=a_u=a_q$, and Claim \ref{cl:r0pqP2} is proved.
\end{proof}

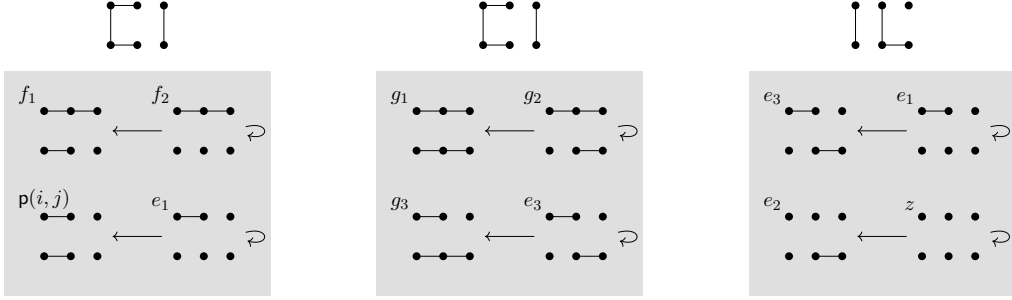
\begin{figure}[t!]
\begin{center}
\scalebox{0.7}{
\begin{tikzpicture}[scale=.5]
\begin{scope}[shift={(0,0)}]
\diagramshading3
\ediagram{f_1}{3}{\partitionedges{}{1/3}{}{1/2}{}}
\fdiagram{f_2}{3}{\partitionedges{}{1/3}{}{}{}}
\gdiagram{}{3}{\partitionedges{}{1/2}{}{1/2}{} \node[above] () at (1,1.5) {$\sfp(i,j)$};}
\hdiagram{e_1}{3}{\partitionedges{}{1/2}{}{}{}}
\udiagram{}{3}{\partitionedges{1/1,3/3}{1/2}{}{1/2}{}}
\RLarrows
\end{scope}
\begin{scope}[shift={(14,0)}]
\diagramshading3
\ediagram{g_1}{3}{\partitionedges{}{1/3}{}{1/3}{}}
\fdiagram{g_2}{3}{\partitionedges{}{1/3}{}{2/3}{}}
\gdiagram{g_3}{3}{\partitionedges{}{1/2}{}{1/3}{}}
\hdiagram{e_3}{3}{\partitionedges{}{1/2}{}{2/3}{}}
\udiagram{}{3}{\partitionedges{1/1,3/3}{1/2}{}{1/2}{}}
\RLarrows
\end{scope}
\begin{scope}[shift={(28,0)}]
\diagramshading3
\ediagram{e_3}{3}{\partitionedges{}{1/2}{}{2/3}{}}
\fdiagram{e_1}{3}{\partitionedges{}{1/2}{}{}{}}
\gdiagram{e_2}{3}{\partitionedges{}{}{}{2/3}{}}
\hdiagram{z}{3}{\partitionedges{}{}{}{}{}}
\udiagram{}{3}{\partitionedges{1/1,2/2}{}{}{2/3}{}}
\RLarrows
\end{scope}
\end{tikzpicture}
}
\caption{The  singular squares featuring in the proof of 
Claim \ref{cl:r0pqP2} in Proposition \ref{pr:r0}.
In each partition, only the points $i,j,k$ and their dashed counterparts are shown (left to right), where $\{i,j,k\}=\{1,2,3\}$.}
\label{fig:r01a}
\end{center}
\end{figure}

We can now complete the proof of the proposition.
We have seen that we can eliminate all the generators, except for a single $a_t$, for $t\in P_2$.
Thus $\GI \cong \sfG(T)$ is cyclic. But $\GI$ is infinite by Proposition \ref{prop:preimage}, so it follows that indeed $\GI\cong\Z$.
\end{proof}
\setcounter{claim}{0}

\begin{rem}\label{rem:eggbox8}
Figure \ref{fig:eggbox8} shows the location in $D(3,0)$ of the idempotents from the singular squares $\smat{f_1}{f_2}{\sfp(i,j)}{e_1}$, $\smat{g_1}{g_2}{g_3}{e_3}$ and $\smat{e_3}{e_1}{e_2}{z}$ used in the proof of Claim \ref{cl:r0pqP2} in Proposition \ref{pr:r0}.  Note that since $r=0$, the $\NTu$ and $\NTd$ parameters range from $1$ to $3$.  Note also that the strip $D^1(3,0)$ is a single $\R$-class, and $D_1(3,0)$ is a single $\L$-class; the unions of these strips is the spanning tree~$T(3,0)$.
\end{rem}

\begin{figure}[t]
\begin{center}
\scalebox{.7}{
\begin{tikzpicture}[scale=.8]

\begin{scope}[shift={(18,0)}]
\fill[gray!10](2,0)--(12,0)--(12,10)--(2,10)--(2,0);
\foreach \x in {2,4,8,12} {\draw[ultra thick] (\x,0)--(\x,10);}
\foreach \x in {2,6} {\draw[ultra thick] (2,10-\x)--(12,10-\x);}
\draw[ultra thick](2,0)--(12,0)--(12,10)--(2,10)--(2,0)--(12,0);
\draw[thick, dashed] (0+2,12-2)--(-2+2,14-2);
\node () at (-.6-1+2,12.22+.5-2) { $\NTu$};
\node () at (-.22-.5+2,12.6+1-2) { $\NTd$};
\foreach \x in {2,4,8,12} {\draw[thick, dashed] (\x,12)--(\x,10) (0,12-\x)--(2,12-\x);}

\node () at (2+1,13-2) { $1$};
\node () at (6,13-2) { $2$};
\node () at (10,13-2) { $3$};
\node () at (-1+2,10-1) { $1$};
\node () at (-1+2,6) { $2$};
\node () at (-1+2,2) { $3$};
\node (p) at (5,6) {\Large$\sfp(i,j)$};
\node (e3) at (7,6) {\Large$e_3$};
\node (e1) at (10,6) {\Large$e_1$};
\node (g3) at (3,6) {\Large$g_3$};
\node (g1) at (3,9) {\Large$g_1$};
\node (f1) at (5,9) {\Large$f_1$};
\node (g2) at (7,9) {\Large$g_2$};
\node (f2) at (10,9) {\Large$f_2$};
\node (e2) at (7,2) {\Large$e_2$};
\node (z) at (10,2) {\Large$z$};

\draw[ultra thick, red] (e3)--(e1)--(z)--(e2)--(e3);
\draw[ultra thick, green] (f1)--(p) (f2)--(e1);
\draw[ultra thick, green] (e1) to[bend right=25] (p);
\draw[ultra thick, green] (f2) to[bend left=25] (f1);
\draw[ultra thick, blue] (e3)--(g2) (g1)--(g3);
\draw[ultra thick, blue] (g2) to[bend right=25] (g1);
\draw[ultra thick, blue] (g3) to[bend right=25] (e3);
\end{scope}

\end{tikzpicture}
}
\caption{Stratified egg-box diagram of the idempotents in the singular squares featuring in the proof of Claim \ref{cl:r0pqP2} from Proposition \ref{pr:r0}.  See Remark \ref{rem:eggbox8} for more details.}
\label{fig:eggbox8}
\end{center}
\end{figure}
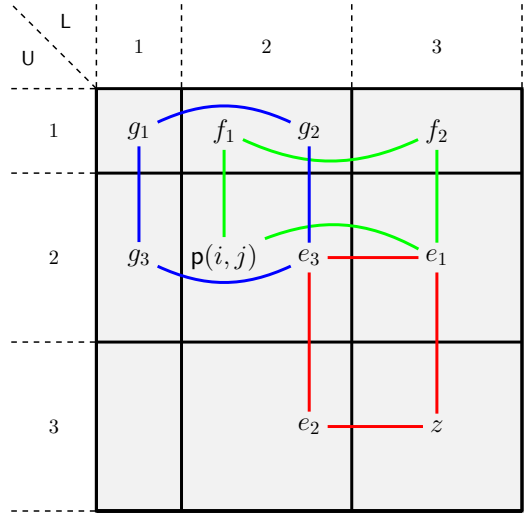

\section{Further applications, concluding remarks and open problems}
\label{sec:conc}

Now that we have completed the proofs of our results, we once again take a broader view, and discuss some implications. In particular we suggest some natural questions, the answers to which would advance our understanding of the free objects that we have been considering -- $\PG(P)$, $\IG(E)$ and $\RIG(E)$ -- and their mutual relationships.

For a projection-generated regular $*$-semigroup $S$ with $P=\sfP(S)$,
we saw in Subsection \ref{ss:PGP} that the Green's $\R$-, $\L$- and $\D$-structure of $\PG(P)$ is the same as the corresponding structure  for $S$.
So, once we also understand the maximal subgroups of $\PG(P)$, its structure is completely determined. 
Theorem \ref{thm:mainPGPn} and Corollary \ref{cor:PGPnn-1} complete this programme when $S$ is the partition monoid $\P_n$.
These results can be interpreted as saying that $\PG(P)$ is  `almost equal' to $\P_n$, in the 
sense that the former
 is obtained from the latter by only changing groups in the top two $\D$-classes 
(to trivial or free  as appropriate), and leaving the rest of $\P_n$ intact.
From this
it follows quickly that $\PG(P)$ has a solvable word problem (see \cite[Remark 4.5]{Ru99}), 
and in fact even admits a presentation by a finite complete rewriting system (using \cite[Theorem 1]{GM11}), and also that it is residually finite (using \cite{Go75}).

We hope that the results of this paper will initiate a systematic study of maximal subgroups of the semigroups $\PG(P)$ where $P$ is the projection algebra of other natural regular $*$-semigroups. We recall that for the Temperley--Lieb monoid we have
$\PG(\sfP(\TL_n))\cong \TL_n$ \cite[Theorem 9.1]{EGMR}, but the following are open:

\begin{prob}
Determine the maximal subgroups of the semigroups $\PG(\sfP(S))$, where $S$ is a diagram monoid such as the Brauer monoid $\B_n$ or Motzkin monoid $\mathcal{M}_n$.
\end{prob}

One difficulty that immediately presents itself in the case of  $\B_n$ is that there is no natural copy of $\T_n$ to work towards, as we did in this paper.
As for $\mathcal{M}_n$, it is known that all of its maximal subgroups are trivial. 
However in \cite[Example 10.12]{EGMR} it is shown that the maximal subgroups of
$\PG(\sfP(\mathcal{M}_4))$ corresponding to idempotents of rank $1$ are infinite.

Like $\PG(\sfP(\P_n))$,  the free idempotent-generated regular semigroup $\RIG(\sfE(\P_n))$
is entirely determined by its maximal subgroups. 
The latter are identical to those of $\IG(\sfE(\P_n))$ by \ref{it:RIG3}, and hence determined by Theorem \ref{thm:mainIGPn} and Remark \ref{rem:noss}.
Thus, as above for $\PG(\sfP(\P_n))$, we quickly deduce that
 $\RIG(\sfE(\P_n))$ has a solvable word problem, a finite complete rewriting system and is residually finite.
So, in the same way as $\PG(\sfP(\P_n))$ was `almost equal' to $\P_n$, here $\RIG(\sfE(\P_n))$ is `almost equal' to $\P_n^\Phi$.
This seems quite remarkable, in that it implies that the twisting map $\Phi$, which counts 
the number of floating components in the product graphs of pairs of partitions, is somehow intrinsically `encoded' in the idempotent structure of the partition monoid $\P_n$, whose multiplication expressly `forgets' this information.
We believe that it would be enlightening to find an intuitive explanation/interpretation of this observation.

For the free idempotent-generated semigroup
$\IG(\sfE(\P_n))$ the knowledge of its maximal subgroups only entails precise understanding of 
the structure of its regular $\D$-classes.
As already pointed out, $\IG(\sfE(\P_n))$ does have non-regular $\D$-classes, and therefore its overall structure remains opaque.
In particular, it does not readily follow that it possesses any of the above nice properties.

\begin{prob}
Does $\IG(\sfE(\P_n))$ have a solvable word problem? 
Does it have a finite complete rewriting system? Is it residually finite?
\end{prob}

The word problem in free idempotent-generated semigroups in general is the subject of \cite{DG17,DD19,Do21}, but none of the results there give an immediate answer to the above questions.

On a more general level, we may ask which groups arise as maximal subgroups of different free objects. 
For example, it is known that every group arises as a maximal subgroup of some $\IG(E)$
\cite{GR12IJM,DR13,DD15}
 and hence also of some $\RIG(E)$. 
 
 \begin{prob}
Is it true that every group is the maximal subgroup of some free projection-generated regular $*$-semigroup $\PG(P)$?
\end{prob}

None of the constructions deployed in \cite{GR12IJM,DR13,DD15} yield regular $*$-semigroups, so the following seems of interest:

\begin{prob}
Which groups arise as maximal subgroups of $\IG(E)$, where $E$ is the biordered set of a regular $*$-semigroup?
\end{prob}

Finally, one could consider maximal subgroups of $\PG(P)$ and $\IG(E)$ simultaneously:

\begin{prob}
Which pairs of groups arise as maximal subgroups of $\PG(P)$ and $\IG(E)$ arising from the same regular $*$-semigroup, and corresponding to the same idempotent?
\end{prob}

From general theory, we know that  the maximal subgroups of $\PG(P)$ are quotients of their 
$\IG(E)$ counterparts, but at this stage it is not clear to us whether there may be further constraints.

\footnotesize
\def\bibspacing{-1.1pt}
\bibliography{proj}
\bibliographystyle{abbrv}

\end{document}